\documentclass{article}
\usepackage{amsthm}
\usepackage{amsmath}
\usepackage{amssymb}
\usepackage{graphicx}
\usepackage{tikz}
\usepackage{cases}
\usepackage{enumitem} 
\usepackage{hyperref}
\hypersetup{colorlinks,linkcolor={red},citecolor={darkgreen},urlcolor={red}} 
\definecolor{darkgreen}{rgb}{0.1,0.5,0.1}

\usetikzlibrary{arrows,positioning,fit,shapes,calc}

\newtheorem{theorem}{Theorem}
\newtheorem{lemma}[theorem]{Lemma}

\newtheorem{proposition}[theorem]{Proposition}
\newtheorem{corollary}[theorem]{Corollary}
\newtheorem{remark}[theorem]{Remark}

\DeclareMathAlphabet{\bi}{OML}{cmm}{b}{it}

\DeclareMathAlphabet\bfcal{OMS}{cmsy}{b}{n} 

\newcommand{\Imag}{\operatorname{Im}}

\newcommand{\bE}{\boldsymbol{E}}
\newcommand{\bH}{\boldsymbol{H}}

\newcommand{\bG}{\boldsymbol{G}}

\newcommand{\bJ}{\boldsymbol{J}}
\newcommand{\bK}{\boldsymbol{K}}
\newcommand{\bx}{\boldsymbol{x}}
\newcommand{\bU}{\boldsymbol{U}}
\newcommand{\bV}{\boldsymbol{V}}

\newcommand{\bu}{\boldsymbol{u}}

\newcommand{\calH}{\mathcal{H}}

\newcommand{\calO}{\mathcal{O}}
\newcommand{\calW}{\mathcal{W}}
\newcommand{\calZ}{\mathcal{Z}}

\newcommand{\bbA}{\mathbb{A}}

\newcommand{\bbC}{\mathbb{C}}
\newcommand{\bbD}{\mathbb{D}}
\newcommand{\bbE}{\mathbb{E}}
\newcommand{\bbF}{\mathbb{F}}

\newcommand{\bbI}{\mathbb{I}}

\newcommand{\bbM}{\mathbb{M}}
\newcommand{\bbP}{\mathbb{P}}
\newcommand{\bbR}{\mathbb{R}}

\newcommand{\bbV}{\mathbb{V}}
\newcommand{\bbW}{\mathbb{W}}

\newcommand{\Om}{\Omega_{\rm m}}
\newcommand{\Oe}{\Omega_{\rm e}}
\newcommand{\Op}{\Omega_{\rm p}}
\newcommand{\Oc}{\Omega_{\rm c}}
\newcommand{\kc}{k_{\rm c}}

\newcommand{\bhatU}{\boldsymbol{\hat{U}}}

\newcommand{\curl}{\operatorname{curl}}
\newcommand{\bcurl}{\operatorname{\bf curl}}

\def\div{\operatorname{div}}

\newcommand{\sgn}{\operatorname{sgn}}

\def\ker{\mathrm{Ker}}

\newcommand{\hatH}{\boldsymbol{\hat{\calH}}}

\newcommand{\Thetakl}{\Theta_{k,\lambda}}

\newcommand{\thetakl}{\theta_{k,\lambda}}
\newcommand{\thetaklp}{\theta_{k,\lambda'}}

\def\Rop{\mathrm{R}}
\def\bR{\boldsymbol{\rm{R}}}
\newcommand{\bPi}{\boldsymbol{\Pi}}

\newcommand{\kE}{k_{\scE}}
\newcommand{\kO}{k_{0}}
\newcommand{\kD}{k_{\scD}}
\newcommand{\kI}{k_{\scI}}
\newcommand{\kp}{k^{+}}

\newcommand{\eps}{\varepsilon}

\newcommand{\rmd}{{\mathrm{d}}}
\newcommand{\rmD}{{\mathrm{D}}}
\newcommand{\rme}{{\rm e}}
\newcommand{\rmi}{{\rm i}}

\newcommand{\Hxy}{\boldsymbol{\mathcal{H}}}
\newcommand{\Hxydiv}{\boldsymbol{\mathcal{H}}_{\rm div0}}

\newcommand{\Hps}{\boldsymbol{\mathcal{H}}_{s}}
\newcommand{\Hms}{\boldsymbol{\mathcal{H}}_{-s}}
\newcommand{\Xps}{\boldsymbol{\mathcal{X}_{s}}}
\newcommand{\Xpsstar}{\boldsymbol{\mathcal{X}_s^*}}

\def\scD{\textsc{d}}
\def\scE{\textsc{e}}
\def\scI{\textsc{i}}
\def\scZ{\textsc{z}}
\def\DD{\textsc{dd}}
\def\DE{\textsc{de}}
\def\EI{\textsc{ei}}
\def\DI{\textsc{di}}
\def\EE{\textsc{ee}}
\def\zDD{\Lambda_\DD}
\def\zDE{\Lambda_\DE}
\def\zEI{\Lambda_\EI}
\def\zDI{\Lambda_\DI}
\def\zEE{\Lambda_\EE}
\def\zZ{\Lambda_\scZ}

\def\JacE{\mathcal{J}_{\scE}}

\def\Xint#1{\mathchoice
	{\XXint\displaystyle\textstyle{#1}}%
	{\XXint\textstyle\scriptstyle{#1}}%
	{\XXint\scriptstyle\scriptscriptstyle{#1}}%
	{\XXint\scriptscriptstyle\scriptscriptstyle{#1}}%
	\!\int}
\def\XXint#1#2#3{{\setbox0=\hbox{$#1{#2#3}{\int}$}
		\vcenter{\hbox{$#2#3$}}\kern-.5\wd0}}

\def\dashint{\Xint-}

\DeclareMathOperator*{\slim}{s-lim}

\newcommand{\thetaklt}{\theta_{k,\tilde{\lambda}}}

\newcommand{\bbml}{\mathbb{M}_{\lambda}}
\newcommand{\bbmlp}{\mathbb{M}_{\lambda'}}

\begin{document}
\title{Spectral theory for Maxwell's equations at the interface of a metamaterial. Part II: Limiting absorption, limiting amplitude principles and interface resonance}
\author{Maxence Cassier$^{a}$, Christophe Hazard$^{b}$ and Patrick Joly$^{b}$ \\ \ \\
{
\footnotesize $^a$ Aix Marseille Univ, CNRS, Centrale Marseille, Institut Fresnel, Marseille, France }\\ 
{\footnotesize $^b$ POEMS$^1$, CNRS, INRIA, ENSTA Paris, Institut Polytechnique de Paris, 91120 Palaiseau, France}\\ 
{\footnotesize (maxence.cassier@fresnel.fr, christophe.hazard@ensta-paris.fr, patrick.joly@inria.fr)}}
\footnotetext[1]{POEMS (Propagation d'Ondes: Etude Math\'ematique et Simulation) is a mixed research team (UMR 7231) between CNRS (Centre National de la Recherche Scientifique), ENSTA Paris (Ecole Nationale Sup\'erieure de Techniques Avanc\'ees) and INRIA (Institut National de Recherche en Informatique et en Automatique).}
\footnotetext[2]{The work of Maxence Cassier was supported in part by Simons Foundation grant $\sharp 376319$ (Michael I. Weinstein).}

\maketitle

\begin{abstract}
This paper is concerned with the time-dependent Maxwell's equations for a plane interface between a negative material described by the Drude model and the vacuum, which fill, respectively, two complementary half-spaces. In a first paper, we have constructed a generalized Fourier transform which diagonalizes the Hamiltonian that represents the propagation of transverse electric waves. In this second paper, we use this transform to prove the limiting absorption and limiting amplitude principles, which concern, respectively, the behavior of the resolvent near the continuous spectrum and the long time response of the medium to a time-harmonic source of prescribed frequency. This paper also underlines the existence of an interface resonance which occurs when  there exists a  particular frequency characterized by a ratio of permittivities and permeabilities equal to $-1$ across the interface.  At this frequency, the response of the system  to a harmonic forcing term blows up linearly in time. Such a  resonance is unusual for wave problem in unbounded domains and corresponds to  a  non-zero embedded eigenvalue of infinite multiplicity of the underlying operator. 
This is the time counterpart of the ill-posdness of the corresponding harmonic problem. 
\end{abstract}



\section{Introduction}
This paper is second of a series of two articles devoted to the mathematical analysis of the transmission of electromagnetic waves through a plane interface separating a standard material (here the vacuum) and a metamaterial. These two papers are an improved version of the preliminary study presented in the PhD thesis \cite{Cas-14}. 
Metamaterials are manufactured materials whose effective behaviour is dispersive (in other words frequency dependent). In particular, the effective permittivity and permeability can be both negative in a certain  range of frequencies  \cite{cas-mil-16,cas-kach-jol-17,Gra-10,Tip-98}. As a consequence such materials support so-called back propagating waves whose phase and group velocities have opposite directions \cite{Zio-01}. This is the reason of apparition of new phenomena at the interface with a dielectric medium such as negative refraction or plasmonic surface waves \cite{Mai-07}. Therefore, these materials have raised a lot interest in the physical literature during the two last decades due to applications in electromagnetism \cite{Nir-94,Smit-04}, in acoustic \cite{Cum-16, Li-04} and also for seismic waves \cite{Bru-14}. Accordingly the study of the corresponding models raised new mathematical questions for transmission problems, to begin with the long time behaviour of the response of such medium to a time-harmonic source of prescribed frequency. More precisely, after a transient regime, does the solution of the time-dependent equation ``converge'' for large times to a stationary regime? In scattering theory, this property is referred to as the \emph{limiting amplitude principle}. It is closely related to another property called the \emph{limiting absorption principle} which defines the stationary regime via the limit of the resolvent of the propagative operator at the frequency of excitation. The question of the validity of both limiting amplitude and limiting absorption principles is precisely the objective of this paper for the particular case where the metamaterial is a Drude material, which can be seen as the simplest metamaterial.\\ 

Limiting absorption and limiting amplitude principles have a long history  in scattering theory and more generally in mathematical physics. In the context of wave phenomena, these principles were first proved to our knowledge by  C. Morawetz \cite{Mora-62} for sound soft obstacles in a homogeneous medium via energy techniques. Then D. Eidus  \cite{Eid-65,Eid-69} constructed an abstract proof which involved the spectral decomposition of the propagative operator and applied it to a class of acoustic media that are locally inhomogeneous. Eidus' approach was then developed by  C. Wilcox \cite{Wil-84}, Y. Dermanjian, and J-C. Guillot  \cite{Der-83,Der-86} and R. Weder \cite{Wed-91} for acoustic and electromagnetic stratified media. Finally, it was extended  to other structures  such as waveguides \cite{Mor-89}, periodic media \cite{Rad-15}  \ldots  and  to other  waves equations: elastic waves \cite{Der-88,San-89}, water waves \cite{Vul-87,Haz-07}, \ldots. The method we use is inspired from Eidus' spectral approach and its extension to stratified media. It is applied for the first time in the context of dispersive Maxwell's equations and  metamaterials. Compared to previous studies, the difficulty and novelty of the analysis relies in the fact that for the Drude material, the permittivity and permittivity  depend on the frequency and become negative for low frequencies. This complicates significantly the establishment of both principles. Finally, we want to mention  that other techniques such as Mourre's commutators \cite{Koj-91,Wed-84} can be used to prove the limiting  absorption and limiting amplitude  principles. These techniques have the advantage to work on non separable geometries but they are based on a  more abstract limit process. Thus, unlike the spectral decomposition approach, they don't provide an explicit modal decomposition of the solution and its limiting stationary regime. Therefore, they are not as precise for applications. \\

In the first paper \cite{Cas-Haz-Jol-17}, we begin by writing the governing equations as a conservative Sch\"odinger equation. We point out that such a reformulation of the time-dependent Maxwell's equation which takes a very explicit expression in \cite{Cas-Haz-Jol-17} for the Drude material, can be applied in the more general setting of linear passive electromagnetic media (including dissipative ones), see \cite{cas-kach-jol-17,Fig-05,Gra-10,Tip-98}. Then, we perform the complete spectral analysis of the corresponding Hamiltonian. In particular, we provide the diagonalization of this operator through the construction of an appropriate generalized Fourier transform. This furnishes the material needed for addressing the question of the limiting amplitude principle which relies on the  existence of a limiting absorption principle. As we shall see, our analysis emphasizes the role of a so-called resonant frequency  corresponding to the case where the ratios between the permittivities and permeabilities across the interface are simultaneously equal to $-1$. At this particular frequency, the limiting principle fails and the solution grows linearly in time. This result in the time-domain is the counterpart of the results concerning the  ill-posedness of the transmission problem in the frequency domain \cite{Bon-14,Bon-14(2),Cos-85,Ngu-16}. 
This  interface resonance phenomenon  has been enlighten in the physical literature in \cite{Gra-12}. We prove here that it is based
on the existence of a non-zero embedded eigenvalue of infinite multiplicity that does not exist in a stratified media composed of non-dispersive dielectrics \cite{Wed-91}  media. It is is due to the presence of a negative dispersive material: the Drude material.
Let us mention than other resonance phenomena which are not linked to eigenvalues are observed in unbounded domains such as waveguides  (see e. g. \cite{Wed-98,Wer-87,Wer-96}) excited at a cut-off frequencies. In this case, the growth  rate is non-linear with the time and depends on the geometry shape: $C\; \log{t}$ for planar waveguides and $C \, t^{{1}/{2}}$ for cylindrical ones. \\

The outline of the paper is as follows. Section \ref{s.math-model-main-results} is devoted to a recap of \cite{Cas-Haz-Jol-17} and the statement of the main results of the present paper.  In \S\ref{ss.math-model}, we recall the formulation of the evolution problem as a generalized Schr\"odinger equation. We then present in \S\ref{ss.main-results} the main theorems of this paper: the limiting absorption principle (Theorem \ref{thm.limabs}) and the limiting amplitude principle (Theorem \ref{th.ampllim}). Their proofs are based on the diagonalization of the Hamiltonian involved in the Schr\"odinger equation, using appropriate generalized eigenfunctions, which is recalled in \S\ref{ss.diag-Hamilt}. 

In \S\ref{sec.spectral-density}, we introduce the fundamental notion of spectral density of the Hamiltonian, as a function of the real (spectral) variable with values in the set of bounded linear operators between two appropriate weighted function spaces on $\mathbb{R}^2$. We give an explicit expression of this spectral density with the help of the generalized eigenfunctions and establish the technical results which are the basic ingredients for the proofs of our main theorems:  bounds of the spectral density  and corresponding (local) H\"older continuity estimates, that themselves rely on similar properties about generalized eigenfunctions. The proofs of the two main theorems are the subject of section \ref{sec.prolimabs}. Finally, section \ref{sec-limampl-limabs-threshold} is devoted to a very specific situation excluded in Theorems \ref{thm.limabs} and \ref{th.ampllim} and which necessitates technical adjustments: this corresponds to the case where the frequency of the source coincides with the so-called \emph{plasmonic frequency}.

\section{Mathematical model and main results}
\label{s.math-model-main-results}
\subsection{Mathematical model}
\label{ss.math-model}
We recall here the mathematical formulation of the problem studied in \cite{Cas-Haz-Jol-17}. In this previous paper, we have  considered the Transverse Electric (TE) transmission problem between a Drude material and the vacuum separated by a planar interface, which reduces to a two-dimensional model where the vacuum and the Drude material fill respectively the half-planes
\begin{equation*} 
\bbR^2_- := \{\bx=(x,y)\in \bbR^2\mid x<0\} 
\quad\text{and}\quad
\bbR^2_+ := \{\bx=(x,y)\in \bbR^2\mid x>0\}.
\end{equation*}
The  physical unknowns are the transverse component of the electric field $E(\bx,t)$, the magnetic field $\bH(\bx,t)=(H_x(\bx,t), H_y(\bx,t) \big)^{\top}$, the induced transverse electric current in the Drude material $J(\bx,t)$ and finally the induced magnetic current in the Drude material $\bK(\bx,t)=(K_x(\bx,t), K_y(\bx,t) \big)^{\top}$. Our problem, which couples these unknowns, can be formulated in a concise form as
\begin{equation}\label{TE}
\left\{ \begin{array}{ll}
\eps_0 \:\partial_t E -\curl \bH + \Pi \, J = - J_{\rm s} & \mbox{in }  \bbR^2, \\[5pt]
\mu_0\: \partial_t \bH + \bcurl E + \bPi\, \bK= 0 & \mbox{in } \bbR^2,\\[5pt]
\partial_t J = \eps_0 \Oe^2 \, \Rop \, E & \mbox{in } \bbR_+^2, \\[5pt]
\partial_t \bK= \mu_0 \Om^2 \, \bR \, \bH & \mbox{in } \bbR_+^2.
\end{array} \right.
\end{equation}
The first two equations derive from Maxwell's equations, whereas the last two are the constitutive laws of the Drude material. In these equations, $\eps_0$ and $\mu_0$ stand for the permittivity and the permeability of the vacuum, whereas $\Oe$ and $\Om$ are positive constants which characterize the Drude material. The operators $\bcurl$ and $\curl$ are respectively defined by
\begin{equation}\label{eq-defcurls}
\bcurl u := (\partial_y u, -\partial_x u)^{\top}
\quad\mbox{and}\quad 
\curl \bu := \partial_x u_y -\partial_y u_x \mbox{ for } \bu=(u_x,u_y)^{\top}.
\end{equation}
The operator $\Pi$ (respectively, $\bPi$) denotes the extension by $0$ of a scalar function (respectively, a 2D vector field) defined on $\bbR^2_+$ to the whole plane $\bbR^2$, whereas $\Rop$ (respectively, $\bR$) stands for the restriction to  $\bbR^2_+$ of a scalar function (respectively, a 2D vector field) defined on $\bbR^2$. Finally, in the right-hand side of the first equation, $J_{\rm s}(\bx,t)$ represents the excitation (current density) which generates an electromagnetic wave. 
We emphasize that all quantities involved in \eqref{TE} will be assumed square-integrable (the equations being understood in the sense of distributions). Thus \eqref{TE} contains implicitly the continuity of the tangential fields across the interface $x=0$, that is,
\begin{equation*}\label{eq.transTE}
[E]_{x=0}=0 \quad\mbox{and}\quad [H_{y}]_{x=0}=0,
\end{equation*}
where $[f]_{x=0}$ denotes the gap of $f$ across the line $x=0.$

When looking for time-harmonic solutions to \eqref{TE} at a given (circular) frequency $\lambda \in \bbR,$ \textit{i.e.}, 
$$(E(\bx,t),\bH(\bx,t), J(\bx,t),K(\bx,t)) =(E_{\lambda}(\bx),  \bH_{\lambda}(\bx) ,  J_{\lambda}(\bx),  \bK_{\lambda}(\bx))   \  \rme^{-\rmi\lambda t}$$ 
for $J_{\rm s}(\bx,t) = J_{{\rm s},\lambda}(\bx)\, \rme^{-\rmi\lambda t},$ one can eliminate 
$J_{\lambda}$ and $K_{\lambda}$ and obtain the following time-harmonic Maxwell equations:
\begin{equation*}
\rmi \, \lambda\, \eps_\lambda(\bx) \, \bE_{\lambda} + \bcurl \bH_{\lambda}= J_{{\rm s},\lambda}
\quad\mbox{and}\quad 
-\rmi \, \lambda\, \mu_\lambda(\bx) \,\bH_{\lambda} + \bcurl\bE_{\lambda} = 0
\quad\mbox{in }\bbR^2,
\end{equation*}
where 
\begin{equation}\label{eq.defepsmu}
\eps_{\lambda}(\bx)  :=  \left\lbrace\begin{array}{ll}
         \eps_{\lambda}^{-}:=\eps_{0} & \mbox{ if } x<0,\\[2pt]
       \displaystyle   \eps_{\lambda}^{+}:=\eps_0 \left( 1-\frac{ \Oe^2}{\lambda^2}\right) & \mbox{ if } x>0,\\
       \end{array}\right.  \ \mbox{ and }
        \ \mu_{\lambda}(\bx)  :=  \left\lbrace\begin{array}{ll}
         \mu_{\lambda}^{-}:=\mu_{0} & \mbox{ if } x<0,\\[2pt]
       \displaystyle   \mu_{\lambda}^{+}:=\mu_0 \left( 1-\frac{ \Om^2}{\lambda^2}\right) & \mbox{ if } x>0.\\
       \end{array}\right.
\end{equation}
The rational functions $\lambda \mapsto \eps_{\lambda}^{+},\mu_{\lambda}^{+}$ characterize the frequency dispersion of the Drude material. Both take negative values for low frequencies (respectively, when $|\lambda| < \Oe$ and $|\lambda| < \Om$). Note that
\begin{equation*}
\frac{\eps_\lambda^+}{\eps_0} = -1  \mbox{ if } |\lambda| = \frac{\Oe}{\sqrt{2}}
\quad\mbox{and}\quad
\frac{\mu_\lambda^+}{\mu_0} = -1  \mbox{ if } |\lambda| = \frac{\Om}{\sqrt{2}}.
\end{equation*}
We see in particular that both ratios can be simultaneously equal to $-1$ at the same frequency if and only if $\Oe = \Om,$ which will be referred to as the \emph{critical case} in the following. 

Our study of the Maxwell's equations \eqref{TE} is based on their reformulation as a \emph{conservative Schr\"{o}dinger equation} (see \cite{Cas-Haz-Jol-17}) 
\begin{equation}\label{eq.schro}
\frac{\rmd \, \bU}{\rmd\, t} + \rmi\, \bbA \, \bU=\bG,
\end{equation}
in the Hilbert space
\begin{equation}\label{eq.defHxy}
\Hxy := L^2(\bbR^2) \times L^2(\bbR^2)^2 \times L^2(\bbR^2_+) \times L^2(\bbR^2_+)^2
\end{equation}
whose inner product is defined for all $\bU:=(E, \bH, J, \bK)^{\top}$ and $\bU':=(E^{\prime}, \bH^{\prime}, J^{\prime}, \bK^{\prime})^{\top}\in \Hxy$ by
\begin{equation*}
(\bU,\bU')_{\Hxy} := 
\int_{\bbR^2}\left(\eps_0\,E\,\overline{E^{\prime}} + \mu_0\,\bH\cdot\overline{\bH^{\prime}}\right)\rmd \bx 
+ \int_{\bbR^2_+}\left(\eps_0^{-1} \Oe^{-2}\,J\,\overline{J^{\prime}} + \mu_0^{-1} \Om^{-2}\, \bK\cdot\overline{\bK^{\prime}}\right)\rmd \bx.
\end{equation*}
The Hamiltonian $\bbA$ is the unbounded \emph{selfadjoint} operator on $\Hxy$ defined by
\begin{align*}
\rmD(\bbA) & := H^{1}(\bbR^2) \times \bH_{\!\curl}(\bbR^2) \times  L^2(\bbR^2_+) \times L^2(\bbR^2_+)^2 \subset \Hxy,\\[5pt]
\bbA & := \ \rmi\, \begin{pmatrix}
0 &\eps_0^{-1}\,\curl & -\eps_0^{-1} \, \Pi & 0\\
- \mu_0^{-1}\,\bcurl& 0 &0 & - \mu_0^{-1} \,\bPi \\
\eps_0 \Oe^2 \, \Rop & 0 & 0 &0 \\
0 & \mu_0 \Om^2\, \bR & 0 & 0
\end{pmatrix},
\end{align*}
where $\bH_{\!\curl}(\bbR^2) := \{ \bu\in  L^2(\bbR^2)^2 \mid \curl \bu \in L^2(\bbR^2)\}$. 
Finally the source term $\bG$ in \eqref{eq.schro} is given by $\bG(t) := (- \eps_0^{-1} \, J_{\rm s}(\cdot,t) \,,0\,, \,0 ,\, 0)^{\top} \in \Hxy$.

Considering for simplicity zero initial conditions (\textit{i.e.}, $\bU(0) = 0$), we know from the Hille--Yosida theorem \cite{Bre-10} that the Schr\"{o}dinger equation \eqref{eq.schro} has a unique solution $\bU \in C^1\big([0,+\infty), \Hxy \big) \cap C^0\big([0,+\infty), \rmD(\bbA) \big) $ which is given by Duhamel's formula
\begin{equation}\label{eq.Duhamel}
\bU(t) = \int_{0}^{t} \rme^{-\rmi \bbA\, (t-s)}\, \bG(t)\,\rmd s,  \quad  \forall  t\geq0. 
\end{equation}

Let us finally notice that Maxwell's equation \eqref{TE} contain implicitly some conditions about the divergence of the magnetic field $\bH$ and of the induced magnetic current $\bK$. Indeed taking the divergence of the second equation of \eqref{TE} restricted to $\bbR^2_-$ shows that $\mu_0\,\partial_t \div \bH = 0$ in $\bbR^2_-$. Hence, as our system starts from rest, we have $\div \bH = 0$ at $t=0$ and the latter equation shows that $\div \bH = 0$ in $\bbR^2_-$ for all $t>0.$ Similarly, taking the divergence of the second and fourth equations of \eqref{TE} restricted to $\bbR^2_+$ yields
\begin{equation}\label{eq.divTE}
\left\{ \begin{array}{ll}
\mu_0\: \partial_t \div \bH + \div \bK= 0 & \mbox{in } \bbR_+^2,\\[5pt]
\partial_t \div \bK = \mu_0 \Om^2 \, \div \bH & \mbox{in } \bbR_+^2.
\end{array} \right.
\end{equation}
Then by differentiating the second equation with respect to $t$, we can eliminate $\div \bH$ and obtain
\begin{equation*}
\partial_t^2 \div \bK + \Om^2 \, \div \bK = 0 \mbox{ in } \bbR_+^2.
\end{equation*}
Hence, as $\div \bK = 0$ and $\partial_t \div \bK = 0$ at $t=0,$ this equation shows that $\div \bK = 0$ for all $t>0$ and we deduce from \eqref{eq.divTE} that $\div \bH = 0$ in $\bbR^2_+$ for all $t>0.$ This explains why in the following, the solution to \eqref{TE} will be searched for in the subspace of $\Hxy$ defined by
\begin{equation}
\Hxydiv :=\{ (\bE, \bH, \bJ, \bK)^\perp \in \Hxy \mid \div \bH=0 \mbox{ in } \bbR^{2}_{\pm} \mbox{ and } \div \bK=0 \mbox{ in } \bbR^2_+\}.
\label{eq.Hxydiv}
\end{equation}
Note that the conditions $\div \bH=0$ in $\bbR^{2}_{\pm}$ does not mean that the divergence of $\bH$ vanishes in the whole plane $\bbR^{2}$: there may be a gap of the normal component of $\bH$ across the line $x=0.$

\subsection{Statement of the main results}
\label{ss.main-results}
In this paper, we are interested in  the long-time behavior of $\bU(t)$ given in \eqref{eq.Duhamel} when the excitation starts at $t=0$ and becomes time-harmonic at a given (circular) frequency $\omega > 0$, that is,
\begin{equation*}
	\bG(t) = \bG_\omega \ H(t) \, \rme^{-\rmi\omega\, t} \quad \text{for some given } \bG_\omega \in \Hxy,
\end{equation*}
where $H$ denotes the Heaviside function (\textit{i.e.}, $H(t) = 0$ if $t<0$ and $H(t) = 1$ if $t\geq 0$). In this case, formula \eqref{eq.Duhamel} can be rewritten equivalently as 
\begin{equation*}
\bU(t) = \phi_{\omega,t}(\bbA)\,\bG_\omega
\end{equation*}
where $\lambda \mapsto \phi_{\omega,t}(\lambda)$ is, for all $t\geq 0$, the bounded continuous function defined by
\begin{equation}\label{eq.phiduhamel2}
\phi_{\omega,t}(\lambda) := 
\rme^{-\rmi \lambda\, t} \int_{0}^{t} \rme^{\rmi(\lambda - \omega)\,s}\,\rmd s
= \left\lbrace\begin{array}{ll}
\displaystyle \rmi \,\frac{\rme^{-\rmi \lambda\, t} -\rme^{-\rmi \omega \, t}}{\lambda-\omega} & \mbox{ if } \lambda \neq \omega , \\[10pt]
t\, \rme^{-\rmi \omega \,t} & \mbox{ if } \lambda =\omega.
\end{array}
\right.
\end{equation}
We intuitively expect that after some transient regime due to the fact that $\bU(t)$ starts from rest at $t=0$, the solution  $\bU(t)$ behaves like a time-harmonic wave 
$\bU_{\omega}^{+}=(E_{\omega}^{+},\bH_{\omega}^{+}, J_{\omega}^{+},\bK_{\omega}^{+})^\perp$, that is,
\begin{equation}
\bU(t)\sim - \rmi  \, \bU_{\omega}^{+}(\cdot)\, \rme^{-\rmi\omega\, t} \quad\text{as } t \to +\infty. 
\label{eq.LAmP-imprecis}
\end{equation} 
Such a property is usually called the \emph{limiting amplitude principle} in mathematical physics. It is closely related to another property, called the \emph{limiting absorption principle}, which provides the time-harmonic behavior $\bU_{\omega}^{+}$ by the formula 
\begin{equation*}
\bU_{\omega}^{+} = \lim_{\eta \searrow 0} \big( \bbA - (\omega + \rmi\eta)\,\bbI \big)^{-1} \bG_\omega.
\end{equation*} 
Our aim is to define a mathematical framework for a rigorous statement of these principles and to make precise the various situations where these principles hold true or not. Our main results are summarized below.


\subsubsection{The limiting absorption principle}\label{sec-thabsl}

We first have to recall some results about the spectrum of $\bbA$ (which is necessarily real since $\bbA$ is a selfadjoint) and introduce some notations, in particular the following particular frequencies $\Op$ (``p'' for ``plasmonic'') and $\Oc$ (``c'' for ``cross point'', see \S\ref{ss.diag-Hamilt}) defined by
\begin{equation}
\Op := \frac{\Om}{\sqrt{2}} \quad\text{and}\quad
\Oc := \frac{\Oe \, \Om}{\sqrt{\Oe^2+\Om^2}}.
\label{eq.def-Op-Oc}
\end{equation}
Note that in the critical case, that is, when $\Oe = \Om$, we have $\Op = \Oc.$ It is also useful to introduce the following set of ``exceptional frequencies'' (whose role will be made clear later):
\begin{equation}
\sigma_{\rm exc} := \left\{
\begin{array}{ll}
\{0,\pm\Op,\pm\Om\} & \text{if } \Oe \neq \Om, \\[5pt] \{0,\pm\Om\} & \text{if } \Oe = \Om.
\end{array}\right.
\label{eq.def-sigma-exc}
\end{equation}

The proposition below gathers various results given in \cite[\S 4]{Cas-Haz-Jol-17}.
\begin{proposition}\label{prop.spectrumA}
	The spectrum of $\bbA$ is the whole real line: $\sigma(\bbA)=\bbR.$ The point spectrum $\sigma_{\rm pt}(\bbA)$ is composed of eigenvalues of infinite multiplicity: 
	\begin{equation}\label{eq.specpt}
	\sigma_{\rm pt}(\bbA) = \left\{
	\begin{array}{ll}
	\{0,\pm\Om\} & \text{if } \Oe \neq \Om, \\[5pt] \{0,\pm\Op,\pm\Om\} & \text{if } \Oe = \Om.
	\end{array}\right.
	\end{equation} 
	The eigenspaces $\ker(\bbA)$ and $\ker(\bbA\pm \Om)$ are respectively given by:
	\begin{align*}
		\ker(\bbA) & =  \{ (0, \widetilde{\bPi}\, \nabla \phi, 0 , 0)^{\top} \ \mid \ \phi\in  W_0^1(\bbR^2_-)\},  \\
		\ker(\bbA \mp \Om ) & = \left\{ \left(0, \,\bPi \,\nabla \phi, \,0 , \pm \rmi\mu_0\Om \,\nabla \phi \right)^{\top} \ \mid \ \phi\in W_0^1(\bbR^2_+)\right\},
	\end{align*}
	where $\widetilde{\bPi}$ is the extension by $0$ of a 2D vector field defined on $\bbR^2_-$ to the whole plane $\bbR^2$ and $W_0^1(\bbR^2_{\pm})$ stands for  the Beppo-Levi space 
	$W_0^1(\bbR^2_{\pm}) := \{ \phi \in L^2_{\rm loc}(\bbR^2_{\pm})\ \mid  \nabla \phi  \in L^2(\bbR_{\pm}^2)^2 \mbox{ and } \phi|_{x=0} = 0 \}.$ Furthermore, the orthogonal complement of the direct sum of the eigenspaces associated to 0 and $\pm\Om$ is the space $\Hxydiv$ defined in \eqref{eq.Hxydiv}:
	\begin{equation*}
	\Hxydiv = \Big( \ker \bbA \oplus \ker(\bbA+\Om)\oplus \ker(\bbA-\Om) \Big)^\perp.
	\end{equation*}
\end{proposition}

In the following, we denote by $\bbP_{\rm div0}$ the orthogonal projection on the subspace $\Hxydiv$ of $\Hxy$,  by $\bbP_{\rm pt}$ the orthogonal projection on the point subspace of $\bbA,$ that is, the direct sum of the eigenspaces associated to the eigenvalues of $\bbA$, and by $\bbP_{\pm\Op}$ is the orthogonal projection on the eigenspace associated to $\pm \Op$. Finally we introduce
\begin{equation}\label{eq.def-Pac}
\bbP_{\rm ac} := \bbI - \bbP_{\rm pt}
= \left\{\begin{array}{ll}
\bbP_{\rm div0} & \text{if } \Oe \neq \Om, \\[5pt] 
\bbP_{\rm div0}  - \bbP_{-\Op} - \bbP_{+\Op} & \text{if } \Oe = \Om,
\end{array}\right. 
\end{equation}
where the last equality follows from Proposition \ref{prop.spectrumA}. We will see in \S\ref{sec.spectral-density} that $\bbP_{\rm ac}$ is actually the orthogonal projection on the \emph{absolutely continuous} subspace associated to $\bbA$, which explains the index ``ac''.

The limiting absorption principle explores the behavior of the resolvent of $\bbA$, \textit{i.e.},
\begin{equation*}
R(\zeta):=(\bbA-\zeta\,\bbI)^{-1} \quad \text{for } \zeta \in \bbC\setminus\bbR
\end{equation*}
near the spectrum of $\bbA.$ More precisely, we investigate the existence of the one-sided limits of the \emph{absolutely continuous} part of the resolvent near some $\omega \in \bbR,$ that is,
\begin{equation}\label{eq.def-Rac}
R^{\pm}_{\rm ac}(\omega) := \lim_{\eta \searrow 0} R_{\rm ac}(\omega \pm \rmi \eta) 
\quad\text{where}\quad
R_{\rm ac}(\zeta) := R(\zeta)\,\bbP_{\rm ac} = \bbP_{\rm ac}\,R(\zeta)
\text{ for }\zeta \in \bbC\setminus\bbR.
\end{equation}
The resolvent $R(\zeta)$ is an analytic function of $\zeta$ in $\bbC\setminus \bbR$ with values in $B(\Hxy)$, the Banach algebra of bounded linear operators in $\Hxy$. Of course, the above one-sided limits, if they exist, are not defined in $B(\Hxy)$ (otherwise, $\omega$ would belong to the resolvent set of $\bbA$), but for a weaker topology. To this aim, we introduce a weighted version of our Hilbert space $\Hxy$ (see \eqref{eq.defHxy}) defined for any $s \in \bbR$ by
\begin{equation*}
\Hps := L^2_{s}(\bbR^2) \times L^2_{s}(\bbR^2)^2 \times L^2_{s}(\bbR^2_+) \times L^2_{s}(\bbR^2_+)^2,
\end{equation*}
where $L^2_{s}(\calO) := \{ u \in L^2_{\rm loc}(\calO) \mid  \eta_{s}\,u \in L^2(\calO)\}$ for $\calO=\bbR^2$ or $\calO=\bbR^2_+$, and the weight $\eta_s$ is given by
\begin{equation*}
\eta_s(x,y) := (1 + x^2)^{s/2}\,(1 + y^2)^{s/2}.
\end{equation*}
The space $L^2_{s}(\calO)$ is naturally endowed with the norm
\begin{equation*}
\|u\|_{L^2_{s}(\calO)}^2 := \|\eta_s\,u\|_{L^2(\calO)}^2 = \int_\calO |\eta_s\,u|^2\,\rmd\bx.
\end{equation*}
Similarly, the Hilbert space $\Hps$ is equipped with the norm
$$
\|\bU\|_{\Hps}  := \|\eta_s \,\bU\|_{\Hxy}.
$$
It is readily seen that for positive ${s},$ the spaces $\Hps$ and $\Hms$ are dual to each other if $\Hxy$ is identified with its own dual space, which yields the continuous embeddings $\Hps \subset \Hxy \subset \Hms.$ The notation $\langle \cdot\,,\cdot\rangle_{s}$ represents the duality product between them. This duality product extends the inner product of $\Hxy$ in the sense that
\begin{equation}\label{eq.innerproduct}
\langle \bU \,, \bU' \rangle_{s} = (\bU \,, \bU')_{\Hxy} \quad\mbox{if }\bU \in \Hps \mbox{ and }\bU' \in \Hxy.
\end{equation}
As a topology for the limits \eqref{eq.def-Rac}, we choose the operator norm  $\| \cdot\|_{\Hps,\Hms}$ of $B(\Hps,\Hms)$, the Banach algebra  of bounded linear operators from $\Hps$ to $\Hms$. The limiting amplitude principle stated in the next subsection requires the H\"{o}lder regularity of the limits. The following theorem, which is proved in \S\ref{sec-thabslproof}, aims at providing an optimal result in this direction.

\begin{theorem}[Limiting absorption principle]\label{thm.limabs}
Let $s>1/2$. For all $\omega \in \bbR \setminus \sigma_{\rm exc},$ the absolutely continuous part of the resolvent 
$R_{\rm ac}(\zeta)$ has one-sided limits 
$R^{\pm}_{\rm ac}(\omega) := \lim_{\eta \searrow 0} R_{\rm ac}(\omega \pm \rmi \eta)$ for the operator norm of $B(\Hps,\Hms)$. Moreover, by denoting
\begin{equation}\label{eq.def-Rac-pm}
R^{\pm}_{\rm ac}(\zeta) := 
R_{\rm ac}(\zeta) \quad \text{if } \zeta \in \bbC^\pm := \{ \zeta\in\bbC \mid \pm \Imag\zeta > 0 \},
\end{equation}
the function  $\zeta \mapsto R^{\pm}_{\rm ac}(\zeta) \in B(\Hps,\Hms)$ is locally H\"{o}lder continuous in $\overline{\bbC^\pm} \setminus \sigma_{\rm exc}.$ More precisely, for any compact set $K \subset \overline{\bbC^\pm} \setminus \sigma_{\rm exc}$, there exists a set $\Gamma_K \subset (0,1)$ of H\"{o}lder exponents such that for any $\gamma \in \Gamma_K,$ there exists $C_{K,\gamma}>0$ such that
\begin{equation*}
\forall (\zeta,\zeta') \in K \times K, \quad 
\Big\| R^{\pm}_{\rm ac}(\zeta') - R^{\pm}_{\rm ac}(\zeta) \Big\|_{\Hps,\Hms} 
\leq C_{K,\gamma} \ |\zeta'-\zeta|^{\gamma}.
\end{equation*}
The set $\Gamma_K$ is defined as follows: 
\begin{equation}
\Gamma_K := \left\{
\begin{array}{ll}
\big(0,\min(s-1/2,1)\big) & \text{if } K \cap \{ \pm\Oe,\pm\Oc \} = \varnothing, \\[5pt] 
\big(0,\min(s-1/2,1/2)\big) & \text{if } K \cap \{ \pm\Oe,\pm\Oc \} \neq \varnothing.
\end{array}\right.
\label{eq.def-Gamma-K}
\end{equation}
\end{theorem}

This abstract theorem provides us the existence of the one-sided limits $R^{\pm}_{\rm ac}(\omega)$, but not an explicit expression of these operators. Section  \ref{sec-thabslproof} will show their respective spectral representations (see Proposition \ref{prop.lim-res}). To understand the physical significance of these limits, first note that $\bU_\zeta = R_{\rm ac}(\zeta)\,\bG$ means equivalently that $(\bbA - \zeta\,\bbI) \,\bU_\zeta = \bbP_{\rm ac}\bG$. Hence one can expect that the limits $\bU_\omega^\pm = R^{\pm}_{\rm ac}(\omega)\,\bG$ satisfy the time-harmonic equation $(\bbA - \omega\,\bbI) \,\bU_\omega^\pm = \bbP_{\rm ac}\bG$. This can be verified rigorously provided $\bbA$ is interpreted in a suitable  distributional sense (since $\bU_\omega^\pm$ does not belong to $\Hxy$ in general). This means that $\bU_\omega^\pm$ both represent time-harmonic solutions of our Maxwell equations \eqref{TE}. Their difference can be understood precisely by the \emph{limiting amplitude principle} will actually tells us that $\bU_\omega^+$ is \emph{outgoing}, in the sense of \eqref{eq.LAmP-imprecis}. It could be seen similarly that $\bU_\omega^-$ is \emph{incoming}, by considering the behavior as $t \to -\infty$ of anti-causal solutions.

Note that Theorem \ref{thm.limabs} excludes the values of $\sigma_{\rm exc}$ defined in \eqref{eq.def-sigma-exc}. In the case where $\Oe \neq \Om,$ it may seem surprising to exclude the values $\omega = \pm\Op$ which are not in the point spectrum of $\bbA.$ As a matter of fact, these values require a special study, which is the subject of \S\ref{sec-limampl-limabs-threshold}.

\begin{remark}
The above formulation of Theorem \ref{thm.limabs} is not entirely optimal in the sense that our definition of $\Gamma_K$ excludes a particular case which could be included. Indeed, when $K$ contains $+\Oe$ or $-\Oe$ but not $\pm\Oc$, the value $\gamma = 1/2$ is allowed, provided that $s > 1$ (note that this situation cannot occur in the critical case $\Oe = \Om$ since in this case, the values  $\pm\Oe = \pm\Om \in \sigma_{\rm exc}$ cannot belong to $K$). The reason why we excluded this particular case is that its proof is slighly more involved than for all the other cases. As the proof of the general situation is already very technical (see \S\ref{sec.spectral-density}), we decided to spare the reader and postpone the proof of this particular case to Appendix \ref{app.cas-un-demi}.
\label{rem.cas-un-demi}
\end{remark}

\subsubsection{The limiting amplitude principle}
We are now able to state our main result concerning the asymptotic behavior of the solution $\bU(t)$ to our Schr\"{o}dinger equation \eqref{eq.schro} with a time-harmonic excitation $\bG(t) = \bG_\omega \ H(t) \, \rme^{-\rmi\omega\, t}$ starting at $t=0$ (recall that $H(t)$ denotes the Heaviside function). Instead of our particular excitation introduced in \S\ref{ss.math-model}, we consider the more general case where $\bG_{\omega} \in \Hxydiv \cap \Hps.$ For simplicity, we assume that $\bU(0)=0$ which corresponds to zero initial conditions on the fields $(E,\bH,J,\bK)$.

\begin{theorem}\label{th.ampllim}
Let $s > 1/2$ and $\omega \in \bbR \setminus \sigma_{\rm exc}$ (see \eqref{eq.def-sigma-exc}).
On the one hand, for any $\bG_\omega \in \Hps$ which belongs to the range of $\bbP_{\rm ac},$ the limiting amplitude principle holds true in the sense that the solution $\bU(t)$ to \eqref{eq.schro} with zero initial conditions has the following asymptotic behavior for large time:
\begin{equation}
\lim_{t\to+\infty}\Big\| \bU(t) + \rmi \, \bU_{\omega}^{+} \, \rme^{-\rmi \omega t} \Big\|_{\Hms} = 0,
\label{eq.lim-ampl}
\end{equation}
where $\bU_{\omega}^{+} := R^{+}_{\rm ac}(\omega) \, \bG_\omega\in \Hms$ is given by the limiting absorption principle. 

On the other hand, in the critical case $\Oe = \Om,$ for any $\bG_\omega \in \Hps \cap \Hxydiv$, the asymptotic behavior for large time of $\bU(t)$ is given by
\begin{equation}
\lim_{t\to+\infty}\Big\| \bU(t) 
- \Big( -\rmi \,\bU_{\omega}^{+} \, \rme^{-\rmi \omega t} + \sum_{\pm} \bbP_{\pm\Op} \bG_\omega\ \phi_{\omega,t}(\pm\Op)\Big) \Big\|_{\Hms} = 0,
\label{eq.lim-ampl-reson}
\end{equation}
where $\bbP_{\pm\Op}$ is the orthogonal projection on the infinite dimensional eigenspace associated to $\pm\Op$, $\phi_{\omega,t}$ is defined in \eqref{eq.phiduhamel2} and $\bU_{\omega}^{+} := R^{+}_{\rm ac}(\omega) \, \bG_\omega$ as above.
\end{theorem}

This theorem, which is proved in \S\ref{sec.limamplproof}, tells us that in the non-critical case, that is $\Oe \neq \Om,$ the limiting amplitude principle holds true for any $\bG_{\omega} \in \Hxydiv \cap \Hps$ (since $\Hxydiv$ is exactly the range of $\bbP_{\rm ac}$ in this case, see \eqref{eq.def-Pac}). The assumption $\bG_{\omega} \in \Hxydiv$ forces the solution to remain orthogonal to the eigenspaces associated to the point spectrum $\{0,\pm\Om\},$ which is a natural physical assumption (see \S\ref{ss.math-model}). The frequencies which are excluded here are $0,$ $\pm\Om$ and $\pm\Op$. Let us mention that the limiting amplitude principle holds also for  $\omega=\pm\Op$. This case, which is not covered by Theorem \ref{th.ampllim}, requires a special treatment that is detailed in \S \ref{sec-limampl-limabs-threshold}.

On the other hand, in the critical case $\Oe=\Om$, the validity of the limiting amplitude principle depends on the spectral content of the excitation $\bG_\omega \in \Hxydiv \cap \Hps$ which can be decomposed as
\begin{equation*}
\bG_\omega = \bbP_{\rm ac}\bG_\omega + \bbP_{-\Op}\bG_\omega + \bbP_{+\Op}\bG_\omega.
\end{equation*}
If $\bbP_{-\Op}\bG_\omega = \bbP_{+\Op}\bG_\omega = 0$ (which means that $\bG_\omega$ belongs to the range of $\bbP_{\rm ac}$), the principle holds true for any $\omega \in \bbR \setminus \{0,\pm\Om\},$ which includes in particular the frequencies $\omega = \pm\Op.$ But if $\bbP_{-\Op}\bG_\omega \neq 0$ or $\bbP_{+\Op}\bG_\omega \neq 0$, the behavior of $\bU(t)$ is no longer time-harmonic at the frequency $\omega$. Two situations may occur. Firstly, if $\omega \in \bbR \setminus \{0,\pm\Op,\pm\Om\},$ the solution $\bU(t)$ remains bounded in time but oscillates at the two frequencies $\omega$ and $\Op$ (see the expression \eqref{eq.phiduhamel2} of $\phi_{\omega,t}(\pm \Op)$ for $\omega\neq \pm \Op$): it is a beat phenomenon. Secondly, if $\omega = \pm\Op,$ there is no stationary regime at all, since $\bU$ blows up linearly in time (since by  \eqref{eq.phiduhamel2}: $\phi_{\omega,t}(\pm \Op)=t\, \rme^{\mp \rmi \Op t}$ for $\omega=\pm \Op$). This conclusion confirms the strong ill-posedness of the time-harmonic problem described in \cite{Bon-14,Bon-14(2),Ngu-16}. This linear growth in time corresponds  to a resonance phenomena. Such a phenomenon is classical for vibration problems in bounded domains but quite unusual for unbounded domains. In our case, the fields $\bbP_{\pm\Op}\bG_\omega$ are trapped waves which belongs to $\Hxy$ and are defined as (continuous) superpositions of functions  which are exponentially decaying with the distance to the interface (\textit{i.e.}, \emph{plasmonic waves}): they give birth to an \emph{interface resonance phenomenon}. The linear behavior in time is characteristic to a resonance due to an eigenvalue of the operator. Here, the eigenvalues $\pm\Op$ are of infinite multiplicities and embedded in the continuous spectrum of $\bbA$. This is a very interesting and new resonance phenomenon for transmission problem in a electromagnetic stratified media which does not occur with standard dielectric materials (see \cite{Wed-91}) since such non-zero eigenvalue of the Maxwell operator does not exist.

\subsection{Recap on the diagonalization of the Hamiltonian $\bbA$}
\label{ss.diag-Hamilt}
The proofs of both limiting absorption and limiting amplitude principles rely on the spectral analysis of the Schr\"{o}dinger equation \eqref{eq.schro} performed in \cite{Cas-Haz-Jol-17}. The final result of this previous paper is the explicit construction of a generalized Fourier transform $\bbF$ which diagonalizes the Hamiltonian operator $\bbA$, in the sense that $\bbF$ is a unitary transformation from the \emph{physical} space $\Hxy$ (defined in \eqref{eq.defHxy}) into a second Hilbert space $\hatH$, named the \emph{spectral} space, in which the Hamiltonian $\bbA$ takes a diagonal form. The introduction of this transformation yields in particular modal representations of both the resolvent of $\bbA$ and  the solution $\bU$  of the evolution equation \eqref{eq.schro}. Thus, it appears as a key tool for proving Theorems \ref{thm.limabs} and \ref{th.ampllim}. Therefore, the goal of this subsection is to recall all the results and notations from \cite{Cas-Haz-Jol-17} that are necessary to use this spectral tool.

\subsubsection{Spectral zones and generalized eigenfunctions}\label{sec-eigenfunc}
The generalized Fourier transform $\bbF$ is based on the knowledge of a family of time-harmonic solutions of our Schr\"{o}dinger equation, which are referred to as generalized eigenfunctions or generalized eigenmodes. These modes are non-zero bounded solutions for real $\lambda$'s of the equation $\bbA \, \mathbb{W}=\lambda\, \mathbb{W}$, which has to be understood in distributional sense since these solutions do not belong to $\Hxy$. Thanks to the stratified geometry, they are expressed here as separable functions of the variables $x$ and $y.$ They appear as superpositions of planes waves on each side of the interface $x=0$. 

In the family defined below, the generalized eigenfunctions functions are denoted by  $\bbW_{k,\lambda,j}$. They are indexed by three variables $k,$ $\lambda$ and $j$. The two first ones are real parameters: $k$ represents a wavenumber in the $y$-direction, that is, the direction of the interface, whereas $\lambda$ is a spectral parameter. The last one $j$ is an integer that indicates a multiplicity: its possible values depend on the pair $(k,\lambda)$. To make this precise, we first have to introduce various subsets of the $(k,\lambda)$-plane, called here \emph{spectral zones}, which correspond to various propagation regimes. In each of them, the set of possible values for $j$ will be constant.

The definition of the spectral zones is linked to the sign of the piecewise-constant function
\begin{equation}\label{eq.defThetalkfunction}
\Thetakl(x) := k^2-\varepsilon_{\lambda}(x)\, \mu_{\lambda}(x) \lambda^2.
\end{equation}
From \eqref{eq.defepsmu}, we have more explicitly
\begin{equation*}\label{eq.defTheta}
\Thetakl(x) = \left\lbrace
\begin{array}{ll} 
\Thetakl^{-} := k^2-\eps_0\,\mu_0 \,\lambda^2 & \mbox{if }x < 0,\\[4pt]
\Thetakl^{+} := \displaystyle - \frac{\eps_0\mu_0\, \lambda^4 - \left(k^2+\eps_0\mu_0(\Oe^2+\Om^2)\right) \lambda^2 + \eps_0\mu_0 \, \Oe^2\Om^2}{\lambda^2} & \mbox{if }x > 0.
\end{array}
\right.
\end{equation*}
Physically $|\Thetakl^{\pm}|$ represents the square of the wavenumber in the $x$-direction inside $\mathbb{R}^2_\pm$, for a plane wave of frequency $\lambda$ whose wavenumber in the $y$-direction is $k$. Notice first that $\Thetakl^{\pm} = \Theta_{|k|,|\lambda|}^\pm$ for all $(k,\lambda) \in \bbR^2,$ so that we can restrict ourselves to the quadrant $k \geq 0$ and $\lambda \geq 0.$ In this quadrant, there are three curves through which the sign of $\Thetakl^{-}$ or $\Thetakl^{+}$ changes. These curves, which are referred to as \emph{spectral cuts}, have been defined in \cite{Cas-Haz-Jol-17} as the graphs of three functions $\lambda_0(k),$ $\lambda_\scD(k)$ and $\lambda_\scI(k)$. In the present paper, we use instead their respective inverses $k_0(\lambda),$ $k_\scD(\lambda)$ and $k_\scI(\lambda)$ defined as follows:
\begin{equation} \label{defk0DI}
\begin{array}{ll}
\Thetakl^{-} = 0 & \Longleftrightarrow \quad |k| = k_0 (\lambda) :=\sqrt{ \eps_0 \, \mu_0} \, |\lambda|, \\ [8pt]
\Thetakl^{+} = 0 & \Longleftrightarrow \quad |k| = \left\{ \begin{array}{ll}
k_{\scD} (\lambda) := \sqrt{\eps_\lambda^+ \, \mu_\lambda^+} \, |\lambda|
& \text{if }  |\lambda| \geq \max(\Oe,\Om)\\
\quad\text{or} &  \\ 
k_{\scI} (\lambda) := \sqrt{\eps_\lambda^+ \, \mu_\lambda^+} \, |\lambda|
& \text{if }  0 < |\lambda| \leq \min(\Oe,\Om).
\end{array} \right.
\end{array}
\end{equation}
Note that $k_{\scD} (\lambda)$ and $k_{\scI} (\lambda)$ are given by the same formula but define two different curves since they differ by their domain of definition. The spectral cuts are represented in Figure \ref{fig.speczones1} in the cases $\Oe < \Om$, $\Oe = \Om$ and $\Oe > \Om$. The grey area represents the part of the quadrant where $\Thetakl^{-} < 0,$ which corresponds to the \emph{propagative} regime along the $x$-direction in the vacuum, whereas the white remaining sector corresponds to the \emph{evanescent} regime (that is, non propagative). Similarly, the hatched areas represent the parts of the quadrant where $\Thetakl^{+} < 0,$ that is, the propagative regime in the Drude material (again along the $x$-direction). In the area with vertical hatches, \emph{direct} propagation occurs, which means that the group and phase velocities of a plane wave have the same direction, as in vacuum. On the other hand, in the area with horizontal hatches, the propagation is called \emph{inverse}, since these velocities point in opposite  directions (which is related to the fact that both $\eps_\lambda^+$ and $\mu_\lambda^+$ are negative in this area, see \cite{Loh-09} or \cite[\S 3.3.2]{Cas-Haz-Jol-17} for more complete explanations). This justifies the use of the indices $\scD,$ $\scI$ and $\scE$, meaning respectively \emph{direct}, \emph{inverse} and \emph{evanescent}, to name the various spectral zones. Each of them is actually indexed by a pair of indices: the first one indicates the behavior in the vacuum ($\scD$ or $\scE$) and the second one, in the Drude material ($\scD$, $\scI$ or $\scE$). 
We thus define
\begin{equation*}
\begin{array}{ll}
\zDD & := \ \left\{(k,\lambda) \in \bbR^2 \mid \  
|\lambda| > \max(\Oe,\Om) \text{ and } |k| < k_\scD(\lambda) \, \right\},\\[4pt]
\zDI & := \ \left\{(k,\lambda) \in \bbR^2 \mid \ 
0 < |\lambda| < \min(\Oe,\Om) \text{ and } |k| < \min\big( k_0(\lambda),k_\scI(\lambda) \big) \, \right\}, \\[4pt]
\zEI & := \ \left\{(k,\lambda) \in \bbR^2 \mid \ 
0 < |\lambda| < \min(\Oe,\Om),\ k_0(\lambda) < |k| < k_\scI(\lambda) \, \right\}, \\[4pt]
\zDE & := \ \left\{(k,\lambda) \in \bbR^2 \mid \ 
|\lambda| \neq \Om  \text{ and } |k| < k_0(\lambda)\, \right\} 
\setminus \overline{\zDD \cup \zDI}.
\end{array}
\end{equation*}
In the following, the above sets will be referred as  {\em surfacic spectral zones}. The parts of these spectral zones located in the quadrant $\bbR^+ \times \bbR^+$ are represented in Figure \ref{fig.speczones1}. 
\begin{figure}\begin{minipage}{8.0cm}%
		\includegraphics[width=1.1\textwidth]{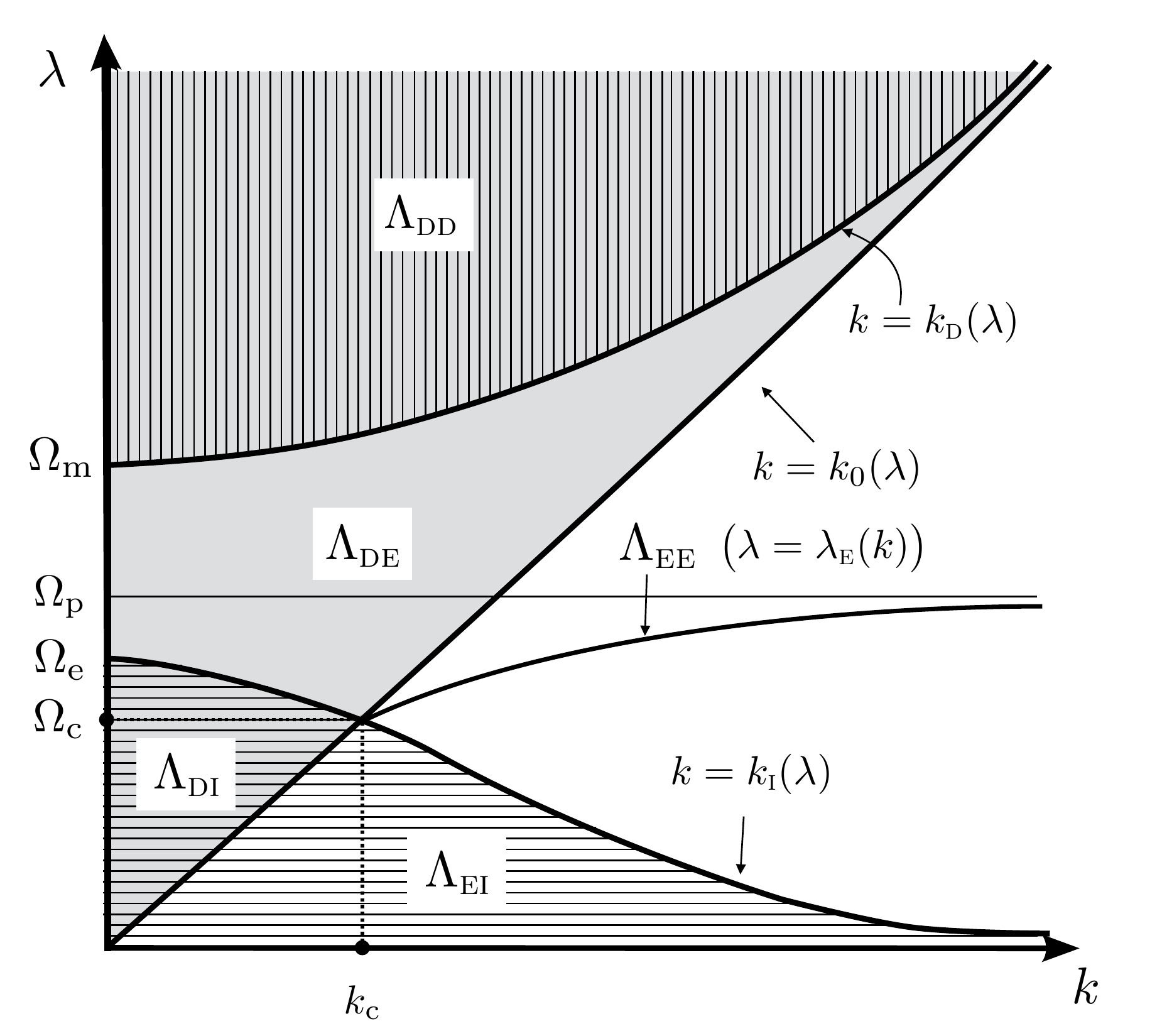}
	\end{minipage}
	\hspace{0.5cm}
	\begin{minipage}{8.0cm}%
		\includegraphics[width=1.1\textwidth]{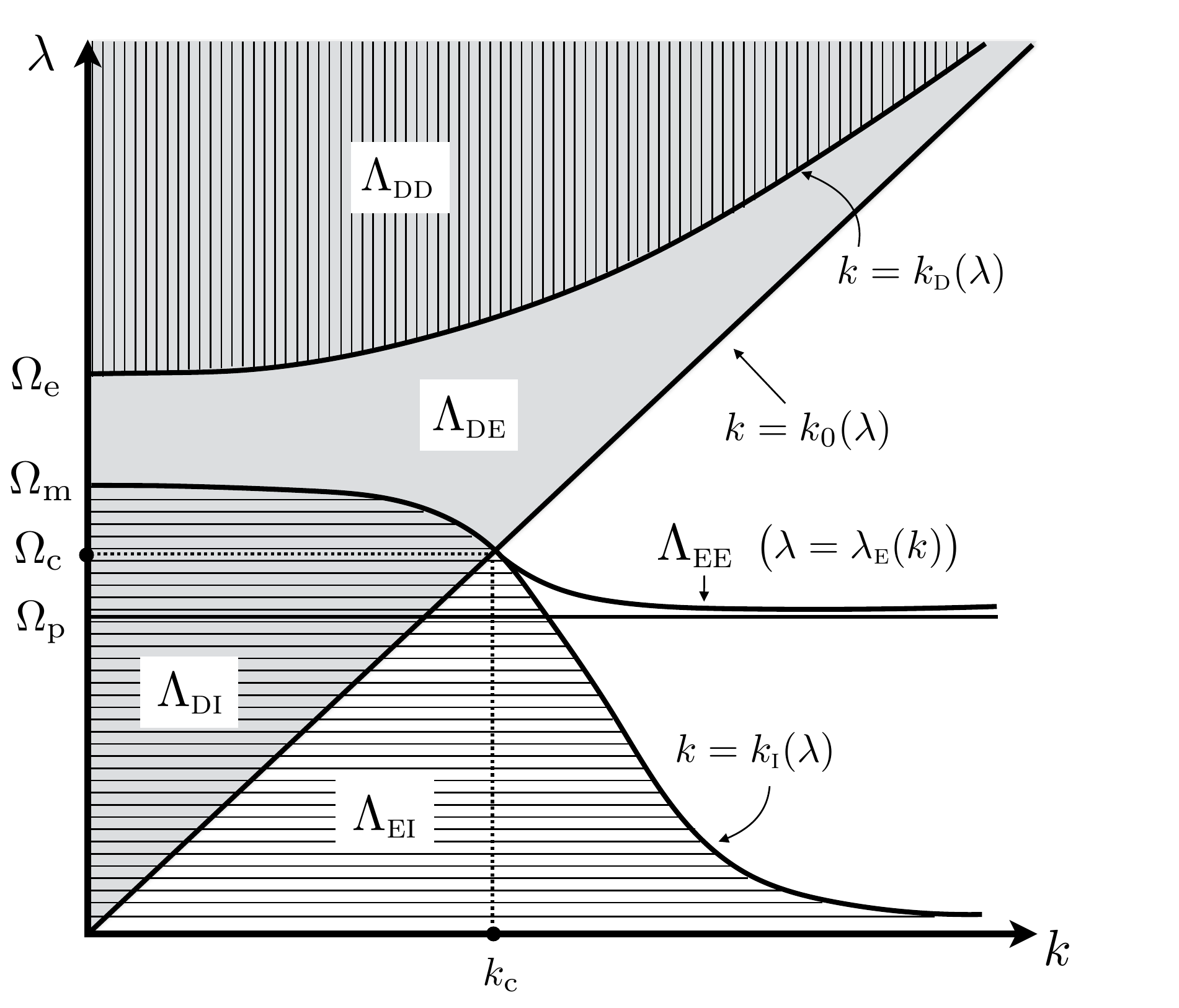}
	\end{minipage}
	\begin{center}
		\includegraphics[width=0.55\textwidth]{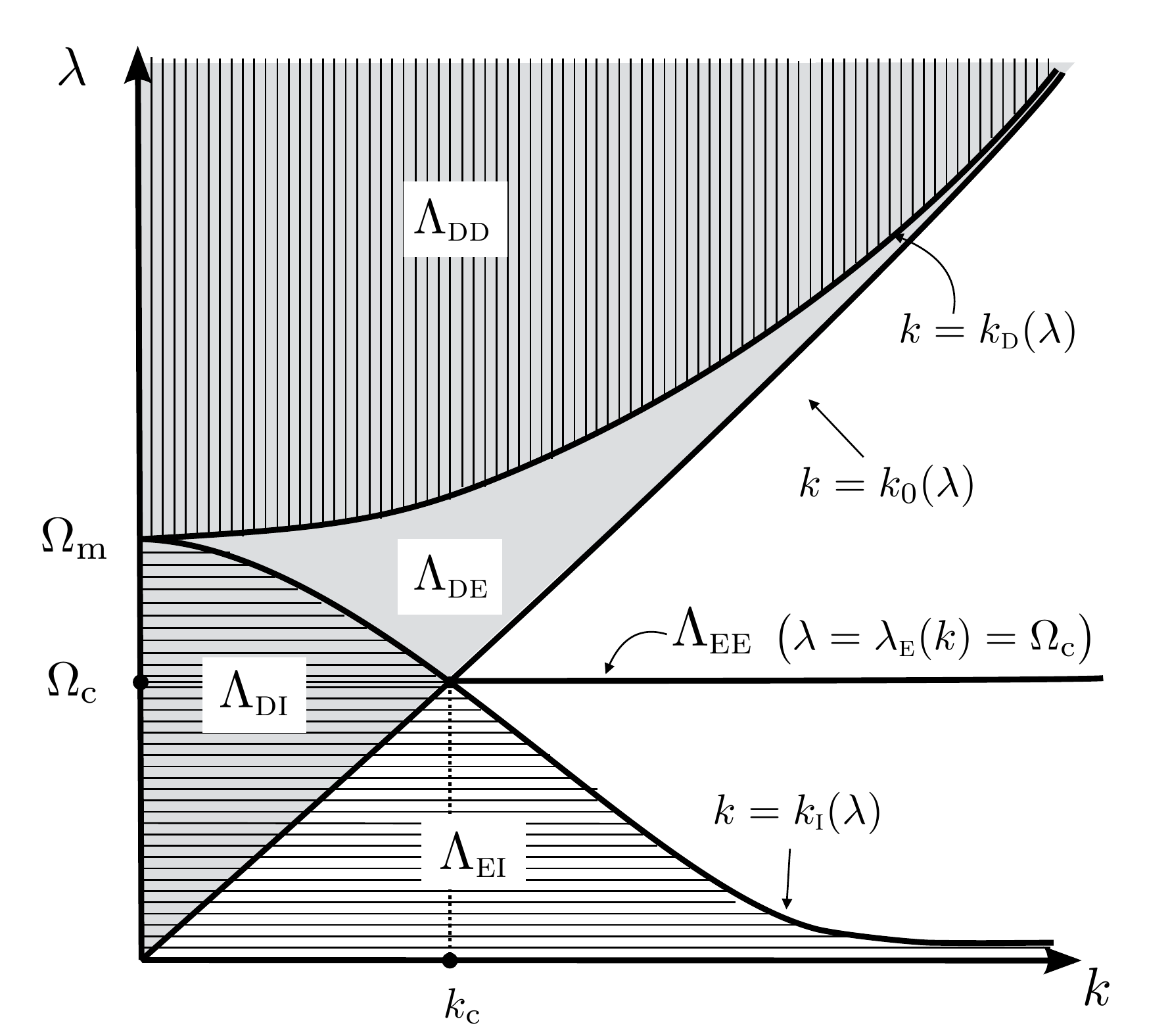}
		
	\end{center}
\caption{Parts of the spectral zones located in the quadrant $\bbR^+ \times \bbR^+$ for $\Oe<\Om$ (top left), $\Oe>\Om$ (top right) and  $\Oe=\Om$ (bottom).}
\label{fig.speczones1}
\end{figure}

The expression of the generalized eigenfunctions given below involves an appropriate square root $\thetakl^\pm$ of $\Thetakl^\pm$  that has the property to be either purely imaginary or positive real (the choice of the square root is justified by a limiting absorption process \cite[\S 3.3.1]{Cas-Haz-Jol-17}). We thus define
\begin{eqnarray}
\thetakl(x) & := &\thetakl^\pm \quad\text{if }\pm x > 0 \quad\text{where} \label{eq.defthetax} \\[8pt]
\thetakl^-  & := & \left\lbrace \begin{array}{ll} 
- \rmi \,\sgn(\lambda)\,|\Thetakl^-|^{1/2} & \mbox{if } (k,\lambda) \in \zDI \cup \zDE \cup \zDD,\\[4pt]
|\Thetakl^-|^{1/2} & \mbox{otherwise},
\end{array}\right.
\label{eq.def-thetam} \\[8pt]
\thetakl^+ & := & \left\lbrace \begin{array}{ll} 
+ \rmi \,\sgn(\lambda)\,|\Thetakl^+|^{1/2} & \mbox{if } (k,\lambda) \in \zEI \cup \zDI,\\[4pt]
- \rmi \,\sgn(\lambda)\,|\Thetakl^+|^{1/2} & \mbox{if } (k,\lambda) \in \zDD,\\[4pt]
|\Thetakl^+|^{1/2} & \mbox{otherwise}.
\end{array}\right.
\label{eq.def-thetap}
\end{eqnarray} 

We have to introduce a last spectral zone $\zEE $, which is associated to \emph{plasmonic waves}, {\it i.e.}, guided modes that are localized and propagates alongside the interface between both media  \cite{Mai-07}. Unlike the four other spectral zones which are surface areas, $\zEE$ is composed of one-dimensional curves which originate at the intersection points of the spectral cuts, called here the \emph{cross points}. These are the points where $\Thetakl^{-} = \Thetakl^{+} = 0,$ that is, the four points $(k,\lambda)$ such that $|k| = \kc$ and $|\lambda| = \Oc$, where $\kc =  k_0 (\Oc) = k_{\scI} (\Oc),$ which yields the definition \eqref{eq.def-Op-Oc} of $\Oc$, that is,
\begin{equation*}
\Oc := \frac{\Oe \, \Om}{\sqrt{\Oe^2+\Om^2}} 
\quad \text{and} \quad
\kc = \sqrt{\eps_0 \mu_0} \, \Oc.
\end{equation*}
The spectral zone $\zEE $ is composed of the solutions $(k,\lambda)$ of the following \emph{dispersion equation}:
\begin{equation}\label{eq.disp}
\calW_{k,\lambda} = 0 
\quad\text{where}\quad
\calW_{k,\lambda} := \frac{\thetakl^{-}}{\mu_{\lambda}^{-}}+ \frac{\thetakl^{+}}{\mu_{\lambda}^{+}} = 0.
\end{equation}
We know from \cite[Lemma 13]{Cas-Haz-Jol-17} that for a given $k$, this equation admits no solution if $|k|<\kc$, and two opposite solutions  $\pm\lambda_\scE(k)$  if $|k|\geq \kc$, where
\begin{equation}\label{eq.expressionlambdae}
\lambda_\scE(k) := 
\left\{ \begin{array}{ll}
\displaystyle \Om \ \sqrt{\frac{1}{2} + \frac{k^2}{K} - \sgn(K)\,\sqrt{\frac{1}{4} + \frac{k^4}{K^2}}} 
& \text{if } \Om \neq \Oe, \\
\displaystyle \Om / \sqrt{2} 
& \text{if } \Om = \Oe,
\end{array}\right.
\end{equation}
with $K := \eps_0 \mu_0 \,(\Om^2-\Oe^2).$ The function $k \mapsto \lambda_\scE(k)$ is strictly decreasing on $[\kc,+\infty)$ if $\Om < \Oe$ and strictly increasing if $\Om > \Oe$. Moreover $\lambda_\scE(k) = \Om / \sqrt{2} + O(k^{-2})$ as $|k| \to +\infty.$ In the case where $\Om \neq \Oe,$ we denote by $\kE$  the inverse of $\lambda_\scE$, originally defined for positive $\lambda$ and $k$ and extended to negative $\lambda$ by setting $\kE(-\lambda)=\kE(\lambda)$, that is,
\begin{equation} \label{defkE}
|\lambda| = \lambda_\scE(k) \quad \Longleftrightarrow \quad |k| = \kE (\lambda) 
\quad\text{if } |k| \in [\kc, +\infty) \text{ and } |\lambda| \in \lambda_\scE\big([\kc, +\infty)\big),
\end{equation}
where $\lambda_\scE\big((\kc, +\infty)\big) = 
\big(\min(\Op,\Oc),\max(\Op,\Oc)\,\big)$. We finally define 
\begin{equation*}\label{eq.defZEE}
\zEE := \left\{ (k,\lambda)\in \bbR^2 \mid |k| > \kc \text{ and } |\lambda| = \lambda_\scE(k) \right\}.
\end{equation*}
Since, it is a curve,  $\zEE$ will be referred  as \emph{the lineic spectral zone}.
Note that, for technical reasons which will appear later, we have excluded the cross points from this definition, although they are also solutions to \eqref{eq.disp}. In other words, $\overline{\zEE}$ yields all the solutions to \eqref{eq.disp}. Figure \ref{fig.speczones1} shows the location of $\zEE$ in the three cases $\Oe<\Om$, $\Oe=\Om$ and $\Oe>\Om$. 

We can now introduce the family of generalized eigenfunctions $\bbW_{k,\lambda,j}$ related to the various spectral zones $\zZ$ for
\begin{equation*}
\scZ \in \calZ := \{ \DD,\DE,\DI,\EI,\EE \}.
\end{equation*}
Before giving their mathematical expression, let us discuss their physical interpretation, which make clear our choice of possible values for the index $j$. Consider first the case of the surface zones, that is, $\scZ \in \calZ \setminus \{ \EE \}$. In this case, each $\bbW_{k,\lambda,j}$ represents an incident plane wave which scatters on the interface between both media and produces a reflected plane wave and a transmitted wave. In the half-plane where both incident and reflected waves coexist, the regime of vibration is necessarily propagative (direct or inverse) in the $x$-direction. On the other hand, in the half-plane where the transmitted wave occurs, the regime can be propagative or evanescent. This explains that for a given pair $(k,\lambda)$ in the spectral zones $\zDD$ and $\zDI$ where both half-planes are propagative, two generalized eigenfunctions $\bbW_{k,\lambda,j}$ are considered: they are indexed by $j=\pm 1$ which indicates the half-plane $\bbR^2_\pm$ where the transmitted wave takes place. Following the same interpretation, for a given pair $(k,\lambda)$ in the spectral zones $\zEI$ and $\zDE$, only one $\bbW_{k,\lambda,j}$ is considered, with $j=-1$ in $\zEI$ and $j=+1$ in $\zDE$. On the other hand, for the one-dimensional spectral zone $\zEE$, the regime is evanescent in both media. For a given pair $(k,\lambda) \in \zEE,$ only one $\bbW_{k,\lambda,j}$ which  represents now a guided wave that propagates along the interface is considered. Since there is no longer transmitted wave, we use the index $j=0$ in this case. Summing up, the set $J_\scZ$ of possible values of $j$ when $(k,\lambda) \in \zZ$ with $\scZ \in \{ \DD, \DE, \DI, \EI\}$ is given by
\begin{equation}\label{eq.def-Jz}
J_\scZ :=
\left\lbrace\begin{array}{ll}
\{ -1,+1 \} & \mbox{ if } \scZ = \DD \mbox{ or }\DI, \\
\{ +1 \}    & \mbox{ if } \scZ = \DE,\\
\{ -1 \}    & \mbox{ if } \scZ = \EI, \\
\{ 0 \}    & \mbox{ if } \scZ = \EE.
\end{array}
\right.
\end{equation}

The generalized eigenfunctions are then defined by
\begin{equation}\label{eq.def-W}
\forall \scZ \in \calZ,\ \forall (k,\lambda) \in \zZ,\ \forall j\in J_\scZ,\quad 
\bbW_{k,\lambda,j} := \bbV_\lambda\ w_{k,\lambda,j},
\end{equation}
where $\bbV_\lambda$ is a ``vectorizator'' in the sense that it expresses each $\bbW_{k,\lambda,j}$ in term of its first scalar component $w_{k,\lambda,j}$ (the component associated with the electrical field), via the formula
\begin{equation}\label{eq.def-Vk}
\bbV_\lambda \, w :=  
\left(  w \,,\, -\frac{\rmi}{\mu_{\lambda}\, \lambda}\,\bcurl w \,,\, \frac{\rmi\, \eps_0 \,\Oe^2 }{\lambda}\,\Rop \,w \,,\, \frac{\mu_0\, \Om^2 }{\mu_{\lambda}^+\, \lambda^2}\,\bR \,\bcurl w \right)^{\top}.
\end{equation}
The scalar function $w_{k,\lambda,j}$ is given by
\begin{equation}\label{eq.def-w}
w_{k,\lambda,j}(x,y) := A_{k,\lambda,j}\ \psi_{k,\lambda,j}(x)\ \rme^{\rmi k y}
\end{equation}
where the expressions of $A_{k,\lambda,j}$ and $\psi_{k,\lambda,j}(x)$ depend on the spectral zones. 
Note that, thanks to \eqref{eq-defcurls} and because of \eqref{eq.def-w}, \eqref{eq.def-W} can be rewritten as 
	\begin{equation}\label{eq.def-Vkbis}
	\bbW_{k,\lambda,j} =  
	\Big(  w_{k,\lambda,j} \,,\, \frac{k \, w_{k,\lambda,j}}{\mu_{\lambda}\, \lambda} \, ,
	\rmi \, \frac{\partial_x w_{k,\lambda,j}}{\mu_{\lambda}\, \lambda}, 
	 \frac{\rmi \eps_0 \,\Oe^2 }{\lambda}\,\Rop \,w_{k,\lambda,j} \,, \,
	\rmi  k \,  \frac{\mu_0\, \Om^2 }{\mu_{\lambda}^+\, \lambda^2}\, \Rop \,w_{k,\lambda,j}\,,
	\, - \frac{\mu_0\, \Om^2 }{\mu_{\lambda}^+\, \lambda^2}\, \Rop \,\partial_x w_{k,\lambda,j}
	\Big)^{\! \!\top}.
	\end{equation}
	
On the one hand, in the surface spectral zones $\zDD$, $\zDE$, $\zDI$ and $\zEI$, we have
\begin{align}
A_{k,\lambda,\pm 1} & := \frac{1}{\pi\,\left| \calW_{k,\lambda} \right|} \,
\left| \frac{\lambda}{2}\, \thetakl^\mp / \mu_\lambda^\mp \right|^{1/2} 
\quad\text{and} \label{def-A-gen} 
\\
\psi_{k,\lambda,\pm 1}(x) & := \cosh\big( \thetakl(x)\, x \big) \mp \frac{ \thetakl^\pm / \mu_\lambda^\pm }{ \thetakl(x) / \mu_\lambda(x)}\ \sinh\big( \thetakl(x)\, x \big), \label{def-psi-gen}
\end{align}
where $\calW_{k,\lambda}$ and $\thetakl(x)$ are defined respectively in \eqref{eq.disp} and \eqref{eq.defthetax}. Note that the latter expression of $\psi_{k,\lambda,\pm 1}$ can be rewritten equivalently
\begin{equation}\label{eq-def-phi}
\psi_{k,\lambda,\pm 1}(x) = 
\left\{ \begin{array}{ll}
\displaystyle \cosh\big( \thetakl^\mp\, x \big) \mp \frac{ \thetakl^\pm / \mu_\lambda^\pm }{ \thetakl^\mp / \mu_\lambda^\mp}\ \sinh\big( \thetakl^\mp\, x \big) & \text{if } \pm x \leq 0, \\
\exp\big( \mp \thetakl^\pm\, x \big) & \text{if } \pm x \geq 0,
\end{array} \right.
\end{equation}
which justifies the above-mentioned physical interpretation of the $\bbW_{k,\lambda,j}$.

On the other hand, in the plasmonic spectral zone $\zEE,$ we have
\begin{align}
A_{k,\lambda,0} & := \frac{\lambda^2\, \left|\mu_{\lambda}^{+} \, \thetakl^{+}\right|^{1/2}}{\sqrt{2\pi}\,\Om \big( 4k^4 +(\eps_0\mu_0)^2(\Oe^2-\Om^2)^2 \big)^{1/4}} \quad\text{and}  \label{def-A-plasm} \\[4pt]
\psi_{k,\lambda,0}(x) & := \exp\big( - \thetakl(x)\, |x| \big),   \label{def-psi-plasm}
\end{align}
which shows clearly that $\bbW_{k,\lambda,0}$ is a guided wave localized near the interface.

\begin{remark}
Let us mention that the notations introduced here are slightly different from \cite{Cas-Haz-Jol-17}, which results from the fact that the scalar function $w_{k,\lambda,j}$ defined in \eqref{eq.def-w} includes the contribution $\rme^{\rmi ky}$. As a consequence, the normalizing coefficient $A_{k,\lambda,j}$ differs from \cite{Cas-Haz-Jol-17} by a factor $\sqrt{2\pi}$. But of course, the $\bbW_{k,\lambda,j}$'s remain unchanged.
\end{remark}

\subsubsection{Generalized Fourier transform and  diagonalization theorem}
We introduce now the spectral space 
\begin{equation*}
\hatH := \bigoplus \limits_{\scZ \in \calZ} L^{2}(\zZ)^{\operatorname{card}(J_{\scZ})}=L^2(\zDD)^2 \oplus L^2(\zDE)\oplus L^2(\zDI)^2 \oplus  L^2(\zEI)\oplus   L^2(\zEE),
\label{eq.ident-hatH}
\end{equation*}
in which the action of the Hamiltonian $\bbA$ will be reduced to a simple multiplication by the spectral variable $\lambda$. This space is a direct sum of $L^2$ spaces of each spectral zone.  More precisely, each $L^2(\zZ)$ for $\scZ \in \calZ$  is repeated $\operatorname{card}(J_{\scZ})$ times, that is, the number of generalized eigenfunctions associated to the spectral zone $\zZ$. As we did for the $\bbW_{k,\lambda,j}$'s, we denote somewhat abusively by $\bhatU(k,\lambda,j)$ the fields of $\hatH$, where it is understood that the set $J_{\scZ}$ of possible values for $j$ depends on the spectral zone $\zZ$ to which the pair $(k,\lambda)$ belongs. Using these notations, the Hilbert space $\hatH$ is endowed with the following norm:
$$
\| \bhatU \|_{\hatH}^2 := \sum_{\scZ \in \calZ\setminus\{\EE\}}\sum_{j\in J_\scZ} \int_{ \zZ} |\bhatU(k,\lambda,j) |^2 \,\rmd \lambda \,\rmd k+  \sum_{\pm }\int_{|k|>k_{\rm c}}| \bhatU(k,\pm \lambda_\scE(k),0) |^2  \,\rmd k.
$$

Theorem \ref{th.diagA} below gathers the results of Theorem 20 and Proposition 21 in \cite{Cas-Haz-Jol-17}. It provides us the expression of the generalized Fourier transform $\bbF$ and its adjoint $\bbF^{*}$. The former appears as a ``decomposition'' operator on the family of generalized eigenfunctions $(\bbW_{k,\lambda,j})$, whereas the latter can be interpreted as a ``recomposition'' operator in the sense that its ``recomposes'' a function $\bU\in \Hxy$ from its spectral components $\bhatU(k,\lambda,j)\in \hatH$ which appear as ``coordinates'' on the ``generalized spectral basis'' $(\bbW_{k,\lambda,j})$. Both of these operators are (partial) isometries and thus bounded. So it is sufficient to know their expression on a dense subspace, exactly as for the usual Fourier transform or its inverse. For $\bbF$, the dense subspace of $\Hxy$ is $\Hps$ (with $s>1/2$) whereas for $\bbF^{*}$, we introduce below $\hatH_{\rm comp}$.

\begin{theorem}[Diagonalization Theorem, cf. \cite{Cas-Haz-Jol-17}]\label{th.diagA}
	Let $s > 1/2.$ 
	
	{\rm (i)} The generalized Fourier transform $\bbF: \Hxy \mapsto \hatH$ is a partial isometry, defined for all $\bU$ in  $\Hps$  by  
	\begin{equation}\label{eq.Four-gen}
	\forall \scZ \in \calZ,\ \forall (k,\lambda) \in \zZ,\ \forall j\in J_{\scZ}, \quad
	\bbF\bU(k,\lambda,j) =\langle  \bU, \bbW_{k,\lambda,j}  \rangle_{s},   
	\end{equation}
	where the $\bbW_{k,\lambda,j}$'s are defined in {\rm (\ref{eq.def-W})}.
	
	{\rm (ii)} Let $\hatH_{\rm comp}$ denote the dense subspace of $\hatH$ composed of compactly supported functions whose supports do not intersect the boundaries of the spectral zones $\zZ$ for $\scZ \in \{\DD,\DE,\DI,\EI\}$ (\textit{i.e.}, the spectral  cuts and the three lines $\bbR \times \{0,\pm\Om \}$). The adjoint $\bbF^*: \hatH \mapsto \Hxy$  of $\bbF$ is an isometry defined for all $\bhatU \in \hatH_{\rm comp}$ by
	\begin{equation}\label{eq.adj-Four-gen}
	\bbF^{*}\bhatU = \sum_{\scZ \in \calZ \setminus \{\EE \} } \sum_{j\in J_{\scZ}} \int_{\zZ } \bhatU(k,\lambda,j)\, \bbW_{k,\lambda, j} \,\rmd \lambda\, \rmd k 
	+ \sum_{\pm} \int_{|k|>k_{\rm c} } \bhatU(k,\pm \lambda_\scE(k) ,0)\, \bbW_{ k,\pm \lambda_\scE(k) ,0} \, \rmd k,
	\end{equation}
	where the integrals are understood as Bochner integrals with values in $\Hms$. 	
	
	{\rm (iii)} Furthermore, we have $ \bbF\,\bbF^{*} =\mathrm{Id}_{\hatH},
	$ while $ \bbF^{*} \bbF\,=\bbP_{\rm div0}$ where we recall that $\bbP_{\rm div0}$ is the orthogonal projector in $\Hxy$ onto $\Hxydiv$ (see \eqref{eq.Hxydiv}).
	In particular, the restriction of $\bbF$ to $\Hxydiv$ is a unitary operator. Furthermore $\bbF$ diagonalizes $\bbA$ in the sense that for any measurable function $f:\bbR \to \bbC$,
	\begin{equation}\label{eq.diagA}
	f(\bbA)\bbP_{\rm div0}=\bbP_{\rm div0} f(\bbA)=\bbF^{*}\,f(\lambda)\,\bbF \ \mbox{ in } \rmD(f(\bbA)).
	\end{equation}
\end{theorem}

\begin{remark} \label{rem.F}
	{\rm (i)} First notice that we use of duality product $\langle \cdot,\cdot \rangle_{s}$ (which extends the inner product of $\Hxy$, see \eqref{eq.innerproduct}) in the definition \eqref{eq.Four-gen} of $\bbF$. The reason is that the $\bbW_{k,\lambda,j}$'s do not belong to $\Hxy$ since their modulus does not decay at infinity (this is why they are called \emph{generalized eigenfunctions}). But the fact that they are bounded shows that they belong to $\Hms$ for any $s > 1/2$ (indeed $L^\infty(\bbR^2) \subset L^2_{-s}(\bbR^2)$ if and only if $s > 1/2$).

	{\rm (ii)} Let us now explain why we restrict ourselves to functions of $\hatH_{\rm comp}$ in \eqref{eq.adj-Four-gen}. First one can easily check that the $\Hms$-norm of $\bbW_{k,\lambda,j}$ remains uniformly bounded if $(k,\lambda)$ is restricted to vary in a compact set of $\bbR^2$ that does not intersect the boundaries of the spectral zones. Hence, for $\bhatU\in \hatH_{\rm comp}$, the integrals considered in \eqref{eq.adj-Four-gen}, whose integrands are valued in $\Hms$, are Bochner integrals {\rm \cite{Hil-96}} in $\Hms$. However, as $\bbF^{*}$ is bounded from $\hatH$ to $\Hxy,$ the values of these integrals belongs to $\Hxy$. The same holds true for all $\bhatU \in \hatH$ such that the integrands are integrable functions valued in $\Hms$.
	
	{\rm (iii)} In the general case, the integrands are not always integrable functions valued in $\Hms$, which  may happen for instance if $\bhatU$ does not vanish near some part of the boundaries of the spectral zones, because of the singular behavior of some $\bbW_{k,\lambda,j}$. For such a $\bhatU$, the integrals in \eqref{eq.adj-Four-gen} are no longer Bochner integrals in $\Hms$, but limits of Bochner integrals. Indeed, thanks to the density of $\hatH_{\rm comp}$ in $\hatH$, we can approximate $\bhatU$ by its restrictions to an increasing sequence of compact subsets of $\cup_{\scZ \in \calZ} \zZ$ as in the definition of $\hatH_{\rm comp}$, which yields an approximation of $\bbF^{*}\bhatU$. Of course, the limit we obtain belongs to $\Hxy$ and does not depend on the sequence. We will indicate this limiting process before each integral as follows:
	\begin{equation}\label{eq.adj-F-lim}
	\bbF^{*}\bhatU = \sum_{\scZ \in \calZ \setminus \{\EE \} } \sum_{j\in J_{\scZ}} \lim_{\Hxy}\int_{\zZ } \bhatU(k,\lambda,j)\, \bbW_{k,\lambda, j} \,\rmd \lambda\, \rmd k 
	+ \sum_{\pm} \lim_{\Hxy}\int_{|k|>k_{\rm c} } \bhatU(k,\pm \lambda_\scE(k) ,0)\, \bbW_{ k,\pm \lambda_\scE(k) ,0} \, \rmd k,
	\end{equation}
	for all $\bhatU \in \hatH$.
\end{remark}

\section{The spectral density}
\label{sec.spectral-density}
\subsection{Motivation and main results}\label{sec.motiv-main-results}
This quite technical section is somehow the keystone of the present paper. It provides us the basic ingredients for the proofs of Theorem \ref{thm.limabs} and \ref{th.ampllim}, which both mainly consist in using Theorem \ref{th.diagA} to investigate the asymptotic behavior of a family of functions of $\bbA.$ On the one hand, for the limiting absorption principle at a given frequency $\omega \in \bbR,$ we have to consider the families of functions $r_{\omega\pm\rmi\eta}: \bbR \to \bbC$ for $\eta > 0$ defined by
\begin{equation}
r_{\omega\pm\rmi\eta}(\lambda) := \frac{1}{\lambda - (\omega \pm \rmi\eta)}
\label{eq.def-r-eta}
\end{equation}
and we study the limits of the resolvent $R(\omega \pm \rmi\eta) = r_{\omega\pm\rmi\eta}(\bbA)$ as $\eta \searrow 0.$ On the other hand, for the limiting amplitude principle, we have to examine the behavior of $\phi_{\omega,t}(\bbA)$ as $t \to +\infty,$ where $\phi_{\omega,t}(\cdot)$ is defined in \eqref{eq.phiduhamel2}. As we will see in \S\ref{sec.limamplproof}, both limiting processes are intimately connected. We focus here on the former to explain the motivation of this section.

Roughly speaking, the basic idea is to rewrite the diagonal expression \eqref{eq.diagA} of $f(\bbA)$ as
\begin{equation}\label{eq.specdensity}
f(\bbA)\,\bbP_{\rm ac} = \int_{\bbR} f(\lambda) \, \bbM_{\lambda} \,\rmd \lambda,
\end{equation}
where $\lambda \mapsto \bbM_{\lambda}$ is a family of bounded operators from $\Hps$ to $\Hms$, and $\bbP_{\rm ac}$ is defined in \eqref{eq.def-Pac}. This formula can be interpreted as a continuous block diagonalization of $f(\bbA)\,\bbP_{\rm ac}$  where the diagonal blocks are the operators $\bbM_{\lambda}$. Using this formula for the functions defined in \eqref{eq.def-r-eta}, the absolutely continuous part of the resolvent of $\bbA$ (see \eqref{eq.def-Rac}) appears as a Cauchy integral
\begin{equation*}
R_{\rm ac}(\omega \pm \rmi\eta) := R(\omega \pm \rmi\eta)\, \bbP_{\rm ac} = 
\int_{\bbR} \frac{\bbM_{\lambda}}{\lambda - (\omega \pm \rmi\eta)} \,\rmd \lambda,
\end{equation*}
whose limits as $\eta \searrow 0$ will be given by a suitable version of the well-known Sokhotski--Plemelj formula \cite{Hen-86}, provided that $\lambda \mapsto \bbM_{\lambda}$ is locally H\"{o}lder continuous. This gives actually the main objectives of the present section which consists first in establishing \eqref{eq.specdensity}, then proving the local H\"{o}lder continuity of $\bbM_{\lambda}.$ These are the respective subjects of Theorems \ref{th.dens-spec} and \ref{th.Holder-dens-spec} below.

Formula \eqref{eq.specdensity} provides us a fundamental property of the spectral measure $\bbE$ (see \cite[\S 2.3]{Cas-Haz-Jol-17} for a brief reminder about this notion) of  $\bbA$, namely the fact that  in the orthogonal complement of the point subspace,  it is \emph{absolutely continuous}, which means that it is ``proportional'' to the Lebesgue measure (see Corollary \ref{cor.specmes}). This explains the terminology \emph{spectral density} for $\bbM_{\lambda},$ as well as the notation $\bbP_{\rm ac}.$\\

\textbf{Formal construction of $\bbM_{\lambda}$ in the non-critical case.} Let us first consider the case $\Oe \neq \Om$ for which the orthogonal projection $\bbP_{\rm ac} := \bbI - \bbP_{\rm pt}$ coincides with $\bbP_{\rm div0}$ (see \eqref{eq.def-Pac}). In order to prove \eqref{eq.specdensity}, we start from the diagonalization Theorem \ref{th.diagA} applied to the spectral measure of $\bbA:$ for any Borel set $S \subset \bbR,$ we have $\bbE(S) = \boldsymbol{1}_{S}(\bbA)$ where $\boldsymbol{1}_{S}$ denotes the indicator function of $S$. We assume here that
\begin{equation}
S \text{ is a bounded set such that } \overline{S} \cap \sigma_{\rm exc} = \varnothing,
\label{e.asumpt-S}
\end{equation}
where we recall that $\sigma_{\rm exc} := \{0,\pm\Op,\pm\Om\}$ in the non-critical case (see \eqref{eq.def-sigma-exc}). In other words, we exclude not only the eigenvalues $0$ and $\pm\Om,$ which shows in particular that
\begin{equation}
\bbE(S)\,\bbP_{\rm div0} = \bbE(S) = \bbE(S)\,\bbP_{\rm ac} 
\label{eq.Eac-non-crit}
\end{equation}
(since $\bbE(S)\,\bbP_{\rm div0} = \bbE(S)\,\bbE\big(\bbR \setminus \{ 0, \pm \Om \}\big) = \bbE\big(S \cap (\bbR \setminus \{ 0, \pm \Om \})\big) = \bbE(S)$), but also the plasmonic frequencies $\pm\Op$. Applying \eqref{eq.diagA} to $\boldsymbol{1}_{S}(\bbA)$ then yields
\begin{equation*}
\bbE(S)\,\bbP_{\rm div0} = \bbF^{*}\,\boldsymbol{1}_{S}(\lambda)\,\bbF.
\end{equation*}
Using the expressions \eqref{eq.Four-gen} and \eqref{eq.adj-F-lim} of $\bbF$ and $\bbF^{*},$ this formula writes more explicitly as
\begin{multline}
\label{eq.spect-rep-E}
\bbE(S)\,\bbP_{\rm div0} \, \bU = \sum_{\scZ \in \calZ\setminus\{\EE\}} \sum_{j \in J_{\scZ}} \lim_{\Hxy} \int_{\zZ } \boldsymbol{1}_{S}(\lambda) \, \langle  \bU, \bbW_{k,\lambda,j} \rangle_{s}  \; \bbW_{k,\lambda,j} \,\rmd\lambda\,\rmd k 
\\
+ \sum_{\pm } \lim_{\Hxy} \int_{|k|>k_{\rm c}} \boldsymbol{1}_{S}(\pm \lambda_\scE(k))\, \langle  \bU,  \bbW_{k,\pm \lambda_\scE(k),0}  \rangle_{s} \,   \bbW_{k,\pm \lambda_\scE(k),0} \,\rmd k,
\end{multline}
for all $\bU \in \Hps$, where we recall that the limit (in $\Hxy$) is obtained by considering an increasing sequence of compact subsets of each $\zZ$ whose union covers $\zZ$. We are going to see that thanks to assumption \eqref{e.asumpt-S}, on the one hand such a limiting process is useless here, and on the other hand, we can apply Fubini's theorem for the surface integrals on the $\zZ$ for $\scZ \in \calZ\setminus\{\EE\}$, as well as the change of variable $k = \pm k_\scE(\lambda)$ in the last integral. Admitting this provisionally and using \eqref{eq.Eac-non-crit}, we finally obtain that for all $\bU \in \Hps$,
\begin{equation}
\label{eq.expr-Eac-non-crit}
\bbE(S)\,\bbP_{\rm ac} \, \bU = \bbE(S) \, \bU = \int_\bbR \boldsymbol{1}_{S}(\lambda) \, \bbM_\lambda \bU \,\rmd\lambda
\quad\text{with}
\end{equation}
\begin{equation}
\label{eq.density-non-crit}
\bbM_\lambda \bU := \sum_{\scZ \in \calZ\setminus\{\EE\}} \sum_{j \in J_{\scZ}} \int_{\zZ(\lambda)} \langle \bU,\bbW_{k,\lambda,j} \rangle_{s}  \; \bbW_{k,\lambda,j} \,\rmd k 
+ \sum_{k \in \zEE(\lambda)} \JacE(\lambda)\ \langle  \bU,  \bbW_{k,\lambda,0}  \rangle_{s} \,   \bbW_{k,\lambda,0},
\end{equation}
for almost every $\lambda \in \bbR,$ where $\JacE(\lambda)$ is the Jacobian of the change of variable $k= \pm k_\scE(\lambda)$ defined by
$$
\JacE(\lambda) := \big| \kE'(\lambda)\big| = \Big|\frac{\rmd \lambda_\scE }{\rmd k} \big(\kE(\lambda)\big)\Big|^{-1} 
$$  
and $\zZ(\lambda)$ denotes the set of $k \in \bbR$ corresponding to the horizontal section of $\zZ$ at the ``height'' $\lambda,$ \textit{i.e.,}
\begin{equation}
\zZ(\lambda) := \left\{ k \in \bbR \mid (k,\lambda) \in \zZ \right\}.
\label{eq.def-section}
\end{equation}
A glance at Figure \ref{fig.speczones1} clearly shows that if $\scZ \in \calZ\setminus\{\EE\}$,  then $\zZ(\lambda)$ is either empty (in this case the corresponding integral vanishes) or a bounded set composed of one or two intervals. For instance, if $\lambda > \max(\Oe,\Om),$ then 
$\zDD(\lambda) = \big( -k_\scD(\lambda),+k_\scD(\lambda) \big)$ whereas 
$\zDE(\lambda) = \big( -k_0(\lambda),-k_\scD(\lambda) \big) \cup 
\big( +k_\scD(\lambda),+k_0(\lambda) \big)$. Moreover, we have
\begin{equation*}
\zEE(\lambda) = \left\{
\begin{array}{ll}
\{ \pm k_\scE(\lambda) \} & \text{if } |\lambda| \in \lambda_\scE\big((\kc, +\infty)\big) = \big(\min(\Op,\Oc),\max(\Op,\Oc)\big),  \\[5pt]
\varnothing & \text{otherwise,}
\end{array}
\right.
\end{equation*}
which shows that the last term in \eqref{eq.density-non-crit} appears only if $|\lambda| \in \lambda_\scE\big((\kc, +\infty)\big).$ \\

\textbf{The critical case.} What about the case $\Oe = \Om$? We keep the assumption \eqref{e.asumpt-S} for $S$, but now $\sigma_{\rm exc} := \{0,\pm\Om\}$ (see \eqref{eq.def-sigma-exc}), so that \eqref{eq.Eac-non-crit} is no longer true. From \eqref{eq.def-Pac}, it has to be replaced by
\begin{equation}
\bbE(S)\,\bbP_{\rm div0} = \bbE(S) = \bbE(S)\,\bbP_{\rm ac} + \bbE\big(S)\,\bbP_{-\Op} + \bbE\big(S)\,\bbP_{+\Op}.
\label{eq.Eac-crit}
\end{equation}
Formula \eqref{eq.spect-rep-E} is still valid. The difference with the non-critical case lies in the last integrals for which it is no longer possible to make the change of variable $k = \pm k_\scE(\lambda)$ since $\lambda_\scE(k) = \Op$ for all $|k| > k_{\rm c}$. These integrals represent exactly the quantities $\bbE\big(S)\,\bbP_{\pm\Op}\, \bU$ related to the eigenvalues $\pm\Om$ of infinite multiplicity. We have actually
\begin{equation*}
\bbP_{\pm\Op} \, \bU 
= \lim_{\Hxy} \int_{|k|>k_{\rm c}} \langle  \bU,  \bbW_{k,\pm \Op,0}  \rangle_{s} \,   \bbW_{k,\pm \Op,0} \,\rmd k,
\end{equation*}
for all $\bU \in \Hps$. Formula \eqref{eq.Eac-crit} then shows that the last integrals in 
\eqref{eq.spect-rep-E} have to be removed to express $\bbE(S)\,\bbP_{\rm ac}.$ Using the same arguments as above for the surface integrals on $\zZ$ for $\scZ \in \calZ\setminus\{\EE\}$, we obtain instead of \eqref{eq.expr-Eac-non-crit}-\eqref{eq.density-non-crit}
\begin{equation}
\label{eq.expr-Eac-crit}
\bbE(S)\,\bbP_{\rm ac} \, \bU = \bbE(S\setminus\{\pm\Op\}) \, \bU = \int_\bbR \boldsymbol{1}_{S}(\lambda) \, \bbM_\lambda \bU \,\rmd\lambda
\quad\text{with}
\end{equation}
\begin{equation}
\label{eq.density-crit}
\bbM_\lambda \bU := \sum_{\scZ \in \calZ\setminus\{\EE\}} \sum_{j \in J_{\scZ}} \int_{\zZ(\lambda)} \langle \bU,\bbW_{k,\lambda,j} \rangle_{s}  \; \bbW_{k,\lambda,j} \,\rmd k.
\end{equation}

\textbf{Main results.} The properties of the spectral density are stated in the two following theorems, which are proved in the remainder of this section.

\begin{theorem}
	Let $s > 1/2.$ For every bounded function $f: \bbR \to \bbC$ with a compact support that does not contain any point of $\sigma_{\rm exc}$ (see \eqref{eq.def-sigma-exc}), the operator $f(\bbA)\,\bbP_{\rm ac}$ is given by
	\begin{equation*}
	f(\bbA)\,\bbP_{\rm ac} = \int_{\bbR} f(\lambda) \, \bbM_{\lambda} \,\rmd \lambda,
	\end{equation*}
	where the spectral density $\bbM_\lambda$ is defined for all $\lambda \in \bbR \setminus \sigma_{\rm exc}$ as a bounded operator from $\Hps$ to $\Hms$ by \eqref{eq.density-non-crit} if $\Oe \neq \Om$ and by \eqref{eq.density-crit} if $\Oe = \Om.$ The above integral is understood as a Bochner integral in $B(\Hps,\Hms)$.
	\label{th.dens-spec}
\end{theorem}

Note that if $f$ is a bounded function whose support $S$ is no longer compact and/or contains points of $\sigma_{\rm exc},$ the expression of $f(\bbA)\,\bbP_{\rm ac}$ follows from Theorem \ref{th.dens-spec} by considering an increasing sequence $(S_n)$ of compacts subsets of $S \setminus \sigma_{\rm exc}$ whose union covers this set. Setting $f_n := f\,\boldsymbol{1}_{S_n},$ Theorem \ref{th.diagA} shows that
\begin{equation*}
\Big\| \big(f(\bbA) - f_n(\bbA)\big)\,\bbP_{\rm ac}\bU \Big\|_{\Hxy} =
\left\{ \begin{array}{ll}
\Big\| \big(f(\lambda) - f_n(\lambda)\big)\,\bbF\bU \Big\|_{\hatH} & \text{if }\Oe \neq \Om, \\[8pt]
\Big\| \big(f(\lambda)\boldsymbol{1}_{\bbR\setminus\{\pm\Op\}} - f_n(\lambda)\big)\,\bbF\bU \Big\|_{\hatH} & \text{if }\Oe = \Om,
\end{array}
\right.
\end{equation*}
which tends to 0 by the Lebesgue dominated convergence theorem. Hence, using the same notation as in \eqref{eq.adj-F-lim}, we have 
\begin{equation*}
\forall \bU \in \Hps, \quad
f(\bbA)\,\bbP_{\rm ac}\bU = \lim_{\Hxy} \int_{\bbR} f(\lambda) \, \bbM_{\lambda}\bU \,\rmd \lambda,
\end{equation*}
that we rewrite in the condensed form
\begin{equation}
f(\bbA)\,\bbP_{\rm ac} = \slim_{B(\Hps,\Hxy)}\ \int_{\bbR} f(\lambda) \, \bbM_{\lambda} \,\rmd \lambda,
\label{eq.calc-fonct-ac-lim}
\end{equation}
where ``$\slim$'' means that the limit is taken for the strong operator topology of $B(\Hps,\Hxy)$.

\begin{theorem}
	Let $s > 1/2.$	The spectral density $\lambda \mapsto \bbM_\lambda \in B(\Hps,\Hms)$ is locally H\"{o}lder-continuous on $\bbR \setminus \sigma_{\rm exc}$. More precisely, let  $[a,b]$ be an interval of $ \bbR \setminus \sigma_{\rm exc}$ and $\Gamma_{[a,b]} \subset (0,1)$ be the set of H\"{o}lder exponents defined by \eqref{eq.def-Gamma-K} for  $K=[a,b]$. Then for any $\gamma \in \Gamma_{[a,b]}$, there exists a constant $C_{a,b}^{\gamma}>0$ such that
	\begin{equation}\label{eq.holderestim-specdensity}
	\forall \lambda', \, \lambda \in [a,b], \quad
	\big\| \bbM_{\lambda'} - \bbM_\lambda \big\|_{\Hps,\Hms} \leq C_{a,b}^{\gamma} \ |\lambda'-\lambda|^{\gamma}.
	\end{equation}
	\label{th.Holder-dens-spec}
\end{theorem}

\begin{remark}
Let us mention that the formulation of Theorem \ref{th.Holder-dens-spec} is not entirely optimal in the sense that it holds true for two particular values of $\gamma$ which are not contained in the definition \eqref{eq.def-Gamma-K} of $\Gamma_{[a,b]}$.

On the one hand, in the case where $[a,b] \cap \{ \pm\Oe,\pm\Oc \} = \varnothing,$ the value $\gamma = 1$ can be included in $\Gamma_{[a,b]}$, provided that $s > 3/2$. The proof of Theorem \ref{th.Holder-dens-spec} presented in the following actually includes this particular case. However, as the proof of the limiting absorption principle (Theorem \ref{thm.limabs}) is no longer valid for this particular value (see \S\ref{sec-thabslproof}), we keep the same definition of $\Gamma_{[a,b]}$ here.

On the other hand, in the case where $[a,b]$ contains $+\Oe$ or $-\Oe$ but not $\pm\Oc$, the value $\gamma = 1/2$ is allowed, provided that $s > 1$. As mentioned in Remark \ref{rem.cas-un-demi}, this particular value has been excluded for the sake of clarity, but we show in Appendix \ref{app.cas-un-demi} how to deal with this case.
\label{rem.excluded}
\end{remark}

The above theorems are based on properties of the generalized eigenfunctions which are studied in the next subsections. These properties will be established in bounded ``slices'' of the spectral zones $\Lambda_\scZ$ defined for any closed interval $[a,b] \subset \bbR$ by
\begin{equation}
\Lambda_\scZ([a,b]) 
:= \Lambda_\scZ \cap \big( \bbR \times [a,b] \big)
= \bigcup_{\lambda\in [a,b]} \{(k,\lambda) \mid k \in \Lambda_{\scZ}(\lambda)\},
\label{eq.defLambdaspeczoneab}
\end{equation}
where $\Lambda_{\scZ}(\lambda)$ is defined in \eqref{eq.def-section}. In particular, we are able to show now that Theorem \ref{th.dens-spec} follows from the following Proposition which is proved in \S\ref{sec_defMlambda}. 

\begin{proposition}
	Let $s>1/2$ and $[a,b] \subset \bbR\setminus\sigma_{\rm exc}$. 
	\begin{enumerate}
		\item If $\scZ \in \calZ\setminus\{\EE\}$ and $\Lambda_{\scZ}([a,b]) \neq \varnothing$, then for $j\in J_{\scZ}$, the map $(k, \lambda) \mapsto  \| \bbW_{k,\lambda,j} \|_{\Hms}$ is square integrable in $\Lambda_{\scZ}([a,b])$.
		\item In the non-critical case $\Oe\neq \Om$, if $\Lambda_{\EE}([a,b]) \neq \varnothing$, then the map $(k, \lambda) \mapsto \| \bbW_{k,\lambda,0} \|_{\Hms}$ is bounded on $\Lambda_{\EE}([a,b])$.
	\end{enumerate}
	\label{p.estim-fpg}
\end{proposition}

\begin{proof}[Proof of Theorem \ref{th.dens-spec}]
The properties of Proposition \ref{p.estim-fpg} allow us to justify the path from \eqref{eq.spect-rep-E} to \eqref{eq.expr-Eac-non-crit}-\eqref{eq.density-non-crit} if $\Oe \neq \Om$ and to \eqref{eq.expr-Eac-crit}-\eqref{eq.density-crit} if $\Oe = \Om$. Indeed, thanks to assumption \eqref{e.asumpt-S}, this lemma tells us that the functions involved in the surface integrals in \eqref{eq.spect-rep-E} are integrable functions valued in $\Hms$. More precisely, for all $\scZ \in \calZ\setminus\{\EE\}$ and $\bU\in \Hps$, the map $(k, \lambda) \mapsto  \boldsymbol{1}_{S}(\lambda) \, \langle  \bU, \bbW_{k,\lambda,j} \rangle_{s}  \; \bbW_{k,\lambda,j}$ belongs to $L^1 \big(\Lambda_{\scZ},\Hms\big)$ since
\begin{equation*}
\left\| \int_{\zZ } \boldsymbol{1}_{S}(\lambda) \, \langle  \bU, \bbW_{k,\lambda,j} \rangle_{s}  \; \bbW_{k,\lambda,j} \,\rmd\lambda\,\rmd k \right\|_{\Hms}
\leq  \|\bU \|_{\Hps} \ \int _{\zZ} \, \boldsymbol{1}_{S}(\lambda)\,\|\bbW_{k,\lambda,j} \|_{\Hms}^2\,\rmd\lambda\,\rmd k.  
\end{equation*}
On the one hand, this shows that the limiting process in $\Hxy$ is useless (by the Lebesgue's dominated convergence theorem for Bochner integrals \cite[Theorem 3.7.9]{Hil-96}). On the other hand, this justifies the use of Fubini's theorem \cite[Theorem 3.7.13]{Hil-96}, which tells us in particular that the integrals on $\zZ(\lambda)$ in \eqref{eq.density-non-crit} or \eqref{eq.density-crit} are defined for almost every $\lambda.$ Noticing that Proposition \ref{p.estim-fpg} holds true for $a=b$, we see that these integrals are actually defined for all $\lambda \notin \sigma_{\rm exc}.$

In the non-critical case, it remains to deal with the last integrals in \eqref{eq.spect-rep-E}, related to the 1D spectral zone $\zEE$. This is where we use the fact that  $\overline{S} \cap \{\pm\Op\} = \varnothing$ (contained in assumption \eqref{e.asumpt-S}), which implies that the support of $k \mapsto \boldsymbol{1}_{S}(\pm \lambda_\scE(k))$ is bounded. In other words, the integrals actually cover a bounded part of $\zEE$. Hence Proposition \ref{p.estim-fpg} tells us that the functions involved in these integrals are integrable functions valued in $\Hms$, which shows again that the limit process is useless and allows the change of variable which yields \eqref{eq.density-non-crit}.

To conclude, we simply have to notice that all the above arguments remain valid if, instead of the spectral measure $\bbE(S),$ one considers the spectral representation of $f(\bbA)\,\bbP_{\rm ac}$ where $f$ is a bounded function with compact support $S$ that satisfies \eqref{e.asumpt-S}.
\end{proof}

The following corollary of Theorem \ref{th.dens-spec} shows that outside the eigenvalues of $\bbA,$ the spectrum of $\bbA$ is absolutely continuous.

\begin{corollary}\label{cor.specmes}
	The spectral measure of $\bbA$ satisfies
	\begin{equation}\label{eq.absolutecontinuousmes2}
	\forall \bU, \bV \in \Hps,\quad 
	\rmd\big(\bbE(\lambda)\bbP_{\rm ac} \bU, \bbP_{\rm ac} \bV\big)_{\Hxy} =
	\rmd\big(\bbE(\lambda)\bbP_{\rm ac} \bU, \bV\big)_{\Hxy} =  \langle\bbM_{\lambda}\bU,\bV\rangle_s \,\rmd\lambda.
	\end{equation}
	Moreover, for any Borel set $S \subset \bbR$ (bounded or not), we have
	\begin{equation}
	\forall \bU \in \Hps,\quad 
	\left\| \bbE(S)\,\bbP_{\rm ac} \, \bU \right\|_{\Hxy}^2 
	= \int_\bbR \boldsymbol{1}_{S}(\lambda) \, \langle\bbM_{\lambda} \bU, \bU \rangle_s\,\rmd\lambda,
	\label{eq.norm-Eac}
	\end{equation}
	where $\lambda \mapsto \langle\bbM_{\lambda} \bU, \bU \rangle_s$ is non-negative and integrable on $\bbR$.
\end{corollary}

\begin{proof}
	By virtue of \eqref{eq.expr-Eac-non-crit} and \eqref{eq.expr-Eac-crit}, for any bounded Borel set $S \subset \bbR$ such that  $\overline{S} \cap \sigma_{\rm exc} = \varnothing$, we have
	\begin{equation*}
	\forall \bU, \bV \in \Hps,\quad 
	\big( \bbE(S)\,\bbP_{\rm ac} \, \bU, \bV \big)_{\Hxy} = 
	\left\langle \int_\bbR \boldsymbol{1}_{S}(\lambda) \, \bbM_\lambda \bU \,\rmd\lambda\,, \bV \right\rangle_s
	\end{equation*}
	As the above integral is Bochner, we can permute it with the duality product \cite[Theorem 3.7.12]{Hil-96}, which yields
	\begin{equation}
	\forall \bU, \bV \in \Hps,\quad 
	\big( \bbE(S)\,\bbP_{\rm ac} \, \bU, \bV \big)_{\Hxy} = 
	\int_\bbR \boldsymbol{1}_{S}(\lambda) \, \langle \bbM_\lambda \bU, \bV \rangle_s \,\rmd\lambda.
	\label{eq.expres-Eac}
	\end{equation}
	
	Besides, we have $\bbE(\sigma_{\rm exc})\,\bbP_{\rm ac} = 0$. Indeed, from  \eqref{eq.def-Pac}, we see that $\bbE(\sigma_{\rm exc})\,\bbP_{\rm ac} = \bbE(\sigma_{\rm exc}\setminus \sigma_{\rm pt}(\bbA))$, where \eqref{eq.def-sigma-exc} and \eqref{eq.specpt} show that $\sigma_{\rm exc}\setminus \sigma_{\rm pt}(\bbA)$ is equal to $\{\pm\Op\}$ for $\Oe \neq \Om$, whereas it is empty  for $\Oe = \Om$ (which implies that  $\bbE\big(\sigma_{\rm exc}\setminus \sigma_{\rm pt}(\bbA)\big)=0$). For $\Oe \neq\Om$, $\{\pm\Op\}$ has zero Lebesgue's measure and does not contain eigenvalues, thus we have also $\bbE(\{\pm\Op\}) = 0$ and the conclusion follows. This shows by sigma-additivity that \eqref{eq.expres-Eac} actually holds true if $\overline{S} \cap \sigma_{\rm exc} \neq \varnothing$, thus for any bounded Borel set $S$, which amounts to the second equality of \eqref{eq.absolutecontinuousmes2}.
	
	The first equality simply follows from the fact that $ \bbP_{\rm ac}$ is an  orthogonal projection which commutes with $\bbE(S)$.
	
	Finally, if we choose $\bV = \bU$ in \eqref{eq.expres-Eac}, we obtain \eqref{eq.norm-Eac} for $S$ bounded. The fact that it holds true for unbounded $S$ follows from the spectral theorem which ensures that for all $\bU \in \Hxy,$ the map $S \mapsto (E(S)\bU,\bU)_{\Hxy}$ defines a non-negative finite measure.
\end{proof}

In what follows, in order to avoid the appearance of non meaningful constants in inequalities in  the {\bf proofs} of the estimates involved in this paper, we use the symbols $\lesssim$ and $\gtrsim$. More precisely, this will be employed for non-negative functions  $f_{\gamma}$ and $g_{\gamma}$ of the real variable $\lambda$, that may depend on a parameter $\gamma\in (0,1]$. By definition, one has:
$$
	f_{\gamma}(\lambda) \lesssim g_{\gamma}(\lambda) \mbox{ on } [a,b] \quad \Longleftrightarrow \quad \exists \ C_{a,b}^{\gamma}>0, \ \forall  \lambda\in[a,b], \quad f_{\gamma}(\lambda) \leq C_{a,b}^{\gamma}  \  g_{\gamma}(\lambda),.
$$
However, for clarity purposes, we decide to keep these constants explicit  in the {\bf statements} of the results.

\subsection{Pointwise estimates of generalized eigenfunctions} \label{sec_EF-estimates}

In this section, we  establish pointwise (in $(k, \lambda) \in \Lambda_\scZ$)) estimates of $\bbW_{k,\lambda,j}$ for $j \in J_\scZ$ and $\scZ \in \calZ$, in the space $\Hms$ for $s>1/2$,  essentially based on the continuous embedding $ L^\infty(\mathbb{R}^2) \subset L^2_{-s}(\mathbb{R}^2)$ which relies on the following inequality:
\begin{equation}\label{Linf-L-s}
 \|\phi \|_{L^2_{-s}(\mathbb{R}^2)}\leq C_s \, \|\phi\|_{ L^\infty(\mathbb{R}^2)}, \  \forall \phi \in L^{\infty}(\mathbb{R}^2) \quad  \mbox{ with } C_s=\Big(\int_{\mathbb{R}^2} \eta_{-s}(x,y)^2 \rmd x \, \rmd y\Big)^{\frac{1}{2}}<+\infty \mbox{ for } s>\frac{1}{2}.
 \end{equation}
According to the formula  \eqref{eq.def-Vkbis}  for $\bbW_{k,\lambda,j}$, this  relies on  estimates on the 2D scalar functions $w_{k,\lambda,j}$ \eqref{eq.def-w}.

\subsubsection{Generalized eigenfunctions of surface spectral zones}
We deal first with the case $  \scZ \in \calZ \setminus \{\EE\}$ and their associated generalized eigenfunctions $\bbW_{k,\lambda,j}$ for $j=\pm 1 \in J_{\scZ}$ (often referred as bulk modes in physics).
In that case, the estimates will be used in integrals along the $\Lambda_{\scZ}(\lambda)$, see \eqref{eq.density-non-crit} and  \eqref{eq.density-crit}, where the integrand is quadratic in the generalized
eigenfunctions $\bbW_{k,\lambda,j}$. Thus, we need estimates of $\bbW_{k,\lambda,j}$  which are {\bf square integrable} over the variable $k$ in each set $\Lambda_{\scZ}(\lambda)$. The pointwise upper bounds at $(k, \lambda)$ will depend on the zone $\Lambda_\scZ$ that contains $(k, \lambda)$ and may blow up when the
will approach  boundary of each spectral zone.  This is more or less clear from the expressions of the  $\bbW_{k,\lambda,j}$ (\eqref{eq.def-W}, \eqref{eq.def-Vk}): these estimates must take care of the presence of negative powers of the functions $\thetakl^\pm$ (see in particular \eqref{eq.def-w}) that precisely vanish on the boundary of the spectral zones. 

More precisely, in our estimates, negative powers of $|\thetakl^{+}|$ and $|\thetakl^{-}|$ can be accepted, but not too large in order to respect the square integrability criterion in $k$ over each $\Lambda_{\scZ}(\lambda)$ (it will be discussed in more details in Lemma \ref{eq.lemsingularitythetalk} and  relation \eqref{L2integrability}). 

The following proposition provides $\Hms$-estimate  for the surface  spectral zones. It is preceded by a preliminary lemma which gives a useful estimate on the Wronskian $\calW_{k,\lambda}$ that appears in the expression of the  generalized eigenfunctions (see \eqref{eq.def-w} and \eqref{eq.def-Vkbis}).

\begin{lemma}\label{lem.boundwronsk}
	Let  $\scZ \in \calZ\setminus\{\EE\}$ and $[a,b] \subset \bbR\setminus\sigma_{\rm exc}$ such that $\Lambda_{\scZ}([a,b]) \neq \varnothing$. Then, there  exists $C_{a,b}>0$ such that:
	\begin{equation}\label{eq.ineqA}
	|\calW_{k,\lambda} |^{-1} \leq C_{a,b}  \, (|\thetakl^+|+|\thetakl^-|) ^{-1}, \   \, \forall (k, \lambda)\in\Lambda_{\scZ}([a,b]).
	\end{equation}
	
\end{lemma}
\begin{proof}
The proof of the estimate  \eqref{eq.ineqA} depends on the spectral zone. More precisely:
\begin{itemize}
		\item  By virtue of (\ref{eq.def-thetam}, \ref{eq.def-thetap}), one has  $\thetakl^{-}/\mu_{\lambda}^{-} \in \rmi \, \mathbb{R}$ and $\thetakl^{+}/\mu_{\lambda}^{+} \in \mathbb{R}$  in $\zDE$ and $\thetakl^{-}/\mu_{\lambda}^{-} \in \mathbb{R}$ and $\thetakl^{+}/\mu_{\lambda}^{+} \in \rmi \, \mathbb{R}$ in $\zEI$. Thus, from \eqref{eq.disp},
		$$ 
		|\calW_{k,\lambda}| \geq \frac{1}{\sqrt{2}} \; \Big( \frac{|\thetakl^{-}|}{|\mu_{\lambda}^{-}|}+ \frac{|\thetakl^{+}|}{|\mu_{\lambda}^{+}|}  \Big)  \gtrsim    |\thetakl^{+}| + |\thetakl^{-}| ,
		$$
		since $\mu_{\lambda}^{-} = \mu_0$ and $\mu_{\lambda}^{+}$ is continuous and does not vanish on $[a,b]$.
		\item In $\zDD$ and $\zDI$, $\thetakl^{\pm}/\mu_\pm(\lambda)$ are imaginary numbers whose imaginary parts have the same sign. Thus,  
		$$
		|\calW_{k,\lambda}| = \;  \frac{|\thetakl^{-}|}{|\mu_{\lambda}^{-}|}+ \frac{|\thetakl^{+}|}{|\mu_{\lambda}^{+}|}  \gtrsim   |\thetakl^{+}| + |\thetakl^{-}| .
		$$
	\end{itemize}
	The above lower bounds for $|\calW_{k,\lambda}|$ prove \eqref{eq.ineqA}.
\end{proof}
\begin{proposition}\label{prop.decay-estim-surfacic-zones}
	Let  $s > 1/2$, $\scZ \in \calZ\setminus\{\EE\}$, and $[a,b] \subset \bbR\setminus\sigma_{\rm exc}$ such that $\Lambda_{\scZ}([a,b]) \neq \varnothing$. Then, there exists $C_{a,b}>0$  such that if  $\pm 1\in J_{\scZ}$ then:
	\begin{equation}\label{eq.ineqbound1_crosspoint}
	\| \bbW_{k,\lambda,\pm 1}\|_{\Hms}\leq C_{a,b} \;  |\thetakl^{\mp}|^{-\frac{1}{2}}, \quad \forall (k, \lambda)\in\Lambda_{\scZ}([a,b]) .
	\end{equation}
     If moreover $\pm \Oc\notin [a,b]$, for any $\gamma \in (0,1]\cap (0,s-1/2)$, one has 
\begin{equation}\label{eq.ineqbound1}
\| \bbW_{k,\lambda,\pm 1}\|_{\Hms}\leq C_{a,b}^\gamma  \;  |\thetakl^{\mp}|^{-\frac{1}{2} + \gamma}, \quad \forall (k, \lambda)\in\Lambda_{\scZ}([a,b]) .
\end{equation}
\end{proposition}
\begin{proof}
From symmetry reasons in the $(k, \lambda)$ plane, it is quite obvious that we can restrict ourselves to the intersections of the spectral zones with the quadrant $\{ \lambda > 0, k \geq 0\}$. Moreover, for simplicity, we give the proof for $j=1$ (which means that $\scZ \in \{\DI, \DE, \DD \}$ since $1 \notin J_\EI$, see the definition \eqref{eq.def-Jz} of $J_\scZ$) and let the reader check by simple symmetry arguments (between $-1$ and $+1$ on one hand, $x< 0$ and $x > 0$ on the other hand) that it also works for $j=-1$.  \\ [12pt]
{\it Step 1.} To begin, we estimate  the first component $w_{k,\lambda,1}= A_{k,\lambda,1}\ \psi_{k,\lambda,1}(x)\ \rme^{\rmi k y}$ of $\bbW_{k,\lambda,1}$ (see \eqref{eq.def-Vkbis}).
First, using \eqref{eq.ineqA} and the fact that $\mu_{\lambda}^{-} = \mu_0$ and $\mu_{\lambda}^{+}$ is continuous and does not vanish on $[a,b]$ (since $\pm \Om\notin [a,b]$), it follows that the coefficient $A_{k,\lambda,1}$ defined by \eqref{def-A-gen} is bounded  by:
\begin{equation}\label{eq.boundAcoefficient}
|A_{k,\lambda,1}|\lesssim \frac{|\thetakl^-|^{\frac{1}{2}}}{|\thetakl^-|+|\thetakl^+|} .
\end{equation}  
Then, we  estimate the function $\psi_{k,\lambda, 1}$:
\begin{itemize} 
	\item{(i)} For $x\geq 0$, by definition (cf. \eqref{eq-def-phi}), $\psi_{k,\lambda, 1}(x)=\rme^{- \thetakl^{+}x}$ with $\thetakl^{+}$ positive or purely imaginary.     		         
	Hence,   
\begin{equation} \label{estpsigeq0} 
\forall \; x \geq 0, \quad |\psi_{k,\lambda, 1}(x)| \leq 1.
\end{equation}
	\item{(ii)} For $x < 0$,  by formula \eqref{eq.def-thetam} for $\scZ \in \{\DI, \DE, \DD\}$, one has $\thetakl^{-} = - \, \rmi  \, |\thetakl^{-}| \in \rmi \, \bbR $.  Thus, by virtue of the inequalities  $|\cosh(u)|\leq 1$ and $|\sinh(u)|\leq 1$ for $u\in \rmi \bbR$ , one deduces from  the expression of \eqref{eq-def-phi} of $\psi_{k,\lambda, 1}$ that:
	\begin{equation} \label{estpsileq0} 
	\forall \; x < 0, \quad |\psi_{k,\lambda, 1}(x)|\lesssim 1+{|\thetakl^+|}/{|\thetakl^-|}.
	\end{equation}
\end{itemize}
From \eqref{estpsigeq0} and \eqref{estpsileq0} , we deduce the uniform bound
\begin{equation}\label{eq-boundpsi}
\forall \; x \in \bbR, \quad |\psi_{k,\lambda, 1}(x)| \lesssim \frac{|\thetakl^+| + |\thetakl^-|}{|\thetakl^-|}.
\end{equation} 
Thus, from \eqref{eq.boundAcoefficient} and \eqref{eq-boundpsi},  we get 
\begin{equation}\label{eq.boundcrosspoint-w}
|w_{k,\lambda,1}(x,y)|\lesssim |\thetakl^-|^{-\frac{1}{2}},  \ \forall (x,y)\in \bbR^2,
\end{equation} 
and as $s>1/2$, it follows that
\begin{equation}\label{eq-boundw4}
\|w_{k,\lambda,1}\|_{L^2_{-s}(\bbR^2)} \lesssim  |\thetakl^-|^{-\frac{1}{2}}.
\end{equation}
{\it Step 2.} 
The second step consists in showing that similar estimates hold for the other five components of the vector $\bbW_{k,\lambda,1}$.  According to \eqref{eq.def-Vkbis}, the second  component  of this vector is
$k  \, (\lambda \, \mu_\lambda)^{-1}   \,w_{k,\lambda,1}.$ 
Since  $k \, (\lambda \, \mu_\lambda)^{-1}$ is bounded in $\Lambda_{\scZ}([a,b])$ (since  $0, \, \pm \Om\notin [a,b]$), the estimate of this component follows from the one for the first component \eqref{eq-boundw4}. 
\\ [12pt]
The third component of $\bbW_{k,\lambda,1}$ is   $ \rmi (\lambda \, \mu_\lambda)^{-1}  \partial_x w_{k,\lambda,1} = \rmi (\lambda \, \mu_\lambda)^{-1}  A_{k,\lambda,1} \, \partial_x \psi_{k,\lambda,1} \, \rme^{\rmi k y} $ (cf. \eqref{eq.def-Vkbis}).
This component is less singular  than $w_{k,\lambda,1} $ since differentiating $\psi_{k,\lambda, 1}$ with respect to $x$ leads to a multiplication by factors $\thetakl^{\pm}$ that regularizes the expression. Indeed, one has
\begin{equation}\label{eq-express-deriv-x}
\partial_x\psi_{k,\lambda,1}(x) = -\thetakl^{+} \; \rme^{- \thetakl^{+}x}, \quad \mbox{for } x > 0, \quad  \thetakl^{-} \sinh\big( \thetakl^{-} x \big)- \thetakl^{+} \, \frac{\mu_{\lambda}^{-}}{\mu_{\lambda}^{+}} \, \cosh\big( \thetakl^{-} x \big) , \quad \mbox{for } x < 0.
\end{equation}
Again, as $ \thetakl^{+}>0$ or $\thetakl^{-} \in \rmi \mathbb{R}$,  the exponential function  and  the hyperbolic sine and cosine involved in \eqref{eq-express-deriv-x} are bounded by $1$.
Furthermore  as $\mu_{\lambda}^{-}/\mu_{\lambda}^{+}$ is bounded (since $\pm \Om\notin [a,b]$), it follows that:
\begin{equation}\label{eq.deriv-1}
\big|\partial_x\psi_{k,\lambda,1}(x)\big| \lesssim |\thetakl^{-}|+ |\thetakl^{+}|.
\end{equation}
Thus, combining with \eqref{eq.boundAcoefficient} we get 
\begin{equation}\label{eq.etimatethirdcomponent}
 \ \forall x\in \mathbb{R}^*, \ \  \ |\displaystyle \rmi (\lambda \, \mu_\lambda)^{-1}  \partial_x w_{k,\lambda,1} |=
\big|\rmi (\lambda \, \mu_\lambda)^{-1}  A_{k,\lambda,1}\, \, \partial_x\psi_{k,\lambda,1}(x)\, \rme^{\rmi k y}\big| \lesssim  |\thetakl^{-}|^{\frac{1}{2}},
\end{equation}
and thus
$
\big\|\rmi (\lambda \, \mu_\lambda)^{-1}  \partial_x w_{k,\lambda,1} \big\|_{L^2_{-s}(\bbR^2)} \lesssim |\thetakl^{-}|^{\frac{1}{2}}\lesssim |\thetakl^{-}|^{-\frac{1}{2}}.
$
Once the first three components have been treated, the work is finished since the last three components are $0$ for $x < 0$ and proportional 
(with a coefficient that depends on $\lambda$ but is bounded in $\Lambda_{\scZ}([a,b])$) to the first three for $x > 0$ (see again  \eqref{eq.def-Vkbis}). Thus one concludes  to the estimate \eqref{eq.ineqbound1_crosspoint}.
\\ [12pt]
{\it Step 3}. We now explain how to obtain the improved estimates when 
$\pm \Oc\notin [a,b]$. These are an improvement in the sense that, since $\gamma \geq 0$, the bound \eqref{eq.ineqbound1} is better than \eqref{eq.ineqbound1_crosspoint} when $\theta_{k,\lambda}^-$ tends to 0. \\ [12pt]
We shall concentrate ourselves on the estimate of the first component $w_{k,\lambda,1}$ of $\bbW_{k,\lambda,1}$. Passing to the other five components is essentially a matter of repeating the arguments of Step 2. \\ [12pt]
The improvement is obtained from a ``new" estimate for $A_{k,\lambda,1}$ that exploits $\pm \Oc\notin [a,b]$, and a different estimate for $\psi_{k,\lambda,1}$ when $x< 0$ in which we introduce $\gamma$:
\begin{itemize} 
	\item because $\pm \Oc\notin [a,b]$, $\overline{\Lambda_{\scZ}([a,b])}$ does not contain any cross-point so that
$\thetakl^{+}$ and  $\thetakl^{-}$ do not vanish simultaneously.  As a consequence $|\thetakl^-|+|\thetakl^+|$ is  bounded from below and 
\begin{equation}\label{eq.boundAcoefficientnew}
|A_{k,\lambda,1}|\lesssim |\thetakl^-|^{\frac{1}{2}}.
\end{equation}  
\item The inequality $| \sinh( \thetakl^{-} x)|\leq 1$, used of {\it Step 1} (ii) for deriving  \eqref{estpsileq0}, is  inaccurate for small  $\thetakl^{-}$. Given $\gamma \in (0,1]$, we can replace it by
$$
|\sinh( \thetakl^{-} x)| = \big|\sin\big(|\thetakl^{-}|\, x)\big|^\gamma \; \big|\sin\big(|\thetakl^{-}|\, x)\big|^{1-\gamma} \leq |\thetakl^{-}|^\gamma |x|^\gamma .
$$
Using this estimate yields $|\psi_{k,\lambda, 1}(x)| \lesssim 1 + |\thetakl^{-}|^{\gamma-1} \, |x|^\gamma$ for $x<0$. This,  together with \eqref{estpsigeq0}, yields the following new estimate for $\psi_{k,\lambda, 1}$ (instead of  \eqref{eq-boundpsi})
\begin{equation}\label{eq-boundpsi-new}
\forall \; x \in \bbR, \quad |\psi_{k,\lambda, 1}(x)| \lesssim 1 + |\thetakl^{-}|^{\gamma-1} \, |x|^\gamma .
\end{equation} 
\end{itemize}
Thus, from \eqref{eq.boundAcoefficientnew} and \eqref{eq-boundpsi-new},  we get 
\begin{equation*}\label{eq.improveestimateinterpola}
|w_{k,\lambda,1}(x,y)|\lesssim  |\thetakl^-|^{\frac{1}{2}} + |\thetakl^{-}|^{\gamma-\frac{1}{2}} |x|^\gamma = |\thetakl^{-}|^{\gamma-\frac{1}{2}} \, \big( |\thetakl^{-}|^{1-\gamma} + |x|^\gamma\big) \lesssim |\thetakl^{-}|^{\gamma-\frac{1}{2}}  (1 + |x|^\gamma)
\end{equation*}
since, as $\gamma \leq 1$, $|\thetakl^{-}|^{1-\gamma}$ is bounded.
As the function $(x,y)\mapsto (1+|x|)^{\gamma}\in L^2_{-s}(\bbR^2)$ for $0\leq\gamma<s-1/2$, this yields immediately the expected estimate
\begin{equation*}\label{eq-firstestim}
\|\, w_{k,\lambda,1}\|_{L^2_{-s}(\bbR^2)} \lesssim  |\thetakl^-|^{-\frac{1}{2}+\gamma} ,
\end{equation*}
which is nothing but the estimate \eqref{eq.ineqbound1} for the first component $w_{k,\lambda,1}$ of $\bbW_{k,\lambda,1}$.
\end{proof}
\subsubsection{Generalized eigenfunctions of the lineic spectral zone}
We deal now with the $\Hms$ estimates of  the generalized  eigenfunctions $\bbW_{k,\lambda,0}$ (also referred as plasmonic waves in the physical literature) when $(k, \lambda)$ belongs to the spectral zone $\zEE$ for $\Oe \neq \Om$. Let us recall that for $(k,\lambda)\in\zEE$, $k$ and $\lambda$ are related by $k = \pm \, k_{\scE}(\lambda)$ (cf. (\ref{eq.expressionlambdae}, \ref{defkE}, \ref{eq.defZEE}). Thus, the set $\Lambda_{\EE}([a,b])$ involved in this  result and  defined by \eqref{eq.defLambdaspeczoneab} can be  rewritten as 
$$ 
\Lambda_{\EE}([a,b])=\{ (\pm k_{\scE}(|\lambda|),\lambda)\in \Lambda_{\EE}  \mid  \lambda \in [a,b]\}.
$$

\begin{proposition}\label{prop.decay-estim-plasmon}
Assume that  $\Oe\neq \Om$ and let $s > 1/2$ and $[a,b] \subset \bbR\setminus\sigma_{\rm exc}$ such that $\Lambda_{\EE}([a,b]) \neq \varnothing$. Then, there exists $C_{a,b}>0$ (depending only on $a, b$) such that 
\begin{equation}\label{eq.ineqboundplasm}
\| \bbW_{k,\lambda,0}\|_{\Hms}\leq C_{a,b} \ (\thetakl^{+})^{\frac{1}{2}}, \quad \forall (k, \lambda)\in \Lambda_{\EE}([a,b]).
\end{equation}
\end{proposition}

\begin{proof}
As $[a,b]$  does not contain $\pm \Op$, $\Lambda_{\EE}([a,b])$ is a bounded subset of $\Lambda_{\EE}$.
One one hand,  from  \eqref{def-A-plasm}, one deduces that $A_{k,\lambda,0}\lesssim (\thetakl^+)^{1/2}$ ($\thetakl^+>0$ in $\Lambda_{\EE}([a,b])$, see \eqref{eq.def-thetap}). 

\noindent On the other hand,  one has $|\psi_{k,\lambda,0}(x)|\lesssim 1$ (cf. \eqref{def-psi-plasm}) since $\thetakl^{\pm}>0$.  It follows that:
\begin{equation}\label{eq.interlinf}
|w_{k,\lambda,0}(x,y)| \lesssim (\thetakl^+)^{1/2} \  \mbox{ and }  \ \ | \partial_x w_{k,\lambda,0}(x,y)|  \lesssim (\thetakl^+)^{1/2} \  , \forall (k, \lambda)\in \Lambda_{\EE}([a,b]),
\end{equation}
(we use here to estimate  $\partial_x w_{k,\lambda,0}$ that $\thetakl^{\pm}$ is bounded as a  continuous function of $(k,\lambda)$ on  $\overline{\Lambda_{\EE}([a,b])}$).
Finally, as the lines $\lambda=0$ and $\lambda=\pm \Om$ do not intersect the compact set $\overline{\Lambda_{\EE}([a,b])}$ all the coefficients  in $(k,\lambda)$ that are involved in the expression  \eqref{eq.def-Vkbis}  of $\bbW_{k,\lambda,0}$ in terms of $w_{k,\lambda,0}$ and $\partial_x w_{k,\lambda,0}$ are bounded. Thus, the inequality \eqref{eq.ineqboundplasm} for $s>1/2$ comes immediately from \eqref{eq.interlinf}  and the relation  \eqref{eq.def-Vkbis} that define the generalized eigenfunction $ \bbW_{k,\lambda,0}$.
\end{proof}

\subsubsection{Proof of Proposition \ref{p.estim-fpg} } \label{sec_defMlambda}
The following preliminary  lemma gives local estimates of the inverse of $|\thetakl^{\pm}|$ with respect to $(k,\lambda)$ in each spectral zone $\zZ$ for $\scZ \in \calZ\setminus\{\EE\}$. 
These estimates will be used to prove Proposition \ref{p.estim-fpg} and thus  Theorems \ref{th.dens-spec} and \ref{th.Holder-dens-spec}.  To simplify their presentation,  we introduce the function $\kp$ defined as follows:
\begin{eqnarray}
\kp(\lambda) & := & \left\lbrace \begin{array}{ll} \kI(\lambda) \in \mathbb{R}^+& \mbox{if } 0 < |\lambda|\leq \min(\Oe,\Om), \\[5pt]
 \rmi \,\sqrt{-\eps^+_\lambda\, \mu^+_\lambda} \,  |\lambda|  \in \rmi \, \mathbb{R}^{+}& \mbox{if } \min(\Oe,\Om) < |\lambda| < \max(\Oe, \Om), \\[5pt]
\kD(\lambda)  \in \mathbb{R}^+  &   \mbox{if } \max(\Oe,\Om)\leq |\lambda|.
\end{array}\right.
\label{eq.def-kp}
\end{eqnarray} 
We notice that the even function $\kp$ vanishes at $\pm \Oe$  and $\pm \Om $ since $\kp(\pm \Oe)=\kI(\pm \min(\Om,\Oe))=\kD(\pm \max(\Om,\Oe))= 0$ (see  \eqref{eq.defepsmu} and  \eqref{defk0DI} and figure \ref{fig.speczones1}). Moreover, one easily checks that $\kp$ is  locally H{\"o}lder of index $\gamma=1/2$ on $\bbR^*$ and $C^{\infty}$ (thus locally Lipschitz continuous) on $\bbR^*\setminus \{\pm \Oe, \pm \Om \}$.

\begin{lemma}\label{eq.lemsingularitythetalk}
For all $\lambda \in \bbR \setminus \{0,\pm\Oe,\pm\Om\},$ we have
\begin{align} 
\forall |k| \neq k_0(\lambda), \quad
\big|\thetakl^{-}\big|^{-1} & \leq k_0(\lambda)^{-1/2}\ \big| |k|-k_0(\lambda) \big|^{-1/2},\label{eq.singtheta-spectralcut-1} \\
\forall |k| \neq k^+(\lambda), \quad
\big|\thetakl^{+}\big|^{-1} & \leq \big|k^+(\lambda)\big|^{-1/2}\ \big| |k|-k^+(\lambda) \big|^{-1/2}.\label{eq.singtheta-spectralcut-2}
\end{align}
\end{lemma}

\begin{proof}
From the definitions \eqref{eq.defThetalkfunction} and \eqref{eq.def-thetam} of $\Thetakl^{+}$ and $\thetakl^{+}$, one has
\begin{equation*}
\forall \lambda \in \bbR, \ \forall |k| \neq k^+(\lambda), \quad |\thetakl^{+}|^{-1}= |\Thetakl^{+}|^{-1/2}= \big||k|-\kp(\lambda)\big|^{-1/2} \ \big||k|+\kp(\lambda)\big|^{-1/2}.
\end{equation*}
Inequality \eqref{eq.singtheta-spectralcut-2} simply follows by noticing that
\begin{equation*}
\big||k| + \kp(\lambda) \big| \geq |\kp(\lambda)|,
\end{equation*}
since $\kp(\lambda)$ is either a positive real number or a purely imaginary number, which does not vanish if $\lambda \notin \{\pm\Oe,\pm\Om\}$.

For inequality \eqref{eq.singtheta-spectralcut-1}, one proceeds similarly by  substituting respectively $\thetakl^{+}$, $\Thetakl^{+}$ and $\kp$ by $\thetakl^{-}$, $\Thetakl^{-}$ and $\kO$ (the only difference is  that $\kO(\lambda)$ is always real-valued).
\end{proof}
\begin{remark}
The estimates of Lemma \ref{eq.lemsingularitythetalk}  are optimal in the sense that they take into account  in a sharp way their singular behaviour of $| \thetakl^{\pm} |^{-1} $ when approaching the boundary of the spectral zone (see \eqref{eq.singtheta-spectralcut}). Indeed, these functions approach $0$ as the square root of the (horizontal) distance to the spectral cuts,  more precisely:
\begin{equation*}\label{eq.singtheta-spectralcut}
\hspace*{-0.4cm} \begin{array}{ll}
| \thetakl^{-} |^{-1} \sim  (2 \kO(\lambda))^{-\frac{1}{2}} \; |k\pm \kO(\lambda)|^{-\frac{1}{2}}  \mbox{ as }  k \in \zZ(\lambda) \rightarrow \mp \;\kO(\lambda),  & \; \scZ = \EI,  \DE, \DI   \\[12pt]
| \thetakl^{+} |^{-1} \sim (2 \kp(\lambda))^{-\frac{1}{2}} \; |k\pm \kp(\lambda)|^{-\frac{1}{2}}  \mbox{ as }  k \in \zZ(\lambda) \rightarrow \mp \; \kp(\lambda),  \,  \lambda\neq \pm \Oe & \scZ = \EI,   \DI, \DD. \\ [12pt]
\end{array}
\end{equation*}
Therefore,  one has for $\lambda\neq \pm \Oe$,
\begin{equation} \label{L2integrability}
  \left\{ \begin{array}{lll} k \mapsto |\thetakl^{-}|^{-r} \in L^2 \big(\Lambda_{\scZ}(\lambda)\big),  \, \scZ =   \DE, \DI  \\ [8pt]
k \mapsto |\thetakl^{+}|^{-r} \in L^2 \big(\Lambda_{\scZ}(\lambda)\big), \; \scZ =  \EI,   \DI, \DD
\end{array}  \right. \quad \Longleftrightarrow \quad r < 1.
\end{equation}
\end{remark}
\noindent As $ |\thetakl^{-}|^{-1}$ is bounded on $\Lambda_{\DD}(\lambda)$, it is straightforward to see that  \eqref{eq.ineqbound1_crosspoint} and   \eqref{L2integrability} (used with $r=1/2$)  imply that  if $\pm 1 \in  J_{\scZ}$, the function  $k\mapsto \bbW_{k,\lambda,\pm 1}$  is {\bf square integrable}  in each set $\Lambda_{\scZ}(\lambda)$.
Thanks to Lemma  \ref{eq.lemsingularitythetalk}, and  Propositions \ref{prop.decay-estim-surfacic-zones} and \ref{prop.decay-estim-plasmon}, we can now prove Proposition  \ref{p.estim-fpg} of  section \ref{sec.motiv-main-results}.

\begin{proof}
{\it Proof of point 1.}
Let  $s>1/2$, $\scZ \in \calZ\setminus\{\EE\}$  and $[a,b] \subset \bbR\setminus\sigma_{\rm exc}$ such that $\Lambda_{\scZ}([a,b]) \neq \varnothing$.
We know from the relation  \eqref{eq.ineqbound1_crosspoint} of Proposition \ref{prop.decay-estim-surfacic-zones} that if $ \pm 1 \in J_{\scZ}$ then
\begin{equation}\label{eq.generalbound}
\| \bbW_{k,\lambda,\pm 1} \|_{\Hms}^2 \lesssim  |\thetakl^{\mp }|^{-1}, \quad  \forall (k, \lambda)\in\Lambda_{\scZ}([a,b])  .
\end{equation}
Thus, it is  simply a matter of using Lemma \ref{eq.lemsingularitythetalk}. We distinguish two cases: \\ [12pt]
(i) : $ j=1$ (which is  only possible by \eqref{eq.def-Jz}  for $\scZ=\DD, \DE, \DI$)\\ [12pt]
\noindent $\bullet$ For $\scZ=\DD$, as $(k,\lambda)\mapsto  |\thetakl^{-}|$ does not vanish in  $\overline{\Lambda_{\DD}([a,b])}$, the 
$\Hms$-norm of $\bbW_{k,\lambda,j} $ is bounded when $(k,\lambda) \in \Lambda_{\scZ}([a,b])$. This proves the point 1 for $\scZ=\DD$  since $\Lambda_{\DD}([a,b])$ has a finite Lebesgue measure. \\ [12pt]
\noindent $\bullet$ Now if $\scZ=\DE$ or $\scZ= \DI$, then
it follows   from   \eqref{eq.singtheta-spectralcut-1} and \eqref{eq.generalbound} that 
\begin{equation}\label{eq.dom}
\forall  (k,\lambda)\in \Lambda_{\scZ}([a,b]), \quad 
\| \bbW_{k,\lambda,1} \|_{\Hms}^2  \lesssim \big| |k| - \kO(\lambda)\big|^{-\frac{1}{2}},
\end{equation}
since $\kO(\lambda)^{-1/2}$ remains bounded in $[a,b]$.
Let $\lambda\in [a,b]$. One has  $\Lambda_{\scZ}(\lambda)\subset \; ] -\kO(\lambda), \kO(\lambda) [ $. Then, by  \eqref{eq.dom} and evenness in $k$, it follows 
$$
\int_{\Lambda_{\scZ}(\lambda)} \| \bbW_{k,\lambda,1} \|_{\Hms}^2  \rmd k \lesssim \int_{0}^{\kO(\lambda)}  (\kO(\lambda)-k)^{-\frac{1}{2}}\, \rmd k \lesssim 2 \kO(\lambda)^{\frac{1}{2}} \lesssim 1 .
$$
Thus, the point 1 is a consequence of the  Fubini-Tonelli theorem and the fact  that  $[a,b]$ is bounded.\\
[6pt]
(ii) : $j=-1$ (which is  only possible by \eqref{eq.def-Jz}  for $\scZ= \DI, \DD, \EI$)\\ [12pt]
The slight difference with the case (i) is when $[a,b]$ contains $\pm \Oe$ where $\kp$ vanish.\\
[6pt]
\noindent $\bullet$ For  the case where $\scZ=\DD, \DI, \EI$, by virtue of the estimates   \eqref{eq.singtheta-spectralcut-2}  and  \eqref{eq.generalbound}, one gets 
\begin{equation}\label{eq.dom1}
\forall  (k,\lambda)\in \Lambda_{\scZ}([a,b]),\quad 
\| \bbW_{k,\lambda,-1} \|_{\Hms}^2 \leq |\kp(\lambda)|^{-\frac{1}{2}} \,
\big| |k| - \kp(\lambda)\big|^{-\frac{1}{2}}.
\end{equation}
Using that $\kp(\lambda)\in \mathbb{R}^{+,*}$ in \eqref{eq.dom1}, $\Lambda_{\scZ}(\lambda) \subset  \; ] \! -\!\kp(\lambda), \kp(\lambda) [ $ (see figure \ref{fig.speczones1}) and a parity argument in $k$  give
$$
\int_{\Lambda_{\scZ}(\lambda)} \| \bbW_{k,\lambda,-1} \|_{\Hms}^2 \, \rmd k \lesssim  \kp(\lambda)^{-\frac{1}{2}}  \int_{0}^{\kp(\lambda)}  (\kp(\lambda)-k)^{-\frac{1}{2}} \, \rmd k \lesssim 1.
$$
This implies the point 1 since the bound is uniform with respect to $\lambda\in [a,b]$.\\
[6pt]

{\it Proof of point 2.} The key point is that, as $[a,b]$ does not contain $\pm \Op$,  $\Lambda_{\EE}([a,b])$ is a bounded subset of $\Lambda_{\EE}$. 
Thus, the point $2$ is an immediate consequence of the  estimate  \eqref{eq.ineqboundplasm}  and the fact that $\thetakl^{+}$ is bounded on $\Lambda_{\EE}([a,b])$.
\end{proof}

\subsection{H\"{o}lder regularity of generalized eigenfunctions}\label{sec.Holdereigenfunctions}

\subsubsection{Orientation} \label{sec_GEForient}
In this section,  we  show what we shall call local ``H\"{o}lder  type estimates'' (see Propositions \ref{prop.holdreggeneralizedeigenfunctions} and \ref{prop.holdestimateplasmon}) on generalized eigenfunctions.
\subsubsection*{The case of the functions $\bbW_{k,\lambda,\pm 1}$} 
In this case we study the functions  $(k, \lambda) \in \Lambda_\scZ \mapsto  \bbW_{k,\lambda,j} \in \Hms$ ,  such that $j \in J_{\scZ}$ and $\scZ \in \calZ\setminus\{\EE\}$. By local ``H\"{o}lder  type estimate''  we mean an estimate in which $k$ plays the role of a parameter with respect to the spectral variable $\lambda$, i.e. of the form (given $a \leq b$)
\begin{equation} \label{formHolder}
\forall \big((k,\lambda), (k,\lambda' )) \in \Lambda_Z([a,b])^2, \ \|  \bbW_{k,\lambda',j}-\bbW_{k,\lambda,j}\|_{\Hms}\leq  F_{\gamma}(k; \lambda, \lambda' ) \,   |\lambda'-\lambda|^{\gamma} ,
\end{equation}
where $F_{\gamma} :  \big\{ (k; \lambda, \lambda' ) \in \mathbb{R}^3 \; / \;  \big((k,\lambda), (k,\lambda' )) \in \Lambda_{\scZ}([a,b])^2\}  \rightarrow \mathbb{R}^+$  is smooth and $\gamma \in (0,1]$.
\\ [6pt]
There are many different ways to obtain estimates of the form \eqref{formHolder}. Before going to technical developments (based, as we shall see, on lengthy hand computations that require to be done with a lot of care) and precise results, it is worthwhile to make three observations that guided us in the derivation of \eqref{formHolder}.
\\ [8pt] 
{\bf Observation 1}: it is clear from the integral expression $\bbM_\lambda$ of the spectral density that
in order to transform the estimates \eqref{formHolder} into H\"{o}lder regularity for  $\lambda \mapsto \bbM_\lambda$ in $B(\Hps, \Hms)$,  for given $(\lambda, \lambda')$ the function $k \mapsto F_{\gamma}(k; \lambda, \lambda' )$ should have appropriate square integrability properties (we shall be more precise later).  Natural candidates for $F_{\gamma}(k; \lambda, \lambda' )$ automatically involve negative powers of $|\thetakl^{\pm}|$ and $|\theta_{k, \lambda'}^{\pm}|$ (cf. section \ref{sec_EF-estimates}), which means that they may blow up when $(k, \lambda)$ or $(k, \lambda')$  approaches a spectral cut. 
That is why we  have to pay attention to control this blow up (i.e. to the allowed powers of $\thetakl^{\pm}$): this brings us back criterion \eqref{L2integrability}.\\ [8pt]
{\bf Observation 2}: similarly to what was done  in section \ref{sec_EF-estimates},  the desired estimates \eqref{formHolder} will be derived from similar pointwise H\"older estimates for the functions $\lambda \mapsto w_{k,\lambda,j}(x,y) \in \mathbb{C}$ and $\lambda \mapsto \partial_x w_{k,\lambda,j}(x,y) \in \mathbb{C}$, that involve the space variable $(x,y)$ as an additional parameter. These estimates will be of the form  
\begin{equation} \label{formHolderpointpointwise}
| v_{k,\lambda',j}(x,y)-v_{k,\lambda,j}(x,y)| \leq  (1+|x|)^{\gamma} \, f_{\gamma}(k; \lambda, \lambda' ) \,   |\lambda'-\lambda|^{\gamma} , 
\end{equation}
for $v_{k,\lambda,j} = w_{k,\lambda,j}$ or $v_{k,\lambda,j} = \partial_x w_{k,\lambda,j}$. 
The limitation for the set of possible H\"older exponents then comes from the double requirement that :
\begin{itemize} 
	\item[(i)] in order to use these estimates to get the $\Hms$-valued estimates, we must have $0<\gamma < s - {1}/{2}$ (the condition for which $ x \mapsto (1+|x|)^{\gamma}$ belongs to $L^2_{-s}(\mathbb{R}^2)$),
	\item[(ii)] the function $ f_{\gamma}(k; \lambda, \lambda' )$ is supposed to have the same square integrability properties in $k$ as $ F_{\gamma}(k; \lambda, \lambda' )$, cf. observation 1, which will generate another limitation on $\gamma$.  
\end{itemize}
{\bf Observation 3}: From the {\bf technical point of view}, the systematic path that we chose to use for getting estimates of the form \eqref{formHolderpointpointwise}, assuming that $\lambda \mapsto v_{k,\lambda,j}(x,y)$ is differentiable,  is the following: 
\begin{enumerate}
	\item Obtain Lipschitz estimates via the mean value theorem from $L^\infty$ estimates of the $\lambda$-derivative $\partial_\lambda v_{k,\lambda,j}$:
	$$
	| v_{k,\lambda',j}(x,y)-v_{k,\lambda,j}(x,y)| \leq \sup_{\tilde{\lambda}\in[\lambda,\lambda']} \big|\partial_\lambda v_{k,\lambda',j}(x,y)\big|  \; |\lambda - \lambda'| .
	$$
	\item Interpolate the previous estimate with $L^\infty$ estimates of $v_{k,\lambda,j}$
	$$
	\left| \begin{array}{lll}
	| v_{k,\lambda',j}(x,y)-v_{k,\lambda,j}(x,y)| & = & | v_{k,\lambda',j}(x,y)-v_{k,\lambda,j}(x,y)|^{\gamma} \;  | v_{k,\lambda',j}(x,y)-v_{k,\lambda,j}(x,y)|^{1-\gamma} \\ [12pt] 
	& \leq & \displaystyle 2^{1-\gamma} \big(\sup_{\tilde{\lambda}\in[\lambda,\lambda']} |(\partial_\lambda v)_{k,\tilde{\lambda},j}(x,y)|^\gamma\big) \; \big(\sup_{\tilde{\lambda}\in[\lambda,\lambda']} | v_{k,\tilde{\lambda},j}(x,y)|^{1-\gamma}\big)  \;  |\lambda - \lambda'|^\gamma
	\end{array} \right.
	$$
\end{enumerate} 
and the estimation of the two above sup leads to estimates of the form \eqref{formHolderpointpointwise}.\\[6pt]  
Even though step 1 seems to provide already the desired type of estimate with $\gamma = 1$,
the interpolation step 2 will be needed to fulfil the  integrability requirements mentioned in observations 1 and 2. \\[6pt] 
Finally, we want to mention that  estimates of the form \eqref{formHolder} with a function $F_{\gamma}$  with $0<\gamma\leq 1$ that involves negative powers  of $|\thetakl^{\pm}|$ and $|\theta_{k, \lambda'}^{\pm}|$  will allow us to obtain (in section \ref{sec.Holderspectraldensity}) better H\"{o}lder exponents for the spectral density (after an integration on the sets $\Lambda_{\scZ}(\lambda)$) than  the actual local H\"{o}lder regularity of the generalized eigenfunctions $\bbW_{k,\lambda,j}$ with respect to $\lambda$. Indeed,  H\"{o}lder estimates  for $\bbW_{k,\lambda,j}$ (with a constant function  $F_{\gamma}$) are limited to $0<\gamma\leq1/4$
since the expression of $\bbW_{k,\lambda,j}$ (see  \eqref{eq.def-Vkbis}, \eqref{eq.def-w}, \eqref{def-A-gen} and \eqref{def-psi-gen}) contains functions $ (k, \lambda) \mapsto |\thetakl^{\pm}|^{1/2}$ that are only   $1/4$ locally H\"{o}lder continuous in $\lambda$ at the vicinity of the spectral cut where $\thetakl^{\pm}$  vanishes.

\subsubsection*{The case of the functions $\bbW_{k,\lambda,0}$} 
In this case, we study the functions  $(k,\lambda) \in \mathbb{R}^2 \mapsto  \bbW_{k,\lambda,0} \in \Hms$ for $(\lambda,k) \in \Lambda_{\EE}$ and $\Oe\neq \Om$. This case is special since $\Lambda_{\EE}$ is a reunion of curves: $k = \pm \, \kE(\lambda)$ for  $|\lambda| \in (\min(\Op,\Oc),\max(\Op,\Oc))$.
Thus, $(k,\lambda)$ cannot be considered as independent variables and $k$ is no longer a parameter. We are in fact interested in the functions 
$$
|\lambda| \in \big(\min(\Op,\Oc),\max(\Op,\Oc)\big) \; \mapsto  \; \bbW_{\pm \kE(\lambda),\lambda,0} .
$$
Then given $[a,b] \subset [\min(\Op,\Oc),\max(\Op,\Oc)] $ which does not contain $\Op$, we look for estimates of the form (these replace \eqref{formHolder})
\begin{equation} \label{formHolder0}
\forall \; (\lambda,\lambda' )\in [a,b]^2 \quad \Longrightarrow \quad \|  \bbW_{\pm \kE(\lambda),\lambda,0} -\bbW_{\pm \kE(\lambda'),\lambda',0}\|_{\Hms}\leq  F_{\gamma}(\lambda, \lambda' ) \,   |\lambda'-\lambda|^{\gamma} ,
\end{equation}
that will be obtained from pointwise estimates (these replace \eqref{formHolderpointpointwise})
\begin{equation} \label{formHolderpoint0}
| v_{\pm \kE(\lambda),\lambda,0}(x,y)-v_{\pm \kE(\lambda),\lambda,0}(x,y)| \leq  (1+|x|+|y|)^{\gamma} \, 
f_{\gamma}( \lambda, \lambda' ) \,   |\lambda'-\lambda|^{\gamma} , 
\end{equation}
for $v_{\pm \kE(\lambda),\lambda,0}= w_{\pm \kE(\lambda),\lambda,0}$ or $v_{\pm \kE(\lambda),\lambda,0}= \partial_x w_{\pm \kE(\lambda),\lambda,0}$. A difference between \eqref{formHolderpointpointwise} and \eqref{formHolderpoint0} is that $1+|x|$ is replaced by $1+|x| + |y|$. However, this does not change the condition  $\gamma < s - {1}/{2}$ raised in the point (ii) of observation 2 since this is also the condition for which $(x,y) \mapsto (1+|x| + |y|)^{\gamma}$ belongs to $L^2_{-s}(\mathbb{R}^2)$. In this case, since $k$ is no longer a parameter of the estimate, the observations 1 and the point (ii) of 2 of the previous paragraph are no longer relevant. However, the technical approach described in observation 3 still makes sense.

\subsubsection{Generalized eigenfunctions of surface spectral zones}
From now on, the forthcoming estimates will be established for  $\scZ \in \calZ\setminus\{\EE\}$,  $j \in J_{\scZ}\subset\{-1,1\}$ and $(k,\lambda) \in \Lambda_{\scZ}([a,b])$ where $[a,b] \subset \bbR\setminus\sigma_{\rm exc}$ is such that $\Lambda_{\scZ}([a,b]) \neq \varnothing$.

Also,  in order to symmetrize our estimates  with respect to $\thetakl^+$ and $\thetakl^-$ we introduce the quantity
\begin{equation}\label{eq.defthetamin}
\thetakl^{\min}:=\min(|\thetakl^{+}|, |\thetakl^{-}|), \quad  \mbox{ for }  \scZ \in \calZ\setminus\{\EE\} \mbox { and  } (k,\lambda)\in\Lambda_{\scZ}.
\end{equation}
Note that while $\thetakl^{\min}$ (and positive powers of $\thetakl^{\min}$)  remain bounded when $(k.\lambda) \in \overline{\Lambda_{\scZ}([a,b])}$ (in other words $\thetakl^{\min} \lesssim 1$),  negative powers of $\thetakl^{\min}$ blow up when $(k, \lambda)$  approaches  a spectral cut. \\ [6pt]
Since pointwise estimates of the functions $ w_{k,\lambda,j}$ and $\partial_x w_{k,\lambda,j}$ have already been obtained in section \ref{sec_EF-estimates} (see \eqref{eq.boundcrosspoint-w} for instance), in order to implement the process described above (observation 3), we simply need to get pointwise estimates of their $\lambda$-derivatives on $\Lambda_{\scZ}([a,b])$  (where these functions are smooth in $\lambda$ since $\Lambda_{\scZ}([a,b])$  does not intersect the spectral cut). These require estimates of $\lambda$-derivatives of various intermediate quantities that are the object of the next subsection. 
\subsubsection*{(I) Preliminary $\lambda$-derivatives estimates.}

\noindent {\bf (Ia) Derivatives of powers of $|\thetakl^{\pm}|$}:
From (\ref{eq.def-thetam}, \ref{eq.def-thetap}), in each $\Lambda_{\scZ}$, one has $ \theta_{k,\lambda }^{\pm}=a |\theta_{k,\lambda }^{\pm}|$ with $a^2=\pm 1$. Thus, it follows 
$$\forall \; \alpha \in \mathbb{R}, \quad |\partial_{\lambda} ( \theta_{k,\lambda }^{\pm})^\alpha | =| \partial_{\lambda} (|\theta_{k,\lambda }^{\pm}|^\alpha)|.$$ Furthermore, the relation $|\thetakl^{\pm}|=|\Thetakl^{\pm}|^{1/2}$ gives:
$$
|\partial_{\lambda} ( \theta_{k,\lambda }^{\pm})^\alpha | =| \partial_{\lambda} (|\theta_{k,\lambda }^{\pm}|^\alpha)| =|\partial_{\lambda}(|\Theta_{k,\lambda }^{\pm}|^{\frac{\alpha}{2}})|=\Big|\frac{\alpha}{2}  \Big|  \      \big|\partial_{\lambda} |\Theta_{k,\lambda }^{\pm}|\big| \  |\theta_{k,\lambda }^{\pm}|^{\alpha-2} \mbox{ for $\alpha\in \bbR$ }.
$$
As $[a,b] \subset \bbR\setminus\sigma_{\rm exc}$, $\Theta_{k,\lambda }^{\pm}$ is $C^{\infty}$ on $\overline{\Lambda_{\scZ}([a,b])}$ so that  $\big|\partial_{\lambda} |\Theta_{k,\lambda }^{\pm}|\big|= | \partial_{\lambda} \Theta_{k,\lambda }^{\pm}\big|\lesssim 1$. Thus, it yields
\begin{equation}\label{eq.deritheta-surfac}
|\partial_{\lambda} ( \theta_{k,\lambda }^{\pm})^\alpha | =| \partial_{\lambda} (|\theta_{k,\lambda }^{\pm}|^\alpha)|  \lesssim  |\theta_{k,\lambda }^{\pm}|^{\alpha-2}, \quad \forall  (k, \lambda) \in \Lambda_{\scZ}([a,b]) \mbox{ and $\alpha\in \bbR$}.
\end{equation}
{\bf(Ib) Derivatives of  the Wronskian $\calW_{k,\lambda}$}:
From the definition  \eqref{eq.disp}  of $\calW_{k,\lambda}$ and $\mu_{\lambda}^-=\mu_0$, one gets:
$$
\partial_{\lambda }\calW_{k,\lambda}=\frac{\partial_{\lambda}   \theta_{k,\lambda }^{-} }{\mu_0}+\frac{\partial_{\lambda}   \theta_{k,\lambda }^{+} }{\mu_{\lambda}^+}+   \theta_{k,\lambda }^{+} \, \partial_{\lambda}\, \Big(\frac{1}{\mu_{\lambda}^+}\Big).
$$
As $[a,b] \subset \bbR\setminus\sigma_{\rm exc}$,  $|\mu_{\lambda}^+|\lesssim 1$,  $|\partial_{\lambda} (1/ \mu_{\lambda}^+)|\lesssim 1$ for $\lambda \in [a,b]$.
Moreover $|\thetakl^{\pm}|\lesssim 1$ since $\thetakl^{\pm}$ is continuous on $\overline{\Lambda_{\scZ}([a,b])}$. Then it follows, with  \eqref{eq.deritheta-surfac} with $\alpha=1$:
\begin{equation}\label{eq.derivwronskian}
|\partial_{\lambda} \calW_{k,\lambda}|\lesssim |\thetakl^-|^{-1}+ |\thetakl^+|^{-1}.
\end{equation}
Furthermore, one has $\partial_{\lambda }(1/\calW_{k,\lambda} )= -\partial_{\lambda }\calW_{k,\lambda}/(\calW_{k,\lambda})^2$. Thus, combining  \eqref{eq.ineqA}  and  \eqref{eq.derivwronskian} leads to:
\begin{equation}\label{eq.derivinversewronskian}
|\partial_{\lambda}( \calW_{k,\lambda}^{-1})|\lesssim \frac{ |\thetakl^-|^{-1}+ |\thetakl^+|^{-1}}{ (|\thetakl^-|+  |\thetakl^+|)^2}, \quad \forall (k, \lambda) \in \Lambda_{\scZ}([a,b]).
\end{equation}
{\bf (Ic) Derivative of  the coefficients  $A_{k,\lambda,\pm 1}$}. 
From formula  \eqref{def-A-gen} and the fact that $\lambda\mapsto |\lambda/(2 \mu_\lambda^\mp) |^{1/2}$ is $C^{\infty}$ smooth for $(k,\lambda)$ in $\overline{\Lambda_{\scZ}([a,b])}$, one gets
$$
|\partial_{\lambda }A_{k,\lambda,\pm 1}| \lesssim (|\partial_{\lambda}(| \thetakl^{\mp}|^{\frac{1}{2}}) | +| \thetakl^{\mp}|^{\frac{1}{2}} )\ | \calW_{k,\lambda}|^{-1} +  |\thetakl^{\mp}|^{\frac{1}{2}}   \, |  \partial_{\lambda}( \calW_{k,\lambda}^{-1}) | .
$$
Thus combining  the estimate \eqref{eq.ineqA} for $\calW_{k,\lambda}^{-1}$,  \eqref{eq.deritheta-surfac} applied with $\alpha=1/2$, $| \thetakl^{\mp}|^{\frac{1}{2}} \lesssim  |\theta_{k\lambda }^{\mp}|^{-\frac{3}{2}}$  and  \eqref{eq.derivinversewronskian} gives:
\begin{equation}\label{eq.deriv-coefA}
|\partial_{\lambda}A_{k,\lambda,\pm 1} |\lesssim  \frac{ |\theta_{k\lambda }^{\mp}|^{-\frac{3}{2}}}{ |\thetakl^-|+ |\thetakl^+|}+ |\thetakl^{\mp}|^{\frac{1}{2}}  \frac{ |\thetakl^-|^{-1}+ |\thetakl^+|^{-1}}{ (|\thetakl^-|+  |\thetakl^+|)^2}, \quad  \forall \; (k, \lambda) \in \Lambda_{\scZ}([a,b]).
\end{equation}
As $|\thetakl^{\mp}|^{1/2}/ (|\thetakl^-|+  |\thetakl^+|)\leq |\thetakl^{\mp}|^{-1/2}$ and  we have by definition \eqref{eq.defthetamin} of $\thetakl^{\min}$:
$$
|\theta_{k\lambda }^{\mp}|^{-\frac{3}{2}} \leq (\thetakl^{\min})^{-\frac{3}{2}} , \quad |\theta_{k\lambda }^{\mp}|^{-\frac{1}{2}} \leq (\thetakl^{\min})^{-\frac{1}{2}} , \quad |\thetakl^-|^{-1}+ |\thetakl^+|^{-1} \lesssim (\thetakl^{\min})^{-1}, \quad |\thetakl^-|+  |\thetakl^+| \geq \thetakl^{\min},
$$
we deduce from \eqref{eq.deriv-coefA} that
\begin{equation}\label{eq.derivA-1}
|\partial_{\lambda} A_{k, \lambda,\pm 1}| \lesssim (\thetakl^{\min})^{-\frac{5}{2}}, \quad \forall (k, \lambda) \in \Lambda_{\scZ}([a,b]).
\end{equation}
Note that moreover, if $\pm \, \Oc \notin [a,b]$, we can exploit in \eqref{eq.deriv-coefA} the fact $|\thetakl^{+}|$ and  $|\thetakl^{-}| $ cannot vanish simultaneously, in other words that $|\thetakl^{+}| + |\thetakl^{-}|$ is bounded from below, to obtain the improved estimate 
\begin{equation}\label{eq.derivA-1.improved}
|\partial_{\lambda} A_{k, \lambda,\pm 1}| \lesssim (\thetakl^{\min})^{-\frac{3}{2}}, \quad \forall (k, \lambda) \in \Lambda_{\scZ}([a,b]).
\end{equation}

\subsubsection*{(II) H\"older-type estimates of  generalized eigenfunctions for $j = \pm 1$}\label{section.spec derivplasm} 
Each $\bbW_{k,\lambda,j}$ is constructed from $w_{k,\lambda,j}$ (and its $x$ derivative), which is itself  constructed from $\psi_{k,\lambda,j}$ (and its $x$ derivative). We study below the $\lambda$-derivatives of these functions in the reverse order. \\ [12pt]
\noindent {\bf (IIa) Pointwise estimates the $\lambda$-derivative of  the functions  $\psi_{k,\lambda,\pm 1}(x)$ and $\partial_x \psi_{k,\lambda,\pm 1}(x)$  }

\begin{lemma} \label{lem_estimderpsiA} 
	Let $\scZ \in \calZ\setminus\{\EE\}$ and $[a,b] \subset \bbR\setminus\sigma_{\rm exc}$ such that $\Lambda_{\scZ}([a,b]) \neq \varnothing$.
	If $\pm 1 \in J_{\scZ}$, then one has  for all $(k,\lambda)\in \Lambda_{\scZ}([a,b])$  the following pointwise estimates 
	\begin{eqnarray}
	& \forall \; x \in \mathbb{R}, \quad 	\displaystyle \big|\partial_{\lambda}\psi_{k,\lambda,\pm 1}(x)\big|& \lesssim (1+|x|)\,  \Big(  \frac{1}{|\thetakl^{+ } |} +  \frac{1}{|\thetakl^{-}| } +  \frac{| \thetakl^\pm|}{|\thetakl^{\mp}|^2}\Big)  \label{eq.maj-deriv-psi-spec},\\[6pt]
	&\forall \; x \in \mathbb{R}, \quad 	\big|\partial_{\lambda}\psi_{k,\lambda,\pm 1}(x)\big| &\lesssim  (1+|x|) \,  (\thetakl^{\min})^{-2} \label{eq.maj-deriv-psi-spec-simpl},\\[6pt]
	& \quad \  \forall \; x \in \mathbb{R}^*,  \quad	\big|  \partial_{\lambda}\partial_x \psi_{k,\lambda,\pm 1}(x)\big|&\lesssim  (1+|x|) \,  (\thetakl^{\min})^{-1}  . \label{eq.estimate-deriv-lambda-deriv-psi-2-mp}
	\end{eqnarray}
\end{lemma}
\begin{proof}
As for proposition \ref{prop.decay-estim-surfacic-zones}, we give the proof for $j=1$ (which means that $\scZ \in \{\DI, \DE, \DD \}$) and let the reader check by simple symmetry arguments (between $-1$ and $+1$ on one hand, $x< 0$ and $x > 0$ on the other hand) that it also works for $j=-1$. \\[6pt]
	{\it Step 1: proof of \eqref{eq.maj-deriv-psi-spec}.}
	\begin{itemize}
		\item[(i)] For $x\geq 0$, one has  $\psi_{k,\lambda, 1}(x)=\rme^{- \thetakl^{+}x}$ (cf. formula \eqref{eq-def-phi}) thus:
		$$
		\partial_{\lambda}\psi_{k,\lambda,1}(x)=- x \, (\partial_{\lambda} \thetakl^{+} )\, \rme^{- \thetakl^{+} x}
		$$
		with $\thetakl^{+}>0$ or $\thetakl^{+}\in \rmi\mathbb{R}$. Thus, combining the inequality \eqref{eq.deritheta-surfac}  for $\alpha=1$ and  $| \rme^{- \thetakl^{+} x}|\leq 1$ yields immediately:
		\begin{equation}\label{eq.mapj-psi-spec-positif}
		\forall \; x \geq 0, \quad \big|\partial_{\lambda}\psi_{k,\lambda,1}(x)\big| \lesssim   |\theta_{k,\lambda }^{+}|^{-1} \,  |x| .
		\end{equation}
		\item[(ii)]  For $x<0$, setting $ q_\lambda :=(\thetakl^+/ \mu_\lambda^+)/( \thetakl^-/ \mu_\lambda^-)$, formula \eqref{eq-def-phi} for $\psi_{k,\lambda,\pm 1}$ that, 
		$$
		\psi_{k,\lambda,1}(x) = \cosh\big( \thetakl^-\, x \big) - q_\lambda\ \sinh\big( \thetakl^-\, x \big).
		$$
		Therefore, one computes that
		\begin{equation}\label{eq.deriv-psi-spec-negatif}
		\partial_{\lambda}\psi_{k,\lambda,1}(x)=x \, \partial_{\lambda}( \thetakl^{-}) \, \sinh(\thetakl^{-} \, x)- 
		\partial_{\lambda} q_\lambda \,  
		\sinh(\thetakl^{-} \, x)- q_\lambda  \, x\,  \partial_{\lambda}( \thetakl^{-})\,  \cosh(\thetakl^- \, x).
		\end{equation}
		We now bound successively the three terms of the right hand side of \eqref{eq.deriv-psi-spec-negatif}. For the first term,  thanks to \eqref{eq.deritheta-surfac}  for $\alpha=1$ and  $| \sinh(\thetakl^{-} \, x)|\leq 1$ since $\thetakl^{-} \in \rmi \mathbb{R}$  (see \eqref{eq.def-thetam}), one gets:
		\begin{equation}\label{eq.maj-deriv-psi-spec-negatif1}
		|x \, \partial_{\lambda}( \thetakl^{-})\, \sinh(\thetakl^{-} \, x)|  \lesssim |x| \,  |\thetakl^{-}|^{-1}.
		\end{equation}
		For the second term,  
		as $[a,b] \subset \bbR\setminus\sigma_{\rm exc}$,  $| \mu_\lambda^-/ \mu_\lambda^{+}|\lesssim 1$ and $|\partial_{\lambda} (\mu_\lambda^-/ \mu_\lambda^{+}) | \lesssim 1$, one obtains  
		$$
		\big|\partial_{\lambda}q_\lambda\big| \lesssim \Big| \frac{\thetakl^+}{\thetakl^-} \Big|+ | \partial_{\lambda} \thetakl^+| \, |\thetakl^-|^{-1}+| \thetakl^+| \, |\partial_{\lambda} (\thetakl^-)^{-1}|,
		$$
		which gives with  \eqref{eq.deritheta-surfac} applied successively  for $\alpha=1$ and $\alpha=-1$:
		$$
		\big|\partial_{\lambda}q_\lambda\big|   \lesssim \Big| \frac{\thetakl^+}{\thetakl^-} \Big|+ \frac{1}{|\thetakl^-| \, |\thetakl^+|}+ \frac{| \thetakl^+|}{|\thetakl^-|^3}  .
		$$
		As $|\thetakl^-|^{-1} \lesssim |\thetakl^-|^{-3} $ (since $|\thetakl^-|^{2}\lesssim 1$), the previous  inequality simplifies to:
		\begin{equation}\label{eq.maj-deriv-psi-spec-negatif2-inter}
		\big|\partial_{\lambda}q_\lambda\big|   \lesssim  \frac{1}{|\thetakl^-| \, |\thetakl^+|}+ \frac{| \thetakl^+|}{|\thetakl^-|^3}  .
		\end{equation}
		Hence, combining \eqref{eq.maj-deriv-psi-spec-negatif2-inter} and $| \sinh(\thetakl^{-} \, x)|\leq  |\thetakl^{-} | \, |x| $ yields:
		\begin{equation}\label{eq.maj-deriv-psi-spec-negatif2}
		\big|\partial_{\lambda}q_\lambda \,  \sinh(\thetakl^{-} \, x) \big| \lesssim \Big(  \frac{1}{|\thetakl^+| } +  \frac{| \thetakl^+|}{|\thetakl^-|^2}\Big) \, |x| .
		\end{equation}
		For the third term,  $ |\mu_\lambda^-|/|\mu_\lambda^+|\lesssim 1$ gives $|q_\lambda| \lesssim |\thetakl^+| / |\thetakl^-|$ and $| \cosh(\thetakl^{+} \, x)|\leq  1$. Thus,  one gets
		\begin{equation}\label{eq.maj-deriv-psi-spec-negatif3}
		\big| q_\lambda \, x\,  \partial_{\lambda}( \thetakl^{-})\,  \cosh(\thetakl^{-} \, x)\big| \lesssim 
		\frac{|\thetakl^+|}{|\thetakl^-|} \; \big|\partial_{\lambda} \thetakl^{-} \big| \, |x| \lesssim  \frac{| \thetakl^+|}{|\thetakl^-|^2}\, |x| .
		\end{equation}
		where we have used \eqref{eq.deritheta-surfac} with $\alpha=1$ for the last inequality.
		Gathering the estimates \eqref{eq.maj-deriv-psi-spec-negatif1}, \eqref{eq.maj-deriv-psi-spec-negatif2} and \eqref{eq.maj-deriv-psi-spec-negatif3} in \eqref{eq.deriv-psi-spec-negatif} gives \eqref{eq.maj-deriv-psi-spec}  for $x<0$. One observes with \eqref{eq.mapj-psi-spec-positif}  that   \eqref{eq.maj-deriv-psi-spec}   holds also for $x\geq 0$ and thus \eqref{eq.maj-deriv-psi-spec} is proved.
	\end{itemize}
	
	\noindent {\it Step 2 : proof of \eqref{eq.maj-deriv-psi-spec-simpl}.}
	Inequality \eqref{eq.maj-deriv-psi-spec-simpl} is an immediate consequence of \eqref{eq.maj-deriv-psi-spec} by using  the bounds: $|\thetakl^{+}|\lesssim 1$ and $|\thetakl^\pm|^{-1}\lesssim (\thetakl^{\min})^{-1}\lesssim (\thetakl^{\min})^{-2}$.\\

	\noindent {\it Step 3: proof of \eqref{eq.estimate-deriv-lambda-deriv-psi-2-mp}.} From formula \eqref{eq-express-deriv-x}, we compute, using the chain rule, that for $x>0$
	$$
	\partial_{\lambda} \partial_x \psi_{k,\lambda,1}(x) = - \, \big( 1 - x \, \thetakl^{+} \big) \, \partial_{\lambda}  \thetakl^{+} \; \rme^{- \thetakl^{+}x}
	$$
	so that, using \eqref{eq.deritheta-surfac}  for $\alpha = 1$,
	\begin{equation*}\label{eq.estimate-deriv-lambda-deriv-psi-2-pm}
	\big|  \partial_{\lambda} \partial_x \psi_{k,\lambda,1}(x) \big| \lesssim |\thetakl^{+}| \, |\partial_{\lambda}  \thetakl^{+} |  \,|x|+ \, |\partial_{\lambda}  \thetakl^{+} |  \lesssim |x|+ |\thetakl^{+}|^{-1} \lesssim (1+|x|) \,  (\thetakl^{\min})^{-1},
	\end{equation*} 
	that is to say \eqref{eq.estimate-deriv-lambda-deriv-psi-2-mp} for $x>0$. On the other hand, for  $x<0$, using again \eqref{eq-express-deriv-x}, one has 
	$$
	\left| \begin{array}{lll}
	\partial_{\lambda} \partial_x \psi_{k,\lambda,1}(x) & = & \partial_{\lambda}  \thetakl^{-} \big(  \sinh (\thetakl^{-} x)+x \, \thetakl^{-} \cosh (\thetakl^{-} x) \big) \\ [10pt]
	& &\displaystyle -\Big[ \partial_{\lambda}  \thetakl^{+}  \; \frac{\mu_\lambda^-}{\mu_\lambda^+}+  \thetakl^{+} \, \partial_\lambda\Big(\frac{\mu_\lambda^-}{\mu_\lambda^+}\Big)  \Big] \; \cosh (\thetakl^{-} x)-  \thetakl^+ \;  \frac{\mu_\lambda^-}{\mu_\lambda^+} \   \partial_{\lambda}  \thetakl^{-}  \ x \, \sinh (\thetakl^{-} x) .
	\end{array} \right.
	$$
	Then, since  $\thetakl^{-} \in \rmi \mathbb{R}$, see (\ref{eq.def-thetam}), $|\mu_{\lambda}^{-}/\mu_{\lambda}^{+}| \lesssim 1$, $|\partial_{\lambda}(\mu_{\lambda}^{-}/\mu_{\lambda}^{+})| \lesssim 1$, $|\thetakl^{\mp}|\lesssim 1 $ and $|\partial_{\lambda}\thetakl^{\mp}| \lesssim  (\thetakl^{\min})^{-1}$ (by \eqref{eq.deritheta-surfac}  for $\alpha=1$),
	one obtains \eqref{eq.estimate-deriv-lambda-deriv-psi-2-mp}  for  $x<0$.
\end{proof}

\noindent{\bf (IIb) Pointwise estimates the $\lambda$-derivative of  the functions  $w_{k,\lambda,\pm 1}(x,y)$ and  $\partial_x w_{k,\lambda,\pm 1}(x,y)$.} 
\begin{lemma} \label{lem_estimderpsiB}
	Let   $\scZ \in \calZ\setminus\{\EE\}$ and $[a,b] \subset \bbR\setminus\sigma_{\rm exc}$ such that $\Lambda_{\scZ}([a,b]) \neq \varnothing$. Then, for $j \in J_{\scZ}$, one has   for all $(k,\lambda)\in \Lambda_{\scZ}([a,b])$  the following pointwise estimates 
	\begin{equation}\label{eq.deriv-w-surfac-inter-crosspointsA}
	\forall (x,y)\in \bbR^2, \quad 	|\partial_{\lambda} w_{k,\lambda, j}(x,y)| \lesssim  (\thetakl^{\min})^{-\frac{5}{2}} \, (1+|x|), \  .
	\end{equation}
	\begin{equation}\label{eq.estimate-deriv-lambda-deriv-psiA}
	\forall \; (x,y) \in \bbR^* \times \bbR, \quad | \partial_{\lambda} \partial_x w_{k,\lambda,j}(x,y)|  \lesssim  (\thetakl^{\min})^{-\frac{3}{2}} (1+|x|) .
	\end{equation} 
	If moreover  $\pm \, \Oc \notin [a,b]$, \eqref{eq.deriv-w-surfac-inter-crosspointsA} can be improved into 
	\begin{equation}\label{eq.deriv-w-surfac-inter-crosspointsA-improved}
	\forall (x,y)\in \bbR^2, \quad 	|\partial_{\lambda} w_{k,\lambda,j}(x,y)| \lesssim  (\thetakl^{\min})^{-\frac{3}{2}} \, (1+|x|)  .
	\end{equation}
\end{lemma}
\begin{proof} As in Lemma \ref{lem_estimderpsiA}, we give only  the proof for $j=1$. We naturally separate this proof in three steps. \\ [12pt]
	{\it Step 1 : proof of \eqref{eq.deriv-w-surfac-inter-crosspointsA}.} It follows from the formula \eqref{eq.def-w} for $w_{k,\lambda, 1}$ that:
	\begin{equation}\label{eq.deriv-w-surfac-inter}
	|\partial_{\lambda} w_{k,\lambda, 1}(x,y)| \lesssim  |\partial_{\lambda} A_{k, \lambda,1}| \,   |\psi_{k,\lambda, 1}(x)|+| A_{k, \lambda,1} | \,  |\partial_{\lambda} \psi_{k,\lambda, 1}(x)|. 
	\end{equation}
	One one hand, we have the estimate \eqref{eq.derivA-1} for $  \partial_{\lambda} A_{k, \lambda,1}$, namely
	\begin{equation}\label{eq.derivA-1bis}
	|\partial_{\lambda} A_{k, \lambda,1}| \lesssim (\thetakl^{\min})^{-\frac{5}{2}}.
	\end{equation}
	On the other hand, we already showed (see \eqref{eq-boundpsi-new}  in the proof of proposition \ref{prop.decay-estim-surfacic-zones}, used for $\gamma = 1$) 
	\begin{equation}\label{eq-boundpsibis}
	\forall \; x \in \bbR, \quad |\psi_{k,\lambda, 1}(x)| \leq 1+|x|.
	\end{equation}  
	Next, using $|\thetakl^-|+|\thetakl^+| \geq |\thetakl^-|$ in \eqref{eq.boundAcoefficient}, we deduce
	\begin{equation}\label{eq.boundAcoefficient2}
	|A_{k,\lambda,1}|\lesssim |\thetakl^-|^{-\frac{1}{2}} \lesssim (\thetakl^{\min})^{-\frac{1}{2}},
	\end{equation}  
	while, from Lemma  \ref{lem_estimderpsiA} (estimate \eqref{eq.maj-deriv-psi-spec-simpl}), we also have 
	\begin{equation*}\label{eq. boundderpsi}
	\forall \; x \in \mathbb{R}, \quad 	\big|\partial_{\lambda}\psi_{k,\lambda,1}(x)\big| \lesssim   (\thetakl^{\min})^{-2} \, (1+|x|).
	\end{equation*} 
	Finally, using (\ref{eq.derivA-1bis}, \ref{eq-boundpsibis}, \ref{eq.boundAcoefficient2}, \ref{eq. boundderpsi}) in \eqref{eq.deriv-w-surfac-inter} yields
	\eqref{eq.deriv-w-surfac-inter-crosspointsA}. \\ [12pt]
	{\it Step 2 : proof of \eqref{eq.estimate-deriv-lambda-deriv-psiA}.}
	Using \eqref{eq.def-w} again, one has for all $x\in \bbR^*$  and  $y \in \bbR$:
	\begin{equation}\label{eq-deriv-lambda-deriv-x-ineq-triang}
	| \partial_{\lambda} \partial_x w_{k,\lambda,1}(x,y)| \leq  |\partial_{\lambda} A_{k, \lambda,1}|  \, \big| \partial_x  \psi_{k,\lambda,1}(x)\big| +   |A_{k,\lambda,1}| \, \big|  \partial_{\lambda}\partial_x  \psi_{k,\lambda,1} (x)\big| .
	\end{equation}
	For the first term, we have to be careful since using directly would lead to a non optimal estimate. Instead of this, we have to take profit of the cancellation of terms when doing the product of the two estimates \eqref{eq.deriv-1} (established in section \ref{sec_EF-estimates}, proof of proposition \ref{prop.decay-estim-surfacic-zones}) and \eqref{eq.deriv-coefA}, namely
	\begin{equation*}\label{twoestimates}
	\big| \partial_x  \psi_{k,\lambda,1}(x)\big|\lesssim |\thetakl^{-}|+ |\thetakl^{+}|, \quad |\partial_{\lambda}A_{k,\lambda,1} |\lesssim  \frac{ |\theta_{k\lambda }^{-}|^{-\frac{3}{2}}}{ |\thetakl^-|+ |\thetakl^+|}+ |\thetakl^{-}|^{\frac{1}{2}}  \frac{ |\thetakl^-|^{-1}+ |\thetakl^+|^{-1}}{ (|\thetakl^-|+  |\thetakl^+|)^2}
	\end{equation*}
	which yields
	\begin{equation} \label{firstterm}
	|\partial_{\lambda} A_{k, \lambda,1}|  \, \big| \partial_x  \psi_{k,\lambda,1}(x)\big| \lesssim |\theta_{k\lambda }^{-}|^{-\frac{3}{2}}+ |\thetakl^{-}|^{\frac{1}{2}}  \frac{ |\thetakl^-|^{-1}+ |\thetakl^+|^{-1}}{ |\thetakl^-|+  |\thetakl^+|} \lesssim   (\thetakl^{\min})^{-\frac{3}{2}}.
	\end{equation}
	The second term is easier. We simply combine \eqref{eq.boundAcoefficient2} with the estimate \eqref{eq.estimate-deriv-lambda-deriv-psi-2-mp} of Lemma \ref{lem_estimderpsiA}, namely 
	\begin{equation*}\label{rappel}
	\Big|  \partial_{\lambda}\partial_x \psi_{k,\lambda,1}(x)\Big| \lesssim  (1+|x|) \,  (\thetakl^{\min})^{-1}  .
	\end{equation*} 
	to obtain 
	\begin{equation}\label{secondterm}
	|A_{k,\lambda,1}| \, \big|  \partial_{\lambda}\partial_x  \psi_{k,\lambda,1} (x)\big| \lesssim (1+|x|) \,  (\thetakl^{\min})^{-\frac{3}{2}}.
	\end{equation}
	~\\[-6pt]Finally,  \eqref{eq.estimate-deriv-lambda-deriv-psiA}  follows from \eqref{eq-deriv-lambda-deriv-x-ineq-triang}, \eqref{firstterm} and \eqref{secondterm}.
	\\ [12pt]
	{\it Step 3: proof of \eqref{eq.deriv-w-surfac-inter-crosspointsA-improved}.}
	Since $\pm \, \Oc \notin [a,b]$, we can use the improved estimates \eqref{eq.derivA-1.improved} for $\partial_{\lambda } A_{k,\lambda,1}$ (instead of \eqref{eq.derivA-1}) and with \eqref{eq.deriv-w-surfac-inter} and \eqref{eq-boundpsibis}, one obtains:
	\begin{equation}\label{eq-deriv-lambda-deriv-x-noOC}
	| \partial_{\lambda} w_{k,\lambda,1}(x,y)| \lesssim   (\thetakl^{\min})^{-\frac{3}{2}} (1+|x|) +   |A_{k,\lambda,1}| \, \big|  \partial_{\lambda}  \psi_{k,\lambda,1} (x)\big| .
	\end{equation}
	For the second term of the right hand side of \eqref{eq-deriv-lambda-deriv-x-noOC}, combining \eqref{eq.boundAcoefficient} and \eqref{eq.maj-deriv-psi-spec} leads to:
	\begin{equation*}\label{eq.secondterm-interm}
	| A_{k, \lambda, 1} | \,  |\partial_{\lambda} \psi_{k,\lambda,1}(x)| \lesssim (|\thetakl^-|+  |\thetakl^+|)^{-1}\Big(  \frac{|\thetakl^-|^{\frac{1}{2}} }{|\thetakl^{+}| } + \frac{1}{|\thetakl^{-}|^{\frac{1}{2}} } +  \frac{| \thetakl^+|}{|\thetakl^-|^{\frac{3}{2}}}\Big) \,(1+ |x|) , \ \forall x\in \bbR.
	\end{equation*}
	We point out that in the last expression, we use the more precise inequality  \eqref{eq.maj-deriv-psi-spec}  instead of  \eqref{eq.maj-deriv-psi-spec-simpl} to simplify in the product a  $|\thetakl^-|^{-1/2}$ term.
	As $( |\thetakl^-|+  |\thetakl^+|)^{-1}\lesssim 1$ (since $\pm \Oc\notin[a,b]$), $|\thetakl^\mp|\lesssim1$ and  $|\thetakl^{\pm}|^{-1}\lesssim   (\thetakl^{\min})^{-1}$, one deduces that:
	\begin{equation}\label{eq.secondterm}
	| A_{k, \lambda, 1} | \,  |\partial_{\lambda} \psi_{k,\lambda, 1}(x)| \lesssim    (\thetakl^{\min})^{-\frac{3}{2}}   \,(1+ |x| ), \ \forall x\in \bbR.
	\end{equation}
	Combining \eqref{eq-deriv-lambda-deriv-x-noOC} and  \eqref{eq.secondterm} yields finally the estimate \eqref{eq.deriv-w-surfac-inter-crosspointsA-improved}.
\end{proof}

\subsubsection*{(IIc) H\"older-type estimates for $\bbW_{k,\lambda,\pm 1}$.} 
We are now in position to prove our ``H\"older type'' inequalities for $\bbW_{k,\lambda,\pm 1}$.
\begin{proposition}\label{prop.holdreggeneralizedeigenfunctions}
	Let  $s > 1/2$, $\scZ \in \calZ\setminus\{\EE\}$, $\gamma \in (0,1]\cap (0,s-1/2)$ and $[a,b] \subset \bbR\setminus\sigma_{\rm exc}$ such that $\Lambda_{\scZ}([a,b]) \neq \varnothing$. Then,  there exists $C_{a,b}^{\gamma} > 0$  such that for  $j\in J_{\scZ}$:
	\begin{equation}\label{eq.ineqboundhold1crosspoint}
	\|  \bbW_{k,\lambda',j}-\bbW_{k,\lambda,j}\|_{\Hms}\leq  C_{a,b}^{\gamma} \, \sup_{\tilde{\lambda}\in [\lambda,\lambda']}  (\thetaklt^{\min})^{-\frac{1}{2}- 2 \gamma} \   |\lambda'-\lambda|^{\gamma}  , \ \forall (k, \lambda), (k,\lambda')\in\Lambda_{\scZ}([a,b]) \mbox{ and } \lambda\leq \lambda'.
	\end{equation}
	If moreover,  $\pm \Oc\notin [a,b]$ , then there exists  $C_{a,b}^{\gamma} > 0$  such that for  $j\in J_{\scZ}$:
	\begin{equation}\label{eq.ineqboundhold1nocrosspoint}
	\| \bbW_{k,\lambda',j}- \bbW_{k,\lambda,j}\|_{\Hms}\leq C_{a,b}^{\gamma} \, \sup_{\tilde{\lambda}\in [\lambda,\lambda']}  (\thetaklt^{\min})^{-\frac{1}{2}- \gamma} \ \,  |\lambda'-\lambda|^{\gamma}  , \ \forall (k, \lambda), (k,\lambda')\in\Lambda_{\scZ}([a,b]) \mbox{ and } \lambda\leq \lambda'.
	\end{equation}
\end{proposition}
\begin{proof}
	We detail the proof for  $j=1$ ($j=-1$ follows by ``symmetry arguments''). Let us first prove \eqref{eq.ineqboundhold1crosspoint}. \\ [12pt] 
	We proceed as explained in observation 3 at the beginning of this section. First of all, by using the mean value Theorem  and the estimate \eqref{eq.deriv-w-surfac-inter-crosspointsA} for $\partial_\lambda w_{k,\lambda,1}(x,y)$, one gets:
	\begin{equation}\label{eq.interpol-hold-1}
	| w_{k,\lambda',1}(x,y) -  w_{k,\lambda,1}(x,y)| \lesssim  (1+|x|) \sup_{\tilde{\lambda}\in [\lambda,\lambda']}  (\thetaklt^{\min})^{-\frac{5}{2}}    \ |\lambda'-\lambda|  . 
	\end{equation}
	On the other hand, from the pointwise estimate \eqref{eq.boundcrosspoint-w} for $w_{k,\lambda,1}$, we also have (by simply bounding the modulus of the difference by the sum of the moduli): 
	\begin{equation}\label{eq.interpol-hold-2}
	| w_{k,\lambda',1}(x,y) -  w_{k,\lambda,1}(x,y)|  \lesssim \sup_{\tilde{\lambda}\in [\lambda,\lambda']}   (\thetakl^{\min})^{-\frac{1}{2}} .
	\end{equation}
	Interpolating between \eqref{eq.interpol-hold-1} and \eqref{eq.interpol-hold-2} with $\gamma \in (0,1]$ , we get, as $\gamma \, (-\frac{5}{2}) + (1-\gamma)(-\frac{1}{2}) =-\frac{1}{2}-2\gamma$,  
	\begin{equation}\label{eq.interpol-hold-3}
	| w_{k,\lambda',1}(x,y) -  w_{k,\lambda,1}(x,y)| \lesssim  (1+|x|)^{\gamma} \sup_{\tilde{\lambda}\in [\lambda,\lambda']}  (\thetaklt^{\min})^{-\frac{1}{2}-2\gamma}   \ |\lambda'-\lambda|^{\gamma}.
	\end{equation}
	Thus, as  $x\mapsto  (1+|x|)^{\gamma}\in L^2_{-s}(\bbR^2)$ thanks to $\gamma < s - 1/2$, \eqref{eq.interpol-hold-3} implies:
	\begin{equation}\label{eq.hold-firstcomponentcrosspoint}
	\| w_{k,\lambda',1}-  w_{k,\lambda,1}\|_{L^2_{-s}(\bbR^2)} \lesssim   \sup_{\tilde{\lambda}\in [\lambda,\lambda']}  (\thetaklt^{\min})^{-\frac{1}{2}-2\gamma}   \ |\lambda'-\lambda|^{\gamma},\quad \forall (k,\lambda), (k,\lambda')\in\Lambda_{\scZ}([a,b])  \mbox{ and } \lambda\leq \lambda'.
	\end{equation}
	With the additional  assumption that $\pm \Oc \notin [a,b]$,  we can use  the better inequality  \eqref{eq.deriv-w-surfac-inter-crosspointsA-improved} for $\partial_\lambda w_{k,\lambda,1}$ instead of  \eqref{eq.deriv-w-surfac-inter-crosspointsA} , which leads to an improved version of \eqref{eq.interpol-hold-1}
	where $-5/2$ is replaced by $-3/2$. Interpolating again with \eqref{eq.interpol-hold-2}, one obtains that  for any $\gamma \in (0,1]$:
	\begin{equation}\label{eq.hold-firstcomponentnocrosspoint}
	\| w_{k,\lambda',1}-  w_{k,\lambda,1}\|_{L^2_{-s}(\bbR^2)} \lesssim   \sup_{\tilde{\lambda}\in [\lambda,\lambda']}  (\thetaklt^{\min})^{-\frac{1}{2}-\gamma}   \ |\lambda'-\lambda|^{\gamma}, \quad \forall (k,\lambda), (k,\lambda')\in\Lambda_{\scZ}([a,b])  \mbox{ and } \lambda\leq \lambda'.
	\end{equation}
	The estimate \eqref{eq.hold-firstcomponentcrosspoint} (resp.  \eqref{eq.hold-firstcomponentnocrosspoint}) is nothing but the inequality  \eqref{eq.ineqboundhold1crosspoint} (resp. \eqref{eq.ineqboundhold1nocrosspoint})  for the first component $w_{k,\lambda,1}$ of $\bbW_{k,\lambda,1}$. 
	It remains to show similar estimates for the other five components of $\bbW_{k,\lambda,1}$.  \\ [6pt] 
	According to \eqref{eq.def-Vkbis}, the second  component  of $\bbW_{k,\lambda,1}$ is
	$k  \, (\lambda \, \mu_\lambda)^{-1}   \,w_{k,\lambda,1}.$ On one hand,  the coefficient $k  \, (\lambda \, \mu_\lambda)^{-1}$ is bounded and smooth on the compact set $\overline{\Lambda_{\scZ}([a,b])}$. 
	On the other hand from \eqref{eq-boundw4}, one has for $\alpha=\gamma$ or $2\gamma$: $$\|w_{k,\lambda,1}\|_{L^2_{-s}(\bbR^2)}\lesssim (\thetakl^{\min})^{-1/2} \lesssim \sup_{\tilde{\lambda}\in [\lambda,\lambda']}  (\thetaklt^{\min})^{-1/2-\alpha}, \quad \forall (k,\lambda), (k,\lambda') \in \overline{\Lambda_{\scZ}([a,b])} .$$ 
	Thus,  from \eqref{eq.hold-firstcomponentcrosspoint} (resp. \eqref{eq.hold-firstcomponentnocrosspoint}), one derives  for the second component  (seen as the product  of $k  \, (\lambda \, \mu_\lambda)^{-1} $ by $ w_{k,\lambda,1}$)  an  H\"older  estimate 
	of the form \eqref{eq.hold-firstcomponentcrosspoint} (resp. \eqref{eq.hold-firstcomponentnocrosspoint} if $\pm \Oc\notin [a,b]$).

	The third component of $\bbW_{k,\lambda,1}$ is given by $\rmi/(\mu_{\lambda}\, \lambda) \, \partial_x w_{k,\lambda,1}$. 
	We first establish an estimates of the form \eqref{eq.hold-firstcomponentcrosspoint}  for $\partial_x w_{k,\lambda,1}$. 
	To do so, we proceed as for $w_{k,\lambda,1}$ at the beginning of this proof. First, by using the mean value Theorem and the estimate \eqref{eq.estimate-deriv-lambda-deriv-psiA} for $\partial_\lambda \partial_x w_{k,\lambda,1}(x,y)$, one gets:
	\begin{equation}\label{eq.interpol-hold-1bis}
	| \partial_x w_{k,
		\lambda',1}(x,y) -  \partial_x w_{k,\lambda,1}(x,y)| \lesssim  (1+|x|) \sup_{\tilde{\lambda}\in [\lambda,\lambda']}  (\thetaklt^{\min})^{-\frac{3}{2}}    \ |\lambda'-\lambda|,
	\end{equation}
	which is ``better'' than the same for $w_{k,\lambda,1}$ (cf. \eqref{eq.interpol-hold-1}) since $(\thetakl^{\min})^{-\frac{5}{2}}$ is replaced  by  $(\thetakl^{\min})^{-\frac{3}{2}}$. \\ [6pt]
	On the other hand, it follows from the estimate  \eqref{eq.etimatethirdcomponent} (since $|\lambda \, \mu_\lambda|\lesssim 1$) that
	\begin{equation}\label{eq.interpol-hold-2bis}
	| (\partial_x w_{k,\lambda',1}-  \partial_x w_{k,\lambda,1})(x,y)|  \leq | \partial_x w_{k,\lambda', 1}(x,y) |+ | \partial_x w_{k,\lambda, 1}(x,y) |\lesssim |\thetakl^{- }|^{\frac{1}{2}} + |\thetaklp^{-}|^{\frac{1}{2}} \lesssim 1 . 
	\end{equation}
	The interpolation between \eqref{eq.interpol-hold-1bis} and \eqref{eq.interpol-hold-2bis}  leads to 
	\begin{equation*}\label{eq.interpol-hold-3bis}
	| \partial_x w_{k,\lambda',1}(x,y) -  \partial w_{k,\lambda,1}(x,y)| \lesssim  (1+|x|)^{\gamma} \sup_{\tilde{\lambda}\in [\lambda,\lambda']}  (\thetaklt^{\min})^{-\frac{3}{2} \, \gamma}   \ |\lambda'-\lambda|^{\gamma}  ,
	\end{equation*}
	and yields 
	\begin{equation}\label{eq.estimderivxl2s}
	\| \partial_x w_{k,\lambda',1}-  \partial_x w_{k,\lambda,1}\|_{L^2_{-s}(\bbR^2)} \lesssim  \sup_{\tilde{\lambda}\in [\lambda,\lambda']}  (\thetaklt^{\min})^{-\frac{3}{2}\gamma}   \ |\lambda'-\lambda|^{\gamma},\  \forall (k,\lambda),  (k,\lambda')\in\Lambda_{\scZ}([a,b])  \mbox{ and } \lambda\leq \lambda'.
	\end{equation}
	Moreover,
	\eqref{eq.interpol-hold-2bis} implies that $\| \partial_x w_{k,\lambda,1}\|_{L^2_{-s}(\bbR^2)} \lesssim 1$ and  $ \lambda \mapsto \rmi/ (\mu_{\lambda}\, \lambda)$ is smooth in $\lambda $ on $\overline{\Lambda_{\scZ}([a,b])}$. 
	Thus,  the third component $\rmi/(\mu_{\lambda}\, \lambda) \, \partial_x w_{k,\lambda,1}$ (seen as the  product of $\rmi/(\mu_{\lambda}\, \lambda) $ by $ \partial_x w_{k,\lambda,1}$) satisfies also an estimate of the form \eqref{eq.estimderivxl2s}. We point out that 
	$$ (\thetaklt^{\min})^{-\frac{3}{2}\gamma}  \lesssim   (\thetaklt^{\min})^{-\frac{1}{2}-\gamma} \mbox{ for }  \gamma\leq 1, $$ thus the weaker estimates  \eqref{eq.hold-firstcomponentcrosspoint} and \eqref{eq.hold-firstcomponentnocrosspoint}  hold for the third component (which is less singular than than the first one). \\

Finally,  the last three components are $0$ for $x < 0$ and proportional  with a coefficient that  is smooth and  bounded in $\lambda$  to the first three components for $x > 0$ (see \eqref{eq.def-Vkbis}). Thus, the  estimates on these components are obtained  by using the estimates on the three first components and \eqref{eq.ineqbound1_crosspoint}.
\end{proof}

\subsubsection{Generalized eigenfunctions of the lineic spectral zone}
\label{sec.Holder-lineic}
According to what we said in section \ref{sec_GEForient}, in order to obtain the desired ${\Hms}$-estimates \eqref {formHolder0} for  $\bbW_{k,\lambda,0}$ via the pointwise estimates \eqref{formHolderpoint0}, we shall first 
obtain estimates of the $\lambda$-derivatives of the functions:
\begin{equation}\label{eq.holdmapw0}
\lambda \mapsto w_{\pm k_E(\lambda),\lambda,0}(x,y)  \quad \mbox{ and }  \lambda \mapsto \partial_x w_{\pm k_E(\lambda),\lambda,0}(x,y).
\end{equation}
By parity arguments in $k$, we  only need to give the proofs  of these  estimates  for $k=k_{\scE}(\lambda)$.

The forthcoming estimates  will be established for  $\Oe\neq \Om$, $\scZ=\EE$,  $j =0$ and $(k,\lambda) \in \Lambda_{\EE}([a,b])$ where $[a,b] \subset \bbR\setminus\sigma_{\rm exc}$ is such that $\Lambda_{\EE}([a,b]) \neq \varnothing$. In particular, $\pm \Op \notin [a,b]$ which ensures  that  $\Lambda_{\EE}([a,b])$ is a bounded subset of $\Lambda_{\EE}$, whereas the fact that $0\notin [a,b]$ implies that all points of $[a,b]$ have the same sign.
Moreover,  as $\overline{\Lambda_{\EE}([a,b])}$ does not intersect the lines $\lambda=0$ and $\lambda=\pm \Om$, the functions $\lambda \mapsto \mu_+(\lambda)^{-1}$ and $\lambda\mapsto \lambda^{-1}$ are bounded and $C^{\infty}$ smooth with respect to $\lambda$ on the compact set  $\overline{\Lambda_{\EE}([a,b])}$.
Concerning the regularity of the function   $\lambda\mapsto  \kE(\lambda)$ (defined by \eqref{defkE}) that appears in \eqref{eq.holdmapw0},  it is continuous (by the bijection theorem) on $[\Oc,+\infty[$ with  value $\kE(\Oc)=k_c$ at $\lambda=\Oc$.
Furthermore,  as $\lambda_{\scE}$  is $C^{\infty}$ on $[k_c,+\infty]$ and $\lambda_{\scE}'(k)\neq 0$ on this interval, using the inversion theorem, $\lambda \mapsto  \kE(\lambda)$ is  $C^{\infty}$ on $(\Oc,+\infty[$ but also on $[\Oc,+\infty[$ since $ \lambda_{\scE}'(k_c)\neq 0$. As $\kE$ is even, $\kE$ is $C^{\infty}$ for $ |\lambda| \in [\Oc,+\infty[$.
It implies in particular that  $| \kE(\lambda)|\lesssim  1$ and $| \kE'(\lambda) |\lesssim  1$ on  $\overline{\Lambda_{\EE}([a,b])}$.

\subsubsection*{(I) Preliminary $\lambda$-derivatives estimates.}
\noindent {\bf (Ia) Derivative of powers of }$\theta_{k, \lambda}^{\pm}$.
\begin{equation}\label{eq.derivativetheta}
\big|\partial_\lambda \big( \, \theta_{k, \lambda}^{\pm}\big)^\alpha\big| \lesssim    (\theta_{k,\lambda }^{+})^{\alpha-2},  \quad  \forall \; (k,\lambda)\in \Lambda_{\EE}([a,b]),
\end{equation}
(where we recall that $\theta^{\pm}_{k,\lambda}>0$ in $\zEE$, see (\ref{eq.def-thetam},\ref{eq.def-thetap})).
Using the chain rule formula, one can write somewhat abusively (see Remark \ref{rem.partialderiv}):
\begin{equation}\label{eq.derivativetheta-compos}
\partial_\lambda (\theta^{\pm}_{k_{\scE}(\lambda), \lambda})^\alpha = \partial_\lambda (\Theta_{k_{\scE}(\lambda),\lambda }^{\pm})^{\frac{\alpha}{2}} = 
\frac{\alpha}{2} \Big[\partial_{k}  \Theta^{\pm}_{k_{\scE}(\lambda),\lambda } \,   k_{\scE}'(\lambda) +\partial_{\lambda}  \Theta^{\pm}_{k_{\scE}(\lambda),\lambda }  \Big]  (\theta_{k_{\scE}(\lambda),\lambda }^{\pm})^{\alpha-2}.
\end{equation}
The quantity inside brackets is a continuous function of $\lambda$ and is thus bounded, hence  it follows that
$$
\big|\partial_\lambda  (\theta_{k_{\scE}(\lambda), \lambda}^{\pm})^\alpha\big|\lesssim   (\theta_{k_{\scE}(\lambda),\lambda }^{\pm})^{\alpha-2}.
$$
This yields \eqref{eq.derivativetheta} since  the dispersion relation \eqref{eq.disp} (or equivalently the definition of $\zEE$):
$$
\theta_{k_{\scE}(\lambda),\lambda }^{-} = -(\mu_\lambda^- /\mu_\lambda^+) \,  \theta_{k_{\scE}(\lambda),\lambda }^{+}
$$
implies that  $\theta_{k_{\scE}(\lambda),\lambda }^{-}\lesssim \theta_{k_{\scE}(\lambda),\lambda }^{+}$.	\\

\begin{remark}\label{rem.partialderiv}
The reader will note that in sections \ref{sec.Holder-lineic} and \ref{sec-Holdregulgeneralized-eigenfunc-thr} the symbol $\partial_{\lambda}$ is used  somewhat abusively for  the total derivative  $\rmd / \rmd \lambda$ except in equations  \eqref{eq.derivativetheta-compos} and   \eqref{eq.partialderivative2} where  $\partial_{\lambda}  \Theta^{\pm}_{k_{\scE}(\lambda),\lambda }$  is the partial derivative in $\lambda$ of the function $(k,\lambda)\mapsto\Theta^{\pm}_{k,\lambda}$ evaluated at $(\kE(\lambda),\lambda)$.
\end{remark}

\noindent {\bf (Ib) Derivatives of powers of $A_{k, \lambda,0}$}: we show that 
\begin{equation}\label{eq.coeffA}
\big|\partial_\lambda A_{k, \lambda,0}\big| \lesssim   (\theta_{k,\lambda }^{+})^{-\frac{3}{2}},\quad  \forall (k,\lambda)\in \Lambda_{\EE}([a,b]).
\end{equation}
We first rewrite the expression \eqref{def-psi-plasm} of 
$A_{k_{\scE}(\lambda), \lambda,0}$ in the form
$$
A_{k_{\scE}(\lambda), \lambda,0} = (\theta_{\kE(\lambda),\lambda }^{+})^{\frac{1}{2}} \, B_{\scE}(\lambda),
$$
where  $\lambda \mapsto B_{\scE}(\lambda) $ is of class $C^{\infty}$ in $[\Omega_c , + \infty)$. Differentiating with respect to $\lambda$,  one gets:
$$
\partial_\lambda A_{\kE(\lambda),\lambda,0}= 
\partial_\lambda (\theta_{\kE(\lambda),\lambda }^{+})^{\frac{1}{2}} \, B_{\scE}(\lambda) +  (\theta_{\kE(\lambda),\lambda }^{+})^{\frac{1}{2}} \;\partial_\lambda B_{\scE}(\lambda).
$$
As $ B_{\scE}(\lambda)$, $\partial_\lambda B_{\scE}(\lambda)$, and $\theta_{\kE(\lambda),\lambda }^{+}$ are bounded in $[a,b]$, using \eqref{eq.derivativetheta} for $\alpha=1/2$ we deduce \eqref{eq.coeffA}.
\\ [6pt]

\noindent{\bf (II) Pointwise estimates of  $\partial_{\lambda} \big(w_{k,\lambda,0}(x,y)\big)$ and $\partial_{\lambda} \big(\partial_x w_{k,\lambda,0}(x,y)\big)$.} \\ 
\begin{lemma} \label{lem_estimderpsiC}  Let $\Oe\neq \Om$  and $[a,b] \subset \bbR\setminus\sigma_{\rm exc}$ such that $\Lambda_{\EE}([a,b]) \neq \varnothing$. Then one has  for all $ (k,\lambda)\in \Lambda_{\EE}([a,b])$ the following  pointwise estimates:
	\begin{eqnarray}
	\forall (x,y)\in \bbR^2, &\quad 	\big|\partial_{\lambda} \big( w_{k,\lambda,0}(x,y)\big)\big| &\lesssim  (\theta_{k,\lambda}^+)^{-\frac{3}{2}} \, (1+|x| + |y|),  \label{eq.deriv-wo-lambda}\\
	\forall \; (x,y) \in \bbR^* \times \bbR, & \quad \big| \partial_{\lambda} \big( \partial_x w_{k,\lambda,0}(x,y)\big)\big|&\lesssim  (\theta_{k,\lambda}^+)^{-\frac{3}{2}} \, (1+|x| + |y|) .  \label{eq.deriv-wo-x-lambda} 
	\end{eqnarray}
	
\end{lemma}
\begin{proof}

	\noindent {\it Step 1 : proof of \eqref{eq.deriv-wo-lambda}.} 
	From the expression (\ref{eq.def-w}) of $w_{\kE(\lambda),\lambda,0}$, one has 
	\begin{equation}\label{eq.express-derivpartialambda-wlo}
	|\partial_\lambda \big(w_{\kE(\lambda),\lambda,0}(x,y)\big) | \leq \big| \partial_\lambda \big(A_{\kE(\lambda),\lambda,0}\big)   \psi_{\kE(\lambda),\lambda,0}(x)\big|+\big|A_{\kE(\lambda),\lambda,0} \, \partial_\lambda \big(\psi_{\kE(\lambda),\lambda,0}(x)\rme^{\rmi k_{\scE}(\lambda)y}\big)\big|. 
	\end{equation} 
	We bound  now the two terms of the right hand side of  \eqref{eq.express-derivpartialambda-wlo}. First, $|\psi_{\kE(\lambda),\lambda,0}(x)| = \rme^{-\theta_{\kE(\lambda),\lambda,0}(x)\, |x|}\leq 1$ (cf. \eqref{def-psi-plasm}), thus by \eqref{eq.coeffA}, one gets:
	\begin{equation}\label{eq.derivwo-term1}
	|\partial_\lambda\big( A_{\kE(\lambda),\lambda,0}\big)  \, \psi_{\kE(\lambda),\lambda,0}(x)| \lesssim  
	(\theta_{\kE(\lambda),\lambda}^+)^{-\frac{3}{2}}.
	\end{equation}
	Then, for the second term, one computes the expression of $\partial_\lambda (\psi_{\kE(\lambda),\lambda,0}(x)\rme^{\rmi k_{\scE}(\lambda)y})$:
	\begin{equation}\label{eq.expression-deriv-lambda-psi}
	\partial_{\lambda}( \psi_{\kE(\lambda),\lambda,0}(x) \, \rme^{\rmi \kE(\lambda) y})= \Big[\mp  \partial_{\lambda} \theta_{\kE(\lambda),\lambda }^{\pm}\,  x +\rmi  \, \kE'(\lambda)   y\Big]\,\psi_{\kE(\lambda),\lambda,0}(x) \, \rme^{\rmi \kE(\lambda) y} \  \mbox{ for } \pm x\geq 0, \,   y\in \bbR.
	\end{equation}
As  $|\psi_{\kE(\lambda),\lambda,0}(x)|\leq 1$, $|\kE'(\lambda)|\lesssim 1$ and $\theta_{\kE(\lambda),\lambda }^{+}\lesssim 1$, by applying \eqref{eq.derivativetheta} for  $\alpha=1$, it follows:
	\begin{equation*}\label{eq.boundphi}
	|\partial_{\lambda}( \psi_{\kE(\lambda),\lambda,0}(x) \, \rme^{\rmi k y}) |\lesssim (|x|+|y|) \, (\theta_{\kE(\lambda),\lambda }^{+})^{-1},   \ \forall (x,y) \in \bbR^2.
	\end{equation*}
	The coefficient $A_{\kE(\lambda),\lambda,0}$ satifies $|A_{\kE(\lambda),\lambda,0}|\lesssim 1$ (see \eqref{def-A-plasm}), hence it leads to
	\begin{equation}\label{eq.derivwo-term2}
	|A_{\kE(\lambda),\lambda,0} \, \partial_\lambda \big(\psi_{\kE(\lambda),\lambda,0}(x)\big)| \lesssim (|x|+|y|) \, (\theta_{\kE(\lambda),\lambda }^{+})^{-1},   \ \forall (x,y) \in \bbR^2.
	\end{equation}
	As $(\theta_{\kE(\lambda),\lambda }^{+})^{-1}\lesssim (\theta_{\kE(\lambda),\lambda }^{+})^{-3/2}$, combining \eqref{eq.express-derivpartialambda-wlo}, \eqref{eq.derivwo-term1} and \eqref{eq.derivwo-term2} yields  the estimate \eqref{eq.deriv-wo-lambda}.\\

	\noindent {\it Step 2 : proof of \eqref{eq.deriv-wo-x-lambda}.} From the expression (\ref{eq.def-w},\ref{def-A-plasm},\ref{def-psi-plasm}) of $w_{\kE(\lambda),\lambda,0}$,
	one has 
	\begin{equation*}
	\partial_x w_{\kE(\lambda),\lambda,0}(x,y)= \mp \theta_{\kE(\lambda),\lambda}^{\pm} w_{\kE(\lambda),\lambda,0}(x,y)\mbox{ for $\pm x>0 \mbox{ and }  y\in \mathbb{R}$.}
	\end{equation*} 
	We let the reader show, by deriving in $\lambda$ this product,  that \eqref{eq.deriv-wo-x-lambda} is a consequence of the estimates:  \eqref{eq.derivativetheta} for $\alpha=1$, \eqref{eq.deriv-wo-lambda} and
	$$
	|\theta_{\kE(\lambda),\lambda}^{\pm}|\lesssim 1 \  \mbox{ and }  \ | w_{\kE(\lambda),\lambda,0}(x,y)|= A_{\kE(\lambda),\lambda,0}  \ \psi_{\kE(\lambda),\lambda,0}(x)\lesssim 1,  \ \forall (x,y) \in \bbR^2.
	$$  
\end{proof}

\noindent Thanks to the previous estimates, we are now able to give  H\"older type  inequalities for $\bbW_{k,\lambda,0}$ in  the following proposition.

\begin{proposition}\label{prop.holdestimateplasmon}
	Let $\Oe\neq \Om$, $s> 1/2$, $\gamma \in (0,1]\cap (0,s-1/2)$ and $[a,b] \subset \bbR\setminus\sigma_{\rm exc}$ such that $\Lambda_{\EE}([a,b]) \neq \varnothing$. Then, there exists a constant $C_{a,b}^{\gamma}>0$   such that  for all $ (\pm \kE(\lambda), \lambda), (\pm \kE(\lambda'),\lambda')\in\Lambda_{\EE}([a,b]) \mbox{ and } \lambda\leq \lambda'$, one has:
	\begin{equation}\label{eq.ineqboundholdplasmcrosspoint}
	\| \bbW_{\pm \kE(\lambda'),\lambda',0}  -\bbW_{\pm \kE(\lambda),\lambda,0}\|_{\Hms}\leq C_{a,b}^{\gamma} \,  \sup_{\tilde{\lambda}\in [\lambda,\lambda']}  (\theta_{\kE(\tilde{\lambda}),\tilde{\lambda}}^{+})^{\frac{1-\gamma}{2}} \, \sup_{\tilde{\lambda}\in [\lambda,\lambda']}  (\theta_{\kE(\tilde{\lambda}),\tilde{\lambda}}^{+})^{-\frac{3}{2}\gamma} \,  |\lambda'-\lambda|^{\gamma}.
	\end{equation}
Moreover,  if $\pm \Oc\notin [a,b]$, there exists $C_{a,b}^{\gamma}>0$  such that for all $(\pm \kE(\lambda), \lambda), (\pm \kE(\lambda'),\lambda')\in\Lambda_{\EE}([a,b]) \mbox{ and } \lambda\leq \lambda'$:
	\begin{equation}\label{eq.ineqboundholdplasm}
	\| \bbW_{\pm \kE(\lambda'),\lambda',0}  -\bbW_{\pm \kE(\lambda),\lambda,0}\|_{\Hms}\leq C_{a,b}^{\gamma} \,  |\lambda'-\lambda|^{\gamma}.
	\end{equation}
\end{proposition}

\begin{proof}
	{\it Step 1: proof of  \eqref{eq.ineqboundholdplasmcrosspoint} and  \eqref{eq.ineqboundholdplasm} for the first component  $w_{k,\lambda,0}$ of $\bbW_{k,\lambda,0}$}.  We use in this proof the notations $k=\kE(\lambda)$ and $k'=\kE(\lambda')$.
	From the mean value Theorem and the estimate and \eqref{eq.deriv-wo-lambda} (and the parity of $\thetakl^+$ with repsect to $k$), one gets:
	\begin{equation}\label{eq.interpolplasm1}
	| w_{\pm k',\lambda',0}(x,y) -  w_{\pm k,\lambda,0}(x,y)| \lesssim  (1+|x|+|y|) \sup_{\tilde{\lambda}\in [\lambda,\lambda']}  (\theta_{\kE(\tilde{\lambda}),\tilde{\lambda}}^{+})^{-\frac{3}{2}}  \ |\lambda'-\lambda|  ,\   \forall (x,y) \in \bbR^2.
	\end{equation}
	Then using \eqref{eq.interlinf}, one immediately obtains:
	\begin{equation}\label{eq.interpolplasm2}
	| w_{\pm k,\lambda',0}(x,y) -  w_{\pm k,\lambda,0}(x,y)|\leq  (\theta_{k',\lambda'}^+)^{\frac{1}{2}} + (\thetakl^+)^{\frac{1}{2}}\lesssim  \sup_{\tilde{\lambda}\in [\lambda,\lambda']}  (\theta_{\kE(\tilde{\lambda}),\tilde{\lambda}}^{+})^{\frac{1}{2}} 
	\end{equation}
	Thus, interpolating inequalities  \eqref{eq.interpolplasm1} and \eqref{eq.interpolplasm2} leads to
	$$
	| w_{\pm k',\lambda',0}(x,y) -  w_{\pm k,\lambda,0}(x,y)| \lesssim   \sup_{\tilde{\lambda}\in [\lambda,\lambda']}  (\theta_{\kE(\tilde{\lambda}),\tilde{\lambda}}^{+})^{\frac{1-\gamma}{2}}   \sup_{\tilde{\lambda}\in [\lambda,\lambda']}  (\theta_{\kE(\tilde{\lambda}),\tilde{\lambda}}^{+})^{-\frac{3\gamma}{2}}  \ |\lambda'-\lambda|^{\gamma}    (1+|x|+|y|)^{\gamma} ,
	$$
	for all $(x,y) \in \bbR^2, (k, \lambda), (k',\lambda')\in\Lambda_{\EE}([a,b]) \mbox{ and } \lambda\leq \lambda'$ and $\gamma\in (0,1]$. As $ (x,y)\to (1+|x|+|y|)^{\gamma}\in L^2_{-s}(\bbR^2)$ for $0<\gamma<s-1/2$, one obtains:
	\begin{equation}\label{eq.holdestfirstcomponent2}
	\| w_{\pm k',\lambda',0}  -w_{\pm k,\lambda,0}\|_{L^2_{s}(\bbR^2)}\lesssim     \sup_{\tilde{\lambda}\in [\lambda,\lambda']}  (\theta_{\kE(\tilde{\lambda}),\tilde{\lambda}}^{+})^{\frac{1-\gamma}{2}}   \sup_{\tilde{\lambda}\in [\lambda,\lambda']}  (\theta_{\kE(\tilde{\lambda}),\tilde{\lambda}}^{+})^{-\frac{3\gamma}{2}} \,  |\lambda'-\lambda|^{\gamma},
	\end{equation}
	for $(k, \lambda), (k',\lambda')\in\Lambda_{\EE}([a,b]) \mbox{ and } \lambda\leq \lambda'$ and $\gamma\in (0,1]\cap (0,s-1/2)$.
	Now, if one makes the additional assumption that $\pm \Oc\notin [a,b]$, it implies that the crosspoints do not belongs to the set $\Lambda_{\EE}([a,b])$. Thus, $\Lambda_{\EE}([a,b])$ is not only a bounded set but also a compact set of $\Lambda_{\EE}$ where  $\thetakl^+$ is a continuous function that does not vanish and therefore the estimate \eqref{eq.holdestfirstcomponent2} simplifies to   
	\begin{equation}\label{eq.holdestfirstcomponent1}
	\| w_{\pm k',\lambda',0}  -w_{\pm k,\lambda,0}\|_{L^2_{s}(\bbR^2)}\lesssim |\lambda'-\lambda|^{\gamma}.
	\end{equation}
	\eqref{eq.holdestfirstcomponent2} and \eqref{eq.holdestfirstcomponent1} are nothing but the estimates \eqref{eq.ineqboundholdplasmcrosspoint} and  \eqref{eq.ineqboundholdplasm} for the first component $w_{k,\lambda,0}$ of the $\bbW_{k,\lambda,0}$. \\
	[12pt]
	{\it Step 2 : Generalization to other components.}
	The estimates of the other components can be performed in the same way.  
	More precisely, the second component: $k \, w_{k,\lambda,0}/(\mu_{\lambda}\lambda)$ is the product
	of the smooth and bounded function $\lambda\mapsto k / (\mu_{\lambda}\lambda)$ for $(k,\lambda)\in \Lambda_{\EE}([a,b])$ by $w_{k,\lambda,0}$. Thus, as   $\|w_{k,\lambda,0}\|_{L^2_{s}(\bbR^2)}\lesssim (\thetakl^{+})^{\frac{1}{2}}$ (by \eqref{eq.interlinf}), an estimate of the form \eqref{eq.holdestfirstcomponent2} (resp. \eqref{eq.holdestfirstcomponent1} if $\pm \Oc\notin [a,b]$) is derived for $k \, w_{k,\lambda,0}/\mu_{\lambda}$   by using \eqref{eq.holdestfirstcomponent2} (resp. \eqref{eq.holdestfirstcomponent1} if $\pm \Oc\notin [a,b]$).
	
	The third component is the product of  the  partial derivative of $ \partial_x w_{k,\lambda,0}$  by  $ \rmi/ (\mu_{\lambda}\, \lambda)$.
	Thus,  in a first time,  one performs  exactly the same reasoning as in step 1 by using the  second estimate of \eqref{eq.interlinf} and  \eqref{eq.deriv-wo-x-lambda} (instead of  \eqref{eq.deriv-wo-lambda}) to obtain  the estimates  \eqref{eq.holdestfirstcomponent2} and \eqref{eq.holdestfirstcomponent1} if $\pm \Oc\notin [a,b]$  but  for $\partial _x w_{\pm k,\lambda,0}$ instead of $w_{\pm k,\lambda,0}$. In a second time, using again the second estimate of \eqref{eq.interlinf}, one observes that the  multiplication by the smooth and bounded coefficient in $\lambda$: $\rmi/ (\mu_{\lambda}\, \lambda)$ only change the constant $C_{a,b}^{\gamma}$ in  these estimates. \\ [6pt]
	\noindent Finally,  the last three components are $0$ for $x < 0$ and proportional  with a coefficient that  is smooth and  bounded in $\lambda$  to the first three components for $x > 0$ (see \eqref{eq.def-Vkbis}). Thus, the  estimates on these components are obtained by using the estimates on the three first components and \eqref{eq.ineqboundplasm}.
\end{proof}

\subsection{Proof of Theorem \ref{th.Holder-dens-spec}}
\label{sec.Holderspectraldensity}

\subsubsection{The various components of the spectral density}
\label{sss.various-comp-M}
We have now all the ingredients to prove the H\"{o}lder regularity of the spectral density, that is to say the local H\"older estimate \eqref{eq.holderestim-specdensity} for $[a,b] \ni \lambda \mapsto \bbM_{\lambda} \in B(\Hps,\Hms)$ where $[a,b]$ is a bounded interval of $\bbR\setminus \sigma_{\rm exc} $. We are going to prove this estimate for the various components of the spectral density which appear in its expression \eqref{eq.density-non-crit} (if $\Oe\neq\Om$), that we rewrite in the form
\begin{align}
	\bbM_\lambda & = \sum_{\scZ \in \calZ} \bbM_\lambda^\scZ \quad\text{where}  \label{eq.spec-meas-decomp}\\
	\bbM_\lambda^\scZ \bU & := \sum_{j \in J_{\scZ}} \int_{\zZ(\lambda)} \langle \bU,\bbW_{k,\lambda,j} \rangle_{s}  \; \bbW_{k,\lambda,j} \,\rmd k 
	\quad\text{if }\scZ \in \calZ\setminus\{\EE\}\quad\text{and} \label{eq.spec-meas-surf}\\
	\bbM_\lambda^\EE \bU & := \sum_{k \in \zEE(\lambda)}  \JacE(\lambda)\ \langle  \bU,  \bbW_{k,\lambda,0}  \rangle_{s} \,   \bbW_{k,\lambda,0}, \label{eq.spec-meas-lin}
\end{align}
where the last component $\bbM_\lambda^\EE$ has to be removed if $\Oe=\Om$ (see \eqref{eq.density-crit}). Thus the aim of this section is to prove that for all $\scZ \in \calZ$ and all $\gamma \in \Gamma_{[a,b]}$ (see \eqref{eq.def-Gamma-K}), we have
\begin{equation}\label{eq.holderestim-specdensity-Z}
	\forall \lambda, \lambda' \in [a,b], \quad
	\big\| \bbM_{\lambda'}^\scZ\bU - \bbM_\lambda^\scZ\bU \big\|_{\Hms} \lesssim \ |\lambda'-\lambda|^{\gamma}\ \|\bU\|_{\Hps}.
\end{equation}
The fact that the set $\Gamma_{[a,b]}$ of possible H\"older exponents depends on $[a,b]$ will become clear in the following. We must distinguish three cases, denoted by (A), (B) and (C), depending on the position of $\{ \pm\Oe,\pm\Oc \}$ with respect to the interval $[a,b]$:
\begin{equation}
\begin{array}{ll}
{\rm (A)}: & [a,b] \cap \{ \pm\Oe,\pm\Oc \} = \varnothing, \\[5pt] 
{\rm (B)}: & [a,b] \cap \{ \pm\Oe \} \neq \varnothing \text{ and } [a,b] \cap \{ \pm\Oc \} = \varnothing, \\[5pt] 
{\rm (C)}: & [a,b] \cap \{ \pm\Oc \} \neq \varnothing.
\end{array}
\label{eq.def-ABC}
\end{equation}
The reader will easily check that these cases 
	are mutually exclusive and cover all possibilities.
According to \eqref{eq.def-Gamma-K}, we have 
\begin{equation*}
	\Gamma_{[a,b]} := \left\{
	\begin{array}{ll}
	\big(0,\min(s-1/2,1)\big) & \text{in case (A),} \\[5pt] 
	\big(0,\min(s-1/2,1/2)\big) & \text{in cases (B) and (C).} 
	\end{array}\right.
\end{equation*}
In the following, \S\ref{sss.Holder-M-surface} and \ref{sss.proof-l-int-theta} are devoted to the proof of \eqref{eq.holderestim-specdensity-Z} for $\scZ \in \calZ\setminus\{\EE\}$, whereas \S\ref{sss.Holder-M-lineic} deals with the case $\scZ = \EE$. Recall that the points of $\sigma_{\rm exc}$ (see \eqref{eq.def-sigma-exc}) are always excluded from the considered interval $[a,b] \subset \bbR$. Moreover, as $\Gamma_{[a,b]}$ depends on whether $[a,b]$ contains $\pm\Oe$ or $\pm\Oc$, we will assume for simplicity that when $[a,b]$ contains one of these  points (that is, in cases (B) or (C)), \emph{it is located at the boundary of the interval}, i.e., equal to $a$ or $b$. There is no loss of generality since, in order to prove \eqref{eq.holderestim-specdensity-Z}, it suffices to prove the same property separately for two intervals $[a,c]$ and $[c,b]$ with $c \in (a,b)$. Then, \eqref{eq.holderestim-specdensity-Z}  follows from the triangle inequality and the fact that
\begin{equation*}
	\forall \lambda \in [a,c], \ \forall \lambda' \in [c,b],\quad
	|\lambda'-c|^\gamma + |c-\lambda|^\gamma \leq 2\,|\lambda'-\lambda|^\gamma.
\end{equation*}

\subsubsection{Components related to the surface spectral zones}
\label{sss.Holder-M-surface}
We focus here on the proof of \eqref{eq.holderestim-specdensity-Z} for $\scZ \in \calZ\setminus\{\EE\}$. From \eqref{eq.spec-meas-surf}, we see that we have to estimate the difference between two integrals defined of different domains, one on $\zZ(\lambda)$, the other on $\zZ(\lambda')$. Each of both contains a common part $\zZ(\lambda) \cap \zZ(\lambda')$ and an own part, respectively, $\zZ(\lambda) \setminus \zZ(\lambda')$ and $\zZ(\lambda') \setminus \zZ(\lambda)$ (some of these sets may be empty). Hence we can write
\begin{equation}
	\bbM_{\lambda'}^\scZ\bU - \bbM_\lambda^\scZ\bU =
	\bbD_{\lambda\cap\lambda'}^\scZ\bU + \bbD_{\lambda'\setminus\lambda}^\scZ\bU - \bbD_{\lambda\setminus\lambda'}^\scZ\bU,
	\label{eq.dif-meas-spec}
\end{equation}
where we have denoted
\begin{align}
	\bbD_{\lambda\cap\lambda'}^\scZ\bU & := \sum_{j \in J_{\scZ}} \int_{\zZ(\lambda) \cap \zZ(\lambda')} \Big\{\langle \bU,\bbW_{k,\lambda',j} \rangle_{s}  \; \bbW_{k,\lambda',j} - \langle \bU,\bbW_{k,\lambda,j} \rangle_{s}  \; \bbW_{k,\lambda,j}\Big\}\,\rmd k, \label{eq.common-part}\\
	\bbD_{\lambda\setminus\lambda'}^\scZ\bU & := \sum_{j \in J_{\scZ}} \int_{\zZ(\lambda) \setminus \zZ(\lambda')} \langle \bU,\bbW_{k,\lambda,j} \rangle_{s}  \; \bbW_{k,\lambda,j} \,\rmd k. \label{eq.own-part}
\end{align}
Figure \ref{fig.zoomholder} illustrates the various sets involved in these integrals in the particular case where
\begin{equation}
	\Om < \Oe \quad\text{and}\quad [a,b]\subset (\Op, \Oc).
	\label{eq.particular-case}
\end{equation}
In this situation, we have $\zDD([a,b])=\zDE([a,b])=\varnothing$, so that \eqref{eq.holderestim-specdensity-Z} has to be proved for $\scZ = \DI$ and $\EI$. For $\lambda$ and $\lambda'$ in $[a,b]$ with $\lambda < \lambda'$, the gray areas in Figure \ref{fig.zoomholder} represent the sets $\zZ([\lambda, \lambda'])$ for $\scZ = \DI$ and $\EI$ (more precisely their intersections with the half-plane $k > 0$), whose lower and upper boundaries are respectively $\zZ(\lambda) \times \{\lambda\}$ and $\zZ(\lambda') \times \{\lambda'\}$. The common part of the domains of integration corresponds to the rectangles in light gray defined by $\big( \zZ(\lambda) \cap \zZ(\lambda') \big) \times [\lambda,\lambda']$, whereas the own parts are associated to the triangles in dark gray. For $\scZ = \DI$, we see that $\zDI(\lambda) \setminus \zDI(\lambda')=\varnothing$ while  $\zDI(\lambda') \setminus \zDI(\lambda) = (\kO(\lambda),\kO(\lambda'))$ corresponds to the upper boundary of the dark triangle. For $\scZ = \EI$, the situation is reversed: $\zEI(\lambda') \setminus \zEI(\lambda)=\varnothing$ while  $\zEI(\lambda') \setminus \zEI(\lambda) = (\kO(\lambda),\kO(\lambda')) \cup (\kI(\lambda'),\kI(\lambda))$ corresponds to the lower boundaries of the two triangles.
\begin{figure}[t!]
	\centering
	\includegraphics[width=0.8\textwidth]{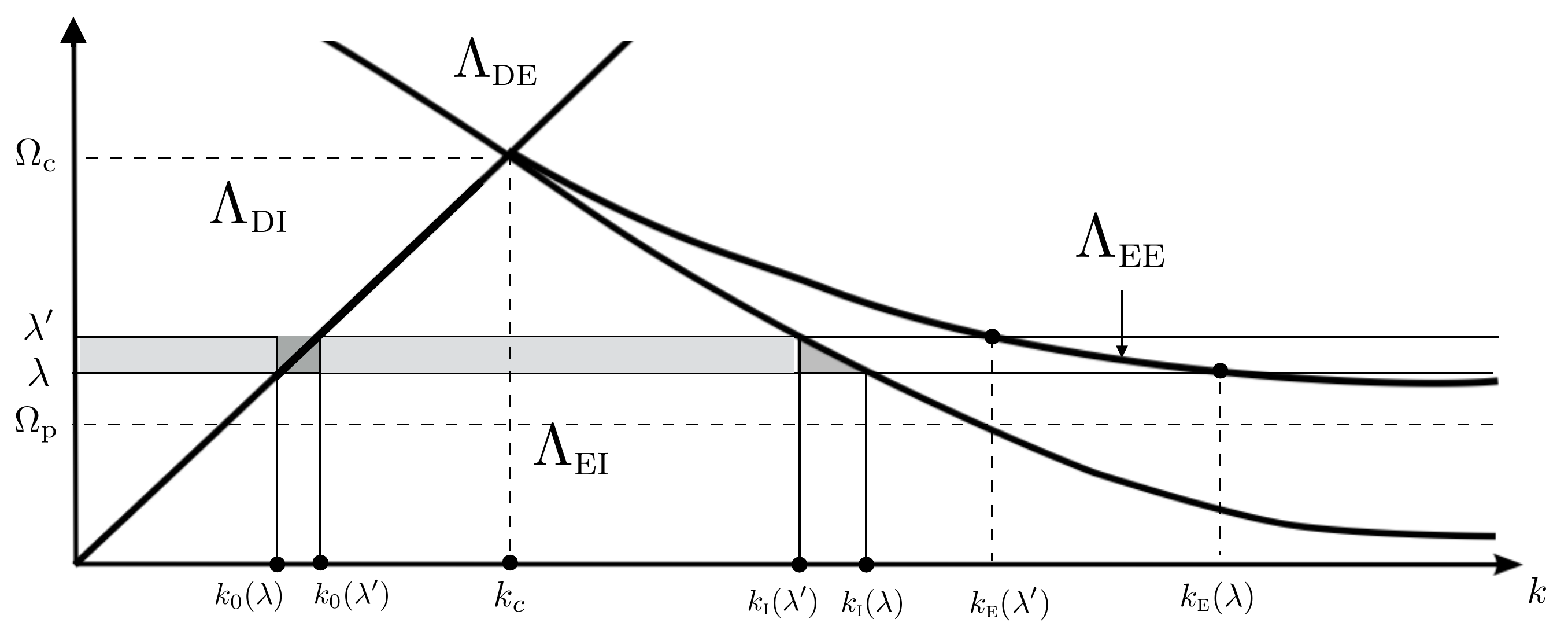}
	\caption{Case where $\Om < \Oe$ and $[a,b]\subset (\Op, \Oc)$.}
	\label{fig.zoomholder}
\end{figure}

\textbf{Step 1.} In the general case of an interval $[a,b] \subset \bbR \setminus \sigma_{\rm exc}$, we first consider the part \eqref{eq.common-part} associated with the common domain of integration. Let us prove that it satisfies the H\"older estimate \eqref{eq.holderestim-specdensity-Z}, i.e.,
\begin{equation}\label{eq.holderestim-common-part}
	\forall \lambda, \lambda' \in [a,b], \quad
	\big\| \bbD_{\lambda\cap\lambda'}^\scZ \bU \big\|_{\Hms} \lesssim \ |\lambda'-\lambda|^{\gamma}\ \|\bU\|_{\Hps},
\end{equation}
for $\gamma \in \Gamma_{[a,b]}$. Using the equality
\begin{multline*}
	\langle \bU,\bbW_{k,\lambda',j} \rangle_{s}  \; \bbW_{k,\lambda',j} - \langle \bU,\bbW_{k,\lambda,j} \rangle_{s}  \; \bbW_{k,\lambda,j}
	\\ = \langle \bU,\{\bbW_{k,\lambda',j} - \bbW_{k,\lambda,j}\}\rangle_{s}  \; \bbW_{k,\lambda',j} + \langle \bU,\bbW_{k,\lambda,j} \rangle_{s}  \; \{\bbW_{k,\lambda',j}-\bbW_{k,\lambda,j}\},
\end{multline*}
we infer that
\begin{align}
	\big\| \bbD_{\lambda\cap\lambda'}^\scZ\bU \big\|_{\Hms}
	& \leq \Big(\sum_{j \in J_{\scZ}} \int_{\zZ(\lambda) \cap \zZ(\lambda')} d_{\lambda,\lambda',j}(k)\,\rmd k \Big) \ \|\bU\|_{\Hps}
	\quad\text{where} \label{eq.estim-common-part} \\
	d_{\lambda,\lambda',j}(k) & := \Big\{ \big\| \bbW_{k,\lambda',j} \big\|_{\Hms} + \big\| \bbW_{k,\lambda,j} \big\|_{\Hms} \Big\} \ \big\| \bbW_{k,\lambda',j} - \bbW_{k,\lambda,j} \big\|_{\Hms}. \label{eq.def-d}
\end{align}
We see here that the H\"older regularity of $\bbD_{\lambda\cap\lambda'}^\scZ$ requires both the pointwise estimates and the H\"older regularity of the generalized eigenfunctions.  Indeed, $d_{\lambda,\lambda',j}(k)$ can be estimated thanks to Propositions \ref{prop.decay-estim-surfacic-zones} and \ref{prop.holdreggeneralizedeigenfunctions}.
We rewrite inequalities \eqref{eq.ineqbound1_crosspoint} and \eqref{eq.ineqbound1}, as well as the H\"older estimates \eqref{eq.ineqboundhold1crosspoint} and \eqref{eq.ineqboundhold1nocrosspoint} in the following condensed expressions, valid for $\gamma \in (0,1]\cap (0,s-1/2)$, $j \in J_\scZ$ and for all $(k,\lambda)$ and $(k,\lambda')$ in $\Lambda_{\scZ}([a,b])$ with $\lambda \leq \lambda'$:
\begin{align}
\big\| \bbW_{k,\lambda,j}\big\|_{\Hms} & \lesssim  \big|\thetakl^{-j}\big|^\alpha  \quad \text{and} \label{eq.estim-fpg} \\[5pt]
\big\| \bbW_{k,\lambda',j} - \bbW_{k,\lambda,j} \big\|_{\Hms}
& \lesssim \sup_{\tilde{\lambda}\in [\lambda,\lambda']}  \big(\thetaklt^{\min}\big)^{\alpha'} \,  |\lambda'-\lambda|^{\gamma}, \nonumber
\end{align}
where we have denoted
\begin{equation}\label{eq.def-alpha}
(\alpha,\alpha') := \left\{
	\begin{array}{ll}
		(-1/2+\gamma,-1/2-\gamma) & \text{if }[a,b] \cap \{\pm\Oc\}=\varnothing, \\[4pt]
		(-1/2,-1/2-2\gamma) & \text{if }[a,b] \cap \{\pm\Oc\} \neq \varnothing.
	\end{array}
	\right.
\end{equation}
Combining these estimates, we obtain 
\begin{equation*}
	d_{\lambda,\lambda',j}(k) \lesssim 
	\Big( \sup_{\tilde{\lambda}\in [\lambda,\lambda']}  \big|\thetaklt^{-j}\big|^\alpha \Big) \ 
	\Big(\sup_{\tilde{\lambda}\in [\lambda,\lambda']}  \big(\thetaklt^{\min}\big)^{\alpha'} \Big)\  |\lambda'-\lambda|^{\gamma}.
\end{equation*}
Hence \eqref{eq.holderestim-common-part} will be proved once we have verified that the integral on $\zZ(\lambda) \cap \zZ(\lambda')$ of the above product of suprema is bounded  by a constant depending only on $a, b$ and $\gamma$. This is the object of Lemma \ref{l.int-theta1} presented in \S\ref{sss.proof-l-int-theta}.

\textbf{Step 2.} Let us prove now the H\"older regularity of the part $\bbD_{\lambda\setminus\lambda'}^\scZ$ defined in \eqref{eq.own-part}, that is,
\begin{equation}\label{eq.holderestim-own-part}
	\forall \lambda, \lambda' \in [a,b], \quad
	\big\| \bbD_{\lambda\setminus\lambda'}^\scZ \bU \big\|_{\Hms} \lesssim \ |\lambda'-\lambda|^{\gamma}\ \|\bU\|_{\Hps}.
\end{equation}
We no longer assume here that $\lambda < \lambda',$ which allows us to simultaneously treat both quantities $\bbD_{\lambda'\setminus\lambda}$ and $\bbD_{\lambda\setminus\lambda'}$ involved in \eqref{eq.dif-meas-spec}. We are going to see that property \eqref{eq.holderestim-own-part} follows now from the smallness of the domain of integration (as $|\lambda'-\lambda| \to 0$) and the bounds of the generalized eigenfunctions given by Proposition \ref{prop.decay-estim-surfacic-zones}.

From the definition \eqref{eq.own-part}, we have 
\begin{equation*}
	\big\| \bbD_{\lambda\setminus\lambda'}^\scZ\bU \big\|_{\Hms}
	\leq \Big(\sum_{j \in J_{\scZ}} \int_{\zZ(\lambda) \setminus \zZ(\lambda')} \big\| \bbW_{k,\lambda,j} \big\|_{\Hms}^2\,\rmd k \Big) \ \|\bU\|_{\Hps}.
\end{equation*}
Using \eqref{eq.estim-fpg} then yields
\begin{equation*}
	\big\| \bbD_{\lambda\setminus\lambda'}^\scZ\bU \big\|_{\Hms}
	\lesssim \Big(\sum_{j \in J_{\scZ}} \int_{\zZ(\lambda) \setminus \zZ(\lambda')} |\thetakl^{-j}|^{2\alpha}\,\rmd k\Big) \ \|\bU\|_{\Hps}.
\end{equation*}
The object of Lemma \ref{l.int-theta2}, that is also presented in \S\ref{sss.proof-l-int-theta}, is to prove that
\begin{equation*}
	\sum_{j \in J_{\scZ}} \int_{\zZ(\lambda) \setminus \zZ(\lambda')} |\thetakl^{-j}|^{2\alpha}\,\rmd k 
	\lesssim \ |\lambda'-\lambda|^{\gamma},
\end{equation*}
for any $\gamma \in \Gamma_{[a,b]}$, which completes the proof of \eqref{eq.holderestim-own-part}. Combining \eqref{eq.dif-meas-spec},  \eqref{eq.holderestim-common-part} and  \eqref{eq.holderestim-own-part} shows \eqref{eq.holderestim-specdensity-Z}  for $\scZ \in \calZ\setminus\{\EE\}$.

\subsubsection{Two technical lemmas}
\label{sss.proof-l-int-theta}
We gather in this subsection the results about integrals of functions $\thetakl^{+}$ and $\thetakl^{-}$ that are needed in the above proof of the H\"older regularity of the components of the spectral density which are related to the surface spectral zones. 
The main ingredients are the estimates \eqref{eq.singtheta-spectralcut-1} and \eqref{eq.singtheta-spectralcut-2}.

\begin{lemma} \label{l.int-theta1}
Let $\scZ \in \calZ\setminus\{\EE\}$, $[a,b] \subset \bbR\setminus \sigma_{\rm exc}$ and $\gamma \in \Gamma_{[a,b]}$. Then for $j \in J_\scZ$ and for all $\lambda, \lambda' \in [a,b]$ such that $\lambda < \lambda'$, we have, with $\alpha$ and $\alpha'$ defined in \eqref{eq.def-alpha}.
	\begin{equation}\label{eq.estim-integ-d}
	\int_{\zZ(\lambda) \cap \zZ(\lambda')} \ \Big( \sup_{\tilde{\lambda}\in [\lambda,\lambda']}  \big|\thetaklt^{-j}\big|^\alpha \Big) \ 
	\Big(\sup_{\tilde{\lambda}\in [\lambda,\lambda']}  \big(\thetaklt^{\min}\big)^{\alpha'} \Big) \ \rmd k
	\lesssim 1.
	\end{equation}
\end{lemma}

 \begin{proof} 
We distinguish here the various cases (A), (B) and (C) defined in \eqref{eq.def-ABC}, which will be themselves divided in several subcases. In most of the subcases, the estimate \eqref{eq.estim-integ-d} will be deduced from a stronger
		result, namely (the fact that \eqref{eq.two-estim-d} implies \eqref{eq.estim-integ-d}  is explained just below)
		\begin{equation}
		\int_{\zZ(\lambda) \cap \zZ(\lambda')} \ \sup_{\tilde{\lambda}\in [\lambda,\lambda']}  \big|\thetaklt^{+}\big|^\beta \ \rmd k \lesssim 1
		\quad\text{and}\quad
		\int_{\zZ(\lambda) \cap \zZ(\lambda')} \ \sup_{\tilde{\lambda}\in [\lambda,\lambda']}  \big|\thetaklt^{-}\big|^\beta \ \rmd k \lesssim 1,
		\label{eq.two-estim-d}
		\end{equation}
		where
		\begin{equation}
		\beta := \min(\alpha,0) + \alpha' = \min(-1,-1/2-\gamma) \leq -1.
		\label{eq.def-beta}
		\end{equation}
		However, in some particular subcases, \eqref{eq.two-estim-d}  will not be true any longer and \eqref{eq.estim-integ-d} 
		will have to be proven directly in a different manner that will be explained separately. \\ [12pt]
		To see why \eqref{eq.two-estim-d} implies \eqref{eq.estim-integ-d}, we remark that  $|\thetakl^{\pm}|^{-1} \leq (\thetakl^{\min})^{-1}$ implies $|\thetakl^{\pm}|^{\alpha} \leq (\thetakl^{\min})^{\alpha}$ if $\alpha < 0$. \\ [12pt] 
		On the other hand, $|\thetakl^{\pm}|$ being bounded, $|\thetakl^{\pm}|^{\alpha} \lesssim 1$ if $\alpha \geq0 $. Gathering these two observations gives
			\begin{equation*}
		 \sup_{\tilde{\lambda}\in [\lambda,\lambda']}  \big|\thetaklt^{\pm}\big|^\alpha  
		\lesssim
		\sup_{\tilde{\lambda}\in [\lambda,\lambda']}  \big(\thetaklt^{\min}\big)^{\min(\alpha,0)}.
		\end{equation*}
		Consequently, by definition \eqref{eq.def-beta} of $\beta$
			\begin{equation*}
	\sup_{\tilde{\lambda}\in [\lambda,\lambda']}  \big|\thetaklt^{-j}\big|^\alpha  \cdot 
	\sup_{\tilde{\lambda}\in [\lambda,\lambda']}  \big(\thetaklt^{\min}\big)^{\alpha'}
		\lesssim \sup_{\tilde{\lambda}\in [\lambda,\lambda']}  \big(\thetaklt^{\min}\big)^{\min(\alpha,0)} \cdot \sup_{\tilde{\lambda}\in [\lambda,\lambda']}  \big(\thetaklt^{\min}\big)^{\alpha'} = \sup_{\tilde{\lambda}\in [\lambda,\lambda']}  \big(\thetaklt^{\min}\big)^{\beta}.
		\end{equation*}
		The last equality being true because both $\min(\alpha,0)$ and $\alpha'$ are negative. \\ [12pt] 
	It is thus clear that \eqref{eq.estim-integ-d} follows from \eqref{eq.two-estim-d} since $\displaystyle \sup_{\tilde{\lambda}\in [\lambda,\lambda']}  \big(\thetaklt^{\min}\big)^\beta \leq \sup_{\tilde{\lambda}\in [\lambda,\lambda']} |\thetaklt^{+}|^\beta + \sup_{\tilde{\lambda}\in [\lambda,\lambda']} |\thetaklt^{-}|^\beta.$ \\ [12pt]
	We shall also use the fact that, since $\beta < 0$, raising  \eqref{eq.singtheta-spectralcut-1}--\eqref{eq.singtheta-spectralcut-2} to the power $- \beta > 0$ yields
	\begin{equation} \label{ineqthetabeta}
	\big|\thetaklt^{-}\big|^{\beta}  \leq k_0(\tilde{\lambda})^{\beta/2}\ \big| |k|-k_0(\tilde{\lambda}) \big|^{\beta/2},\quad
	\big|\thetaklt^{+}\big|^{\beta}  \leq \big|k^+(\tilde{\lambda})\big|^{\beta/2}\ \big| |k|-k^+(\tilde{\lambda}) \big|^{\beta/2}.
	\end{equation}
	\textbf{Case (A)}. This corresponds $[a,b] \cap \{\pm\Oe,\pm\Oc\} = \varnothing$ and $\Gamma_{[a,b]} = \big(0,\min(s-1/2,1))$ (cf. \eqref{eq.def-Gamma-K}).
 We shall additionnally  consider the particular case \eqref{eq.particular-case} illustrated by Figure \ref{fig.zoomholder}. The reader will rely on the authors about the fact that this case actually involves all the technical difficulties that can be met in all other situations of case (A). \\ [12pt] In this particular case, we are going to prove \eqref{eq.two-estim-d} for $\scZ = \EI$ and $\DI$, the only zones that intersect $[a,b] \times \mathbb{R}$ (see Figure \ref{fig.zoomholder}). In the following, $\lambda$ and $\lambda'$ denote two points of $[a,b]$ such that $\lambda < \lambda'$.
	\\ [12pt]
	\textbf{(i)} The case $\scZ = \DI$.  Figure \ref{fig.zoomholder} shows that $\zDI(\lambda) \cap \zDI(\lambda') = (-\kO(\lambda),+\kO(\lambda))$.  As $k^+(\tilde{\lambda}) = \kI(\tilde{\lambda})$ (by definition \eqref{eq.def-kp} of $k^+(\lambda) \mbox{ for }  |\lambda|\leq \min(\Oe,\Om))$, and as both functions $\kO(\tilde{\lambda})^{\beta/2}$ and $ \kI(\tilde{\lambda})^{\beta/2}$ are  bounded (since $\{0,\pm\Om,\pm\Oe\} \cap [a,b] = \varnothing$), we deduce from \eqref{ineqthetabeta} that 
	\begin{equation} \label{eq.ingred-thetap-bis}
	(a) \quad \big|\thetaklt^{-}\big|^\beta 
	\lesssim \big| |k|-k_0(\tilde{\lambda}) \big|^{\beta/2}, \quad (b) \quad 
	\big|\thetaklt^{+}\big|^\beta 
	\lesssim \big| |k|-\kI(\tilde{\lambda}) \big|^{\beta/2}.
	\end{equation}	
	We next prove successively the two inequalities of \eqref{eq.two-estim-d}. \\ [12pt]
	\textbf{(i.1)} As the closure of $\zDI([a,b])$ does not intersect the spectral cuts $|k|=\kI(\tilde{\lambda})$, the right-hand side of \eqref{eq.ingred-thetap-bis}(b) also remains bounded whatever the sign of $\beta$, thus
	\begin{equation*} 
	\sup_{\tilde{\lambda}\in [\lambda,\lambda']}  \big|\thetaklt^{+}\big|^\beta \lesssim 1.
	\end{equation*}
	As $\lambda \mapsto \kO(\lambda)$ is bounded on $[a,b]$, this yields the first inequality of \eqref{eq.two-estim-d}. \\ [12pt]
	\textbf{(i.2)} Oppositely,  the right-hand side of \eqref{eq.ingred-thetap-bis}(a)  tends to $+ \infty$ when $|k|$  approaches $\kO(\tilde{\lambda})$. \\[12pt]
	However as  $\tilde{\lambda} \mapsto \kO(\tilde{\lambda})$ is increasing on $[a,b]$ and  $\beta < 0,$  $\tilde{\lambda} \mapsto |\,|k|-\kO(\tilde{\lambda})|^{\beta/2}$ is decreasing and  therefore
	\begin{equation*} 
	\sup_{\tilde{\lambda}\in [\lambda,\lambda']}  \big|\thetaklt^{-}\big|^\beta 
	\lesssim \big| |k|-\kO(\lambda)\big|^{\beta/2},
	\end{equation*}
	from which we de deduce that
	\begin{equation} 
	\int_{-\kO(\lambda)}^{+\kO(\lambda)} \ \sup_{\tilde{\lambda}\in [\lambda,\lambda']}  \big|\thetaklt^{-}\big|^\beta \ \rmd k 
	\lesssim \int_{-\kO(\lambda)}^{+\kO(\lambda)}\big| |k|-\kO(\lambda)\big|^{\beta/2}\ \rmd k
	= \frac{2\,\kO(\lambda)^{\beta/2+1}}{\beta/2+1},
	\label{eq.case-a-i}
	\end{equation}
	which is bounded since $\beta/2+1 = \min(1/2,3/4-\gamma/2) > 0$ for any $\gamma \in (0,1]\cap (0,s-1/2)$ i. e. $\gamma \in \Gamma_{[a,b]}$. Hence the second inequality of \eqref{eq.two-estim-d} is proved.
	\\ [12pt]
	\textbf{(ii)} The case $\scZ = \EI$.  In this case $\zEI(\lambda) \cap \zEI(\lambda')= (\kO(\lambda'),\kI(\lambda')) \cup (-\kI(\lambda'),-\kO(\lambda'))$. By parity in  $k$, to prove \eqref{eq.two-estim-d}, it suffices to consider the integrals on $(\kO(\lambda'),\kI(\lambda'))$. We use again inequalities \eqref{eq.ingred-thetap-bis}, but now  both right-hand sides tend to $+ \infty$  when $|k|$ approaches  $\kO(\tilde{\lambda})$ or $\kI(\tilde{\lambda})$. Let us show only the first inequality of \eqref{eq.two-estim-d}, the second one can be treated in the same way.  As this time $\tilde{\lambda} \mapsto \kI(\tilde{\lambda})$ is decreasing on $[a,b]$, $\tilde{\lambda} \mapsto |\,|k|-\kI(\tilde{\lambda})|^{\beta/2}$ is increasing. As a consequence,
	\begin{equation*}
	\sup_{\tilde{\lambda}\in [\lambda,\lambda']}  \big|\thetaklt^{+}\big|^\beta 
	\lesssim \big| |k|-\kI(\lambda')\big|^{\beta/2},
	\end{equation*}
	from which we infer that, since $\beta/2 + 1 > 0$ (cf \eqref{eq.def-beta})
	\begin{equation*} 
	\int_{\kO(\lambda')}^{\kI(\lambda')} \ \sup_{\tilde{\lambda}\in [\lambda,\lambda']}  \big|\thetaklt^{+}\big|^\beta \ \rmd k 
	\lesssim \int_{\kO(\lambda')}^{\kI(\lambda')}\big| |k|-\kI(\lambda')\big|^{\beta/2}\ \rmd k
	= \frac{\big(\kI(\lambda')-\kO(\lambda')\big)^{\beta/2+1}}{\beta/2+1}
	\lesssim 1,
	\end{equation*}
that is to say the first inequality of \eqref{eq.two-estim-d}. \\ [12pt]
	\textbf{Case (B)}. In this case, $[a,b]$ contains $+\Oe$ or $-\Oe$,  not $\pm\Oc$ and $\Gamma_{[a,b]} = \big(0,\min(s-1/2,1/2)\big)$.
	To fix ideas, we suppose that $\Oe > \Om$, but it is easy to  reproduce the same argument if $\Oe < \Om$. \\ [12pt] 
	According to the last paragraph of section \ref{sss.various-comp-M} and parity, we can restrict ourselves to $a = \Oe$ or $b=\Oe$. \\ [12pt]
	\textbf{ (B1)} $[a,b] = [\Oe,b]$. This case is illustrated by Figure \ref{fig.Cas-Omega_e} (left) which shows that we have to prove \eqref{eq.estim-integ-d} for $\scZ = \DD$ and $\DE$.  \\ [12pt]
	\begin{figure}[t!]
		\begin{center}
			\includegraphics[width=1\textwidth]{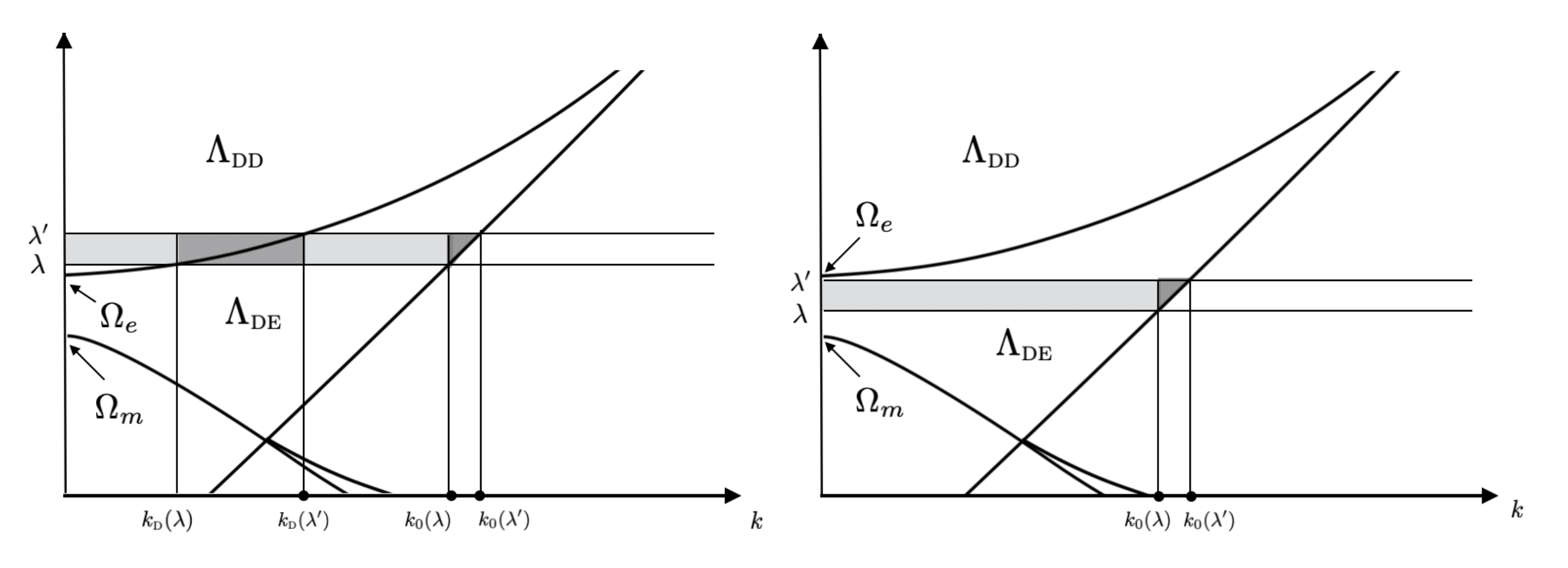}
		\end{center}
		\caption{Case where $\Om < \Oe$ with $[a,b] = [\Oe,b]$ (left) and $[a,b] = [a,\Oe]$ (right).}
		\label{fig.Cas-Omega_e}
	\end{figure}
	\textbf{(i)} The case $\scZ = \DD$. Figure \ref{fig.Cas-Omega_e} (left) then shows that 
	$\zDD(\lambda) \cap \zDD(\lambda') = (-\kD(\lambda),+\kD(\lambda))$. As in case (A), we shall prove \eqref{eq.two-estim-d}.
	\\ [12pt]
	Note first that the second inequality of \eqref{eq.two-estim-d} is easy in this case, since, as  the closure of $\zDD([\Oe,b])$ does not intersect the spectral cuts $|k|=\kO(\tilde{\lambda})$,  the arguments used in item (i.1) of case {\bf (A)}  for proving this inequality apply with obvious changes. 
\\ [12pt]
	The main difference with case {\bf (A)}  concerns the first inequality of \eqref{eq.two-estim-d}. Indeed, the function $k^+(\tilde{\lambda})^{\beta/2} = \kD(\tilde{\lambda})^{\beta/2}$ is no longer bounded in $[\Oe,b]$ (since $\kD(\tilde{\lambda})$ vanishes at $\tilde{\lambda}=\Oe$), so that \eqref{eq.ingred-thetap-bis} is no longer true. We can use instead \eqref{ineqthetabeta}
that gives, since $k^+(\tilde{\lambda}) = \kD (\tilde{\lambda})$ for $\tilde{\lambda} \geq \Omega_e$ (cf. \eqref{eq.def-kp})
\begin{equation*} \label{eq.ingred-thetap-ter}
	\forall \; \tilde{\lambda} \in [\Oe,b], \ \forall |k| \; \neq \kD(\tilde{\lambda}), \quad
	\big|\thetaklt^{+}\big|^\beta 
	\leq \kD(\tilde{\lambda})^{\beta/2}\ \big| |k|-\kD(\tilde{\lambda}) \big|^{\beta/2}.
	\end{equation*}
As $\tilde{\lambda} \mapsto \kD(\tilde{\lambda})$ is increasing on $[\Oe,b]$, the function in the right-hand side of this inequality is a decreasing function of $\tilde{\lambda} \in [\Oe,b]$ for any $k \in (-\kD(\lambda),+\kD(\lambda))$, therefore
	\begin{equation*} 
	\sup_{\tilde{\lambda}\in [\lambda,\lambda']}  \big|\thetaklt^{+}\big|^\beta 
	\leq \kD(\lambda)^{\beta/2}\ \big| |k|-\kD(\lambda)\big|^{\beta/2}.
	\end{equation*}
	We deduce that
	\begin{equation*} 
	\int_{-\kD(\lambda)}^{+\kD(\lambda)} \ \sup_{\tilde{\lambda}\in [\lambda,\lambda']}  \big|\thetaklt^{+}\big|^\beta \ \rmd k 
	\leq \kD(\lambda)^{\beta/2}\ \int_{-\kD(\lambda)}^{+\kD(\lambda)}\big| |k|-\kD(\lambda)\big|^{\beta/2}\ \rmd k
	= \frac{2\,\kD(\lambda)^{\beta+1}}{\beta/2+1},
	\end{equation*}
	which proves the first inequality of \eqref{eq.two-estim-d} provided that $\beta+1 = \min(0,1/2-\gamma) \geq 0$. This is why the possible H\"older exponents $\gamma$ are restricted to the interval $(0,\min(1/2,s-1/2))$ in this situation (recall that  $\gamma = 1/2$ is not considered here, see Appendix \ref{app.cas-un-demi}). \\ [12pt]
		\textbf{(ii)} The case $\scZ = \DE$ for which $\zDE(\lambda) \cap \zDE(\lambda') = (\kD(\lambda'),\kO(\lambda)) \cup(-\kO(\lambda), -\kD(\lambda'))$ and $J_{\EE}=\{1\}$ (by \eqref{eq.def-Jz}).  Again, by a parity argument in $k$, to prove it suffices to consider the integral over  $(\kD(\lambda'),\kO(\lambda))$.
		\\ [12pt]
		The main difference with previous cases is that we can no longer prove \eqref{eq.estim-integ-d} through \eqref{eq.two-estim-d} since the first inequality of \eqref{eq.two-estim-d} becomes false. Indeed, unlike for the case  (B1)(i), the measure of  $\zZ(\lambda) \cap \zZ(\lambda')$ does not shrink to $0$ when $\lambda'\to \Oe^{+}$ which makes the first integral of \eqref{eq.two-estim-d} divergent.
		That is why we shall use an alternative couple of inequalities namely 
		\begin{equation}
		\int_{\kD(\lambda')}^{\kO(\lambda)} \sup_{\tilde{\lambda}\in [\lambda,\lambda']}  \big|\thetaklt^{-}\big|^\alpha \ 
		\sup_{\tilde{\lambda}\in [\lambda,\lambda']}  \big|\thetaklt^{+}\big|^{\alpha'} \ \rmd k \lesssim 1
		\quad\text{and}\quad
		\int_{\kD(\lambda')}^{\kO(\lambda)} \sup_{\tilde{\lambda}\in [\lambda,\lambda']}  \big|\thetaklt^{-}\big|^{-1} \ \rmd k \lesssim 1,
		\label{eq.two-estim-DE}
		\end{equation}
		where we recall that $\alpha = -1/2+\gamma$	and $\alpha' = -1/2-\gamma$, so that $\alpha+\alpha'=-1.$ \\ [12pt]
		The first inequality of \eqref{eq.two-estim-DE} is derived from  \eqref{eq.estim-integ-d} for $j = 1$.
		More precisely,
		$$\sup_{\tilde{\lambda}\in [\lambda,\lambda']}(\thetaklt^{\min})^{\alpha'} \leq \sup_{\tilde{\lambda}\in [\lambda,\lambda']}|\thetaklt^{+}|^{\alpha'} + \sup_{\tilde{\lambda}\in [\lambda,\lambda']}|\thetaklt^{-}|^{\alpha'}$$
		so that, using $\alpha + \alpha'= -1$ with $\alpha,\alpha'<0$, one gets  
		\begin{equation*} \label{inequtile}
		\sup_{\tilde{\lambda}\in [\lambda,\lambda']}  \big|\thetaklt^{-}\big|^\alpha \cdot 
		\sup_{\tilde{\lambda}\in [\lambda,\lambda']}  \big(\thetaklt^{\min}\big)^{\alpha'} \leq \sup_{\tilde{\lambda}\in [\lambda,\lambda']}|\thetaklt^{-}|^{-1}+ \sup_{\tilde{\lambda}\in [\lambda,\lambda']}|\thetaklt^{-}|^{\alpha} \cdot \sup_{\tilde{\lambda}\in [\lambda,\lambda']}  \big|\thetaklt^{+}\big|^{\alpha'}
		\end{equation*}
		which shows that \eqref{eq.two-estim-DE} implies \eqref{eq.estim-integ-d} for $j = 1$ .
		\\[12pt]
		Let us now prove \eqref{eq.two-estim-DE}. For the second inequality, we can use the  arguments of item (i.2) of case {\bf(A)}. \\ [12pt] 
		Only the first inequality requires a new argument. We look separately at the integrand in the left-hand side of the second inequality of \eqref{eq.two-estim-DE} on the intervals $I_1 := \big(\kD(\lambda'),(\kD(\lambda')+\kO(\lambda))/2\big)$ and $I_2 := \big((\kD(\lambda')+\kO(\lambda))/2,\kO(\lambda)\big)$ respectively. More precisely:
		\begin{itemize}
			\item In $I_1$, $|\thetaklt^{-}|^{-1}$ is bounded so $\; \displaystyle \sup_{\tilde{\lambda}\in [\lambda,\lambda']}  \big|\thetaklt^{-}\big|^\alpha \cdot \sup_{\tilde{\lambda}\in [\lambda,\lambda']}  \big|\thetaklt^{+}\big|^{\alpha'} \lesssim \sup_{\tilde{\lambda}\in [\lambda,\lambda']}  \big|\thetaklt^{+}\big|^{\alpha'}$,
			\item In $I_2$, $|\thetaklt^{+}|^{-1}$ is bounded so $ \; \displaystyle \sup_{\tilde{\lambda}\in [\lambda,\lambda']}  \big|\thetaklt^{-}\big|^\alpha \cdot \sup_{\tilde{\lambda}\in [\lambda,\lambda']}  \big|\thetaklt^{+}\big|^{\alpha'} \lesssim \sup_{\tilde{\lambda}\in [\lambda,\lambda']}  \big|\thetaklt^{-}\big|^{\alpha}$.
		\end{itemize} 
	From the above observation, we deduce
		\begin{equation} \label{inequtile2}
		\int_{\kD(\lambda')}^{\kO(\lambda)} \sup_{\tilde{\lambda}\in [\lambda,\lambda']}  \big|\thetaklt^{-}\big|^\alpha \ 
		\sup_{\tilde{\lambda}\in [\lambda,\lambda']}  \big|\thetaklt^{+}\big|^{\alpha'} \ \rmd k \lesssim 
		\int_{\kD(\lambda')}^{\kO(\lambda)} \sup_{\tilde{\lambda}\in [\lambda,\lambda']}  \big|\thetaklt^{-}\big|^\alpha \ \rmd k
		+
		\int_{\kD(\lambda')}^{\kO(\lambda)} 
		\sup_{\tilde{\lambda}\in [\lambda,\lambda']}  \big|\thetaklt^{+}\big|^{\alpha'} \ \rmd k.
		\end{equation}	
		Using again the arguments of item (i.2) in case (A), we see that the first integral of the right-hand side of \eqref{inequtile2} is bounded. For the second integral, we come back to the definition \eqref{eq.def-thetap} of $\thetakl^{+}$ which shows that
		\begin{equation*} 
		\forall \; k > \kD(\tilde{\lambda}), \quad 
		\big|\thetaklt^{+}\big|^{\alpha'} 
		= \big( k^2-\kD(\tilde{\lambda})^2 \big)^{\alpha'/2}
		\leq \big( k-\kD(\tilde{\lambda}) \big)^{\alpha'}.
		\end{equation*}	
		Since $\alpha'<0$, as $\tilde{\lambda} \mapsto \kD(\tilde{\lambda})$ is increasing on $[\Oe,b]$, $\big( k-\kD(\tilde{\lambda}) \big)^{\alpha'}$ is  increasing and  therefore
		\begin{equation*} 
		\sup_{\tilde{\lambda}\in [\lambda,\lambda']}  \big|\thetaklt^{+}\big|^{\alpha'} 
		\leq \big( k-\kD(\lambda') \big)^{\alpha'}.
		\end{equation*}
		The right-hand side of this inequality is integrable in the interval $(\kD(\lambda'),\kO(\lambda))$ provided that $\alpha' > -1$, that is, $\gamma < 1/2.$ This completes the proof of \eqref{eq.two-estim-DE}.
		\\ [12pt]
		\textbf{(B2)} $[a,b] = [a,\Oe]$. Figure \ref{fig.Cas-Omega_e} (right) tells us that we have to prove \eqref{eq.estim-integ-d} only for $\scZ = \DE$ and that $\zDE(\lambda) \cap \zDE(\lambda') = (-\kO(\lambda),+\kO(\lambda))$. \\ [12pt]
		The treatment of this case is similar to the case {(\bf B1)(ii)} : instead of \eqref{eq.estim-integ-d}, we shall prove the stronger estimate (that differs from  \eqref{eq.two-estim-DE} only by integration bounds) 
		\begin{equation}
		\int_{0}^{\kO(\lambda)} \sup_{\tilde{\lambda}\in [\lambda,\lambda']}  \big|\thetaklt^{-}\big|^\alpha \ 
		\sup_{\tilde{\lambda}\in [\lambda,\lambda']}  \big|\thetaklt^{+}\big|^{\alpha'} \ \rmd k \lesssim 1
		\quad\text{and}\quad
		\int_{0}^{\kO(\lambda)} \sup_{\tilde{\lambda}\in [\lambda,\lambda']}  \big|\thetaklt^{-}\big|^{-1} \ \rmd k \lesssim 1.
		\label{eq.two-estim-DEbis}
		\end{equation}
		The second inequality of \eqref{eq.two-estim-DEbis} is treated exactly as in the item (i.2) of case {\bf(A)}. For the first one, as in the case {(\bf B1)(ii)}, by splitting $\big(0, \kO(\lambda)\big)$ into $I_1=(0,\kO(\lambda)/2)$ and $I_2=(\kO(\lambda)/2,\kO(\lambda))$ to separate the singularities at $k=0$ and $k=\kO(\lambda)$, one shows with the same arguments that the inequality  \eqref{inequtile2} holds by replacing the interval of integration   $(\kD(\lambda'),\kO(\lambda))$ by $\big(0, \kO(\lambda)\big)$. Then, using the arguments of item (i.2)  in case \textbf{(A)}, we see that the first integral of the right-hand side of \eqref{inequtile2} (with this new interval of integration)  is bounded. For the second one, from the definition \eqref{eq.def-kp} of $k^+(\lambda)$, we have
	\begin{equation*}
		\forall (k,\tilde{\lambda}) \in \Big(\bbR\times[a,\Oe]\Big) \setminus \{(0,\Oe)\}, \quad
		\big|\thetaklt^{+}\big|^{\alpha'} 
		= \Big( k^2 + \big|k^+(\tilde{\lambda})\big|^2 \Big)^{\alpha'/2}.
	\end{equation*}
and 
\begin{equation*} 
		\sup_{\tilde{\lambda}\in [\lambda,\lambda']}  \big|\thetaklt^{+}\big|^{\alpha'} 
		\leq  \Big( k^2 + \big|k^+(\lambda')\big|^2 \Big)^{\alpha'/2} \leq k^{\alpha'}
		\end{equation*}
is integrable in the interval $(0,\kO(\lambda))$ provided that $\alpha' > -1$, that is, $\gamma < 1/2.$ Thus the second integral of the right-hand side of \eqref{inequtile2} with interval  of  integration $(0,\kO(\lambda))$  is bounded. This completes the proof.\\[12pt]
	\textbf{Case (C)}. It remains to study the case where $[a,b]$ contains $\pm\Oc$, more precisely when $a$ or $b$ is equal to $\pm\Oc$. Fortunately, all the work has already been done in case {\bf (A)}. The same lines are still valid in the present case, the only difference is that the parameter $\beta$ defined by the first equality of \eqref{eq.def-beta} is now equal to $\beta = -1 -2\gamma$ (see \eqref{eq.def-alpha}).We have seen that the boundedness of both integrals in \eqref{eq.two-estim-d} is ensured provided $\beta/2 + 1 > 0$ (see \eqref{eq.case-a-i}), which explains why the possible H\"older exponents $\gamma$ are restricted to the interval $(0,\min(1/2,s-1/2))$ in this situation.
\end{proof}

\begin{lemma} \label{l.int-theta2}
	Let $\scZ \in \calZ\setminus\{\EE\}$, $[a,b] \subset \bbR\setminus \sigma_{\rm exc}$ and $\gamma \in \Gamma_{[a,b]}$. Then for all $\lambda, \lambda' \in [a,b],$ we have
	\begin{equation}\label{eq.estim-int-theta2}
		\sum_{j \in J_{\scZ}} \int_{\zZ(\lambda) \setminus \zZ(\lambda')} \big|\thetakl^{-j}\big|^{2\alpha} \ \rmd k \lesssim |\lambda'-\lambda|^{\gamma},
	\end{equation}
	where $\alpha$ is defined in \eqref{eq.def-alpha}.
\end{lemma}
\begin{proof}
	We distinguish again the  cases (A), (B) and (C) defined in \eqref{eq.def-ABC}. In order to use Figures \ref{fig.zoomholder} and \ref{fig.Cas-Omega_e}, we assume again that $\lambda < \lambda'$ and prove separately \eqref{eq.estim-int-theta2} for the domains $\zZ(\lambda) \setminus \zZ(\lambda')$ and  $\zZ(\lambda') \setminus \zZ(\lambda)$, which amount to the statement of the lemma.
\\ [12pt]
	\textbf{Case (A)}. As in the proof of Lemma \ref{l.int-theta1}, we can restrict ourselves to the particular situation \eqref{eq.particular-case} illustrated by Figure \ref{fig.zoomholder}, which is representative of all cases where $[a,b] \cap \{\pm\Oe,\pm\Oc\} = \varnothing$. We have to prove \eqref{eq.estim-int-theta2} for $\scZ = \DI$ and $\EI$.
\\ [12pt]	
	\textbf{(i)} Let us start with the case $\scZ = \DI$ for which $J_{\scZ}=\{-1,+1\}$. Figure \ref{fig.zoomholder} shows that $\zDI(\lambda) \setminus \zDI(\lambda') = \varnothing$, whereas $\zDI(\lambda') \setminus \zDI(\lambda) = (\kO(\lambda),\kO(\lambda')) \cup (-\kO(\lambda'),-\kO(\lambda))$ whose  measure is $2(\kO(\lambda')-\kO(\lambda))$. Since $|k|$ does not approach the spectral cut $|k|=\kI(\lambda),$ $|\thetaklp^{+}|^{2\alpha} \lesssim 1$ whatever the sign of $\alpha$. Therefore,
	\begin{equation*}
		\int_{\zDI(\lambda') \setminus \zDI(\lambda)} \big|\thetaklp^{+}\big|^{2\alpha} \,\rmd k \lesssim |\kO(\lambda')-\kO(\lambda)| \lesssim |\lambda'-\lambda| \lesssim |\lambda'-\lambda|^\gamma,
	\end{equation*}
	since $\lambda \to k_0(\lambda)$ is differentiable in $[a,b]$ and $\gamma < 1$.
\\ [12pt]
	For $\thetaklp^{-}$, $|k|$ can approach the spectral cut $|k|=\kO(\lambda')$, thus $|\thetaklp^{-}|^{-1}$ is no longer bounded. Nevertheless
          $$
	\big|\thetaklp^{-}\big|^{2\alpha} \lesssim \big|\thetaklp^{-}\big|^{\min(2\alpha,0)}
	\lesssim \big| |k|-k_0(\lambda') \big|^{\tilde{\alpha}}, \quad \tilde{\alpha} :=\min(\alpha,0),
	$$
	where the first inequality holds because $\thetaklp^{-}$ is bounded and the second one follows from \eqref{eq.singtheta-spectralcut-1}.  Thus	
	\begin{equation}
		\int_{\zDI(\lambda') \setminus \zDI(\lambda)} \big|\thetaklp^{-}\big|^{2\alpha} \,\rmd k 
		\lesssim \int_{\kO(\lambda)}^{\kO(\lambda')} \big| |k|-\kO(\lambda') \big|^{\tilde{\alpha}} \,\rmd k
		\lesssim \big( \kO(\lambda')-\kO(\lambda) \big)^{\tilde{\alpha}+1}
		\lesssim |\lambda'-\lambda|^{\tilde{\alpha}+1}.
		\label{eq.int-cas-a}
	\end{equation}
	As $\tilde{\alpha} + 1 = \min(\gamma+1/2,1) \geq \gamma$, inequality \eqref{eq.estim-int-theta2} follows.  \\ [12pt]
	\textbf{(ii)} Consider now the case $\scZ = \EI$ for which $J_{\scZ}=\{-1\}$. Figure \ref{fig.zoomholder} shows that $\zEI(\lambda') \setminus \zEI(\lambda) = \varnothing$, whereas $\zEI(\lambda) \setminus \zEI(\lambda')$ is composed by the set $(\kO(\lambda),\kO(\lambda')) \cup (\kI(\lambda'),\kI(\lambda))$ and its symmetric with respect to $k=0$. For each interval, we can use the same arguments as in item {\bf (i)}, depending on whether $k$ approaches or not the spectral cut near which $|\thetakl^{+}|^{2\alpha}$ becomes unbounded.
\\ [12pt]
	\textbf{Case (B)}. We examine now the case where $a$ or $b$ is equal to $\pm\Oe$. As in the proof of Lemma \ref{l.int-theta1}, we assume that $\Oe > \Om$ and detail the proof for $a = \Oe$ (Figure \ref{fig.Cas-Omega_e}) (the other cases can be done analogously). We have to prove \eqref{eq.estim-int-theta2} for $\scZ = \DD$ and $\DE$.
\\ [12pt]
	\textbf{(i)} Suppose that $\scZ = \DD$.  Then $J_{\scZ}=\{-1,+1\}$, 
	 $\zDD(\lambda') \setminus \zDD(\lambda) = (\kD(\lambda),\kD(\lambda')) \cup ( -\kD(\lambda'), -\kD(\lambda) ) $ while $\zDD(\lambda) \setminus \zDD(\lambda') = \varnothing$ (Figure \ref{fig.Cas-Omega_e}). For $j=1$, as  $\thetaklp^{-}$ does not vanish on $\overline{\Lambda_{\DD}[a,b])}$, $|\thetaklp^{-}|^{2\alpha}$ is bounded and thus  it follows that
	 \begin{equation*}
	\int_{\zDD(\lambda) \setminus \zDE(\lambda')} \big|\thetaklp^{-}\big|^{2\alpha} \,\rmd k \lesssim |\kD(\lambda')-\kD(\lambda)| \lesssim |\lambda'-\lambda|^{1/2} \lesssim |\lambda'-\lambda|^\gamma,
	\end{equation*}
	since $\lambda \to k_D(\lambda)$ is $1/2$-H\"older continuous in $[\Oe,b]$ (see \eqref{eq.kD-Holder}) and $\gamma \leq 1/2$. \\ [12pt]
Indeed, the main difference with case {\bf (A)} concerns \eqref{eq.estim-int-theta2} for $\thetaklp^{+}$ for the case $j=-1$. By definition \eqref{eq.def-Gamma-K} of $\Gamma_{[\Oe,b]}$, $\gamma < 1/2$ hence $\alpha = \gamma - 1/2 < 0$ (cf. \eqref{eq.def-alpha}). Since $k^+ = \kD$ (cf \eqref{eq.def-kp}),  \eqref{eq.singtheta-spectralcut-2}  implies
	 $$
	 \big|\thetaklp^{+}\big|^{2\alpha} 
	 \lesssim  \kD(\lambda')^{\alpha}\ \big| k-\kD(\lambda') \big|^{\alpha}.
	 $$
	 Therefore, we obtain
	\begin{equation} \label{latterestimate}
		\int_{\zDD(\lambda') \setminus \zDD(\lambda)} \big|\thetaklp^{+}\big|^{2\alpha} \,\rmd k 
		\lesssim \int_{\kD(\lambda)}^{\kD(\lambda')} \kD(\lambda')^{\alpha}\ \big| k-\kD(\lambda') \big|^{\alpha} \,\rmd k
		\lesssim \kD(\lambda')^{\alpha}\ \big( \kD(\lambda')-\kD(\lambda) \big)^{\alpha+1}.
	\end{equation}
	Even though $\lambda \mapsto \kD(\lambda)$ is no longer differentiable on $[\Oe,b]$, it is $1/2$-H\"older continuous. Indeed it is easy to see from the definition \eqref{defk0DI} of $\kD(\lambda)$ that
	\begin{equation}
	\kD(\lambda) = \sqrt{2 \varepsilon_0 \, \Oe \, \mu^+_{\Oe} } \ (\lambda-\Oe)^{1/2} \ \big(1+o(1)\big) 
	\quad\text{as } \lambda \searrow \Oe.
	\label{eq.kD-Holder}
	\end{equation}
	Hence, to get an estimate in $|\lambda - \lambda'|^\gamma$, it is natural to make appear $( \kD(\lambda')-  \kD(\lambda))^{2\gamma}$, which we do with the following trick 
	$$
	( \kD(\lambda')-  \kD(\lambda))^{1/2+\gamma} = ( \kD(\lambda')-  \kD(\lambda))^{2\gamma} \; ( \kD(\lambda')-  \kD(\lambda))^{1/2-\gamma}
	$$
	that we substitute into \eqref{latterestimate} to obtain, using again $\alpha = \gamma - 1/2$,
	\begin{equation*}
		\int_{\zDD(\lambda') \setminus \zDD(\lambda)} \big|\thetaklp^{+}\big|^{2\alpha} \,\rmd k 
		\lesssim F(\lambda,\lambda')\ \big( \kD(\lambda')-\kD(\lambda))^{2\gamma}
		\quad\text{where}\quad
		F(\lambda,\lambda') := \Big(1-\frac{\kD(\lambda)}{\kD(\lambda')}\Big) ^{1/2-\gamma}.
	\end{equation*}
	As $\gamma\leq 1/2$  and $F(\lambda,\lambda') \leq 1$, this completes the proof of \eqref{eq.estim-int-theta2} ($\kD$ is $1/2$-H\"older continuous on $[\Oe,b]$).
\\ [12pt]
\textbf{(ii)} For $\scZ = \DE,$ we have $J_{\scZ}=\{+1\}$ and  $\zDE(\lambda) \setminus \zDE(\lambda') = (\kD(\lambda),\kD(\lambda')) \cup  (-\kD(\lambda'), -\kD(\lambda))$ (see Figure \ref{fig.Cas-Omega_e}). Since  $\thetakl^{-}$ does not vanish, this case can be treated as the case \textbf{(B)}(i) for $j=1$.\\[6pt]
On the other hand $\zDE(\lambda') \setminus \zDE(\lambda) = (\kO(\lambda),\kO(\lambda')) \cup  (-\kO(\lambda'), -\kO(\lambda))$. This case can be dealt with as in case {\bf (A)-(i)}, see \eqref{eq.int-cas-a}.
\\ [12pt]
	\textbf{Case (C)}. As in the proof of Lemma \ref{l.int-theta1}, the only difference with case {\bf (A)} is that the parameter $\alpha$ is now defined by $\alpha = -1/2$ (see \eqref{eq.def-alpha}). Hence the most restrictive situation corresponds to \eqref{eq.int-cas-a} with $\alpha + 1 = 1/2$, which shows that all H\"older exponents $\gamma \in (0,\min(1/2,s-1/2))$ are allowed here.
\end{proof}

\subsubsection{Component related to the lineic spectral zone}
\label{sss.Holder-M-lineic}
We focus now on the proof of \eqref{eq.holderestim-specdensity-Z} for $\scZ = \EE$, which has to be considered only if $\Oe \neq \Om$. We consider a positive interval
\begin{equation*}
	[a,b] \subset \left\{  
	\begin{array}{ll}
		[\Oc,\Op) & \text{if }\Oe < \Om, \\[4pt]
		(\Op,\Oc] & \text{if }\Oe > \Om, 
	\end{array}\right.
\end{equation*}
but of course, the same approach applies for negative intervals. We can reduce to  such intervals since $\Lambda_{\EE}(\lambda)=\Lambda_{\EE}(\lambda')=\varnothing$ for $|\lambda|, |\lambda'|\notin\big(\min(\Oc, \Op),\max(\Oc,\Op)\big)$ and thus $\bbml^{\EE}=\bbmlp^{\EE}=0$ for these values of $\lambda$ and $\lambda'$ (see Figure \ref{fig.speczones1}).
Thus,  to show  the H\"{o}lder  continuity of  $\lambda \mapsto \bbml^{\EE}$  on a vicinity  of $\Oc$ (case (C)), it is sufficient to prove  the H\"{o}lder  continuity  on  $[a,\Oc)$ if $\Oe>\Om$ (resp. on $(\Oc, b]$ if $\Oe<\Om$) and  to show that the limit of $\bbml$ when  $\lambda\to \Oc^{-}$ (resp. $\lambda\to \Oc^{+}$) is zero to ensure that  $\bbml^{\EE}$ is continuous at $\Oc$.

We proceed as we did for obtaining \eqref{eq.estim-common-part}, but here one additional term appears due to the presence of the Jacobian $\JacE(\lambda)$ as a multiplicative factor. For conciseness, we drop the index 0 in the notation of the generalized eigenfunctions $\bbW_{k,\lambda,0}$. 
From the expression \eqref{eq.spec-meas-lin} of $\bbM_\lambda^\EE$ and the fact that $k \in \zEE(\lambda)$ if and only if $k = \pm \kE(\lambda)$, we have
\begin{multline*}
	\bbM_{\lambda'}^\EE \bU - \bbM_\lambda^\EE \bU
	= \sum_\pm \{ \JacE(\lambda') - \JacE(\lambda)\}\ \langle \bU,\bbW_{\pm \kE(\lambda'),\lambda'}\rangle_{s}  \; \bbW_{\pm \kE(\lambda'),\lambda'} \\
	+ \JacE(\lambda)\ \Big( \langle \bU,\{\bbW_{\pm \kE(\lambda'),\lambda'} - \bbW_{\pm \kE(\lambda),\lambda}\}\rangle_{s}  \; \bbW_{\pm \kE(\lambda'),\lambda'} + \langle \bU,\bbW_{\pm \kE(\lambda),\lambda} \rangle_{s}  \; \{\bbW_{\pm \kE(\lambda'),\lambda'}-\bbW_{\pm \kE(\lambda),\lambda}\} \Big).
\end{multline*}
As a consequence,
\begin{equation*}\label{eq.inequality-gen-Hol-ZEE}
	\big\| \bbM_{\lambda'}^\EE \bU - \bbM_\lambda^\EE \bU \big\|_{\Hms}
	\leq  \sum_\pm d_{\lambda,\lambda'}^\pm \ \|\bU\|_{\Hps},
\end{equation*}
where we have denoted
\begin{multline}
	d_{\lambda,\lambda'}^\pm  := 
	\big|\JacE(\lambda') - \JacE(\lambda)\big| \ \big\| \bbW_{\pm \kE(\lambda'),\lambda'} \big\|_{\Hms}^2 \\ + \big|\JacE(\lambda)\big| \
	\Big\{ \big\| \bbW_{\pm \kE(\lambda'),\lambda'} \big\|_{\Hms} + \big\| \bbW_{\pm \kE(\lambda),\lambda} \big\|_{\Hms} \Big\} \ \big\| \bbW_{\pm \kE(\lambda'),\lambda'} - \bbW_{\pm \kE(\lambda),\lambda} \big\|_{\Hms}.
	\label{eq.def-d-lineic}
\end{multline}
Hence, the proof of \eqref{eq.holderestim-specdensity-Z} for $\scZ = \EE$ will be complete once we have established the following property:
\begin{equation}\label{eq.Holder-d-lineic}
	\forall \lambda,\lambda' \in [a,b], \quad
	d_{\lambda,\lambda'}^\pm \lesssim |\lambda'-\lambda|^{\gamma},
\end{equation}
for any $\gamma \in \Gamma_{[a,b]}$.
We have now to distinguish cases (A) and (C) (case (B) cannot occur here).\\[6pt]
\textbf{Case (A)}. Using the estimate \eqref{eq.ineqboundplasm} for 
$\| \bbW_{\pm\kE(\lambda),\lambda}\|_{\Hms}$ and $\| \bbW_{\pm\kE(\lambda'),\lambda'}\|_{\Hms}$, the fact that $|\theta_{\pm\kE(\lambda),\lambda}| \lesssim 1$ (since $\lambda$ remains far from $\Op$, so $\kE(\lambda)$ remains in a compact set), as well as the fact that $\lambda \to \JacE(\lambda)$ is bounded in $[a,b]$, we have 
\begin{equation}
	d_{\lambda,\lambda'}^\pm \lesssim \big| \JacE(\lambda') - \JacE(\lambda) \big| +
	\| \bbW_{\pm\kE(\lambda'),\lambda'} - \bbW_{\pm\kE(\lambda),\lambda} \|_{\Hms}.
	\label{eq.estim-d-lineic}
\end{equation}
Therefore, using the H\"older estimate \eqref{eq.ineqboundholdplasm} of Proposition \ref{prop.holdestimateplasmon} and the fact that $\lambda \to \JacE(\lambda)$ is differentiable in $[a,b]$ (see the beginning of \S\ref{sec.Holder-lineic}), we obtain property \eqref{eq.Holder-d-lineic} for any $\gamma \in (0,1]\cap (0,s-1/2)$.\\[6pt]
\textbf{Case (C)}. The estimation of $d_{\lambda,\lambda'}^\pm$ is more delicate if $\Oc \in [a,b],$ that is, if $a=\Oc$ or $b=\Oc$.  To fix ideas, we assume that $\Oe>\Om$  and $[a,b]=[a,\Oc]$ (we can proceed similarly if $\Om>\Oe$ and $[a,b]=[\Oc,b]$). As we discussed  above, it is sufficient  to show that: $\lambda \mapsto \bbml$  is H\"{o}lder continuous on $[a,\Oc)$ and that $\bbml$ tends to $0$  at $\Oc^{-}$.\\[6pt]
\textbf{(i) a)} We show first that $\lambda \mapsto \bbml$  is H\"{o}lder continuous on $[a,\Oc)$.   Assume that $\lambda, \lambda'\in [a,\Oc)$ with  $\lambda \leq \lambda'$. We use again \eqref{eq.ineqboundplasm}, which shows that
\begin{equation*}
	\Big\{ \big\| \bbW_{\pm \kE(\lambda'),\lambda'} \big\|_{\Hms} + \big\| \bbW_{\pm \kE(\lambda),\lambda} \big\|_{\Hms} \Big\} 
	\lesssim \sup_{\tilde{\lambda}\in [\lambda,\lambda']}  (\theta_{\kE(\tilde{\lambda}),\tilde{\lambda}}^{+})^{1/2} 
\end{equation*}
The main difference with case (A) is that we can no longer use \eqref{eq.ineqboundholdplasm}. We have to use instead the H\"older estimate \eqref{eq.ineqboundholdplasmcrosspoint}. As  $\lambda \mapsto \JacE(\lambda)$ is still  differentiable in $[a,\Oc]$, the definition \eqref{eq.def-d-lineic} of $d_{\lambda,\lambda'}^\pm$ yields now
\begin{equation*}
	d_{\lambda,\lambda'}^\pm \lesssim \big(1+p_{\lambda,\lambda'}\big)\ |\lambda'-\lambda|^{\gamma},
\end{equation*}
where
\begin{equation*}
	p_{\lambda,\lambda'} := 
	\sup_{\tilde{\lambda}\in [\lambda,\lambda']}  (\theta_{\kE(\tilde{\lambda}),\tilde{\lambda}}^{+})^{1-\gamma/2} 
	\ \sup_{\tilde{\lambda}\in [\lambda,\lambda']}  (\theta_{\kE(\tilde{\lambda}),\tilde{\lambda}}^{+})^{-3\gamma/2}.
\end{equation*}
Thus, we need to prove that $p_{\lambda,\lambda'} \lesssim 1.$ To do so, we need to know the behaviour of $\theta_{\kE(\lambda),\lambda}^+$ near $\lambda = \Oc$. This is the object of  Lemma \ref{Dl-disp-kc} below. From \eqref{eq.asymptvoisinagekc1}, we deduce that
\begin{equation*}
	\sup_{\tilde{\lambda}\in [\lambda,\lambda']}  (\theta_{\kE(\tilde{\lambda}),\tilde{\lambda}}^{+})^{1-\gamma/2}\lesssim  (\Oc- \lambda)^{1/2-\gamma/4}
	\quad\mbox{and}\quad
	\sup_{\tilde{\lambda}\in [\lambda,\lambda']}  (\theta_{\kE(\tilde{\lambda}),\tilde{\lambda}}^{+})^{-3\gamma/2} \lesssim  (\Oc- \lambda')^{-3\gamma/4},
\end{equation*}
since $\lambda \leq \lambda'$ and $1/2-\gamma/4\geq 0$ while $-3/2\gamma\leq 0$. As a consequence
\begin{equation*}
	p_{\lambda,\lambda'} \lesssim  (\Oc- \lambda)^{1/2-\gamma/4}\ (\Oc- \lambda')^{-3\gamma/4}.
\end{equation*}
The problem is that the right-hand side of this inequality is not a continuous function on the line $\lambda'= \Oc$. However, using polar coordinates for $(\Oc - \lambda, \Oc - \lambda')$, we see that for any fixed $\kappa \in (0,1)$, it is continuous, thus bounded, in any domain of the form 
\begin{equation*}
	D_\kappa := \big\{ (\lambda, \lambda') \in [a, \Oc)^2\; \big| \; \kappa \; (\Oc - \lambda) < \Oc - \lambda' \leq \Oc - \lambda\big\} 
\end{equation*}
since $(1/2-\gamma/4)-3\gamma/4 = 1/2 - \gamma > 0$. This means that the above lines yield property \eqref{eq.Holder-d-lineic} not for the whole domain $\{ (\lambda, \lambda') \in [a, \Oc]^2\; | \; \lambda\leq \lambda'\}$, but only in $D_\kappa$.\\[6pt]
\textbf{(i) b)} To conclude, we have to prove that \eqref{eq.Holder-d-lineic} holds true in the complement of $D_\kappa$, that is, 
\begin{equation*}
	D_\kappa^{\rm c} := \big\{ (\lambda, \lambda') \in [a, \Oc)^2\; \big| \; \kappa \; (\Oc - \lambda) \geq \Oc - \lambda'\big\}.
\end{equation*}
The idea is to restart for instance from \eqref{eq.estim-d-lineic} (which is still valid here) and use simply the triangle inequality to bound the last term of the right-hand side, which yields
\begin{equation*}
	d_{\lambda,\lambda'}^\pm \lesssim |\lambda' - \lambda| + \big(\theta_{\kE(\lambda'),\lambda'}^+ + \theta_{\kE(\lambda),\lambda}^+\big),
\end{equation*}
thanks to the differentiability of $\lambda \to \JacE(\lambda)$ in $[a,\Oc]$ and the estimate \eqref{eq.ineqboundplasm}. Applying again Lemma \ref{Dl-disp-kc} below, we deduce that
\begin{equation*}
	\theta_{\kE(\lambda'),\lambda'}^+ + \theta_{\kE(\lambda),\lambda}^+
	\lesssim (\Oc-\lambda' )^{1/2}+(\Oc-\lambda)^{1/2}.
\end{equation*}
Noticing that the inequality $\kappa \, (\Oc - \lambda) \geq \Oc - \lambda'$ which characterizes points of $D_\kappa^{\rm c}$ can be written equivalently as
\begin{equation*}
	\Oc - \lambda \leq (1-\kappa)^{-1}\ (\lambda'-\lambda)
	\quad\text{or}\quad
	\Oc - \lambda' \leq (\kappa^{-1}-1)^{-1}\ (\lambda'-\lambda),
\end{equation*}
we finally obtain 
\begin{equation*}
	d_{\lambda,\lambda'}^\pm \lesssim |\lambda'-\lambda|^{1/2}.
\end{equation*}
As $\gamma < 1/2$, this implies \eqref{eq.Holder-d-lineic} for all $(\lambda,\lambda') \in D_\kappa^{\rm c}$.  \\[6pt]
{\bf (ii)} To  complete the proof for case (C), it only remains to prove that  that the limit of $\bbml$ at $\Op^{-}$ is zero. This is an immediate consequence of the expression \eqref{eq.spec-meas-lin} of $\bbM_\lambda^\EE$,  the fact that $\mathcal{J}_\scE$ is bounded on $[a,\Oc]$, the estimate \eqref{eq.ineqboundplasm}  and  the Lemma \ref{Dl-disp-kc} which implies that:
$$
\| \bbml\|_{\Hps,\Hms} \leq  \mathcal{J}_\scE(\lambda) \sum_{\pm }  \|\bbW_{\pm \kE(\lambda),\lambda}\|^2 \lesssim \thetakl^{+} \lesssim (\Oc-\lambda)\to 0 \ \ \mbox{ as } \lambda\to \Oc^{-}.
$$  

\begin{lemma}\label{Dl-disp-kc}
	For $\Oe\neq \Om$, the function $\lambda \mapsto \theta^+_{ \kE(\lambda), \lambda}$ can be continuously extended by zero at  $\Oc$ and admits  the following asymptotic behaviour
	\begin{equation}\label{eq.asymptvoisinagekc1}
		\theta_{ \kE(\lambda), \lambda}^+=A \, \big|\lambda-\Oc\big|^{1/2}+ o\big(|\lambda-\Oc| ^{1/2}\big) \quad\mbox{when } \lambda\to \Oc,
	\end{equation}
	for some $A > 0$.	 
\end{lemma}
\begin{proof}
	Using the dispersion relation \eqref{eq.disp}, one gets that 
	$$
	\theta_{ \kE(\lambda),\lambda}^+=- \frac{\mu^+_\lambda}{\mu_0} \ \theta_{\kE(\lambda),\lambda}^-=\frac{|\mu^+_\lambda|}{\mu_0} \ \big|\kE(\lambda) -\sqrt{\eps_0 \, \mu_0}\,  \lambda\big|^{1/2} \ \big|\kE(\lambda) +\sqrt{\eps_0 \, \mu_0}\,  \lambda  \big|^{1/2},
	$$
	(where we use the fact that $\mu_+(\lambda)<0$ in $\Lambda_{\EE}$). Hence, it follows that
	\begin{equation}\label{eq.asymptvoisinagekc2}
		\theta_{ \kE(\lambda), \lambda}^+=(2 \kc)^{1/2}  \,  \frac{ \big|\mu^+_{\Oc}\big|}{\mu_0} \   \big|\kE(\lambda) -\sqrt{\eps_0 \, \mu_0}\, \lambda \big|^{1/2}  \  (1+o(1)) 
		\quad\mbox{when } \lambda\to \Oc.
	\end{equation}
	As $\kE$ is $C^\infty$ at $ \Oc$ and  $\kE( \Oc)=\kc=\sqrt{\eps_0 \, \mu_0}\, \Oc$  by definition of the cross points, one has
	\begin{equation}\label{eq.asymptvoisinagekc3}
		\kE(\lambda) -\sqrt{\eps_0 \, \mu_0}\,\lambda = \big(\kE'(\Oc) -\sqrt{\eps_0 \, \mu_0}\, \big) (\lambda- \Oc) + o\big(\lambda-\Oc\big)  
		\quad\mbox{as } \lambda\to \Oc.
	\end{equation}
	If $\Oe > \Om$, we know that $\kE'(\Oc) < 0$, so that $\kE'(\Oc) -\sqrt{\eps_0 \, \mu_0}$ cannot vanish. This holds true if  $\Oe < \Om$, but it is not so obvious. Indeed, starting from the definition \eqref{eq.expressionlambdae} of $\lambda_\scE(k)$, a simple but somewhat tedious calculation shows that
	\begin{equation*}
		\kE'(\Oc) = \sqrt{\eps_0 \, \mu_0}\ F\big(\Oe^2/\Om^2\big)
		\quad\text{where}\quad 
		F(X) := \frac{1-X}{1-2\big(X+X^{-1} \big)^{-1}}.
	\end{equation*}
	As $\Oe^2/\Om^2 < 1,$ it is easy to see that $F(\Oe^2/\Om^2) > 1$, which confirms that $\kE'(\Oc) -\sqrt{\eps_0 \, \mu_0}$ cannot vanish. The asymptotic behaviour \eqref{eq.asymptvoisinagekc1} then follows from \eqref{eq.asymptvoisinagekc2} and \eqref{eq.asymptvoisinagekc3} with
	\begin{equation*}
		A := (2 \kc)^{1/2}  \,  \frac{ \big|\mu^+_{\Oc}\big|}{\mu_0} \ \big|\kE'(\Oc) -\sqrt{\eps_0 \, \mu_0}\,\big|^{1/2}.
	\end{equation*}
\end{proof}


\section{The limiting absorption and limiting amplitude principles}
\label{sec.prolimabs}
The hardest part has been done. We are now able to prove both limiting principles, which is the subject of the present section. The main arguments of the proofs are classical (see, e.g., \cite{Der-86,Eid-65,Eid-69,Haz-07,Wed-91,Wil-84}). The basic ingredient is the H\"older regularity of the spectral density (Theorem \ref{th.Holder-dens-spec}).

\subsection{Proof of Theorem \ref{thm.limabs}}
\label{sec-thabslproof}
As briefly sketched in \S\ref{sec.motiv-main-results}, the proof of the limiting absorption principle amounts to a simple limiting process in a Cauchy integral. For the sake of clarity, we begin with the case where $\Oe \neq \Om.$

\textbf{The non-critical case}. In this case, we know that $\bbP_{\rm ac} = \bbP_{\rm div0}$ (see \eqref{eq.def-Pac}). The starting point is the Fourier representation of the absolutely continuous part of the resolvent (see \eqref{eq.def-Rac}) given by Theorem \ref{th.diagA}, which writes as 
\begin{equation}
R_{\rm ac}(\zeta) := R(\zeta)\,\bbP_{\rm ac} = \bbF^{*}\,\frac{1}{\lambda-\zeta}\ \bbF
\quad\text{for }\zeta \in \bbC\setminus\bbR.
\label{eq.Fourier-res}
\end{equation}
By virtue of Theorem \ref{th.dens-spec} (more precisely formula \eqref{eq.calc-fonct-ac-lim}), it can be rewritten as a Cauchy integral as follows:
\begin{equation}
R_{\rm ac}(\zeta)\ = \slim_{B(\Hps,\Hxy)} \ \int_{\bbR} \frac{\bbM_{\lambda}}{\lambda-\zeta} \,\rmd \lambda.
\label{eq.diag-res}
\end{equation}
We want to find the limit of the above integral  when $\zeta\in \bbC\setminus\bbR$ tends to a given $\omega \in \bbR\setminus\sigma_{\rm exc}$, which makes the function $\lambda \mapsto (\lambda-\zeta)^{-1}$ singular at $\lambda=\omega$. In order to isolate the role of this singularity, we choose some $\rho > 0$ small enough so that the interval $J := [\omega-\rho,\omega+\rho]$ does not contain any point of $\sigma_{\rm exc}$ and we decompose the latter function as
\begin{equation}
\frac{1}{\lambda-\zeta} = f_\zeta^{\rm sin}(\lambda) + f_\zeta^{\rm reg}(\lambda)
\quad\text{where}\quad
f_\zeta^{\rm sin}(\lambda) := \frac{\boldsymbol{1}_{J}(\lambda)}{\lambda-\zeta}
\text{ and }
f_\zeta^{\rm reg}(\lambda) := \frac{\boldsymbol{1}_{\bbR\setminus J}(\lambda)}{\lambda-\zeta}.
\label{eq.def-f-sin-reg}
\end{equation}
This leads us to split the resolvent into the sum of a ``singular part'' and a ``regular part'':
\begin{equation*}
R_{\rm ac}(\zeta) = R_{\rm sin}(\zeta) + R_{\rm reg}(\zeta)
\quad\text{where}\quad
\left\{\begin{array}{l}
R_{\rm sin}(\zeta) := \bbF^{*}\,f_\zeta^{\rm sin}(\lambda)\,\bbF, \\[4pt]
R_{\rm reg}(\zeta) := \bbF^{*}\,f_\zeta^{\rm reg}(\lambda)\,\bbF.
\end{array}\right.
\end{equation*}

On the one hand, the family of functions $\zeta \mapsto f_\zeta^{\rm reg}(\cdot)$ is differentiable with respect to $\zeta$ in a vicinity of $\omega$ uniformly with respect to $\lambda \in \bbR$. Hence, in this vicinity, the operator of multiplication by $f_\zeta^{\rm reg}(\cdot)$ is a holomorphic function of $\zeta$ for the operator norm of $B(\hatH)$. As $\bbF$ and $\bbF^*$ are bounded operators, $\zeta \mapsto R_{\rm reg}(\zeta)$ is also holomorphic in this vicinity for the operator norm of $B(\Hxy)$, thus a fortiori for the one of $B(\Hps,\Hms)$  for $s>1/2$. Its limit value at $\omega$ is simply given by 
\begin{equation}
R_{\rm reg}(\omega) = \bbF^{*}\,f_\omega^{\rm reg}(\lambda)\,\bbF 
= \slim_{B(\Hps,\Hxy)} \ \int_{\bbR\setminus J} \frac{\bbM_{\lambda}}{\lambda-\omega} \,\rmd \lambda,
\label{eq.lim-res-reg}
\end{equation}
where the last equality is obtained via the formula \eqref{eq.calc-fonct-ac-lim} applied to $f=f_\omega^{\rm reg}$.

On the other hand, the ``singular part'' $R_{\rm sin}(\zeta)$ is no longer continuous on the real axis. We denote by $R_{\rm sin}^\pm$ the restrictions of $\zeta \mapsto R_{\rm sin}(\zeta)$ to the complex half-planes $\bbC^\pm := \{ \zeta\in\bbC \mid \pm \Imag\zeta > 0 \},$ i.e.,
\begin{equation*}
\forall \zeta \in \bbC^\pm, \quad
R_{\rm sin}^\pm(\zeta) = \bbF^{*}\,f_\zeta^{\rm sin}(\lambda)\,\bbF 
= \int_{J} \frac{\bbM_{\lambda}}{\lambda-\zeta} \,\rmd \lambda.
\end{equation*}
Note that the ``$\slim$'' symbol has been removed here. Indeed, function $\lambda \mapsto f_\zeta^{\rm sin}(\lambda)$ is bounded and compactly supported in $\bbR \setminus \sigma_{\rm exc}$, so that Theorem \ref{th.dens-spec} applies: the  latter integral is a Bochner integral valued in $B(\Hms,\Hms)$. The local H\"{o}lder regularity of the spectral density $\lambda \mapsto \bbM_{\lambda}$, given in Theorem \ref{th.Holder-dens-spec}, allows to apply the Sokhotski--Plemelj formula \cite[theorem 14.1.c, p. 94]{Hen-86} which ensures the existence of the one-sided limits of $R_{\rm sin}^\pm(\zeta)$ when $\bbC^{\pm} \ni \zeta \to \omega $ for the operator norm of $B(\Hps,\Hms)$ for $s>1/2$. This formula provides us an explicit expression of these limits given by
\begin{equation}
R_{\rm sin}^\pm(\omega) 
= \dashint_{J} \frac{\bbM_{\lambda}}{\lambda-\omega}\, \rmd\lambda\  \pm\ \rmi \pi \, \bbM_{\omega}  \ \in B(\Hps,\Hms),
\label{eq.lim-res-sin}
\end{equation}
where the dashed integral denotes a Cauchy principal value at $\lambda=\omega$ (cf. Remark \ref{rem-Principalvalue}). Furthermore (see \cite{Hen-86,McL-88}), the local H\"{o}der regularity of the spectral density $\lambda \mapsto \bbM_{\lambda}$ also ensures that $\zeta \mapsto R_{\rm sin}^\pm(\zeta)$ is locally H\"{o}lder continuous on $\overline{\bbC^{\pm}} \setminus \sigma_{\rm exc}$ for the operator norm of $B(\Hps,\Hms)$, with the same H\"{o}lder exponents $\gamma \in (0,1)$ as those of $\lambda \mapsto \bbM_{\lambda}$. Let us notice that even if $\lambda \mapsto \bbM_{\lambda}$ was locally Lipschitz continuous (i.e., $\gamma = 1$), the Sokhotski--Plemelj theorem would not ensure that so is $\zeta \mapsto R_{\rm sin}^\pm(\zeta)$. This explains why the particular value $\gamma = 1$ has not been considered in Theorem \ref{th.Holder-dens-spec} (see Remark \ref{rem.excluded}).

Combining \eqref{eq.lim-res-reg} and \eqref{eq.lim-res-sin} yields the following complementary proposition of Theorem \ref{thm.limabs}, in the case where $\Oe \neq \Om$. 

\begin{proposition}
Let $s>1/2$. For all $\omega \in \bbR \setminus \sigma_{\rm exc},$ $R_{\rm ac}(\zeta)$ has one-sided limits 
$R^{\pm}_{\rm ac}(\omega) := \lim_{\eta \searrow 0} R_{\rm ac}(\omega \pm \rmi \eta)$ for the operator norm of $B(\Hps,\Hms)$, which are given by
\begin{equation}
R^{\pm}_{\rm ac}(\omega) = \slim_{B(\Hps,\Hms)}\ 
\dashint_{\mathbb{R}} \frac{\bbM_{\lambda}}{\lambda-\omega}\, \rmd \lambda\  \pm \ \rmi\pi\, \bbM_{\omega}.
\label{eq.lim-res}
\end{equation}
\label{prop.lim-res}
\end{proposition}

These one-sided limits provide the respective continuous extensions of both restrictions $R^{\pm}_{\rm ac}(\zeta)$ of $\zeta \mapsto R_{\rm ac}(\zeta)$ to the complex half-planes $\bbC^\pm$ (see \eqref{eq.def-Rac-pm}). By virtue of the above mentioned properties of $R_{\rm reg}(\zeta)$ and $R^{\pm}_{\rm sin}(\zeta)$, these extensions are locally H\"older continuous respectively in $\overline{\bbC^\pm} \setminus \sigma_{\rm exc}$, with corresponding  local admissible H\"{o}lder exponents given by \eqref{eq.def-Gamma-K}. This completes the proof of Theorem \ref{thm.limabs} in the case where $\Oe \neq \Om$.

\begin{remark}\label{rem-Principalvalue}
Let us  recall that in formula \eqref{eq.lim-res}, the combination of the ``$\slim$'' symbol with the dashed integral implies two limit processes that can be considered independently by isolating a vicinity $J$ of $\omega$, exactly as we did above. The principal value represented by the dashed integral means that we remove from $J$ a symmetric neighborhood of $\omega$, i.e., considering $J_\delta := J \setminus (\omega-\delta,\omega+\delta)$, and we take the limit of the integral on $J_\delta$ as $\delta \searrow 0$ in the operator norm of $B(\Hps,\Hms)$. The ``$\slim$'' symbol means that we introduce an increasing sequence of compact subsets of $\bbR \setminus (\sigma_{\rm exc}\cup J)$ whose union covers this set, and we take the limit of the integral on these compacts subsets for the strong operator topology of $B(\Hps,\Hms)$. In both limit processes, the integrals on compact sets are Bochner integrals valued in $B(\Hps,\Hms)$.

We point out that to gather the terms \eqref{eq.lim-res-reg} and \eqref{eq.lim-res-sin} to obtain \eqref{eq.lim-res}, we replace  the $\slim_{B(\Hps,\Hxy)}$ by the   $\slim_{B(\Hps,\Hms)}$  to ensure the existence of the principal value in \eqref{eq.lim-res-sin} (this is justified since the existence of the  $\slim_{B(\Hps,\Hxy)}$ in \eqref{eq.lim-res-reg} implies  a fortiori the existence of the $\slim_{B(\Hps,\Hms)}$ of this term as the $\Hxy$-norm dominates  the $\Hms$-norm).
\end{remark}

\textbf{The critical case}. Suppose now that $\Oe = \Om$. In this case, $\bbP_{\rm ac}$ and $\bbP_{\rm div0}$ no longer coincide. They actually differ from the sum of the eigenprojection associated to $\pm\Op$ (see \eqref{eq.def-Pac}), which are eigenvalues of infinite multiplicity. Hence, the Fourier representation \eqref{eq.Fourier-res} of $R_{\rm ac}(\zeta)$ has to be replaced here by
\begin{equation*}
R_{\rm ac}(\zeta) = \bbF^{*}\,\frac{\boldsymbol{1}_{\bbR\setminus \{\pm\Op\}}(\lambda)}{\lambda-\zeta}\ \bbF
\quad\text{for }\zeta \in \bbC\setminus\bbR.
\end{equation*}
On the other hand, the diagonal expression \eqref{eq.diag-res} of $R_{\rm ac}(\zeta)$ holds true. As a consequence, the above proof remains valid if we simply replace the definition \eqref{eq.def-f-sin-reg} of $f_\zeta^{\rm sin}$ and $f_\zeta^{\rm reg}$ by
\begin{equation*}
f_\zeta^{\rm sin}(\lambda) := \boldsymbol{1}_{\bbR\setminus \{\pm\Op\}}(\lambda)\,\frac{\boldsymbol{1}_{J}(\lambda)}{\lambda-\zeta}
\text{ and }
f_\zeta^{\rm reg}(\lambda) := \boldsymbol{1}_{\bbR\setminus \{\pm\Op\}}(\lambda)\,\frac{\boldsymbol{1}_{\bbR\setminus J}(\lambda)}{\lambda-\zeta}.
\end{equation*}

\begin{remark}
Notice that we can get with little effort an improved version of the limiting absorption principle. The improvement lies in the fact that we can replace the topology of $B(\Hps,\Hms)$ by a finer one, namely that of $B(\Hps,\rmD(\bbA)_{-s})$ for $s > 1/2,$ where $\rmD(\bbA)_{-s}$ is simply a weighted version of the domain of $\bbA$ (see \S\ref{ss.math-model}) defined by
$$
\rmD(\bbA)_{-s} := \{\bU \in \Hms \mid \bbA  \, \bU \in \Hms \}
$$
(where $\bbA\, \bU$ has to be understood in the sense of distributions). This improvement results from the fact that the resolvent naturally appears as a bounded operator from $\Hxy$ to $\rmD(\bbA)$ equipped with the graph norm, since $\bbA R(\zeta) = I+\zeta R(\zeta)$ for $\zeta \in \bbC \setminus \bbR.$ The improved version is easily deduced from the use of the latter relation in the limiting process.
\label{rem.LAbP-improved}
\end{remark}

\subsection{Proof of Theorem \ref{th.ampllim}}
\label{sec.limamplproof}

In the following, $\bU(t)$ denotes the solution to our Schr\"{o}dinger equation \eqref{eq.schro} starting from rest (i.e., $\bU(0) = 0$) with a time-harmonic excitation $\bG(t) = \bG_\omega \, \rme^{-\rmi\omega\, t}$ kicked off at $t=0$. We assume that the circular frequency $\omega$ belongs to $\bbR \setminus \sigma_{\rm exc}$.

{\bf The general case}. We first prove the general result \eqref{eq.lim-ampl} which applies in both critical and non-critical cases, assuming that the excitation $\bG_\omega \in \Hps$ belongs to the range of $\bbP_{\rm ac}$. 

We have seen in \S\ref{ss.main-results} that $\bU(t) = \phi_{\omega,t}(\bbA) \,\bG_{\omega}$ where $\phi_{\omega,t}(\cdot)$ is defined in \eqref{eq.phiduhamel2}. As $\bG_{\omega}$ is in the range of $\bbP_{\rm ac}$, one has $\bG_{\omega}=\bbP_{\rm ac}\bG_{\omega}$. Thus, we can use the spectral representation \eqref{eq.calc-fonct-ac-lim} applied to the bounded function $\phi_{\omega,t}(\cdot)$, which yields
\begin{equation}\label{eq.decompUt}
\forall t \geq 0, \quad
\bU(t)=\phi_{\omega,t}(\bbA) \,  \bbP_{\rm ac} \,\bG_{\omega}= \lim_{\Hxy}\int_{\mathbb{R}} \rmi \,\frac{\rme^{-\rmi\lambda\, t} -\rme^{-\rmi\omega \, t}}{\lambda-\omega} \,\bbM_{\lambda}\bG_{\omega}  \ \rmd \lambda.
\end{equation}
The first step is to relate this expression to the time-harmonic solution $\bU_{\omega}^{+} := R^{+}_{\rm ac}(\omega) \, \bG_\omega\in \Hms$ given by the limiting absorption principle for $s>1/2$. Proposition \ref{prop.lim-res} provides us its spectral representation:
\begin{equation}
\bU_{\omega}^{+} = \lim_{\Hms}\ 
\dashint_{\mathbb{R}} \frac{\bbM_{\lambda}\bG_{\omega}}{\lambda-\omega}\, \rmd \lambda\  + \ \rmi\pi\, \bbM_{\omega}\bG_{\omega}.
\label{eq.decompU+}
\end{equation}
The idea is to rewrite \eqref{eq.decompUt} as
\begin{equation*}
\bU(t) = - \rmi\rme^{-\rmi\omega \, t} \lim_{\Hxy} \int_{\mathbb{R}}  \left( \frac{-\rme^{-\rmi \, (\lambda-\omega) \,  t}  }{\lambda-\omega}\, \bbM_{\lambda} \bG_{\omega}   + \frac{\bbM_{\lambda}\bG_{\omega}}{\lambda-\omega} \right) \rmd \lambda
\end{equation*}
and to split the integral into two parts, by integrating separately both functions inside the parentheses. Of course, this splitting has to be done carefully, since both functions are singular at $\lambda=\omega$, while $\phi_{\omega,t}(\lambda)$ is not. The proper way to do this is to introduce two Cauchy principal values at $\lambda=\omega$ defined in $\Hms$, i.e.,
\begin{equation*}
\bU(t) = - \rmi\rme^{-\rmi\omega \, t} 
\left( \lim_{\Hms} \dashint_{\mathbb{R}}   \frac{-\rme^{-\rmi \, (\lambda-\omega) \,t}}{\lambda-\omega}\, \bbM_{\lambda}\bG_{\omega}  \,\rmd \lambda 
+ \lim_{\Hms} \dashint_{\mathbb{R}} \frac{\bbM_{\lambda}  \bG_{\omega}}{\lambda-\omega} \, \rmd \lambda\right).
\end{equation*}
As the second Cauchy principal value is exactly the one involved in the expression \eqref{eq.decompU+} of $\bU_{\omega}^{+}$, we obtain
\begin{align}
\bU(t) & = - \rmi\rme^{-\rmi\omega \, t} \,\Big( \bU_{\omega}^{+} - \bV(t) -  \rmi\pi\, \bbM_{\omega}\,\bG_{\omega}\Big)
\quad\text{where}\quad \nonumber \\[4pt]
\bV(t) & := 
\lim_{\Hms} \dashint_{\mathbb{R}}   \frac{\rme^{-\rmi \, (\lambda-\omega) \,t} }{\lambda-\omega}\, \bbM_{\lambda}\bG_{\omega} \,\rmd \lambda.
\label{eq.def-V}
\end{align}
It is then clear that the proof of \eqref{eq.lim-ampl} will be complete once we have proved the following lemma.

\begin{lemma}\label{lem-lim-ampl}
Let $s>1/2$ and $\omega \in \bbR\setminus\sigma_{\rm exc}.$	Then for all $\bG_{\omega} \in \Hps,$ we have
\begin{equation*}
\lim_{t\to +\infty} \big\| \bV(t) +  \rmi\pi\, \bbM_{\omega}\,\bG_{\omega}\big\|_{\Hms} = 0,
\end{equation*}
where $\bV(t)$ is defined in \eqref{eq.def-V}.
\end{lemma}

\begin{proof}
As in the proof of the limiting absorption principle shown in \S\ref{sec-thabslproof}, we can separate the Cauchy principal value at $\omega$ from the limit in $\Hms$ in the definition of $\bV(t)$. We simply have to choose some $\rho > 0$ small enough so that the interval $J := [\omega-\rho,\omega+\rho]$ does not contain any point of $\sigma_{\rm exc}$, which leads us to decompose $\bV(t)$ in the form
\begin{align}
\bV(t) & = \bV^{\rm sin}(t) + \bV^{\rm reg}(t)
\quad\text{where}\quad \nonumber \\[4pt]
\bV^{\rm sin}(t) := \dashint_{J}   \frac{\rme^{-\rmi \, (\lambda-\omega) \,t}}{\lambda-\omega} \, \bbM_{\lambda}\bG_{\omega}\,\rmd \lambda \quad &\text{and}\quad 
\bV^{\rm reg}(t) := \lim_{\Hms} \int_{\mathbb{R} \setminus J}   \frac{\rme^{-\rmi \, (\lambda-\omega) \,t}   }{\lambda-\omega}\, \bbM_{\lambda}\bG_{\omega}  \,\rmd \lambda.
\label{eq.def-V-reg}
\end{align}
We are going to prove successively that
\begin{equation}
\lim_{t\to +\infty} \big\| \bV^{\rm sin}(t) +  \rmi\pi\, \bbM_{\omega}\,\bG_{\omega}\big\|_{\Hms} = 0
\quad\text{and}\quad
\lim_{t\to +\infty} \big\| \bV^{\rm reg}(t) \big\|_{\Hms} = 0,
\label{eq.two-limits}
\end{equation}
which of course implies the statement of the lemma thanks to the triangle inequality.

\textbf{(i)} Let us first consider $\bV^{\rm sin}(t)$ that we rewrite as
\begin{align*}
\bV^{\rm sin}(t) = v^{\rm sin}(t) \ \bbM_{\omega}\bG_{\omega} \ + \  \rme^{\rmi\omega\,t}\ \widetilde{\bV}^{\rm sin}(t) 
\quad & \text{where}\quad
v^{\rm sin}(t) := \dashint_{J}  \frac{\rme^{-\rmi \, (\lambda-\omega) \,t}}{\lambda-\omega}\,\rmd \lambda \quad \text{and}\\[4pt]
\widetilde{\bV}^{\rm sin}(t) := \int_J \rme^{-\rmi\lambda\,t} \,\widetilde{\bV}_\lambda \,\rmd\lambda 
\quad & \text{with}\quad
\widetilde{\bV}_\lambda := \frac{\big(\bbM_{\lambda}-\bbM_{\omega}\big)\bG_{\omega}}{\lambda-\omega}.
\end{align*}
Note that the latter integral is no longer a Cauchy principal value since the function $J \ni \lambda \mapsto \widetilde{\bV}_\lambda \in \Hms$ is Bochner integrable. Indeed by virtue of the H\"older continuity of $\bbM_{\lambda}$ (Theorem \ref{th.Holder-dens-spec}), for any given $\gamma\in\Gamma_J$, there exists a constant $C^\gamma_J>0$ such that
\begin{equation*}
\forall \lambda \in J\setminus\{\omega\},\quad
\big\| \widetilde{\bV}_\lambda \big\|_{\Hms} 
\leq C^\gamma_J \ |\lambda-\omega |^{-1+\gamma}\, \| \bG_\omega \|_{\Hps}.
\end{equation*}
As a consequence, the Riemann-Lebesgue theorem (applied to $\Hms$-valued Bochner integrals) gives us
\begin{equation}
\lim_{t\to +\infty} \big\| \widetilde{\bV}^{\rm sin}(t) \big\|_{\Hms}= 0.
\label{eq.lim-tildeV}
\end{equation}
Besides, using the change of variable $\xi=(\lambda-\omega)t$, we have
\begin{equation*}
v^{\rm sin}(t) := \dashint_{-\rho t}^{+\rho t} \frac{\rme^{-\rmi \, \xi}}{\xi}\,\rmd \xi,
\end{equation*}
where the Cauchy principal value is now at $\xi = 0$. Using standard complex integration on a suitable contour (see for instance section 6.5 of \cite{Staff-02}), one easily shows that
\begin{equation*}
\lim_{t\to +\infty} \dashint_{-\rho t}^{+\rho t} \frac{\rme^{-\rmi \, \xi}}{\xi}\,\rmd \xi = -\rmi \pi.
\end{equation*}
Together with \eqref{eq.lim-tildeV}, this yields the first statement of \eqref{eq.two-limits}.

\textbf{(ii)} Consider now the part $\bV^{\rm reg}(t)$ defined in \eqref{eq.def-V-reg}. In view of formula \eqref{eq.calc-fonct-ac-lim}, we can rewrite it as
\begin{equation*}
\bV^{\rm reg}(t) = f^{\rm reg}_t(\bbA) \bbP_{\rm ac}\bG_{\omega} 
\quad\text{where}\quad
f^{\rm reg}_t(\lambda) := \boldsymbol{1}_{\mathbb{R} \setminus J}(\lambda) \    \frac{\rme^{-\rmi \, (\lambda-\omega) \,t}   }{\lambda-\omega},
\end{equation*}
since $\lambda \mapsto f^{\rm reg}_t(\lambda)$ is a bounded function on $\bbR.$ This shows in particular that $\bV^{\rm reg}(t)$ actually belongs to $\Hxy$ and that the limit in \eqref{eq.def-V-reg} can be taken in $\Hxy$ instead of $\Hms$. This limit is constructed by considering an increasing sequence $(S_n)$ of compact subsets of $S := \mathbb{R} \setminus (J\cup\sigma_{\rm exc})$ whose union covers $S$, so that
\begin{equation}
\bV^{\rm reg}(t) = \lim_{n\to\infty} \bV^{\rm reg}_n(t)
\quad\text{where}\quad
\bV^{\rm reg}_n(t) := \bbE(S_n) \bV^{\rm reg}(t) = \int_{S_n}   \frac{\rme^{-\rmi \, (\lambda-\omega) \,t}   }{\lambda-\omega}\, \bbM_{\lambda}\bG_{\omega}  \,\rmd \lambda.
\label{eq.def-Vnreg}
\end{equation}
From the above definitions of $\bV^{\rm reg}(t)$ and $\bV^{\rm reg}_n(t)$, we have
\begin{equation*}
\bV^{\rm reg}(t) - \bV^{\rm reg}_n(t) = 
f^{\rm reg}_t(\bbA) \bbE(S \setminus S_n)\bbP_{\rm ac}\bG_{\omega},
\end{equation*}
from which we deduce that
\begin{equation*}
\left\| \bV^{\rm reg}(t) - \bV^{\rm reg}_n(t) \right\|_{\Hms} 
\leq \left\| \bV^{\rm reg}(t) - \bV^{\rm reg}_n(t) \right\|_{\Hxy}
\leq \left\|f^{\rm reg}_t\right\|_\infty \left\|\bbE(S \setminus S_n)\bbP_{\rm ac}\bG_{\omega}\right\|_{\Hxy},
\end{equation*}
where $\left\|f^{\rm reg}_t\right\|_\infty = \rho^{-1}.$ Moreover we know from \eqref{eq.norm-Eac} that
\begin{equation*}
\left\| \bbE(S \setminus S_n)\bbP_{\rm ac}\bG_{\omega} \right\|_{\Hxy}^2 
= \int_\bbR \boldsymbol{1}_{S \setminus S_n}(\lambda) \, \langle\bbM_{\lambda} \bG_{\omega}, \bG_{\omega} \rangle_s\,\rmd\lambda,
\end{equation*}
which tends to 0 by definition of the sequence $(S_n)$ since $\lambda\mapsto \langle\bbM_{\lambda} \bG_{\omega}, \bG_{\omega} \rangle_s \in L^1(\bbR)$ by Corollary \ref{cor.specmes}. As the latter quantity is independent of $t$, this means that the convergence of $\bV^{\rm reg}_n(t)$ to $\bV^{\rm reg}(t)$ is uniform with respect to $t$.

For any given $\delta>0$, we know now that we can find an integer $n_\delta$ such that
\begin{equation*}
\forall t \geq 0,\quad 
\left\| \bV^{\rm reg}(t) - \bV^{\rm reg}_{n_\delta}(t) \right\|_{\Hms} 
\leq \delta/2.
\end{equation*}
Furthermore, as $S_{n_\delta}$ is bounded, we can use as in (i) the Riemann-Lebesgue theorem applied now to the Bochner integral \eqref{eq.def-Vnreg} for $n = n_\delta$, which shows that $\bV^{\rm reg}_{n_\delta}(t)$ tends to 0 in $\Hms$ as $t \to +\infty.$ As a consequence, we can find $T_\delta > 0$ such that
\begin{equation*}
\forall t \geq T_\delta,\quad 
\left\| \bV^{\rm reg}_{n_\delta}(t) \right\|_{\Hms} 
\leq \delta/2.
\end{equation*}
By the triangle inequality, we conclude that for any $\delta>0$, we can find $T_\delta > 0$ such that $\| \bV^{\rm reg}(t) \|_{\Hms} \leq \delta$ for all $t \geq T_\delta.$ In other words, we have proved the second statement of \eqref{eq.two-limits}, which completes the proof of the lemma.
\end{proof}

{\bf The critical case}. We assume now that $\Oe=\Om$ and prove the asymptotic behavior \eqref{eq.lim-ampl-reson} for an excitation $\bG_\omega \in \Hps \cap \Hxydiv$. From \eqref{eq.def-Pac}, we see that $\bG_\omega$ can be decomposed as
\begin{equation}\label{eq.decompPac-PP}
\bG_\omega = \bbP_{\rm ac}\bG_\omega + \bbP_{-\Op}\bG_\omega + \bbP_{+\Op}\bG_\omega.
\end{equation}
Hence the solution $\bU(t) = \phi_{\omega,t}(\bbA) \,\bG_{\omega}$ to our Schr\"{o}dinger equation \eqref{eq.schro} can be decomposed accordingly:
\begin{align*}
& \bU(t) = \bU_{\rm ac}(t) + \bU_{-\Op}(t) + \bU_{+\Op}(t)
\quad\text{where} \\[4pt]
& \bU_{\rm ac}(t) := \phi_{\omega,t}(\bbA)\,\bbP_{\rm ac}\bG_\omega 
\quad\text{and}\quad
\bU_{\pm\Op}(t) := \phi_{\omega,t}(\bbA)\,\bbP_{\pm\Op}\bG_\omega.
\end{align*}

On the one hand, the asymptotic behaviour of $\bU_{\rm ac}(t)$ results from the previous lines (the only difference is that we do not assume here that $\bbP_{\rm ac} \, \bG \in \Hms$, but this assumption is actually not needed, since  \eqref{eq.decompUt} holds true by replacing  $\bU(t)$  by $\bU_{\rm ac}(t)$). We obtain
\begin{equation*}
\lim_{t\to+\infty}\Big\| \bU_{\rm ac}(t) + \rmi \, R^{+}_{\rm ac}(\omega) \, \bG_\omega \, \rme^{-\rmi \omega t} \Big\|_{\Hms} = 0.
\end{equation*}

On the other hand, Theorem \ref{th.diagA} tells us that the operator $\phi_{\omega,t}(\bbA)$ is a multiplication by $\phi_{\omega,t}(\pm \Op)$ in the range of the spectral projection $ \bbP_{\pm \Op}$ associated to the eigenvalues $\pm \Op$). Hence
\begin{equation*}
\bU_{\pm\Op}(t) = \phi_{\omega,t}(\pm\Op)\,\bbP_{\pm\Op}\bG_\omega.
\end{equation*}
The conclusion follows.

\begin{remark}
\textbf{(i)} For simplicity, the limiting amplitude principle is formulated in Theorem \ref{th.ampllim} for zero initial conditions. But, one can easily restate this result for nonzero initial conditions $\bU(0)=\bU_0$. Indeed it is readily seen that the transient contribution due to non-vanishing initial conditions in the range of $\bbP_{\rm ac}$  is ``locally evanescent'', in the sense that the local energy decays: $\| \exp(-\rmi\bbA\,t) \bbP_{\rm ac} \bU_0 \|_{\Hms} \to 0$ as $t \to 0$ for all $\bU_0\in\Hps$, contrary to the total energy, which is conserved: $\rmd_t \| \exp(-\rmi\bbA\,t) \bbP_{\rm ac} \bU_0 \|_{\Hxy} = 0$. This property is a consequence of the Riemann-Lebesgue theorem, exactly as we did for some evanescent components of $\bU(t)$ in the above lines (see proof of Lemma \ref{lem-lim-ampl}, point (ii)). 
Thus, for non  a zero initial condition $\bU_0\in \Hps$  in the range of $\bbP_{\rm ac}$ (which coincides with  $ \Hxydiv$ for $\Oe\neq \Om$), the  formula 
\eqref{eq.lim-ampl}	 remains unchanged. Whereas for the critical case $\Oe=\Om$  and  an initial condition $\bU_0\in  \Hps \cap\Hxydiv$, one deduces from the decomposition \eqref{eq.decompPac-PP} applied to $\bG_{\omega}=\bU_0$,   the fact that $\| \exp(-\rmi\bbA\,t) \bbP_{\rm ac} \bU_0 \|_{\Hms} \to 0$ and  the diagonalization Theorem \ref{th.diagA} (to rewrite  the terms  $\exp(-\rmi\bbA\,t) \,\bbP_{\pm\Op}\bG_\omega$) that
	$$
\lim_{t\to+\infty}\Big\| \exp(-\rmi\bbA\,t) \bU_0- \sum_{\pm} \exp(\mp \rmi \Op \,t) \bbP_{\pm\Op}\bU_0 \Big\|_{\Hms} = 0	.$$
Hence, one has to add to the asymptotic expansion \eqref{eq.lim-ampl-reson} the contribution of the  the plasmonic waves: $\sum_{\pm} \exp(\mp  \rmi \Op \,t) \bbP_{\pm\Op}\bU_0$ due to the initial condition $\bU_0$. 

\textbf{(ii)} Following the comments of Remark \ref{rem.LAbP-improved}, we can also obtain an improved version of the limiting amplitude principle. Indeed, formulas \eqref{eq.lim-ampl} and  \eqref{eq.lim-ampl-reson} holds true for the graph norm of the space $\rmD(\bbA)_{-s}$ defined in Remark \ref{rem.LAbP-improved} instead of the norm of $\Hms$. Moreover, the asymptotic behaviour of $\rmd \bU/\rmd t$  in the norm of $\Hms$  can be derived formally  by differentiating the formulas \eqref{eq.lim-ampl} and  \eqref{eq.lim-ampl-reson} with respect to $t$.
We finally mention that these  results can be shown rigorously  in a similar way as in the proof of Lemma \ref{lem-lim-ampl}.
\label{rem.LAmP-improved}
\end{remark}

\section{Limiting absorption and amplitude principles at $\omega = \pm \Op$} \label{sec-limampl-limabs-threshold}
This section deals with a case that has not be treated up to now, namely in the case where the source frequency is $\omega = \pm \Op$ and  $\Oe \neq \Om$, which corresponds to the case where $\zEE$ is not a union of straight lines (see Figure \ref{fig.speczones1}), the latter case being already covered by  Theorems \ref{thm.limabs} and \ref{th.ampllim}. 

It is clear by parity arguments that $+ \Op$ and $- \Op$ can be treated
similarly. For the ease of the reader, in the following lemmas, propositions and
proofs, we shall consider only the vicinity of $\Op$ while the statements of
the main theorems (namely Theorems \ref{thm.limabsamplthr}, \ref{th.dens-spec-thr} and \ref{thm.regholdestimthr}) will consider both 
$+ \Op$ and $- \Op$.

\subsection{New functional framework and main results}

As we saw in this paper (see sections \ref{sec.motiv-main-results} and  \ref{sec.prolimabs}),  the  key assumption to prove limiting absorption and limiting amplitude results  at $ \omega$ is to establish   the local  H\"{o}lder regularity of the spectral density $\lambda \mapsto \bbM_{\lambda}$ on  a vicinity of $\omega$. When $\omega= \Op$, one cannot obtain such property in $B(\Hps, \Hms)$ since the function $\lambda \mapsto \bbM_{\lambda}$ does not have a limit in $B(\Hps, \Hms)$ when $\lambda\to \Op$.
This singular behavior  concerns only the part $\bbM_\lambda^{\EE}$   associated  with the spectral zone $\Lambda_{\EE}$ in the decomposition \eqref{eq.spec-meas-decomp} of $\bbM_{\lambda}$. The other components $\bbM_\lambda^{\scZ}$ for $\scZ\in \calZ\setminus\{\EE\}$ (see \eqref{eq.spec-meas-surf})   admit a limit in $B(\Hps, \Hms)$ when $\lambda\to \Op$ and furthermore the function $\lambda \mapsto \bbM_\lambda^{\scZ}$ is H\"{o}lder continuous   in $B(\Hps, \Hms)$ on a  vicinity of $\Op$.

The singular behaviour of 
\begin{equation}
\label{eq.density-MEE}
\bbM_\lambda^\EE \bU  := \sum_{k \in \zEE(\lambda)}  \JacE(\lambda)\ \langle  \bU,  \bbW_{k,\lambda,0}  \rangle_{s} \,   \bbW_{k,\lambda,0}
\end{equation}
is linked to the fact that the function $k\mapsto\lambda_\scE(k)$ has a finite limit, namely $\Op$, when $|k| \rightarrow + \infty$ (cf. the asymptote of the curves composing the zone $\zEE\cap \{(k,\lambda)\in \bbR^2\mid \lambda>0\}$ in Figure \ref{fig.speczones1}). In particular $\JacE(\lambda)=|\lambda'_\scE(\lambda_\scE^{-1}(|\lambda|)) |^{-1}$ becomes singular at the vicinity of $\Op$. More precise results are provided by  the following lemma (whose proof is given in Appendix \ref{sec.appendixA1}).
\begin{lemma}{(Asymptotic at the  frequency $\Op$)}\label{lem.asymptestimates}
	Let $\Oe\neq \Om$ and $K := \eps_0 \mu_0 \,(\Om^2-\Oe^2)$. One has the following asymptotic formula for $\lambda\to \Op^-$ (if $K<0$) and for $\lambda\to \Op^+$ (if $K>0$):
	\begin{eqnarray} 
	&&\kE(\lambda)= \displaystyle \Big(\frac{\Op \, |K|}{8}\Big)^{\frac{1}{2}}\,| \lambda-\Op |^{-\frac{1}{2}}\, (1+o(1)), \label{eq.estimasymp1}\\[4pt]
	&&\mathcal{J}_\scE (\lambda) =\frac{1}{2}\displaystyle\Big( \frac{\Op \, |K|}{8}\Big)^{\frac{1}{2}} \,   \, | \lambda -\Op |^{-\frac{3}{2}}\, (1+o(1)), \label{eq.estimasymp2}\\[4pt]
	&& \big|\kE''(\lambda)\big|  =\frac{3}{4} \Big(\frac{\Op \, |K|}{8}\Big)^{\frac{1}{2}}  \, | \lambda -\Op |^{-\frac{5}{2}} \, (1+o(1)), \label{eq.estimasympsecondderiv}\\[4pt]
	&& \theta_{\lambda,\kE(\lambda)}^{\pm}= \displaystyle \Big(\frac{\Op \, |K|}{8}\Big)^{\frac{1}{2}}\, | \lambda -\Op |^{-\frac{1}{2}} \, (1+o(1)),  \label{eq.estimasymp3}\\[1pt]
	&& A_{\kE(\lambda),\lambda,0}=\frac{\mu_0^{\frac{1}{2}}\Op}{2 \sqrt{2\pi}}  \Big( \frac{\Op \, |K|}{8}\Big)^{-\frac{1}{4}} \, | \lambda -\Op|^{\frac{1}{4}}\,  (1+o(1)). \label{eq.estimasymp4}
	\end{eqnarray}
\end{lemma}
\noindent A consequence of lemma \ref{lem.asymptestimates} is that the function $\lambda \mapsto \bbM_{\lambda}^{\EE} \in B(\Hps, \Hms)$ for $\lambda \neq \Op$ does not have a limit in $B(\Hps, \Hms)$ when $\lambda \rightarrow \Op$. For ensuring the existence of such a limit, it is necessary, as we shall see later, to change the functional framework by considering for $s\geq 0$ a smaller Hilbert space $\Xps$ for $s\geq 0$. To this aim, it  is first useful to reinterpret, via Fubini's theorem  the weighted space $\Hps$ as a weighted space of functions of $x$ with values in a weighted space of functions of $y$: 
$$
\Hps= L_{s}^2\big(\bbR_x, L_{s}^2(\bbR_y)\big)   \times \big(  L_{s}^2\big(\bbR_x, L_{s}^2(\bbR_y)\big) \big)^2     \times  L_{s}^2\big(\bbR_{+,x}, L_{s}^2(\bbR_y)\big) \times L_{s}^2\big(\bbR_{+,x}, L_{s}^2(\bbR_y) \big)^2 
$$
and to define (with obvious notation) 
\begin{equation*} \label{defHs12}
\Xps:= L_{s}^2\big(\bbR_x, H_{s}^1(\bbR_y)\big)   \times  L_{s}^2\big(\bbR_x, H_{s}^2(\bbR_y) \big)^2     \times  L_{s}^2\big(\bbR_{+,x}, H_{s}^1(\bbR_y)\big)  \times    L_{s}^2\big(\bbR_{+,x},H_{s}^2(\bbR_y) \big)^2 
\end{equation*}
where $H^N_s(\bbR_y) := \big\{ v\mid  \partial_y^j v \in  L^2_s(\bbR_y), 0\leq \ j \leq N \}$. The space $\Xps$ is
 endowed with the norm
\begin{equation*}
\begin{array}{lll}
	\|\bU\|^2_{\Xps} & := & \eps_0 \, \|E\|_{L_{s}^2(\bbR_x, H_{s}^1(\bbR_y))}^2+ \mu_0 \, \|  \bH\|_{L_{s}^2(\bbR_x, H_{s}^2(\bbR_y))^2}^2\\[12pt] 
		&+  &\eps_0^{-1} \Oe^{-2} \,  \| J\|^2_{L_{s}^2(\bbR_x, H_{s}^1(\bbR_y))}+  \mu_0^{-1} \Om^{-2} \, \| \bK\|^2_{L_{s}^2(\bbR_x, H_{s}^2(\bbR_y))^2},
\end{array}
\end{equation*}
where $ \bU=(E,\bH,J,\bK) \in \Xps$  and where  for a positive integer $N$ and the open set $\calO=\bbR_x,\bbR_{+,x}$: 
\begin{eqnarray*}
	\|u\|_{L^2_{s}(\mathcal{O}, H^{N}_{s}(\bbR_y))}^2&:=&\int_{\mathcal{O}} \, \sum_{n=0}^N \big\|\eta_{s}(x,\cdot) \, \partial^{n}_y u(x,\cdot)\big\|_{L^2(\bbR_y)} ^2 \, \rmd x .
\end{eqnarray*}

The dual of $ \Xps$ is denoted  $\Xpsstar.$ $\Xps$  has been chosen dense in $\Hps$ so that $\Hms$ can be identified to a subspace  of $\Xpsstar$, with continuous embedding. For this reason, and for simplicity of notation, the duality product between $\Xpsstar$ and $\Xps$ will still be denoted $\langle \cdot, \cdot \rangle_s$. 
One denotes by $\| \cdot \|_{\Xpsstar}$ the associated dual norm defined for all $\bU\in \Xpsstar$ by:
\begin{equation}\label{def.dualnorm}
\| \bU\|_{\Xpsstar}:= \sup_{\| \bV\|_{\Xps}\leq 1} |\langle \bV, \bU \rangle_{s}| .
\end{equation}

 We then introduce the bounded operators  $B(\Xps,\Xpsstar)$ from $\Xps$ into $\Xpsstar$ endowed with the operator norm  $\| \cdot\|_{\Xps,\Xpsstar}$.  Then, one has the obvious continuous embedding 
\begin{equation*}\label{eq.continousembedop}
 B(\Hps, \Hms)  \subset  B(\Xps, \Xpsstar),
 \end{equation*}
with the convention that one identifies an operator of  $ B(\Hps, \Hms)$ defined on $\Hps$ with  its restriction on the dense subset $\Xps$.  We shall then be able to prove that the limit of the resolvent exists as an element of $B(\Xps, \Xpsstar)$. Moreover, the function $\lambda \mapsto \bbM_{\lambda}^{\EE} \in B(\Xps, \Xpsstar)$ will be shown to be H\"{o}lder continuous at the neighbourhood of $ \Op$.

\begin{remark} \label{rem-seuils}
 In the literature concerning the limiting absorption principle at a point where the dispersion curves admits a local extremum is referred  to as a threshold. Here 
  $\Op$  is a  specific form of  threshold associated to 
  a dispersion curve that has an horizontal asymptote. 
  For such points,  the local behavior of the spectral density depends not only on the regularity of the generalized eigenfunctions but also of the convergence speed of the  dispersion curve  to its horizontal asymptote. This type of thresholds arises also  in the context  of magnetic Schr\"{o}dinger operator \cite{Soc-16}.
    In \cite{Soc-16}, the convergence of the corresponding dispersion curves to their horizontal asymptote is super-exponential. Thus, to obtain the H\"{o}lder regularity  of the spectral density, the authors consider the smaller Hilbert space of functions whose   ``Fourier  components''    associated to the considered dispersion curve  decrease supra-exponentially when $\lambda$ tends to its threshold.  Here,  we meet a new situation since this convergence speed (given by the Lemma \ref{lem.asymptestimates}) is polynomial in $k$. We could  consider also a space of functions $\bU \in \mathcal{H}$  whose generalized Fourier transform $\bbF\bU(k,\lambda,0) =\langle  \bU, \bbW_{k,\lambda,0}  \rangle_{s}$  decays as $1/\kE(\lambda)^{\alpha}$ for some $\alpha>0$ when $\lambda\to \Op$. As in  \cite{Soc-16}, this approach would lead to a non explicit space in the physical variables $(x,y)$ but  the space will be  ``more optimal''.
  Here, we prefer to adopt another strategy which yields the construction  of a non-optimal space but 
  which has the great advantage to be  explicit in the $(x,y)$ variable.
More precisely, we give with our functional framework sufficient  conditions on  the regularity  in the $y$ direction to insure  the  decay of  the part of the generalized Fourier transform associated to $\zEE$ at the vicinity of $\Op$ and thus obtain the local  H\"{o}lder  regularity of $\bbml$ at $ \Op$.
 \end{remark}
\noindent In the rest of this section, our main goal will be to establish the existence and local H\"{o}lder regularity of $\lambda \mapsto \bbml$ on $( \bbR \setminus \sigma_{\rm exc}) \cup \{\Op\}$ for $\Oe\neq \Om$ for the topology of $B(\Xps, \Xpsstar)$. Obviously, we need to do so only in an interval of the form $[ \Op - \eta,  \Op + \eta]$  for $\eta$ small enough. \\ [6pt]
All the difficulties are concentrated on the study of the generalized eigenfunctions $\bbW_{\pm k_\scE(\lambda),\lambda,0}$, called plasmonic generalized eigenfunctions in the rest of this section. \\ [6pt]

The next theorem is a limiting absorption and limiting amplitude result at the threshold frequencies $\pm \Op$ for $\Oe \neq \Om$.  It replaces  Theorems   \ref{thm.limabs} and \ref{th.ampllim} for this new case. Its proof is  the same as the one performed for Theorems   \ref{thm.limabs} and \ref{th.ampllim} in section \ref{sec.prolimabs} and will not be reproduced here. However, its validity relies on two results Theorems  \ref{th.dens-spec-thr} and \ref{thm.regholdestimthr} which are counterpart of Theorems \ref{th.dens-spec} and \ref{th.Holder-dens-spec}. The proof of Theorems \ref{th.dens-spec-thr} and \ref{thm.regholdestimthr} is the object  of the following sections.
\begin{theorem}\label{thm.limabsamplthr}{(Limiting absorption and limiting amplitude principles at $\pm \Op$ for $\Oe \neq \Om$)}\\
	Let $\Oe\neq \Om$, $s>1/2$ and $\omega= \pm \Op$.
On one hand,  the limiting absorption principle holds at $\omega$. More precisely, the absolutely continuous part of the resolvent 
$R_{\rm ac}(\zeta)$ has one-sided limits 
$R^{\pm}_{\rm ac}(\omega) := \lim_{\eta \searrow 0} R_{\rm ac}(\omega \pm \rmi \eta)$ for the operator norm of $B(\Xps,\Xpsstar)$. Moreover, using the notation \eqref{eq.def-Rac-pm},
the function  $\zeta \mapsto R^{\pm}_{\rm ac}(\zeta) \in B(\Xps,\Xpsstar)$ is locally H\"{o}lder continuous in $\overline{\bbC^\pm} \setminus \{0, \pm \Om \}$ of index $\gamma\in (0,s-1/2)\cap (0,1/3]$. Namely,  for any compact $K$ of $\overline{\bbC^\pm} \setminus \{0, \pm \Om \}$, there exists $\gamma\in (0,s-1/2)\cap (0,1/3]$ and $C_{K,\gamma}>0$ such that
\begin{equation*}
\forall (\zeta,\zeta') \in K \times K, \quad 
\Big\| R^{\pm}_{\rm ac}(\zeta') - R^{\pm}_{\rm ac}(\zeta) \Big\|_{\Xps,\Xpsstar} 
\leq C_{K,\gamma} \ |\zeta'-\zeta|^{\gamma}.
\end{equation*}  
On the other hand, for any $\bG_\omega \in \Xps$ which belongs to $\Hxydiv$ (i.e. the range of $\bbP_{\rm ac}$),  the limiting amplitude principle holds true in the sense that the solution $\bU(t)$ to \eqref{eq.schro} with zero initial conditions satisfies:
\begin{equation*}
\lim_{t\to+\infty}\Big\| \bU(t) + \rmi \, \bU_{\omega}^{+} \, \rme^{-\rmi \omega t} \Big\|_{\Xpsstar} = 0,
\label{eq.lim-ampl-thr}
\end{equation*}
where $\bU_{\omega}^{+} := R^{+}_{\rm ac}(\omega) \, \bG_\omega\in \Xpsstar$.
\end{theorem}

\begin{remark}
We point out that compared to Theorem  \ref{thm.limabs}, we have not only replaced in Theorem \ref{thm.limabsamplthr} the spaces  $\Hps$ and $\Hms$ by the space $\Xps$ and its dual $\Xpsstar$, but also the available set of admissible H\"{o}lder exponents $\gamma$  which is now $(0,s-1/2)\cap (0,1/3]$.  
\end{remark}

In section \ref{sec-sepc-density-thr}, we give the key properties of the spectral density $\bbml$ in this new setting: namely  the construction of a functional calculus with $\bbml$ and its  H\"{o}lder regularity  that are used to prove Theorem \ref{thm.limabsamplthr}. The remaining sections are devoted to prove the results of section \ref{sec-sepc-density-thr}.
In section \ref{sec-Hmsnewestimate}, we provide new $\Xpsstar$ estimates for $\bbW_{\pm k_\scE(\lambda),\lambda,0}$  and corresponding H\"{o}lder regularity estimates in section \ref{sec-Holdregulgeneralized-eigenfunc-thr}. Finally, all these estimates  are used to prove the H\"{o}lder regularity of the spectral density $\lambda \mapsto \bbml$ in the space $B(\Xps,\Xpsstar)$ in section \ref{sec.Hold-reg-spec-dens-thr}.

\subsection{The spectral density in the new functional framework}\label{sec-sepc-density-thr}
In order to understand the need of a new functional framework, let us come back to \S\ref{sec.motiv-main-results} and see what happens if, instead of assumption \eqref{e.asumpt-S}, we choose a Borel set $S \subset \bbR$ whose closure contains $\Op$  (but not 0 or $\pm\Om$). This amounts to allow in Theorem \ref{th.dens-spec} bounded functions with a compact support that contains $\Op$, or to allow in Proposition \ref{p.estim-fpg} intervals $[a,b]$ that contain $\Op$. 

On the one hand, it is easy to see that all the arguments that concern the surface spectral zones (that is, for $\scZ \in \calZ\setminus\{\EE\}$) remain unchanged, since $\Op$ do not give rise to a singular behavior of the associated generalized eigenfunctions. In particular, the first statement of Proposition \ref{p.estim-fpg} holds true if $[a,b]$ contains $\Op$. On the other hand, the arguments that concern the lineic spectral zones (that is, for $\scZ = \EE$) are no longer valid since $\Lambda_{\EE}([a,b])$ becomes unbounded. In particular, the second statement of Proposition \ref{p.estim-fpg} cannot justify the use of the Lebesgue's dominated convergence theorem and the change of variable in the last integral of \eqref{eq.spect-rep-E}. As for the surface spectral zone (see the proof of Theorem \ref{th.dens-spec}), using the fact that for all $\bU \in \Hps$, we have
\begin{equation*}
\| \bbM_\lambda^\EE \bU \|_{\Hms} \leq 
\sum_{k \in \zEE(\lambda)}  \JacE(\lambda)\ \|\bbW_{k,\lambda,0}\|_{\Hms}^2,
\end{equation*}
this justification would require now to verify that the map 
\begin{equation*}
\lambda \mapsto \sum_{k \in \zEE(\lambda)}  \JacE(\lambda)\ \|\bbW_{k,\lambda,0}\|_{\Hms}^2  
\end{equation*}
belongs to $L^1([a,b],\Hms)$. Unfortunately, using Lemma \ref{lem.asymptestimates} and similar calculations as in the proof of Proposition \ref{prop.decayestimplasmon}, one can see (after small computations) that the above map is $O(|\lambda-\Op|^{-3/2})$ near $\Op$, which is not integrable. Hence Theorem \ref{th.dens-spec} is no longer true for functions whose support contains $\Op$. 

The good news is that such a result becomes true if the spectral density $\bbM_\lambda$ is considered as an element of  $B(\Xps,\Xpsstar)$, instead of $B(\Hps,\Hms)$. Indeed, Proposition \ref{prop.decayestimplasmon} together with \eqref{eq.estimasymp2} show that the above map with $\Hms$ replaced by $\Xps$ is now $O(|\lambda-\Op|^{1/2})$ near $\Op$, which is of course integrable, and actually continuous: it implies that $\bbM_\lambda^\EE$ tends to 0 in $B(\Xps,\Xpsstar)$. As regards the other components of the spectral measure, that is, $\bbM_\lambda^\scZ$ for $\scZ \in \calZ\setminus\{\EE\}$, the results of \S\ref{sec.spectral-density} proved in $B(\Hps,\Hms)$ clearly holds in the weaker topology  of $B(\Xps,\Xpsstar)$. Hence, we have the following theorem, which is the counterpart of Theorem \ref{th.dens-spec}.

\begin{theorem}
	Let $s > 1/2$ and $\Oe \neq \Om$. For every bounded function $f: \bbR \to \bbC$ with a compact support that does not contain any point of $\{0,\pm \Om\}$, the operator $f(\bbA)\,\bbP_{\rm ac}$ is given by
	\begin{equation*}\label{eq.calcfoncthr}
	f(\bbA)\,\bbP_{\rm ac} = \int_{\bbR} f(\lambda) \, \bbM_{\lambda} \,\rmd \lambda,
	\end{equation*}
	where the spectral density $\bbM_\lambda$ is defined for all $\lambda \in \bbR \setminus \{0,\pm \Om\}$ as a bounded operator from $\Xps$ to $\Xpsstar$. The above integral is understood as a Bochner integral in $B(\Xps,\Xpsstar)$.
	\label{th.dens-spec-thr}
\end{theorem}
We deduce from the previous theorem, in the same way as for formula \eqref{eq.calc-fonct-ac-lim} in section \ref{sec.spectral-density},  that for any bounded function $f$ whose support $S$ is no longer compact and/or contains points of $\{ 0,\pm\Om \}$, one has:
\begin{equation*}
f(\bbA)\,\bbP_{\rm ac} = \slim_{B(\Xps,\Hxy)}\ \int_{\bbR} f(\lambda) \, \bbM_{\lambda} \,\rmd \lambda.
\label{eq.calc-fonct-ac-lim-thr}
\end{equation*}
We point out that  from Theorem \ref{th.dens-spec-thr}, we deduce exactly in the same way as in Section \ref{sec.motiv-main-results}   an equivalent of Corollary \ref{cor.specmes}. We simply have to replace in the formulation of Corollary \ref{cor.specmes} the space $\Hps$ by $\Xps$ and  the duality product $\langle \cdot , \cdot \rangle_s $ has to be understood as the duality product between $\Xps$ and $\Xpsstar$.

The  following theorem is the counterpart of Theorem \ref{th.Holder-dens-spec}. Its proofs is given in section \ref{sec.Hold-reg-spec-dens-thr}.
\begin{theorem}\label{thm.regholdestimthr}
	Let $\Oe\neq \Om$ and $s>1/2$. 
	The spectral density  $\lambda \mapsto \bbml \in B(\Xps,\Xpsstar)$, given by \eqref{eq.density-non-crit} is  locally H\"{o}lder-continuous on $\mathbb{R}\setminus \{ -\Om, 0,\Om\}$ of index $\gamma\in (0,s-1/2)\cap (0,1/3]$. In other words, for any interval $[a,b]\subset \mathbb{R}\setminus \{ -\Om, 0,\Om\}$, it exists $C_{a,b}^{\gamma}>0$ such that
	\begin{equation}\label{eq.specdensholdthreshold}
	\| \bbmlp  -\bbml \|_{\Xpsstar, \Xps}\leq C_{a,b}^{\gamma}  |\lambda'-\lambda|^{\gamma}, \quad\forall \lambda',  \lambda \in [a,b].
	\end{equation}
\end{theorem}

\subsection{$\Xpsstar$-estimates of the plasmonic generalized eigenfunctions}\label{sec-Hmsnewestimate}
The following proposition replaces Proposition \ref{prop.decay-estim-plasmon} in which  the right-hand side of estimate \eqref{eq.ineqboundplasm}  blows up when $|k| \rightarrow + \infty$ (see \eqref{eq.estimasymp3}).
\begin{proposition}\label{prop.decayestimplasmon}
Let $\Oe \neq \Om$, $s> 1/2$ and $[a,b]=[ \Op -\eta,  \Op+\eta]$ with $\eta>0$ sufficiently small such that $0, \Oc\notin [a,b]$.
Then, there  exists a constant $C_{\eta}>0$ (depending only on $\eta$) such that 
\begin{equation}\label{eq.ineqboundplasmasympt}
\| \bbW_{k,\lambda,0}\|_{\Xpsstar}\leq C_{\eta}  \, | \lambda-\Op|, \quad \forall (k, \lambda)\in \Lambda_{\EE}([a,b]).
\end{equation}
\end{proposition}

\begin{proof}
Obtaining $\Xpsstar$ estimates relies on a bound for the duality product $\langle \bU, \bbW_{k,\lambda,0}\rangle_s$, for $\bU \in \Xps$, of the form 
 \begin{equation*} 
\big| \langle \bU, \bbW_{k,\lambda,0}\rangle_s \big| \leq C(\lambda) \; \| \bU\|_{\Xps}
 \end{equation*}  
 that yields, by definition of the dual norm \eqref{def.dualnorm}, the estimate $\|\bbW_{k,\lambda,0} \|_{\Xps} \leq C(\lambda).$
By \eqref{eq.def-Vkbis}, it requires to estimate all the duality  products between the components of  $ \bbW_{k,\lambda,0}$ and a vector $\bU=(E,\bH,J,\bK)\in \Xps$ since
\begin{equation} \label{dualityprod}
\begin{array}{lll} 
|\langle \bU, \bbW_{k,\lambda,0}\rangle_s| & \lesssim & \displaystyle \Big| \int_{\bbR^2} E \, \overline{w_{k,\lambda,0}} \, \mathrm{d}x\mathrm{d}y \Big| + |k| \; \Big|\int_{\bbR^2}  H_x \, \overline{w_{k,\lambda,0}} \,\mathrm{d}x\mathrm{d}y \Big| + \Big| \int_{\bbR^2} H_y \, \overline{\partial_x w_{k,\lambda,0}} \, \mathrm{d}x\mathrm{d}y\Big| \\ [12pt]
& + &  \displaystyle \Big| \int_{\bbR^2_+} J \,\overline{ w_{k,\lambda,0} }\mathrm{d}x\mathrm{d}y\Big| + |k| \; \Big|\int_{\bbR^2_+}  K_x \,\overline{  w_{k,\lambda,0}}\mathrm{d}x\mathrm{d}y \Big| + \Big| \int_{\bbR^2_+} K_y \,\overline{ \partial_x w_{k,\lambda,0}} \, \mathrm{d}x\mathrm{d}y\Big|.
\end{array} 
\end{equation} 
In the following, we are going to estimate each of the terms of the right hand side of \eqref{dualityprod}, regrouping them column by column.
These are obtained by standard manipulations such as Fubini's theorem (since the product of a $L^2_s$ and a $L^2_{-s}$ function for $s\geq 0$ is $L^1$) and integration by parts    that are justified  by the fact that ${\cal D}(\bbR)$  is dense in $H^k_s(\bbR)$, $k=1,2$ and $L^2_{s}(\bbR)$ is continuously embedded in $L^1(\bbR)$ for $s>1/2$.
\\ [12pt]
	$(i)$ By definition of  $w_{k,\lambda,0}$ (cf. \eqref{eq.def-w},  \eqref{def-A-plasm} and \eqref{def-psi-plasm}), using that $A_{k,\lambda,0}$ and $\psi_{k,\lambda,0}$ are even in $k = \pm \kE(\lambda)$ along $\zEE$, we have
	$$
\Big| \int_{\bbR^2} E \, \overline{ w_{k,\lambda,0}} \ \mathrm{d}x\mathrm{d}y \Big| = A_{\kE(\lambda),\lambda,0} \,  \Big| \int_{\bbR^2}  \psi_{\kE(\lambda),\lambda,0}( x)  E(x,y)   \, \rme^{ \mp \rmi k_\scE(\lambda)  y}  \mathrm{d} x \,  \mathrm{d} y  \Big|,
	$$
 A naive estimate, using $\| \psi_{\kE(\lambda),\lambda,0} \|_{L^2_{-s}(\bbR^2)}  \lesssim 1$ (because $|\psi_{\kE(\lambda),\lambda,0}| \leq 1$ and \eqref{Linf-L-s}) would lead to 
$$
\Big| \int_{\bbR^2} E \, \overline{ w_{k,\lambda,0}}  \, \mathrm{d}x\mathrm{d}y\Big| = | A_{\kE(\lambda),\lambda,0} |  \;  \|E\|_{L^2_s(\bbR^2)} \lesssim |\lambda - \Op|^{\frac{1}{4}} \;  \|E\|_{L^2_s(\bbR^2)}
$$
which would not be sufficient to compensate the blow up of $\JacE(\lambda)$ in the expression \eqref{eq.density-MEE}. To get a sharper estimate (when $\lambda \rightarrow \Op$), we will use two properties:
\begin{itemize} 

\item the fact that $k_\scE(\lambda) \rightarrow + \infty$ when $\lambda \rightarrow \Op$, which can be exploited through an integration by parts in the $y$-variable (this is where we use the fact that $\bU \in \Xpsstar$). More precisely, one has
$$
\Big| \int_{\bbR^2} E \, \overline{w_{k,\lambda,0} }\,\mathrm{d}x\mathrm{d}y \Big| = \frac{A_{\kE(\lambda),\lambda,0}}{\kE(\lambda)}  \,  \Big| \int_{\bbR^2}  \psi_{\kE(\lambda),\lambda,0}(x)\,  \partial_y E (x,y)  \, \rme^{ \mp \rmi k_\scE(\lambda)  y}   \, \mathrm{d}x\mathrm{d}y\Big|,
$$
thus, by  the duality between  $L^2_s$ and $L^2_{-s}$,
\begin{equation*} \label{intpart}
\Big| \int_{\bbR^2} E \, \overline{w_{k,\lambda,0} }\, \mathrm{d}x\mathrm{d}y \Big| \leq  \frac{A_{\kE(\lambda),\lambda,0}}{\kE(\lambda)}  \; \|\psi_{\kE(\lambda),\lambda,0}\|_{L^2_{-s}(\bbR^2)} \; \|\partial_y E\|_{L^2_{s}(\bbR^2)}. 
\end{equation*}

	\item  $\psi_{\kE(\lambda),\lambda,0} \in L^2(\bbR_x)$ and, by an explicit computation and \eqref{eq.estimasymp3}, 
	\begin{equation*} \label{estipsi} \|\psi_{\kE(\lambda),\lambda,0}\|_{L^2(\bbR_x)} =(1/\sqrt{2})\,\big((\theta_{\kE(\lambda),\lambda}^{+})^{-1}+(\theta_{\kE(\lambda),\lambda}^{-})^{-1}\big)^{\frac{1}{2}} \lesssim |\lambda - \Op|^{\frac{1}{4}} .\end{equation*}
Therefore, we have a sharper $L^2_{-s}$-estimate of $ \psi_{\kE(\lambda),\lambda,0}$ (better than $\| \psi_{\kE(\lambda),\lambda,0} \|_{L^2_{-s}(\bbR^2)}  \lesssim 1$ )
	\begin{equation} \label{sharperHms}
	\|\psi_{\kE(\lambda),\lambda,0}\|_{L^2_{-s}(\bbR^2)} \lesssim \|\psi_{\kE(\lambda),\lambda,0}\|_{L^2_{-s}(\bbR_x)} \lesssim \|\psi_{\kE(\lambda),\lambda,0}\|_{L^2(\bbR_x)}  \lesssim |\lambda - \Op|^{\frac{1}{4}} .
	\end{equation}
\end{itemize}
Finally, using \eqref{sharperHms}, together with  the asymptotic behaviours  \eqref{eq.estimasymp1} and \eqref{eq.estimasymp4} for $A_{\kE(\lambda),\lambda,0}$ and $\kE(\lambda)$ respectively, we obtain
\begin{equation} \label{est1}
\Big| \int_{\bbR^2} E \, \overline{ w_{k,\lambda,0}}  \, \mathrm{d}x\mathrm{d}y \Big| \lesssim | \lambda- \Op | \; \|\partial_y E\|_{L^2_{s}(\bbR^2)} .
\end{equation}
Proceeding exactly in the same way, we obtain the companion estimate 
\begin{equation} \label{est1bis}
\Big| \int_{\bbR^2_+} J \, \overline{ w_{k,\lambda,0}}  \, \mathrm{d}x\mathrm{d}y \Big| \lesssim | \lambda - \Op | \; \|\partial_y J\|_{L^2_{s}(\bbR^2_+)} .
\end{equation}
	$(ii)$ For the terms in the second column of \eqref{dualityprod}, since an additional factor $|k|$ has appeared, we need an additional integration by parts to compensate it. It leads to
	\begin{equation} \label{est2}
	|k| \Big| \int_{\bbR^2} H_x \, \overline{ w_{k,\lambda,0}} \,\mathrm{d}x\mathrm{d}y \Big| \lesssim | \lambda- \Op | \; \|\partial^2_y H_x\|_{L^2_{s}(\bbR^2)},
	\end{equation}
	\begin{equation} \label{est2bis}
	|k| \Big| \int_{\bbR^2_+} K_x \,  \overline{ w_{k,\lambda,0}}  \, \mathrm{d}x\mathrm{d}y\Big| \lesssim | \lambda - \Op | \; \|\partial^2_y K_x\|_{L^2_{s}(\bbR^2_+)} .
	\end{equation}
	$(iii)$ For the terms of third column of \eqref{dualityprod}, the problem comes from the apparition of the $x$-derivative of $w$ that makes appear an additional 
	factor $\theta_{k, \lambda}$. More precisely,   one has for $x\neq 0$:
	$$
	\partial_x w_{k, \lambda, 0} = -\operatorname{sgn}(x) \; \thetakl(x)  \; w_{k, \lambda, 0}.
	$$
	Since  $\theta_{k, \lambda}^{\pm}$ behaves proportionally to $|k|$ far large $k$ (see \eqref{eq.estimasymp1} and \eqref{eq.estimasymp3}), 
	the same procedure as in point $(ii)$ can be applied. One then obtains 
	\begin{equation} \label{est3}
	\Big| \int_{\bbR^2} H_y \,\overline{\partial_x w_{k,\lambda,0}} \, \mathrm{d}x\mathrm{d}y\Big| \lesssim | \lambda - \Op | \; \|\partial^2_y H_y\|_{L^2_{s}(\bbR^2)}, 
	\end{equation}
	\begin{equation} \label{est3bis}
	|k| \Big| \int_{\bbR^2_+} K_y \, \overline{ \partial_x w_{k,\lambda,0}} \, \mathrm{d}x\mathrm{d}y | \lesssim | \lambda - \Op | \; \|\partial^2_y K_y\|_{L^2_{s}(\bbR^2_+)}. 
	\end{equation}
	Finally, the desired estimate \eqref{eq.ineqboundplasmasympt} follows from substituting (\ref{est1}, \ref{est1bis}, \ref{est2}, \ref{est2bis}, \ref{est3}, \ref{est3bis}) into \eqref{dualityprod}.
\end{proof}

\subsection{H\"older estimates of ``plasmonic generalized eigenfunctions'' at $ \Op$}\label{sec-Holdregulgeneralized-eigenfunc-thr}
For establishing such H\"older estimates, we essentially use the same strategy as the one exposed in section \ref{sec_GEForient}, which rests upon
preliminary estimates of $\lambda$-derivatives of various functions. Such estimates are the object of section  \ref{specder} whose results are exploited in section 
\ref{Holderplasmonic} to establish H\"older estimates for the function $\lambda \mapsto \bbW_{\pm\kE(\lambda),\lambda,0} \in \Xpsstar$.

We assumes that
$[a,b]=[ \Op-\eta,  \Op+\eta]$ with $\eta>0$ sufficiently small such that  $0 , \Oc, \notin [a,b]$.  
The difference with the estimates of section \ref{section.spec derivplasm} relies on the fact that $ \Op\in[a,b]$. Hence the set $\Lambda_{\EE}([a,b])$ is not bounded  and contains points arbitrarily closed to the  horizontal asymptote $\lambda=  \Op$. 
One deals with the  case $(\kE(\lambda),\lambda)\in \Lambda_{\EE}([a,b])$, but by parity arguments in $k$, one checks easily that these estimates hold  also for $(-\kE(\lambda),\lambda)\in \Lambda_{\EE}([a,b])$ (that is by replacing $\kE(\lambda)$ by $-\kE(\lambda)$ in the left-hand side in the  inequalities (\ref{eq.derivthetaplasmasympt}, \ref{eq.coeffAasympt}, \ref{eq.boundphiasympt}, \ref{eq.derivfirstcomponentasympt}, \ref{eq.derivothercomponentasympt})).
\subsubsection{$\lambda$-derivatives of the ``plasmonic generalized eigenfunctions''} \label{specder}
\noindent $\bullet$ {\bf Derivative of powers of } $\theta^{\pm}_{\kE(\lambda), \lambda}$.
We use (see Remark \ref{rem.partialderiv}) formula \eqref{eq.derivativetheta-compos}  namely
\begin{equation}\label{eq.partialderivative2}
\partial_\lambda \big[ (\theta^{\pm}_{k_{\scE}(\lambda), \lambda})^\alpha \big] = 
\frac{\alpha}{2} \Big[\partial_{k}  \Theta^{\pm}_{k_{\scE}(\lambda),\lambda } \, k_{\scE}'(\lambda) +\partial_{\lambda}  \Theta^{\pm}_{k_{\scE}(\lambda),\lambda }  \Big]  (\theta_{k_{\scE}(\lambda),\lambda }^{\pm})^{\alpha-2}.
\end{equation}
Using \eqref{eq.estimasymp1}, \eqref{eq.estimasymp2} (since $\mathcal{J}_\scE (\lambda)=|\kE'(\lambda)|$), \eqref{eq.estimasymp3}, $|\partial_{k}  \Theta^{\pm}_{k_{\scE}(\lambda),\lambda }|=2k_{\scE}(\lambda)\lesssim  | \lambda-\Op |^{-1/2}$ and $|\partial_{\lambda}  \Theta^{\pm}_{k_{\scE}(\lambda),\lambda }| \lesssim 1$ leads to
\begin{equation}\label{eq.derivthetaplasmasympt}
\big|\partial_\lambda \big[ (\theta^{\pm}_{k_{\scE}(\lambda), \lambda})^\alpha \big]\big|  \lesssim | \lambda-\Op |^{-1-\frac{\alpha}{2}}.
\end{equation}

\noindent $\bullet$ {\bf Derivative of the coefficient} $A_{\kE(\lambda),\lambda,0}$. Its expression  is given by 
$$
\partial_\lambda \big(A_{\kE(\lambda),\lambda,0}\big)= 
\partial_\lambda (\theta_{\kE(\lambda),\lambda }^{+})^{\frac{1}{2}} \, B_{\scE}(\lambda) +  (\theta_{\kE(\lambda),\lambda }^{+})^{\frac{1}{2}} \;\partial_\lambda B_{\scE}(\lambda)
$$
where $B_{\scE}(\lambda)$ is defined from $A_{k_{\scE}(\lambda), \lambda,0} = (\theta_{\kE(\lambda),\lambda }^{+})^{\frac{1}{2}} \, B_{\scE}(\lambda)$  with $A_{k_{\scE}(\lambda), \lambda,0}$ given by 
\eqref{def-A-plasm}. 
By virtue of  \eqref{eq.estimasymp1} and \eqref{eq.estimasymp2}, one gets easily (details are omitted) that 
$$ B_{\scE}(\lambda)=O(|\lambda-\Op|^{\frac{1}{2}})  \mbox{ and } \partial_{\lambda} \, B_{\scE}(\lambda)= O(|\lambda-\Op|^{-\frac{1}{2}}) \mbox{ when } \lambda\to \Op.$$
Combining these estimates with the asymptotic formula  \eqref{eq.estimasymp3} and \eqref{eq.derivthetaplasmasympt} for $\alpha=1/2$ finally leads to:
\begin{equation}\label{eq.coeffAasympt}
\big|\partial_\lambda \big(A_{\kE(\lambda),\lambda,0}\big)\big| \lesssim    | \lambda-\Op |^{-\frac{3}{4}}.
\end{equation}
\noindent $\bullet$ {\bf Derivative of } $ \psi_{\kE(\lambda),\lambda,0}(x) \, \rme^{\rmi k y}$.
From the expression  \eqref{eq.expression-deriv-lambda-psi}, one gets  using  \eqref{eq.estimasymp1} and  \eqref{eq.derivthetaplasmasympt}  for $\alpha=1$:
\begin{equation}\label{eq.boundphiasympt}
|\partial_{\lambda}( \psi_{\kE(\lambda),\lambda,0}(x) \, \rme^{\rmi k y}) |\lesssim (|x|+|y|) \,  \psi_{\kE(\lambda),\lambda,0}(x) \,  | \lambda-\Op |^{-\frac{3}{2}} ,   \ \forall (x,y)\in \bbR^2.
\end{equation}
\noindent $\bullet$ {\bf Derivative of} $ w_{\kE(\lambda),\lambda,0}(x,y)$. 
Applying  \eqref{eq.coeffAasympt}, \eqref{eq.boundphiasympt} and \eqref{eq.estimasymp4} on the expression  \eqref{eq.express-derivpartialambda-wlo} of $\partial_{\lambda}w_{k,\lambda,0}$ gives:
\begin{equation}\label{eq.derivfirstcomponentasympt}
|\partial_{\lambda}w_{\kE(\lambda),\lambda,0}(x,y)| \lesssim  (1+|x|+|y|) \   \psi_{\kE(\lambda),\lambda,0}(x) \,  | \lambda-\Op |^{-\frac{5}{4}},\   \forall (x,y) \in \bbR^2.
\end{equation}
\noindent $\bullet$ {\bf Derivative  of} $ \partial_x w_{\kE(\lambda),\lambda,0}(x,y)$. One has $\partial_x w_{\kE(\lambda),\lambda,0}(x,y)= \operatorname{sgn}(x) \, \thetakl(x) w_{\kE(\lambda),\lambda,0}(x,y)$ for $x\neq 0$.
Furthermore, with the expression of $w_{\kE(\lambda),\lambda,0}$ (see (\ref{eq.def-w}, \ref{def-psi-plasm})), one sees that  
$$ |w_{\kE(\lambda),\lambda,0}(x,y)| \lesssim A_{\kE(\lambda),\lambda,0}  \,  \psi_{\kE(\lambda),\lambda,0}(x). $$
It is then easy to show that (again the details are left to the reader), using  \eqref{eq.derivfirstcomponentasympt}, \eqref{eq.estimasymp3},   \eqref{eq.estimasymp4} and  \eqref{eq.derivthetaplasmasympt}  for $\alpha=1$,  
\begin{equation}\label{eq.derivothercomponentasympt}
|\partial_{\lambda} \partial_x w_{\kE(\lambda),\lambda,0}(x,y)| \lesssim  (1+|x|+|y|) \   \psi_{\kE(\lambda),\lambda,0}(x) \,  \big| \lambda-\Op \big|^{-\frac{7}{4}}, \  \forall (x,y)\in \bbR^*\times \bbR. 
\end{equation}

\subsubsection{H\"older estimates of the plasmonic generalized eigenfunctions} \label{Holderplasmonic}
The following proposition replaces Proposition \ref{prop.holdestimateplasmon} which does not hold if $\pm \Op\in [a,b]$.
\begin{proposition}\label{prop.holdestimateplasmonasymp}
	Let  $\Oe \neq \Om$, $s> 1/2$, $\gamma \in (0,1]\cap (0,s-1/2)$ and $[a,b]=[ \Op -\eta, \Op+\eta]$ with $\eta>0$ sufficiently small such that $0, \Oc\notin [a,b]$. Then, there exists a constant $ C_{\eta}^{\gamma}>0$ (depending only on $\eta$ and $\gamma$) such that for all $ (\pm \kE(\lambda), \lambda), (\pm \kE(\lambda'),\lambda')\in\Lambda_{\EE}([a,b]) \mbox{ and } \lambda\leq \lambda'$:
	\begin{equation}\label{eq.ineqboundholdplasmassympt}
	\| \bbW_{\pm \kE(\lambda'),\lambda',0}  -\bbW_{\pm \kE(\lambda),\lambda,0}\|_{\Xpsstar}\leq C_{\eta}^{\gamma} \,    \sup_{\tilde{\lambda}\in [\lambda,\lambda']} |\tilde{\lambda}-\Op |^{1-\frac{3}{4} \gamma} \sup_{\tilde{\lambda}\in [\lambda,\lambda']} |\tilde{\lambda}-\Op|^{-\frac{3}{4} \gamma} \,  |\lambda'-\lambda|^{\gamma}.
	\end{equation}
\end{proposition}
\begin{proof}
Let $(k=\kE(\lambda), \lambda), (k'=\kE(\lambda'),\lambda')\in\Lambda_{\EE}([a,b]) \mbox{ and } \lambda\leq \lambda'$ and $\bU=(E,\bH,J,\bK)\in \Xps$, then one has
	$$
	\hspace{-0.075cm}	|\langle \bU, \bbW_{\pm k',\lambda,0} - \bbW_{\pm k,\lambda',0}\rangle_s|  \leq \sum_{\ell = 1}^3 \big| \int_{\bbR^2} \hspace{-0.15cm} \bU^\ell \: (\overline{\bbW_{\pm k',\lambda',0}^\ell} - \overline{\bbW_{\pm k,\lambda,0}^\ell} ) \mathrm{d}x  \mathrm{d}y  \big| + \sum_{\ell = 4}^6 \big| \int_{\bbR^2_+} \hspace{-0.15cm}\bU^\ell \: (\overline{\bbW_{\pm k',\lambda',0}^\ell} - \overline{\bbW_{\pm k,\lambda,0}^\ell} )  \mathrm{d}x \mathrm{d}y \big| .
$$
Thus, to establish the H\"{o}lder estimate on the vector-valued function of $\lambda$: $\bbW_{k,\lambda,0}$ , we derive such an estimate on all the components of $\bbW_{k,\lambda,0}$.\\[6pt]
	$(i)$ Estimate of the first term. Since $\bU^1 = E$ and $\bbW_{k,\lambda,0}= w_{k,\lambda,0}$, one has 
	$$
	\Big| \int_{\bbR^2} \bU^1 \, (\overline{\bbW_{\pm k',\lambda',0}^1} - \overline{\bbW_{\pm k,\lambda,0}^1} )  \mathrm{d}x \mathrm{d}y \Big| = \Big| \int_{\bbR^2} E \, (\overline{w_{\pm k',\lambda',0}} - \overline{w_{\pm k,\lambda,0}} )  \mathrm{d}x  \mathrm{d}y \Big|.
	$$
By integrating by parts in $y$ (as in the proof of Proposition \ref{prop.decayestimplasmon}), one obtains:
	\begin{equation}\label{eq.asymptholdpremcomp1}
	\Big| \int_{\bbR^2} \bU^1 \, (\overline{\bbW_{\pm k',\lambda',0}^1} - \overline{\bbW_{\pm k,\lambda,0}^1} )  \mathrm{d}x \mathrm{d}y \Big|= \left|\int_{\bbR^2} \partial_y E   \,\Big( \frac{\overline{ w_{\pm k',\lambda',0}}}{\kE(\lambda')}-\frac{ \overline{w_{\pm k,\lambda,0}}}{\kE(\lambda)} \Big)  \mathrm{d}x  \mathrm{d}y \right|.
	\end{equation}
	We follow again the method described  in section \ref{sec_GEForient} which consist to obtain a H\"{o}lder estimate via an interpolation between a $L^\infty $-estimate and a Lipschitz estimate given by the mean value theorem.
	First, one obtains from the definition of  \eqref{eq.def-w} of $w_{\pm k,\lambda,0}$ and the asymptotic behavior  \eqref{eq.estimasymp4} of $A_{k,\lambda,0}$ that:
	\begin{equation}\label{eq.asymwlkasympt}
	| w_{\pm k,\lambda,0}(x,y)|\lesssim     \psi_{k,\lambda,0}(x) \,  |\lambda-\Op |^{\frac{1}{4}} ,\   \forall (x,y) \in \bbR^2,
	\end{equation}
	and thus with  the estimate \eqref{eq.estimasymp1}, it follows immediately that:
	\begin{equation}\label{eq.thirdcompaasympt1-Linf}
		\Big|\frac{ w_{\pm k,\lambda,0}}{\kE(\lambda)}(x,y)\big|\lesssim    \psi_{k,\lambda,0}(x)  \, |\lambda-\Op |^{\frac{3}{4}} ,\   \forall (x,y) \in \bbR^2.
		\end{equation}
It yields the $L^\infty $-estimate: 
\begin{equation}\label{eq.Linf-asympt}
	\Big|\frac{ w_{\pm k,\lambda',0}}{\kE(\lambda')}(x,y)-\frac{ w_{\pm k,\lambda,0}}{\kE(\lambda)}(x,y) \big|\lesssim   \sup_{\tilde{\lambda} \in [\lambda, \lambda']}  \psi_{\kE(\tilde{\lambda}),\tilde{\lambda}}(x) \sup_{\tilde{\lambda} \in [\lambda, \lambda']} |\tilde{\lambda}-\Op |^{\frac{3}{4}}.
\end{equation}
Combining the  estimates  \eqref{eq.estimasymp1}, \eqref{eq.estimasymp2} (as $\mathcal{J}_\scE(\lambda)=|\kE'(\lambda)|$), \eqref{eq.asymwlkasympt} and \eqref{eq.derivfirstcomponentasympt} gives after simple computations:
\begin{equation}\label{eq.estimatederic--firsterm-asympt}
\Big|\partial_{\lambda}\Big(\frac{ w_{\pm k,\lambda,0}}{\kE(\lambda)}\Big)(x,y) \Big| \lesssim  (1+|x|+|y|) \   \psi_{k,\lambda,0}(x) \, |\lambda-\Op  |^{-\frac{3}{4}},\   \forall (x,y) \in \bbR^2.
\end{equation}
By applying the  mean value inequality with the estimate \eqref{eq.estimatederic--firsterm-asympt}, one gets:
	\begin{equation}\label{eq.Lipschiz-asympt}
	\Big|\frac{ w_{\pm k,\lambda',0}}{\kE(\lambda')}(x,y)-\frac{ w_{\pm k,\lambda,0}}{\kE(\lambda)}(x,y) \big| \lesssim  (1+|x|+|y|) \  \sup_{\tilde{\lambda} \in [\lambda, \lambda']} 
	\psi_{\kE(\tilde{\lambda}),\tilde{\lambda}}(x) \,  \sup_{\tilde{\lambda} \in [\lambda, \lambda']} |\tilde{\lambda}-\Op  |^{-\frac{3}{4}}\, |\lambda'-\lambda|. 
	\end{equation}
	We introduce the function $h_{\lambda,\lambda'}:(x,y)\mapsto  (1+|x|+|y|)^{\gamma} \, \sup_{\tilde{\lambda} \in [\lambda, \lambda']}  \psi_{\kE(\tilde{\lambda}),\tilde{\lambda}}(x)$. Interpolating between the inequalities \eqref{eq.Linf-asympt} and \eqref{eq.Lipschiz-asympt} yield:
	\begin{equation*}\label{eq.interpolasympt}
	\Big|\frac{ w_{\pm k',\lambda',0}}{\kE(\lambda')}(x,y)-\frac{ w_{\pm k,\lambda,0}}{\kE(\lambda)}(x,y) \Big| \lesssim h_{\lambda,\lambda'}(x,y)  \sup_{\tilde{\lambda}\in [\lambda,\lambda']} |\tilde{\lambda}-\Op|^{\frac{3}{4} (1-\gamma)} \sup_{\tilde{\lambda}\in [\lambda,\lambda']} |\tilde{\lambda}-\Op|^{-\frac{3}{4} \gamma}   |\lambda'-\lambda|^{\gamma},
	\end{equation*}
	for all  $(x,y)\in \bbR^2$.
It follows from \eqref{eq.asymptholdpremcomp1}  that
\begin{equation}\label{L-s-asympt--firstterm-Hold}
\Big\| \frac{ w_{\pm k',\lambda',0}}{\kE(\lambda')}-\frac{ w_{\pm k,\lambda,0}}{\kE(\lambda)}  \Big\|_{L^2_{-s}(\bbR^2)} \lesssim  \| h_{\lambda,\lambda'} \|_{L^2_{-s}(\mathbb{R}^2)}
\sup_{\tilde{\lambda}\in [\lambda,\lambda']} |\tilde{\lambda}-\Op|^{\frac{3}{4} (1-\gamma)}  \sup_{\tilde{\lambda}\in [\lambda,\lambda']} |\tilde{\lambda}-\Op |^{-\frac{3}{4} \gamma} \  |\lambda'-\lambda|^{\gamma}. 
\end{equation}
We point out that  $\sup_{\tilde{\lambda} \in [\lambda, \lambda']}  \psi_{\kE(\tilde{\lambda}),\tilde{\lambda}}$ is uniformly bounded by $1$, thus  $h_{\lambda,\lambda'} \in  L^2_{-s}(\mathbb{R}^2)$ for $\gamma \in (0,s-1/2)$.  However, we do not dominate this supremum by $1$ since one wants to benefit from the decay of $ \| h_{\lambda,\lambda'} \|_{L^2_{-s}(\mathbb{R}^2)}$ when $\lambda$ is close to $\Op$. \\[6pt]
\noindent We now bound $ \| h_{\lambda,\lambda'} \|_{L^2_{-s}(\mathbb{R}^2)}$. Clearly, it 
exists $\lambda_{\pm} \in  [\lambda, \lambda']$  such that 
	$\displaystyle \min_{\tilde{\lambda}\in [\lambda,\lambda']} \theta^{\pm}_{\kE(\tilde{\lambda}), \tilde{\lambda} }=\theta^{\pm}_{\kE(\lambda_{\pm}), \lambda_{\pm}}$. Thus, one has
$$
\sup_{\tilde{\lambda} \in [\lambda, \lambda']}  \psi_{\kE(\tilde{\lambda}),\tilde{\lambda}}(x)= \rme^{-\theta^{\pm}_{\kE(\lambda_{\pm}), \lambda_{\pm}} |x|}, \, \mbox{ for } \pm x\geq 0.$$
Hence,  using the inequality $(1+|x|+|y|) \leq (1+|x|) (1+|y|)$ for $x, y\in \bbR$ and the facts that $(1+|x|)^{2\gamma}/(1+|x|^2)^{s}\lesssim 1$ and $y\mapsto (1+|y|)^{\gamma}\in L^2_s(\bbR_y)$ for $\gamma<s-1/2$ leads to
	\begin{equation*}\label{eq.bornsupphi}
	\| h_{\lambda,\lambda'}\|_{L^2_{-s}(\mathbb{R}^2)} \lesssim \Big\| \sup_{\tilde{\lambda} \in [\lambda, \lambda']}  \psi_{\kE(\tilde{\lambda}),\tilde{\lambda}} \Big\|_{L^2_{(\mathbb{R}_x)}} =\frac{1}{\sqrt{2}} \,\big((\theta^{-}_{\kE(\lambda_-),\lambda_-})^{-1}+(\theta^{+}_{\kE(\lambda_+),\lambda_+})^{-1}\big)^{\frac{1}{2}}
	\end{equation*}
and thus it follows with \eqref{eq.estimasymp3} that:
\begin{equation}\label{eq.bornsupphi}
\| h_{\lambda,\lambda'}\|_{L^2_{-s}(\mathbb{R}^2)}\lesssim \sup_{\tilde{\lambda}\in [\lambda,\lambda']} |\tilde{\lambda}-\Op |^{\frac{1}{4} }.
\end{equation}
One concludes from \eqref{eq.asymptholdpremcomp1},  \eqref{L-s-asympt--firstterm-Hold} and  \eqref{eq.bornsupphi} that for $\gamma \in (0,1]\cap (0,s-1/2)$ and $(k, \lambda), (k',\lambda')\in\Lambda_{\EE}([a,b]) \mbox{ and } \lambda\leq \lambda'$:
	\begin{equation} \label{eq.firstcomponenthold}
 \hspace{-0.2cm}	\Big| \int_{\bbR^2} \bU^1 \, (\overline{\bbW_{\pm k,\lambda',0}^1} - \overline{\bbW_{\pm k,\lambda,0}^1} )  \mathrm{d}x  \mathrm{d}y \Big|\lesssim  \sup_{\tilde{\lambda}\in [\lambda,\lambda']} |\tilde{\lambda}-\Op|^{1-\frac{3}{4} \gamma} \hspace{-0.2cm}\sup_{\tilde{\lambda}\in [\lambda,\lambda']} |\tilde{\lambda}-\Op |^{-\frac{3}{4} \gamma} \,  |\lambda'-\lambda|^{\gamma} \,\|\partial_y E\|_{L^2_{s}(\bbR^2)} .
	\end{equation}
$(ii)$ Estimate on the second term. Integrating by part twice in $y$  to compensate the  additional factor $k$ in $\bbW_{\pm k,\lambda,0}^2 $ (see \eqref{eq.def-Vkbis})  gives
	$$
	\Big| \int_{\bbR^2} \bU^2 \, (\overline{\bbW_{\pm k',\lambda',0}^2} - \overline{\bbW_{\pm k,\lambda,0}^2} ) \, \mathrm{d}x  \mathrm{d}y \Big|=   \Big| \int_{\bbR^2} \partial_y^2 H_x  \Big( \frac{\overline{w_{\pm k',\lambda',0}}}{\mu_{\lambda'} \lambda'\, \kE(\lambda')} -  \frac{\overline{w_{\pm k,\lambda,0)}}}{\mu_{\lambda} \lambda\, \kE(\lambda)} \Big)   \, \mathrm{d}x  \mathrm{d}y\Big|.
	$$
	As the functions $\lambda \to 1/ (\mu_{\lambda}^{\pm}\lambda)$ are $C^{\infty}$ on $[a,b]$, the estimates \eqref{eq.Linf-asympt} and \eqref{eq.Lipschiz-asympt} still holds by replacing $w_{\pm k,\lambda,0}/ \kE(\lambda)$  by $w_{\pm k,\lambda,0}/(\mu_{\lambda}\, \lambda \kE(\lambda))$. Hence, one gets  with the same reasoning:
	\begin{equation}\label{eq.secondcomponenthold}
	\Big| \int_{\bbR^2} \bU^2 \, (\overline{\bbW_{\pm k',\lambda',0}^2} - \overline{\bbW_{\pm k,\lambda,0}^2} ) \, \mathrm{d}x  \mathrm{d}y\lesssim  \sup_{\tilde{\lambda}\in [\lambda,\lambda']} |\tilde{\lambda}-\Op |^{1-\frac{3}{4} \gamma} \sup_{\tilde{\lambda}\in [\lambda,\lambda']} |\tilde{\lambda}-\Op |^{-\frac{3}{4} \gamma} \,  |\lambda'-\lambda|^{\gamma}\, \|\partial_y^2 H_x \|_{L^2_{s}(\bbR^2)}.
	\end{equation}
$(iii)$ Estimate on the third term. As in the proof of Proposition \ref{prop.decayestimplasmon}, we need  to  integrate by part twice in $y$ to compensate the additional 
	factor $\theta_{k, \lambda}$ that appears in the expression of $\partial_x w_{\pm k,\lambda',0}$. It yields:
	$$
	\Big| \int_{\bbR^2} \bU^3 \, (\overline{\bbW_{k,\lambda,0}^3} - \overline{\bbW_{k',\lambda',2}^3} )  \mathrm{d}x  \mathrm{d}y \Big|= \Big|  \int_{\bbR^2}  \partial_y^2 H_y   \,  \Big( \frac{ \overline{\partial_x w_{\pm k',\lambda',0}}}{\mu_{\lambda'}\, \lambda' \kE(\lambda')^2} -  \frac{\overline{\partial_x w_{\pm k,\lambda,0)}}}{\mu_{\lambda}\, \lambda   \, \kE(\lambda)^2} \Big)  \, \mathrm{d}x \mathrm{d}y \Big|,
	$$
with $\partial_x w_{\pm k,\lambda,0}(x,y)=\operatorname{sgn}(x) \, \thetakl(x) w_{\kE(\lambda),\lambda,0}(x,y)$ for $x\neq 0$. \\[6pt]
On one hand, from the above formula, \eqref{eq.asymwlkasympt},  \eqref{eq.estimasymp1} and  \eqref{eq.estimasymp3}, one gets 
	\begin{equation}\label{eq.thirdcompaasympt3-Linf}
	\Big| \frac{\partial_x w_{\pm k,\lambda,0}(x,y)}{\mu_{\lambda}\, \lambda   \, \kE(\lambda)^2}\Big|\lesssim     \psi_{k,\lambda,0}(x)  \, |\lambda-\Op  |^{\frac{3}{4}} ,\   \forall (x,y) \in \bbR^*\times \bbR.
	\end{equation}
	On the other hand,  from \eqref{eq.thirdcompaasympt3-Linf}, \eqref{eq.asymwlkasympt},  \eqref{eq.estimasymp1}, \eqref{eq.estimasymp2}  and \eqref{eq.derivothercomponentasympt}, one obtains that
	\begin{eqnarray}\label{eq.thirdcompaasympt2-Lipsch}
		\Big| \partial_{\lambda} \Big(\frac{\partial_x w_{\pm k,\lambda,0}(x,y)}{\mu_{\lambda}\, \lambda   \, \kE(\lambda)^2} \Big)\Big|&\leq &\Big| \frac{\partial_{\lambda} \partial_x w_{\pm k,\lambda,0}(x,y)}{\mu_{\lambda}\, \lambda   \, \kE(\lambda)^2} \Big|  + \Big| \partial_x w_{\pm k,\lambda,0}(x,y) \partial_{\lambda}\Big(\frac{1}{\mu_{\lambda}\, \lambda   \, \kE(\lambda)^2}\Big)  \Big|  \nonumber \\& \lesssim &     (1+|x|+|y|)\, \psi_{k,\lambda,0}(x)  \,  |\lambda-\Op |^{-\frac{3}{4}} ,\   \forall (x,y) \in \bbR^2 .
	\end{eqnarray}
\eqref{eq.thirdcompaasympt3-Linf} and \eqref{eq.thirdcompaasympt2-Lipsch} are the equivalent of the estimates \eqref{eq.thirdcompaasympt1-Linf} and \eqref{eq.estimatederic--firsterm-asympt}  for the first component.	Thus following the same reasoning as for the first component  gives
	\begin{equation}\label{eq.thirdcomponenthold}
 \hspace{-0.2cm}\Big| \int_{\bbR^2} \bU^3 \, (\overline{\bbW_{k,\lambda,0}^3} - \overline{\bbW_{k',\lambda',2}^3} )  \mathrm{d}x  \mathrm{d}y \Big|	\lesssim  \sup_{\tilde{\lambda}\in [\lambda,\lambda']} |\tilde{\lambda}-\Op |^{1-\frac{3}{4} \gamma} \hspace{-0.2cm}\sup_{\tilde{\lambda}\in [\lambda,\lambda']} |\tilde{\lambda}-\Op|^{-\frac{3}{4} \gamma} \,  |\lambda'-\lambda|^{\gamma} \, \|\partial_y^2 H_y \|_{L^2_{s}(\bbR^2)}.
	\end{equation}
$(iv)$ As the three last components of $ \bbW_{k,\lambda,0}$ are only given (up to a multiplication by $C^{\infty}$ function  in $\lambda$ on $[a,b]$) by a restriction to $\bbR^+$ of the three first components, one easily deduced a similar estimates as the ones in \eqref{eq.firstcomponenthold},  \eqref{eq.secondcomponenthold} and \eqref{eq.thirdcomponenthold} for these components. Thus combining all these estimates and the definition of  the dual norm \eqref{def.dualnorm} leads to  \eqref{eq.ineqboundholdplasmassympt}.
\end{proof}

\subsection{Proof of the Theorem \ref{thm.regholdestimthr}  
}\label{sec.Hold-reg-spec-dens-thr}

\begin{figure}[h!]\label{fig.Cas-Asymptote}
	\begin{center}
		\includegraphics[width=0.75\textwidth]{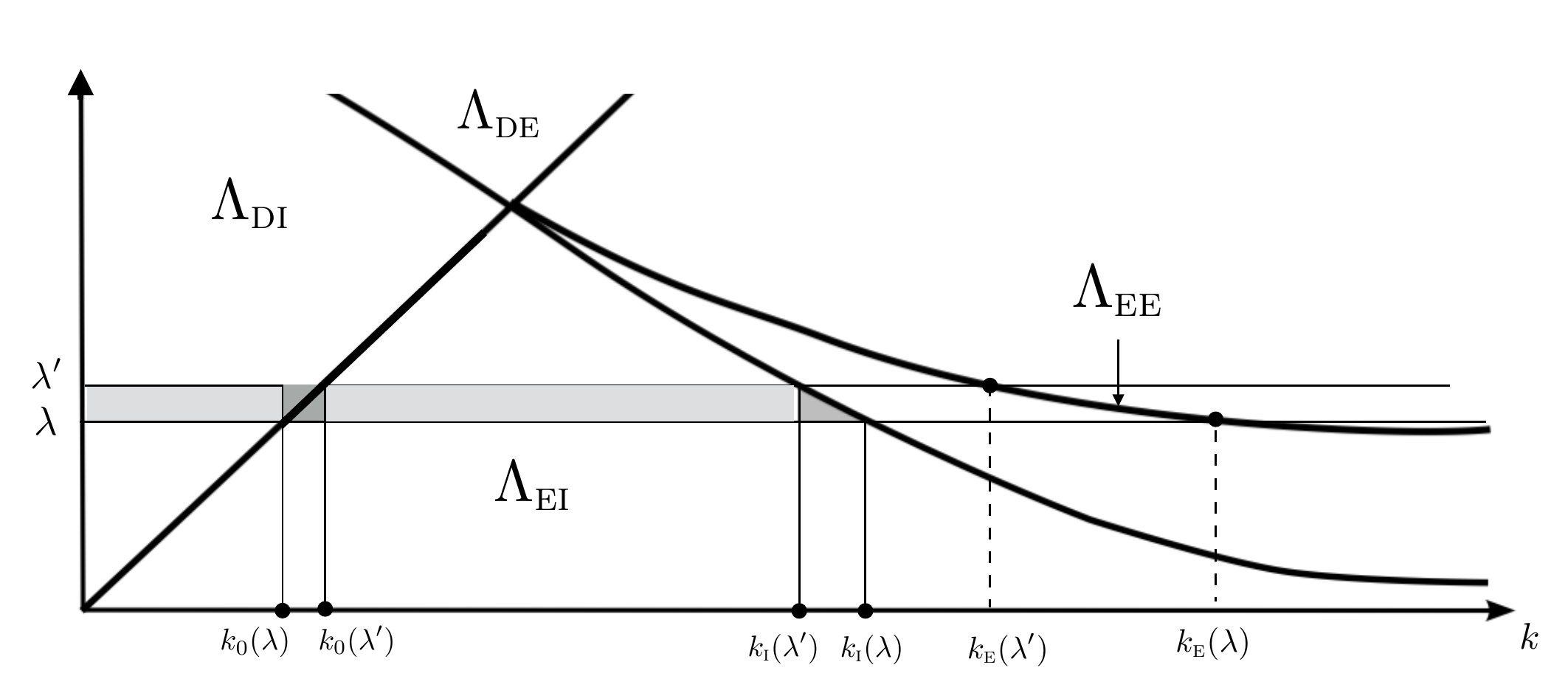}
	\end{center}
	\caption{Zoom for the case $[a,b]=[\Op-\eta,\Op]$ and $\Om< \Oe$.}
	\label{fig.asymptote}
\end{figure}

\begin{proof}
 Using the Theorem \ref{th.Holder-dens-spec},  it only remains to prove a local estimate at $ \Op$, that is on an interval  $K=[\Op-\eta, \Op+\eta]$ with $\eta$ sufficiently small such that $0, \Om, \Oe,\Oc \notin K$ and $\Lambda_{\DD}(K)=\emptyset$. 

As the change of topology is only due to the spectral zone $\Lambda_{\EE}$ at the vicinity of $ \Op$ for $\Oe\neq \Om$,  it requires only to obtain new estimate on the part of the spectral density  $\bbml^{\EE}$  associated with $\Lambda_{\EE}$ for $\Oe\neq \Om$.
For the other spectral zones $\calZ\setminus\{\EE\}$, we can show that the estimate \eqref{eq.holderestim-specdensity-Z} is satisfied   for  $[a,b]=[ \Op -\eta,  \Op+\eta]$ and   H\"{o}lder indexes  $\gamma\in (0,s-1/2) \cap (0,1)$ if $\Op\neq \Oe$ or $\gamma\in (0,s-1/2) \cap (0,1/2)$ if $\Oe=\Op$ (this latter case can occur only if $\Oe<\Om$).
The proof is exactly the same proof as the one of section \ref{sec.Holderspectraldensity}  performed for intervals $[a,b]\subset  \bbR\setminus \sigma_{\rm exc}$ that does not contain $\pm \Oe$  and $\pm \Oc$  (Case A) or does not contain $\pm \Oc$ for the particular case $\Oe=\Op $(Case B). It does not depend on the fact that  $\pm \Op$ is contained in the considered interval.

We consider   the case $\Oe>\Om$ (see Figure \ref{fig.asymptote}). Thus, one can choose $\eta$ sufficiently small such that $\Lambda_{\DE}(\lambda)=\Lambda_{\DE}(\lambda')=\varnothing$ and $\Oe\notin [a,b]$.
The  proof for $\Om>\Oe$ is obtained  in the same way.

For $\lambda, \lambda'\in K$, one has for all $\bU \in \Xps$:
	\begin{eqnarray*}\label{eq.asymptMlmabda}
	\| \bbmlp \bU  -\bbml \bU \|_{\Xpsstar} &  \leq & \sum_{\scZ \in \{\DI, \EI\}}  \| \bbM^{\scZ}_{\lambda'} \bU- \bbM^{\scZ}_{\lambda}\bU \|_{\Xpsstar} + \| \bbM^{\EE}_{\lambda'} \bU- \bbM^{\EE}_{\lambda} \bU\|_{\Xpsstar},
	\end{eqnarray*}
	Indeed,  in the proof of  Theorem \ref{th.Holder-dens-spec} (case A), the   H\"{o}lder  estimate in $B(\Hps, \Hms)$ of the terms related to the spectral zones $\Lambda_{\DI}$ and $\Lambda_{\EI}$  holds on $K$  (which implies the same estimate in  the norm $B(\Xps,\Xpsstar)$). Thus, we  obtain that for $0<\gamma<s-1/2$ and $\gamma <1$: 
	\begin{equation}\label{eq.densitespecasymptote}
	\| \bbmlp \bU  -\bbml \bU \|_{\Xpsstar}   \lesssim  | \lambda'- \lambda |^{\gamma} \, \| \bU\|_{\Xps}+\| \bbM^{\EE}_{\lambda'} \bU- \bbM^{\EE}_{\lambda} \bU\|_{\Xpsstar}.
	\end{equation}
	Hence, as explained at the beginning of the proof, one only needs to derive a H\"{o}lder   estimate for $\lambda\mapsto \bbml^{\EE}$ on $K$.
As  $\Lambda_{\EE}(\lambda)=\Lambda_{\EE}(\lambda')=\varnothing$ and thus $\bbM^{\lambda}=\bbM^{\lambda'}=0 $ for $\lambda,\lambda'\in [\Op-\eta,\Op]$, it is sufficient to prove  this estimate  for $\lambda,\lambda'\in (\Op,\Op+\eta]$ with $\lambda\leq \lambda'$. We point out that we already  show that  the limit $\bbM^{\EE}_{\lambda}$ when $\lambda \to\Op^+$ is zero in $B(\Xps,\Xpsstar)$ to ensure the continuity of $\bbM^{\EE}_{\lambda}$ at $\Op$, see section \ref{sec-sepc-density-thr}.

(i) a) Now, we show that $\bbM^{\EE}_{\lambda}$ is H\"{o}lder continuous with an index $0<\gamma<s-1/2$ and $\gamma\leq 1/3$ on the set $(\Op,\Op+\eta]$. First, one has from the expression \eqref{eq.density-MEE}  of $\bbM^{\EE}_{\lambda}$

\begin{equation}\label{eq.inequality-gen-Hol-ZEE-thr}
	\big\| \bbM_{\lambda'}^\EE \bU - \bbM_\lambda^\EE \bU \big\|_{\Xpsstar}
	\leq  \sum_\pm q_{\lambda,\lambda'}^\pm \ \|\bU\|_{\Xps},
\end{equation}
where we have denoted
\begin{multline}
	q_{\lambda,\lambda'}^\pm  := 
	\big|\JacE(\lambda') - \JacE(\lambda)\big| \ \big\| \bbW_{\pm \kE(\lambda),\lambda} \big\|_{\Xpsstar}^2 \\ + \big|\JacE(\lambda')\big| \
	\Big\{ \big\| \bbW_{\pm \kE(\lambda'),\lambda'} \big\|_{\Xpsstar} + \big\| \bbW_{\pm \kE(\lambda),\lambda} \big\|_{\Xpsstar} \Big\} \ \big\| \bbW_{\pm \kE(\lambda'),\lambda'} - \bbW_{\pm \kE(\lambda),\lambda} \big\|_{\Xpsstar}.
	\label{eq.def-d-lineic-thr}
\end{multline}
We deal with the first term of  \eqref{eq.def-d-lineic-thr}. Thanks to the asymptotic formula  \eqref{eq.estimasymp2} and \eqref{eq.estimasympsecondderiv}, that for $\eta$ sufficiently small and since $ \Op<\lambda\leq \lambda'$:
	$$
	|\mathcal{J}_{\EE}(\lambda')-\mathcal{J}_{\EE}(\lambda)| \lesssim ( \lambda-\Op)^{-\frac{3}{2}} \ \mbox { and } \  |\mathcal{J}_{\EE}(\lambda')-\mathcal{J}_{\EE}(\lambda)| \lesssim ( \lambda-\Op)^{-\frac{5}{2}} |\lambda'-\lambda|.
	$$
	Thus, by interpolation, it follows immediately that for $0\leq \gamma\leq 1$:
	$$
	|\mathcal{J}_{\EE}(\lambda')-\mathcal{J}_{\EE}(\lambda)|  \leq (\lambda-\Op)^{-\frac{3}{2}-\gamma}  |\lambda'-\lambda|^{\gamma}.
	$$
	It gives with  \eqref{eq.ineqboundplasmasympt} that for $\eta$  sufficiently small, $0\leq \gamma\leq 1/2$ and  $\Op<\lambda\leq \lambda'$
	\begin{equation}\label{eq.firsttremholdproofdensity}
	|\mathcal{J}_{\EE}(\lambda')-\mathcal{J}_{\EE}(\lambda)|   \, \| \bbW_{\pm \kE(\lambda),\lambda,0}  \|_{\Xpsstar}^2  \lesssim  (\lambda-\Op)^{1/2-\gamma}   |\lambda'-\lambda|^{\gamma} \lesssim  |\lambda'-\lambda|^{\gamma}  .
	\end{equation}
For the second term of the right hand side of  \eqref{eq.def-d-lineic-thr}, as  $\Op<\lambda\leq  \lambda'$, one has 
	$$
	\lambda-\Op  \leq \lambda'-\Op   \ \mbox{ and }    \ \sup_{\tilde{\lambda}\in [\lambda,\lambda']} \big| \tilde{\lambda}-\Op \big|^{1-\frac{3}{4}\gamma} \sup_{\tilde{\lambda}\in [\lambda,\lambda']} \big|\tilde{\lambda}-\Op \big|^{-\frac{3}{4} \gamma}=(\lambda'-\Op)^{1-\frac{3}{4}\gamma} \,  (\lambda-\Op)^{-\frac{3}{4} \gamma}
	$$
	and thus by virtue of \eqref{eq.ineqboundplasmasympt}, \eqref{eq.ineqboundholdplasmassympt} and \eqref{eq.estimasymp2}, it yields
	\begin{eqnarray}\label{eq.boundplasmon1asympt}
	& & \mathcal{J}_{\EE}(\lambda') \big( \| \bbW_{\pm k',\lambda',0}  \|_{\Xpsstar}  +  \| \bbW_{\pm k,\lambda,0}  \|_{\Xpsstar} ) \|  \bbW_{\pm k',\lambda',0}- \bbW_{\pm k,\lambda,0}\|_{\Xpsstar}\nonumber \\  
	&  &\lesssim  (\lambda'- \Op)^{-\frac{3}{2}} \,  (\lambda'-\Op) \,  (\lambda'-\Op )^{1-\frac{3}{4}\gamma} \,  (\lambda-\Op)^{-\frac{3}{4} \gamma} |\lambda'-\lambda|^{\gamma} \nonumber\\
	& &\lesssim  r_{\lambda, \lambda'}   |\lambda'-\lambda|^{\gamma} ,
	\end{eqnarray} 
with 
$$
 r_{\lambda, \lambda'} :=(\lambda'-\Op)^{\frac{1}{2}-\frac{3}{4}\gamma}  (\lambda-\Op)^{-\frac{3}{4} \gamma}.
$$	
Here, it remains to show that  $r_{\lambda, \lambda'} $ is bounded. The problem (which is similar to the one of section \ref{sss.Holder-M-lineic} for the case (C)) is that $r_{\lambda, \lambda'} $ is not a continuous function on the line $\lambda= \Op$. However, we see that for any fixed $\kappa \in (0,1)$, it is continuous, thus bounded, in any domain of the form 
\begin{equation*}
	\tilde{D}_\kappa := \big\{ (\lambda, \lambda') \in (\Op, \Op+\eta]^2\; \big| \; \kappa \; ( \lambda'-\Op) <  \lambda- \Op  \leq  \lambda'-\Op \big\} 
\end{equation*}
since $(1/2-3\gamma/4)-3\gamma/4 = 1/2 - 3/2\gamma \geq 0$ for $\gamma\in (0,1/3]$. Combining  \eqref{eq.firsttremholdproofdensity}, \eqref{eq.boundplasmon1asympt}, \eqref{eq.def-d-lineic-thr}  and \eqref{eq.inequality-gen-Hol-ZEE-thr} yields
that for any $\gamma\in (0,s-1/2)\cap (0,1/3]$ and $\lambda, \lambda'\in  \tilde{D}_\kappa $, one has 
\begin{equation}\label{eq.specdensity-Hold-EE-thr}
\big\| \bbM_{\lambda'}^\EE \bU - \bbM_\lambda^\EE \bU \big\|_{\Xps,\Xpsstar} \lesssim | \lambda'-\lambda|^{\gamma}.
\end{equation}

(i) b) To conclude, we have to prove that \eqref{eq.specdensity-Hold-EE-thr} holds true in the complement of $\tilde{D}_\kappa$, that is, 
\begin{equation*}
	\tilde{D}_\kappa^{\rm c} := \big\{ (\lambda, \lambda') \in (\Op, \Op+\eta]^2\; \big| \; \kappa \; (\lambda' - \Op) \geq \lambda-\Op\big\}.
\end{equation*}
The idea is to use directly \eqref{eq.densitespecasymptote} (and not \eqref{eq.boundplasmon1asympt}) for $\eta$ sufficiently small  with the following inequality
$$
\|\bbM^{\EE}_{\lambda'} - \bbM^{\EE}_{\lambda} \|_{\Xps,\Xpsstar} \lesssim  \|\bbM^{\EE}_{\lambda'} \|_{\Xps,\Xpsstar} + \|\bbM^{\EE}_{\lambda} \|_{\Xps,\Xpsstar} \lesssim (\lambda'-\Op)^{\frac{1}{2}}  +(\lambda-\Op)^{\frac{1}{2}}  
$$
and to notice that the inequality $\kappa \, (\lambda'- \Op) \geq  \lambda-\Op$ which characterizes points of $\tilde{D}_\kappa^{\rm c}$ can be written equivalently as
\begin{equation*}
	\lambda'-\Op \leq (1-\kappa)^{-1}\ (\lambda'-\lambda)
	\quad\text{or}\quad
	 \lambda-\Op \leq (\kappa^{-1}-1)^{-1}\ (\lambda'-\lambda).
\end{equation*}
Thus, we finally obtain for any $\gamma\in (0,s-1/2)\cap (0,1/3]$:
\begin{equation*}
	\|\bbM^{\EE}_{\lambda'} - \bbM^{\EE}_{\lambda} \|_{\Xps,\Xpsstar} \lesssim |\lambda'-\lambda|^{1/2}\lesssim |\lambda'-\lambda|^{\gamma}, \,   \forall (\lambda,\lambda') \in D_\kappa^{\rm c}.
\end{equation*}
 Combining this inequality, \eqref{eq.specdensity-Hold-EE-thr} and  \eqref{eq.densitespecasymptote} completes the proof of \eqref{eq.specdensholdthreshold}.
\end{proof}

\appendix
\section{Appendix}\label{sec.appendix}

\subsection{Asymptotic formulas at the vicinity of $\pm \Op$}
\label{sec.appendixA1}

We prove here  the Lemma \ref{lem.asymptestimates} that gives some asymptotic expansions which are useful to establish a H\"{o}lder estimate of the spectral density at $\pm \Op$.

\begin{proof}
To prove the asymptotic expansion \eqref{eq.estimasymp1}, one first establishes the asymptotic expansion of $\lambda_\scE(k)$ when $k\to+ \infty$.
From the expression \eqref{eq.expressionlambdae} of $\lambda_\scE$, one obtains after some simple computations:
\begin{equation}\label{eq.asymplambdaE}
\lambda_\scE(k)=\Op-\frac{K \, \Op}{8} \ k^{-2}+ O(k^{-4}), \ \mbox{ as } k \to \infty.
\end{equation}
We deduce  from \eqref{eq.asymplambdaE} an asymptotic expansion of $k=\kE(\lambda)$, where $\kE=\lambda_{\scE}^{-1}$ is defined by \eqref{defkE}. Using  the equivalence between the limits $k=\kE(\lambda)\to +\infty$ and $\lambda\to \Op^{\pm}$ for $\pm K>0$,  one gets from \eqref{eq.asymplambdaE}: 
$$
\big|\lambda-\Op \big|=8^{-1} |K|  \Op \ \kE(\lambda)^{-2}+ O(\kE(\lambda)^{-4}),  \ \mbox{ as } \lambda \to \Op^{\pm} \ \mbox{ for } \ \pm K>0,
$$
and it yields
$$
\kE(\lambda)=\Big(\frac{ |K| \, \Op}{8} \Big)^{\frac{1}{2}}\, \big |\lambda-\Op \big|^{-\frac{1}{2}} \left(1+O(\kE(\lambda)^{-2}) \right) .
$$
One concludes finally to \eqref{eq.estimasymp1} by using the fact that $\kE(\lambda) \to \infty $ for $\lambda \to \Op^{\pm}$ for $\pm K>0$. 

$\bullet$  To compute the asymptotic  expansion  \eqref{eq.estimasymp2} of $\JacE(\lambda)=| \lambda'_{\scE} (\kE(\lambda))|^{-1}$, one first computes $\lambda_{\scE}'$ thanks to \eqref{eq.expressionlambdae}. After some simplifications, one gets that:
\begin{equation}\label{eq.firstderiv}
 \lambda_\scE'(k)=\frac{g(k)}{ \lambda_{\scE}(k)}   \mbox{ with } \ g(k):=\Om^2  \Big( \frac{k}{K}-  \frac{k}{K} \Big(1+\frac{K^2}{4k^4}\Big)^{-\frac{1}{2}} \Big).
\end{equation}
As  
\begin{equation}\label{eq.dlg}
g(k)= \frac{\Om^2}{8 k^3} \left(K+O(k^{-4})\right)  \ \mbox{ as } k\to +\infty,
\end{equation}
it follows that:
\begin{equation}\label{eq.dlfirstderiv}
 \lambda_\scE'(k)=\frac{\Om^2}{8 \lambda_{\scE}(k) k^3} \left(K+O(k^{-4})\right),  \ \mbox{ as } k\to +\infty.
\end{equation}
Using  again that for $k=k_{\scE}(\lambda)$, the limit $k\to +\infty$ is equivalent to the limit $\lambda=\lambda_{\scE}(k) \to \Op^{\pm}$ for $\pm K>0$, one deduces an asymptotic in $\lambda$ for $k=k_{\scE}(\lambda)$. Thanks to \eqref{eq.estimasymp1} and the definition \eqref{eq.def-Op-Oc} of $\Op$, one obtains:
$$
 \lambda_\scE'(\kE(\lambda))=\frac{\Op}{4\, \kE(\lambda)^3}  (K+o(1))=2 \Big(\frac{\Op |K|}{8}\Big)^{-\frac{1}{2}} \big |\lambda-\Op \big|^{\frac{3}{2}} \big(\operatorname{sgn}(K)+o(1)\big).
$$
One arrives finally to  the asymptotic expansion \eqref{eq.estimasymp2} by using that  $\JacE(\lambda)=| \lambda'_{\scE} (\kE(\lambda))|^{-1}$.

$\bullet$  To show \eqref{eq.estimasympsecondderiv}, one uses the relation between the derivatives of $\kE$ and its inverse $\lambda_{\scE}$ which gives
\begin{equation}\label{eq.derivsecondKE}
|\kE''(\lambda)|=  \big|\lambda_{\scE}''(\kE(\lambda)) \big| \, \mathcal{J}_{\scE}(\lambda)^3.
\end{equation}
Thus, as one knows the asymptotic behaviour of $\mathcal{J}_{\scE}$, it only remains to compute the asymptotic of  the second derivative of $\lambda_{\scE}$.
To this aim, differentiating with respect to $k$ the relation \eqref{eq.firstderiv}, one gets that:
\begin{equation}\label{eq.formulderivsecondlambdaE}
\lambda_{\scE}''(k)=\big(g'(k) \lambda_{E}(k)-\lambda_{\scE}'(k) g(k)\big) \,\lambda_{\scE}(k)^{-2}.
\end{equation}
By differentiating $g$ and computing its asymptotic at $+ \infty$,  one shows after simplifications that:
\begin{equation}\label{eq.dlerivg}
g'(k)=\Om^2 \, \Big(-\frac{3\, K }{8k^4}+o\Big(\frac{1}{k^4} \Big) \Big), \mbox{ as } k \to  +\infty .
\end{equation}
Using the asymptotic expansions \eqref{eq.asymplambdaE},  \eqref{eq.dlg},  \eqref{eq.dlfirstderiv} and  \eqref{eq.dlerivg} in \eqref{eq.formulderivsecondlambdaE} gives:
\begin{equation*}
\lambda_{\scE}''(k)=-\frac{3K\Op}{4\, k^4}\Big(1+o(1)\Big) \mbox{ as } k \to  +\infty .
\end{equation*}
Applying the asymptotic expansion \eqref{eq.estimasymp1} in the latter expression for $k=\kE(\lambda)$ yields
\begin{equation}\label{eq.seconddlerivg}
\big|\lambda_{\scE}''\big(\kE(\lambda)\big)\big|= 6 \Big(\frac{\Op \, |K|}{8}\Big)^{-1}  \big| \lambda-\Op\big|^{2} (1+o(1)).
\end{equation}
Thus, using the asymptotic formula  \eqref{eq.estimasymp2} and \eqref{eq.seconddlerivg} in  \eqref{eq.derivsecondKE} leads after simplifications to \eqref{eq.estimasympsecondderiv}.

$\bullet$ For \eqref{eq.estimasymp3}, one deals first with the expansion of  $\theta_{\lambda,\kE(\lambda)}^{-}$. By \eqref{eq.def-thetam} and \eqref{eq.defThetalkfunction}, one has 
$$
\theta_{\lambda,\kE(\lambda)}^{-}=\sqrt{\kE(\lambda)^2-\varepsilon_0 \mu_0 \lambda^2}=|\kE(\lambda)| \, \sqrt{1-\varepsilon_0 \mu_0 \big(\lambda\,\kE(\lambda)^{-1}\big)^2}.
$$
Thus, it yields immediately with  \eqref{eq.estimasymp1} the asymptotic formula \eqref{eq.estimasymp3}. One deduces the expansion of $\theta_{\lambda,\kE(\lambda)}^{+}$ from the one of $\theta_{\lambda,\kE(\lambda)}^{-}$ by using the dispersion relation \eqref{eq.disp} which gives:
$$
\theta_{\lambda,\kE(\lambda)}^{+}=-\frac{ \mu^+(\lambda)}{ \mu_0} \theta_{\lambda,\kE(\lambda)}^{-}= (1+o(1)) \, \theta_{\lambda,\kE(\lambda)}^{-} \, \mbox{ as } \lambda \to \Op^{\pm} \ \mbox{ for } \ \pm K>0.
$$

$\bullet$ Concerning the asymptotic formula \eqref{eq.estimasymp4}, one deduces from  the definition \eqref{def-A-plasm} of $A_{\kE(\lambda),\lambda,0}$ that :
$$
A_{\kE(\lambda),\lambda,0}=\frac{\lambda^2\, |\mu_{\lambda}^{+} \, \theta_{\lambda,\kE(\lambda)}^{+}|^{1/2}}{\sqrt{2\pi}\,\Om (4\kE(\lambda)^4 +(\eps_0\mu_0)^2(\Oe^2-\Om^2)^2)^{1/4}} =\frac{\Op \mu_{0}^{\frac{1}{2}}}{2 \sqrt{2 \pi}}\,\frac{ | \theta_{\lambda,\kE(\lambda)}^{+}|^{1/2}}{ \kE(\lambda)}  (1+o(1)),
$$
as $ \lambda \to \Op^{\pm} \ \mbox{ for } \ \pm K>0$.
Combining the latter expression, \eqref{eq.estimasymp1} and  \eqref{eq.estimasymp3} gives immediately \eqref{eq.estimasymp4}.
\end{proof}

\subsection{The limiting absorption principle near the frequency $\pm \Oe$}
\label{app.cas-un-demi}
In this paragraph, we indicate what changes have to be made in order to adapt the proof of Theorem \ref{thm.limabs} for the particular value $\gamma = 1/2$ of the H\"older exponent in a vicinity of $+\Oe$ or $-\Oe$ (see Remark \ref{rem.cas-un-demi}). Indeed the approach we propose in \S\ref{sec.spectral-density} to prove the H\"older regularity of the spectral density (Theorem \ref{th.Holder-dens-spec}) is not valid in this particular case. This is actually due to our choice of introducing the function $\thetakl^{\min}$ (see \eqref{eq.defthetamin}) which yields more concise but less precise estimates. To deal with this particular case, we have to distinguish between $\thetakl^{+}$ and $\thetakl^{-}$ in the H\"older estimates.

We assume again for simplicity that $\Oe > \Om$ and we consider an interval $[a,b]$ of the form $[\Oe,b]$ or $[a,\Oe]$ contained in $\bbR \setminus (\sigma_{\rm exc}\cup\{\Oc\})$. We see in Figure \ref{fig.Cas-Omega_e} that the two spectral zones which are concerned are $\Lambda_\DD$ and $\Lambda_\DE$. Only the latter is problematic. Indeed, the core of the problem is that for $\gamma = 1/2$, Lemma \ref{l.int-theta1} is wrong for $\scZ = \DE$ (whereas it holds true for $\scZ = \DD$). To remedy, we have to come back to the very first estimates where we have used $\thetakl^{\min}$. This started with the $\lambda$-derivative of $A_{k, \lambda,\pm 1}$
given in \eqref{eq.deriv-coefA}. Here, $J_\DE = \{+1\}$ so that instead of \eqref{eq.derivA-1.improved}, we obtain the more precise estimate
\begin{equation*}
\forall (k, \lambda)\in\Lambda_{\DE}([a,b]), \quad
|\partial_{\lambda} A_{k, \lambda,1}| \lesssim |\theta_{k\lambda }^{-}|^{-3/2}+ |\thetakl^{+}|^{-1},
\end{equation*}
which can be then used in the proof of Lemma \ref{lem_estimderpsiB}. Instead of \eqref{eq.estimate-deriv-lambda-deriv-psiA} and 
\eqref{eq.deriv-w-surfac-inter-crosspointsA-improved}, we infer that
\begin{align*}
\forall (x,y)\in \bbR^2, \quad 	\big|\partial_{\lambda} w_{k,\lambda,1}(x,y)\big| & \lesssim \big(|\thetakl^{-}|^{-3/2}+|\thetakl^{+}|^{-1}\big)\, (1+|x|) , \\
\forall (x,y) \in \bbR^* \times \bbR, \quad \big| \partial_{\lambda} \partial_x w_{k,\lambda,1}(x,y)\big| &   \lesssim \big(|\thetakl^{-}|^{-3/2}+|\thetakl^{+}|^{-1}\big)\, (1+|x|)  .
\end{align*}
Proceeding as in the proof of Proposition \ref{prop.holdreggeneralizedeigenfunctions}, we deduce that for all $(k, \lambda), (k,\lambda')\in\Lambda_{\DE}([a,b])$ such that $\lambda\leq \lambda'$,
\begin{equation*}
\big\| \bbW_{k,\lambda',1}- \bbW_{k,\lambda,1}\big\|_{\Hms}\lesssim  \left(\sup_{\tilde{\lambda}\in [\lambda,\lambda']}  |\thetaklt^{-}|^{-1/2- \gamma}+\sup_{\tilde{\lambda}\in [\lambda,\lambda']}  |\thetaklt^{+}|^{-\gamma} \sup_{\tilde{\lambda}\in [\lambda,\lambda']}  |\thetaklt^{-}|^{-1/2+\gamma/2}\right) \ \,  |\lambda'-\lambda|^{\gamma},
\end{equation*}
which is more precise than \eqref{eq.ineqboundhold1nocrosspoint}. As we are only interested in the case where $\gamma = 1/2,$ from now on we choose this particular value. We move to \S\ref{sss.Holder-M-surface}, where the definition \eqref{eq.def-d}  of $d_{\lambda,\lambda',j}$  and \eqref{eq.ineqbound1} (for $\gamma=1/2$) yield now
\begin{equation*}
d_{\lambda,\lambda',1} \lesssim  \left(\sup_{\tilde{\lambda}\in [\lambda,\lambda']}  |\thetaklt^{-}|^{-1}+\sup_{\tilde{\lambda}\in [\lambda,\lambda']}  |\thetaklt^{+}|^{-1/2} \sup_{\tilde{\lambda}\in [\lambda,\lambda']}  |\thetaklt^{-}|^{-3/4}\right) \ \,  |\lambda'-\lambda|^{1/2}.
\end{equation*}
Hence the relation \eqref{eq.holderestim-common-part} for $\scZ = \DE$ will be proved once we have verified that the integral on $\zDE(\lambda) \cap \zDE(\lambda')$ of the quantity inside the parentheses is bounded. As $\thetaklt^{+}$ and $\thetaklt^{-}$ cannot vanish simultaneously, this amounts to proving separately
\begin{equation*}
\int_{\zDE(\lambda) \cap \zDE(\lambda')} \ \sup_{\tilde{\lambda}\in [\lambda,\lambda']}  \big|\thetaklt^{+}\big|^{-1/2} \ \rmd k \lesssim 1
\quad\text{and}\quad
\int_{\zZ(\lambda) \cap \zDE(\lambda')} \ \sup_{\tilde{\lambda}\in [\lambda,\lambda']}  \big|\thetaklt^{-}\big|^{-1} \ \rmd k \lesssim 1,
\end{equation*}
instead of Lemma \ref{l.int-theta1}. In the present case, the statement of this lemma is equivalent to \eqref{eq.two-estim-d}  with exponent $\beta = -1/2$ for $\big|\thetaklt^{+}\big|$ and $\beta=-1$ for $\big|\thetaklt^{-}\big|$, that is, replacing in the first integral the exponent $-1$ (for which the inequality is false) by $-1/2$ (for which the inequality is true). Finally it is easy to see that on one hand we can reuse exactly the arguments of the proof of Lemma \ref{l.int-theta1} for case (B) to prove the above inequalities and  on the other hand that the proof of Lemma \ref{l.int-theta2} for case (B) holds for $\gamma=1/2$.



\begin{thebibliography}{00}



\bibitem{Bon-14} 
A.-S. Bonnet-Ben Dhia, L. Chesnel and P. Ciarlet Jr., Two-dimensional Maxwell's equations with sign-changing coefficients, Appl. Numer. Math., 79 (2014), 29--41.

\bibitem{Bon-14(2)} 
A.-S. Bonnet-Ben Dhia, L. Chesnel and P. Ciarlet Jr., T-coercivity for the Maxwell problem with sign-changing coefficients, Commun. Part. Diff. Eq., 39 (2014), 1007--1031.

\bibitem{Bre-10}
H. Brezis, Functional analysis, Sobolev spaces and partial differential equations., Springer Science \& Business Media, New York,  2010.

\bibitem{Bru-14}
S. Br\^ul\'e, E. H. Javelaud, S. Enoch and S. Guenneau, Experiments on seismic metamaterials: molding surface
waves, Phys. Rev. Lett., 112 (13) (2014), 133901.

\bibitem{Cos-85} 
M. Costabel and E. Stephan, A direct boundary integral equation method for transmission problems, J. Math. Anal. Appl., 106 (2) (1985), 367--413.

\bibitem{Cas-14} 
M. Cassier, \'{E}tude de deux probl\`{e}mes de propagation d'ondes transitoires. 1: Focalisation spatio-temporelle en acoustique. 2: Transmission entre
un di\'{e}lectrique et un m\'{e}tamat\'eriau (in French). Ph.D. thesis, \'{E}cole Polytechnique, 2014, available online at https://pastel.archives-ouvertes.fr/pastel-01023289.

\bibitem{cas-mil-16} M. Cassier and G. W. Milton, Bounds on Herglotz functions and fundamental limits of broadband passive quasistatic cloaking,
J. Math. Phys,  58 (7) (2017), 071504.


\bibitem {cas-kach-jol-17}M. Cassier, P. Joly and M. Kachanovska, {\it Mathematical
models for dispersive electromagnetic waves: an overview}, Computers \&
Mathematics with Applications 74 (11) (2017), pp. 2792--2830.


\bibitem{Cas-Haz-Jol-17}  M. Cassier, C. Hazard and P. Joly, Spectral theory for Maxwell's equations at the interface of a metamaterial. Part I: Generalized Fourier transform., Commun. Part. Diff.  Eq.  42 (11) (2017), 1707--1748.


\bibitem{Cum-16}
 S. A. Cummer, J. Christensen and  A. Alu, Controlling sound with acoustic metamaterials, Nat. Rev. Mater., 1 (3) (2016), 16001.

\bibitem{Der-83} 
Y. Dermenjian and J.C. Guillot, Th\'{e}orie spectrale de la propagation des ondes acoustiques dans un milieu stratifi\'{e} perturb\'{e}, Pr\'{e}publications du d\'{e}partement de math\'{e}matiques de l'Universit\'{e} Paris Nord, 44, 1983.

\bibitem{Der-86} 
Y. Dermenjian and J.C. Guillot, Th\'{e}orie spectrale de la propagation des ondes acoustiques dans un milieu stratifi\'{e} perturb\'{e}, J. Diff. Equat.  62 (3) (1986), 357--409.



\bibitem{Der-88}  
Y. Dermenjian and J. C. Guillot, Scattering of elastic waves in a perturbed isotropic
half space with a free boundary. The limiting absorption principle. Math. Methods Appl. Sci. 10 (2) (1988), 87--124.

\bibitem{Eid-65}D. M. Eidus, The principle of limiting absorption, Amer. Math. Soc. Transl., 47 (2) (1965), 157--191. 

\bibitem{Eid-69}D. M. Eidus, The principle of limit amplitude, Russ. Math. Surv., 24 (3) (1969), 97--167.  

\bibitem{Fig-05}
A. Figotin and J. H. Schenker, {\it Spectral theory of time dispersive and dissipative
  systems, J. Stat. Phys.}, 118 (1)  (2005), 199--263.
  
   

\bibitem{Gra-10} B. Gralak and A. Tip, Macroscopic Maxwell's equations and negative index materials, J. Math. Phys., 51 (2010), 052902.  

\bibitem{Gra-12}B. Gralak and D. Maystre, Negative index materials and time-harmonic electromagnetic field,  C.R. Physique, 13 (8) (2012), 786--799.  


 

\bibitem {Haz-07}
C. Hazard and F. Loret, Generalized eigenfunction expansions for conservative scattering problems
with an application to water waves. Proc. R. Soc. Edinb.: Section
A Mathematics, 137 (5) (2007), 995--1035. 

\bibitem{Hil-96}
E. Hille and R. S. Phillips. Functional analysis and semi-groups (Vol. 31). American Mathematical Soc., 1996.

\bibitem{Hen-86}
P. Henrici, Applied and computational complex analysis, vol. 3,  Wiley, New York, 1986.


\bibitem{Koj-91}
Koji Kikuchi and Hideo Tamura, Limiting amplitude principle for acoustic propagators in perturbed stratified fluids, J. Differ. Equations, 93 (2) (1991), 260--282.




\bibitem {Loh-09}
P. R. Loh,  A. F. Oskooi, M. Ibanescu, M. Skorobogatiy and S. G. Johnson,
Fundamental relation between phase and group velocity, and application to the failure of perfectly matched layers in backward-wave structures, Phys. Rev. E, 79 (6) (2009), 065601.

\bibitem{Li-04} J. Li and C.T. Chan, Double-negative acoustic metamaterial.,  Phys. Rev. E., 70 (5) (2004), 055602.




\bibitem{Mai-07}S. A. Maier, Plasmonics: fundamentals and applications, Springer, New York, 2007. 


\bibitem{McL-88}
W. McLean, H\"{o}lder estimates for the Cauchy integral on a Lipschitz contour., J. Integral Equ. Appl., 1 (3) (1988), 435--451.

\bibitem{Mora-62} C. S. Morawetz, The limiting amplitude principle, Comm. Pure Appl. Math., 15 (3) (1962), 349--361.



\bibitem{Mor-89}
K. Morgenr\"{o}ther and P. Werner, On the principles of limiting absorption and limit amplitude for a class of locally perturbed waveguides. Part 2: Time-dependent theory, Math. Method Appl. Sci., 11 (1) (1989), 1--25.




\bibitem{Ngu-16}
H.-M. Nguyen, Limiting absorption principle and well-posedness for the Helmholtz equation with sign changing coefficients, J. Math. Pures Appl., 106 (2016), 342--374.

\bibitem{Nir-94}N. A. Nicorovici, R. C. McPhedran and G. W. Milton, Optical and dielectric properties of partially resonant composites, Phys. Rev. B, 49 (12) (1994), 8479--8482.



\bibitem{Rad-15}
M. Radosz, New limiting absorption and limit amplitude principles for periodic operators, Z. Angew. Math. Phys., 66 (2) (2015), 253--275. 


\bibitem{Staff-02} E. B. Saff and A. D. Snider, Fundamentals of Complex Analysis with Applications to Engineering, Science and Mathematics, third edition, Pearson Modern Classic,   2002.


\bibitem {San-89} J. Sanchez Hubert and E. Sanchez Palencia, Vibration and Coupling of Continuous Systems, Asymptotic Methods, Springer-Verlag, Berlin, 1989. 
\bibitem{Soc-16}N. Popoff and E. Soccorsi
Limiting absorption principle for the magnetic Dirichlet Laplacian in a half-plane, Communications in Partial Differential Equations,  41 (6) (2016),  879--893.


\bibitem{Smit-04}
D. R. Smith, John B. Pendry and M. C. K. Wiltshire, Metamaterials and negative refractive index, Science, 305 (5685) (2004), 788--792.


\bibitem{Tip-98} A. Tip,  {\it Linear absorptive dielectrics}. Phys. Rev. A, 57 (6) (1998), 4818.



\bibitem{Vul-87} M. Vullierme-Ledard, The limiting amplitude principle applied to the motion of floating bodies, Math. Model. Numer. Analysis, 21 (1987), 125--170.


\bibitem{Wed-98} R. Weder and  P. Wener, The behavior of real resonances under perturbation in a
              semi-strip, Math. Methods Appl. Sci., 21 (1) (1998), 1--24.


\bibitem {Wed-91} R. Weder, Spectral and scattering theory for wave propagation in perturbed stratified media, Applied mathematical sciences 87, Springer-Verlag, New York, 1991.

\bibitem {Wed-84} R. Weder, Spectral analysis of strongly propagative systems, J. Reine Angew. Math, 354 (1984), 95--122.


\bibitem{Wer-87}
P. Werner, Resonance phenomena in cylindrical waveguides, J. Math. Anal. Appl., 121 (1) (1987), 173--214.

\bibitem{Wer-96}
P. Werner, Resonance phenomena in local perturbations of parallel-plane waveguides., Math. Method Appl. Sci., 19 (10) (1996), 773--823.



\bibitem {Wil-84} C. H. Wilcox, Sound propagation in stratified fluids, Applied mathematical sciences 50, Springer-Verlag, New York, 1984.


\bibitem{Zio-01}
R. W. Ziolkowski et E. Heyman, Wave propagation in media having negative permittivity
and permeability, Phys. Rev. E, 64 (5) (2001), 056625.


\end{thebibliography}
\end{document}